\documentclass[a4paper,twoside]{article}

\usepackage{amssymb}
\usepackage{amsmath}
\usepackage{amsthm}

\usepackage{mathrsfs}

\usepackage[T1]{fontenc}
\usepackage{lmodern}

\usepackage{natbib}
\usepackage{url}

\usepackage{microtype}

\usepackage{esint}                % load after the AMS packages

\newtheoremstyle{citing}% name
  {3pt}%      Space above, empty = `usual value'
  {3pt}%      Space below
  {\itshape}% Body font
  {}%         Indent amount (empty = no indent, \parindent = para indent)
  {\bfseries}% Thm head font
  {.}%        Punctuation after thm head
  {.5em}%     Space after thm head: " " = normal interword space;
  {\thmnote{#3}}% Thm head spec
\theoremstyle{citing}
\newtheorem*{citing}{}

\theoremstyle{definition}

\swapnumbers
\theoremstyle{plain}
\newtheorem{theorem}{Theorem}[section]
\newtheorem{lemma}[theorem]{Lemma}
\newtheorem{corollary}[theorem]{Corollary}

\theoremstyle{remark}
\newtheorem{remark}[theorem]{Remark}
\newtheorem{example}[theorem]{Example}

\theoremstyle{definition}
\newtheorem*{intro_definition}{Definition}
\newtheorem{definition}[theorem]{Definition}
\newtheorem{miniremark}[theorem]{}

\newcounter{counter1}
\newcounter{counter2}

\newcommand{\Var}{\mathbf{V}}     % general varifolds
\newcommand{\RVar}{\mathbf{RV}}   % rectifiable varifolds
\newcommand{\IVar}{\mathbf{IV}}   % integral varifolds

\newcommand{\Lp}[1]{\mathbf{L}_{#1}}

\newcommand{\Lploc}[1]{\mathbf{L}_{#1}^{\mathrm{loc}}}

\newcommand{\trunc}{\mathbf{T}}

\newcommand{\nat}{\mathscr{P}}

\newcommand{\rel}{\mathbf{R}}

\newcommand{\complex}{\mathbf{C}}

\newcommand{\grass}[2]{\mathbf{G}(#1,#2)}

\newcommand{\oball}[2]{\mathbf{U}(#1,#2)}

\newcommand{\cball}[2]{\mathbf{B}(#1,#2)}

\newcommand{\density}{\boldsymbol{\Theta}}

\newcommand{\unitmeasure}[1]{\boldsymbol{\alpha}(#1)}

\newcommand{\besicovitch}[1]{\boldsymbol{\beta}(#1)}

\newcommand{\isoperimetric}[1]{\boldsymbol{\gamma}(#1)}

\newcommand{\id}[1]{\mathbf{1}_{#1}}

\newcommand{\weakD}{\mathbf{D}}

\newcommand{\derivative}[2]{{#1}\,\weakD{#2}}

\newcommand{\boundary}[2]{{#1}\,\partial{#2}}

\newcommand{\ud}{\ensuremath{\,\mathrm{d}}}

\DeclareMathOperator{\with}{:}
\newcommand{\classification}[3]{{#1} \cap \{ {#2} \with {#3} \}}
\newcommand{\eqclassification}[3]{{(#1)} \cap \{ {#2} \with {#3} \}}

\newcommand{\project}[1]{#1_\natural}

\newcommand{\eqproject}[1]{(#1)_\natural}

\newcommand{\perpproject}[1]{#1_\natural^\perp}

\newcommand{\lIm}{[}
\newcommand{\rIm}{]}

\newcommand{\vdim}{{m}}
\newcommand{\codim}{{n-m}}
\newcommand{\adim}{{n}}

\newcommand{\intertextenum}[1]{\setcounter{counter2}{\value{enumi}}\end{enumerate}#1\begin{enumerate}\setcounter{enumi}{\value{counter2}}}

\newcommand{\printRoman}[1]{\setcounter{counter1}{#1}\Roman{counter1}}

\DeclareMathOperator{\without}{\sim}

\newcommand{\restrict}{\mathop{\llcorner}}

\newcommand{\tint}[2]{{\textstyle\int_{#1}^{#2}}}

\newcommand{\tsum}[2]{{\textstyle\sum_{#1}^{#2}}}

\DeclareMathOperator{\card}{card}

\DeclareMathOperator{\Clos}{Clos}

\DeclareMathOperator{\Bdry}{Bdry}

\newcommand{\measureball}[2]{{#1}\,{#2}}

\DeclareMathOperator{\Nor}{Nor}     % normal space
\DeclareMathOperator{\Tan}{Tan}     % tangent space
\DeclareMathOperator{\spt}{spt}     % support
\DeclareMathOperator{\im}{im}       % image
\DeclareMathOperator{\Int}{Int}     % interior of a set
\DeclareMathOperator{\diam}{diam}   % diameter
\DeclareMathOperator{\Lip}{Lip}     % Lipschitz constant
\DeclareMathOperator{\dmn}{dmn}     % domain
\DeclareMathOperator{\dist}{dist}   % distance
\DeclareMathOperator{\Hom}{Hom}     % Vectorspace of homomorphisms
\DeclareMathOperator{\sign}{sign}   % sign of
\DeclareMathOperator{\Span}{span}   % linear span
\DeclareMathOperator{\ap}{ap}       % approximate
\DeclareMathOperator*{\aplim}{\mathrm{ap}\, \lim}   % approximate limit

\newcommand{\Lpnorm}[3]{{#1}_{({#2})}({#3})}
\newcommand{\eqLpnorm}[3]{{(#1)}_{({#2})}({#3})}

\theoremstyle{definition}
\swapnumbers
\newtheorem{step}{Step}	% the environment for steps
\newenvironment{proofinsteps}{\begin{proof}\setcounter{step}{0}}{\end{proof}}

\begin{document}

%%%%%%%%%%%%%%%%%%%%%%%%%%%%%%%%%%%%%%%%%%%%%%%%%%%%%%%%%%%%%%%%%%%%%%%%%%%%%%

\title{Weakly differentiable functions on varifolds}
\author{Ulrich Menne\thanks{\textit{AEI publication number:} AEI-2014-059.}}

%%%%%%%%%%%%%%%%%%%%%%%%%%%%%%%%%%%%%%%%%%%%%%%%%%%%%%%%%%%%%%%%%%%%%%%%%%%%%%

\maketitle
\begin{abstract}
	The present paper is intended to provide the basis for the study of
	weakly differentiable functions on rectifiable varifolds with locally
	bounded first variation. The concept proposed here is defined by means
	of integration by parts identities for certain compositions with
	smooth functions. In this class the idea of zero boundary values
	is realised using the relative perimeter of superlevel sets. Results
	include a variety of Sobolev Poincar{\'e} type embeddings, embeddings
	into spaces of continuous and sometimes H{\"o}lder continuous
	functions, pointwise differentiability results both of approximate and
	integral type as well as coarea formulae.

	As prerequisite for this study decomposition properties of such
	varifolds and a relative isoperimetric inequality are established.
	Both involve a concept of distributional boundary of a set introduced
	for this purpose.
	
	As applications the finiteness of the geodesic distance associated to
	varifolds with suitable summability of the mean curvature and a
	characterisation of curvature varifolds are obtained.
\end{abstract}
\tableofcontents
\addcontentsline{toc}{section}{\numberline{}Introduction}
\section*{Introduction} \subsection*{Overview}
The main purpose of this paper is to present a concept of weakly
differentiable functions on nonsmooth ``surfaces'' in Euclidean space with
arbitrary dimension and codimension arising in variational problems involving
the area functional. The model used for such surfaces are rectifiable
varifolds whose first variation with respect to area is representable by
integration (that is, in the terminology of Simon \cite[39.2]{MR756417},
rectifiable varifolds with locally bounded first variation). This includes
area minimising rectifiable currents\footnote{See Allard
\cite[4.8\,(4)]{MR0307015}.}, in particular perimeter minimising ``Caccioppoli
sets'', or almost every time slice of Brakke's mean curvature
flow\footnote{See Brakke \cite[\S 4]{MR485012}.} just as well as surfaces
occurring in diffuse interface models\footnote{See for instance Hutchinson and
Tonegawa \cite{MR1803974} or R\"oger and Tonegawa \cite{MR2377408}.} or image
restoration models\footnote{See for example Ambrosio and Masnou
\cite{MR1959769}.}. The envisioned concept should be defined without reference
to an approximation by smooth functions and it should be as broad as possible
so as to still allow for substantial positive results.

In order to integrate well the first variation of the varifold into the
concept of weakly differentiable function, it appeared necessary to provide an
entirely new notion rather than to adapt one of the many concepts of weakly
differentiable functions which have been invented for different purposes. For
instance, to study the support of the varifold as metric space with its
geodesic distance, stronger conditions on the first variation are needed, see
Section \ref{sec:geodesic_distance}.

\subsection*{Description of results} \subsubsection*{Setup and basic results}
To describe the results obtained, consider the following set of hypotheses;
the notation is explained in Section \ref{sec:notation}.
\begin{citing} [General hypothesis]
	Suppose $\vdim$ and $\adim$ are positive integers, $\vdim \leq \adim$,
	$U$ is an open subset of $\rel^\adim$, $V$ is an $\vdim$ dimensional
	rectifiable varifold in $U$ whose first variation $\delta V$ is
	representable by integration.
\end{citing}
The study of weakly differentiable functions is closely related to the study
of connectedness properties of the underlying space or varifold. Therefore it
is instructive to begin with the latter.
\begin{intro_definition} [see \ref{def:indecomposable}]
	If $V$ is as in the general hypothesis, it is called
	\emph{indecomposable} if and only if there is no $\| V \| + \| \delta
	V \|$ measurable set $E$ such that $\| V \| ( E ) > 0$, $\| V \| ( U
	\without E ) > 0$ and $\delta ( V \restrict E \times
	\grass{\adim}{\vdim} ) = ( \delta V ) \restrict E$.
\end{intro_definition}
The basic theorem involving this notion is the following.
\begin{citing} [Decomposition theorem, see \ref{remark:decomp_rep} and \ref{thm:decomposition}]
	If $\vdim$, $\adim$, $U$, and $V$ are as in the general hypothesis,
	then there exists a countable disjointed collection $G$ of Borel sets
	whose union is $U$ such that $V \restrict E \times
	\grass{\adim}{\vdim}$ is nonzero and indecomposable and $\delta ( V
	\restrict E \times \grass{\adim}{\vdim}) = ( \delta V ) \restrict E$
	whenever $E \in G$.
\end{citing}
Employing the following definition, the Borel partition $G$ of $U$ is required
to consist of members whose distributional $V$ boundary vanishes and which
cannot be split nontrivially into smaller Borel sets with that property.
\begin{intro_definition} [see \ref{def:v_boundary}]
	If $\vdim$, $\adim$, $U$, and $V$ are as in the general hypothesis and
	$E$ is $\| V \| + \| \delta V \|$ measurable, then the
	\emph{distributional $V$ boundary of $E$} is defined by
	\begin{gather*}
		\boundary{V}{E} = ( \delta V ) \restrict E - \delta ( V
		\restrict E \times \grass{\adim}{\vdim} ) \in \mathscr{D}' (
		U, \rel^\adim ).
	\end{gather*}
\end{intro_definition}
If the varifold is sufficiently regular a version of the Gauss Green
theorem can be proven which both justifies the terminology and links the
concept of boundary to the one for currents, see
\ref{thm:basic_structure_v_boundary} and \ref{remark:link}.

In the terminology of \ref{def:component} and \ref{def:decomposition} the
varifolds $V \restrict E \times \grass \adim \vdim$ occurring in the
decomposition theorem are components of $V$ and the family $\{ V \restrict E
\times \grass \adim \vdim \with E \in G \}$ is a decomposition of $V$.
However, unlike the decomposition into connected components for topological
spaces, the preceding decomposition is nonunique in an essential way. In fact,
the varifold corresponding to the union of three lines in $\rel^2$ meeting at
the origin at equal angles may also be decomposed into two ``Y-shaped''
varifolds each consisting of three rays meeting at equal angles, see
\ref{remark:nonunique_decomposition}.

The seemingly most natural definition of weakly differentiable function would
be to require an integration by parts identity involving the first variation
of the varifold. However, the resulting class of real valued functions is
neither stable under truncation of its members nor does a coarea formula hold,
see \ref{remark:too_big_sobolev} and
\ref{remark:no_coarea_ineq_for_too_big_sobolev}. Therefore, instead one
requires an integration by parts identity for the composition of the function
in question with smooth functions whose derivative has compact support, see
\ref{def:v_weakly_diff}. Whenever $Y$ is a finite dimensional normed
vectorspace the resulting class of functions of $Y$ valued functions is
denoted by
\begin{gather*}
	\trunc (V,Y).
\end{gather*}
Whenever $f$ belongs to that class there exists a $\| V \|$ measurable $\Hom (
\rel^\adim, Y )$ valued function $\derivative V f$, called the
(generalised) weak derivative of $f$, which is $\| V \|$ almost uniquely
determined by the integration by parts identity. In defining $\trunc (
V,Y )$, it seems natural not to require local summability of
$\derivative{V}{f}$ but only
\begin{gather*}
	\tint{K \cap \{ x \with |f(x)| \leq s \}}{} | \derivative{V}{f} | \ud
	\| V \| < \infty
\end{gather*}
whenever $K$ is a compact subset of $U$ and $0 \leq s < \infty$; this is in
analogy with definition of the space ``$\mathscr{T}_{\mathrm{loc}}^{1,1} ( U
)$'' introduced by B{\'e}nilan, Boccardo, Gallou{\"e}t, Gariepy, Pierre and
V{\'a}zquez in \cite[p.~244]{MR1354907} for the case of Lebesgue measure, see
\ref{remark:comparison_trunc_spaces}. In both cases the letter ``T'' in the
name of the space stands for truncation.

Stability properties under composition (for example truncation) then follow
readily, see \ref{lemma:basic_v_weakly_diff} and \ref{lemma:comp_lip}. Also,
it is evident that members of $\trunc(V,Y)$ may be defined on components
separately, see \ref{thm:tv_on_decompositions}. The space $\trunc
(V,Y)$ is stable under addition of a locally Lipschitzian function but it
is not closed with respect to addition in general, see
\ref{thm:addition}\,\eqref{item:addition:add} and \ref{example:star}. A
similar statement holds for multiplication of functions, see
\ref{thm:addition}\,\eqref{item:addition:mult} and \ref{example:star}.
Adopting an axiomatic viewpoint, consider the class of functions which
satisfies the following three conditions for a given varifold:
\begin{enumerate}
	\item Each function which is constant on the components of some
	decomposition of the varifold belongs to the class.
	\item The class is closed under addition.
	\item The class is closed under truncation.
\end{enumerate}
Then there exists a stationary one dimensional integral varifold in $\rel^2$
such that the associated class necessarily contains characteristic functions
with nonvanishing distributional derivative representable by integration, see
\ref{example:axioms}. These characteristic functions should not belong to a
class of ``weakly differentiable'' functions but rather to a class of
functions of ``bounded variation''. The reason for this phenomenon is the
afore-mentioned nonuniqueness of decompositions of varifolds.

Much of the development of the theory rests on the following basic theorem.
\begin{citing} [Coarea formula, functional analytic form, see
\ref{remark:associated_distribution} and \ref{thm:tv_coarea}]
	Suppose $\vdim$, $\adim$, $U$, and $V$ are as in the general
	hypothesis, $f \in \trunc(V)$, and $E(y) = \{ x \with f(x) > y \}$ for
	$y \in \rel$.

	Then there holds
	\begin{gather*}
		\tint{}{} \left < \phi (x,f(x)), \derivative{V}{f}(x) \right >
		\ud \| V \| x = \tint{}{} \boundary{V}{E(y)} ( \phi (\cdot, y
		) ) \ud \mathscr{L}^1 y, \\
		\tint{}{} g (x,f(x)) | \derivative{V}{f}(x) | \ud \| V \| x =
		\tint{}{} \tint{}{} g (x,y) \ud \| \boundary{V}{E(y)} \| x \ud
		\mathscr{L}^1 y.
	\end{gather*}
	whenever $\phi \in \mathscr{D} (U \times \rel, \rel^\adim )$ and
	$g : U \times \rel \to \rel$ is a continuous function with compact
	support.
\end{citing}
An example of such a development is provided by passing from the notion of
zero boundary values on a relatively open part $G$ of the boundary of $U$ for
sets to a similar notion for weakly differentiable functions. For $\| V \| +
\| \delta V \|$ measurable sets $E$ such that $\boundary VE$ is representable
by integration, the condition is defined in \ref{miniremark:zero_boundary}.
In the special case that $\| V \|$ is associated to $M \cap U$ for some
properly embedded $\vdim$ dimensional submanifold $M$ of $\rel^\adim
\without (( \Bdry U ) \without G )$ of class $2$ without boundary it is
equivalent to $E$ being of locally finite perimeter in $M$ and its
distributional boundary being measure theoretically contained in $U$, see
\ref{example:smooth_case_zero_boundary}. In order to define such a notion for
nonnegative members $f$ of $\trunc (V,\rel)$, one requires for $\mathscr{L}^1$
almost all $0 < y < \infty$ that the set $E(y) = \{ x \with f(x)>y \}$
satisfies the corresponding zero boundary value condition.

This gives rise to the subspaces $\trunc_G ( V )$ of $\trunc (V, \rel ) \cap
\{ f \with f \geq 0 \}$ with $|f| \in \trunc_\varnothing ( V )$ whenever $f
\in \trunc (V,Y)$, see \ref{def:trunc_g} and \ref{remark:trunc}. The space
$\trunc_G (V)$ satisfies useful truncation and closure properties, see
\ref{lemma:trunc_tg} and \ref{lemma:closeness_tg}. Moreover, under a natural
summability hypothesis, the multiplication of a member of $\trunc_G (V)$ by a
nonnegative Lipschitzian function belongs to $\trunc_G (V)$, see
\ref{thm:mult_tg}. Whereas the usage of level sets in the definition of
$\trunc_G (V)$ is tailored to work nicely in the proofs of embedding results
in Section~\ref{sec:embeddings}, the stability property under multiplication
requires a more delicate proof in turn.

\subsubsection*{Embedding results and structural results} To proceed to deeper
results on functions in $\trunc (V,Y)$, the usage of the isoperimetric
inequality for varifolds seems indispensable. The latter works best under the
following additional hypothesis.
\begin{citing} [Density hypothesis]
	Suppose $\vdim$, $\adim$, $U$, and $V$ are as in the general
	hypothesis and satisfy
	\begin{gather*}
		\density^\vdim ( \| V \|, x ) \geq 1 \quad \text{for $\| V \|$
		almost all $x$}.
	\end{gather*}
\end{citing}
The key to prove effective versions of Sobolev Poincar{\'e} type embedding
results is the following theorem.
\begin{citing} [Relative isoperimetric inquality, see \ref{corollary:rel_iso_ineq}]
	Suppose $\vdim$, $\adim$, $U$, and $V$ satisfy the general hypothesis
	and the density hypothesis, $E$ is $\|V \| + \| \delta V \|$
	measurable, $1 \leq Q \leq M < \infty$, $\adim \leq M$, $\Lambda =
	\Gamma_{\ref{thm:rel_iso_ineq}} ( M )$, $0 < r < \infty$,
	\begin{gather*}
		\| V \| (E) \leq (Q-M^{-1}) \unitmeasure{\vdim} r^\vdim, \quad
		\| V \| ( \classification{E}{x}{ \density^\vdim ( \| V \|, x )
		< Q } ) \leq \Lambda^{-1} r^\vdim,
	\end{gather*}
	and $A = \{ x \with \oball{x}{r} \subset U \}$.

	Then there holds
	\begin{gather*}
		\| V \| ( E \cap A)^{1-1/\vdim} \leq \Lambda \big ( \|
		\boundary{V}{E} \| (U) + \| \delta V \| ( E ) \big ),
	\end{gather*}
	where $0^0 = 0$.
\end{citing}
In the special case that $\boundary{V}{E} = 0$, $\| \delta V \| (E) = 0$ and
\begin{gather*}
	\density^\vdim ( \| V \|, x ) \geq Q \quad \text{for $\| V \|$ almost
	all $x$}
\end{gather*}
the varifold $V \restrict E \times \grass{\adim}{\vdim}$ is stationary and the
support of its weight measure, $\spt ( \| V \| \restrict E)$, cannot intersect
$A$ by the monotonicity identity and the upper bound on $\| V \|(E)$. The
value of the theorem lies in quantifying this behaviour. Much of the
usefulness of the result for the purposes of the present paper stems from the
fact that values of $Q$ larger than $1$ are permitted. This allows to
effectively apply the result in neighbourhoods of points where the density
function $\density^\vdim ( \| V \|, \cdot )$ has a value larger than $1$ and
is approximately continuous. The case $Q=1$ is partly contained in a result of
Hutchinson \cite[Theorem 1]{MR1066398} which treats Lipschitzian functions.

If the set $E$ satisfies a zero boundary value condition, see
\ref{miniremark:zero_boundary}, on a relatively open subset $G$ of $\Bdry U$,
the conclusion may be sharpened by replacing $A = \{ x \with \oball{x}{r}
\subset U \}$ by
\begin{gather*}
	A'= U \cap \{ x \with \oball{x}{r} \subset \rel^\adim \without B \},
	\quad \text{where $B = ( \Bdry U ) \without G$}.
\end{gather*}

In order to state a version of the resulting Sobolev Poincar{\'e}
inequalities, recall a less known notation from Federer's treatise on
geometric measure theory, see \cite[2.4.12]{MR41:1976}.
\begin{intro_definition}
	Suppose $\mu$ measures $X$ and $f$ is a $\mu$ measurable function
	with values in some Banach space $Y$.

	Then one defines $\mu_{(p)} (f)$ for $1 \leq p \leq \infty$ by the
	formulae
	\begin{gather*}
		\mu_{(p)}(f) = ( \tint{}{} |f|^p \ud \mu )^{1/p}
		\quad \text{in case $1 \leq p < \infty$}, \\
		\mu_{(\infty)}(f) = \inf \big \{ s \with \text{$s \geq 0$,
		$\mu ( \{ x \with |f(x)| > s \} ) = 0$} \big \}.
	\end{gather*}
\end{intro_definition}
In comparison to the more common notation $\| f \|_{\mathbf{L}_p ( \mu, Y
)}$ it puts the measure in focus and avoids iterated subscripts.

\begin{citing} [Sobolev Poincar{\'e} inequality, zero median version, see
\ref{remark:mod_tv}, \ref{remark:trunc}, and \ref{thm:sob_poin_summary}\,\eqref{item:sob_poin_summary:interior:p=1}]
	Suppose $1 \leq M < \infty$.

	Then there exists a positive, finite number $\Gamma$ with the
	following property.

	If $\vdim$, $\adim$, $U$, and $V$ satisfy the general hypothesis and
	the density hypothesis, $\adim \leq M$, $Y$ is a finite dimensional
	normed vectorspace, $f \in \trunc ( V, Y )$, $1 \leq Q \leq M$, $0 < r
	< \infty$, $E = U \cap \{ x \with f(x) \neq 0 \}$,
	\begin{gather*}
		\| V \| ( E ) \leq ( Q-M^{-1} ) \unitmeasure{\vdim} r^\vdim,
		\\
		\| V \| ( E \cap \{ x \with \density^\vdim ( \| V \|, x ) < Q
		\} ) \leq \Gamma^{-1} r^\vdim, \\
	 	\text{$\beta = \infty$ if $\vdim = 1$}, \quad
		\text{$\beta = \vdim/(\vdim-1)$ if $\vdim > 1$},
	\end{gather*}
	$A = \{ x \with \oball{x}{r} \subset U \}$, then
	\begin{gather*}
		\eqLpnorm{\| V \| \restrict A }{\beta}{f} \leq \Gamma \big (
		\Lpnorm{\| V \|}{1}{ \derivative{V}{f} } + \Lpnorm{\| \delta V
		\|}{1} {f} \big ).
	\end{gather*}
\end{citing}
Again, apart from extending the result to the class $\trunc (V,Y)$, the main
improvement upon known results such as Hutchinson \cite[Theorem 1]{MR1066398}
is the applicability with values $Q > 1$. If $f$ belongs to $\trunc_G(V)$,
then $A$ may be replaced by $A'$ as in the relative isoperimetric inequality.
Not surprisingly, there also exists a version of the Sobolev inequality for
members of $\trunc_{\Bdry U} ( V)$.
\begin{citing} [Sobolev inequality, see \ref{thm:sob_poin_summary}\,\eqref{item:sob_poin_summary:global:p=1}]
	Suppose $1 \leq M < \infty$.

	Then there exists a positive, finite number $\Gamma$ with the
	following property.

	If $\vdim$, $\adim$, $U$, and $V$ satisfy the general hypothesis and the
	density hypothesis, $\adim \leq M$, $f \in \trunc_{\Bdry U} ( V )$,
	\begin{gather*}	
		E = U \cap \{ x \with f(x) \neq 0 \}, \quad \| V \| ( E ) <
		\infty, \\
	 	\text{$\beta = \infty$ if $\vdim = 1$}, \quad
		\text{$\beta = \vdim/(\vdim-1)$ if $\vdim > 1$},
	\end{gather*}
	then there holds
	\begin{gather*}
		\Lpnorm{\| V \|}{\beta}{f} \leq \Gamma \big ( \Lpnorm{\| V
		\|}{1}{ \derivative{V}{f} } + \Lpnorm{\| \delta V \|}{1} {f}
		\big ).
	\end{gather*}
\end{citing}
For Lipschitzian functions Sobolev inequalities were obtained by Allard
\cite[7.1]{MR0307015} and Michael and Simon \cite[Theorem 2.1]{MR0344978} for
varifolds and by Federer in \cite[\S 2]{MR0388226} for rectifiable currents
which are absolutely minimising with respect to a positive, parametric
integrand.

Coming back the Sobolev Poincar{\'e} inequalities, one may also establish a
version with several ``medians''. The number of medians needed is controlled
by the total weight measure of the varifold in a natural way.
\begin{citing} [Sobolev Poincar{\'e} inequality, several medians, see
\ref{thm:sob_poincare_q_medians}\,\eqref{item:sob_poincare_q_medians:p=1}]
	Suppose $1 \leq M < \infty$.

	Then there exists a positive, finite number $\Gamma$ with the
	following property.

	If $\vdim$, $\adim$, $U$, and $V$ satisfy the general hypothesis and
	the density hypothesis, $Y$ is a finite dimensional normed
	vectorspace, $\sup \{ \dim Y, \adim \} \leq M$, $f \in \trunc (V,Y)$,
	$1 \leq Q \leq M$, $N$ is a positive integer, $0 < r < \infty$,
	\begin{gather*}
		\| V \| ( U ) \leq ( Q-M^{-1} ) (N+1) \unitmeasure{\vdim}
		r^\vdim, \\
		\| V \| ( \{ x \with \density^\vdim ( \| V \|, x ) < Q \} )
		\leq \Gamma^{-1} r^\vdim, \\
	 	\text{$\beta = \infty$ if $\vdim = 1$}, \quad
		\text{$\beta = \vdim/(\vdim-1)$ if $\vdim > 1$},
	\end{gather*}	
	and $A = \{ x \with \oball{x}{r} \subset U \}$, then there exists a
	subset $\Upsilon$ of $Y$ such that $1 \leq \card \Upsilon \leq N$
	and
	\begin{gather*}
		\eqLpnorm{\| V \| \restrict A}{\beta}{f_\Upsilon} \leq \Gamma
		N^{1/\beta} \big ( \Lpnorm{\| V \|}{1}{\derivative{V}{f}} +
		\Lpnorm{\| \delta V \|}{1}{f_\Upsilon} \big),
	\end{gather*}
	where $f_\Upsilon (x) = \dist (f(x),\Upsilon)$ for $x \in \dmn f$.
\end{citing}
If $Y = \rel$, the approach of Hutchinson, see \cite[Theorem 3]{MR1066398},
carries over unchanged and, in fact, yields a somewhat sharper estimate, see
\ref{thm:sob_poin_several_med}\,\eqref{item:sob_poin_several_med:p=1}. If
$\dim Y \geq 2$ and $N \geq 2$ the selection procedure for $\Upsilon$ is more
delicate.

In order to precisely state the next result, the concept of approximate
tangent vectors and approximate differentiability (see
\cite[3.2.16]{MR41:1976}) will be recalled.
\begin{intro_definition}
	Suppose $\mu$ measures an open subset $U$ of a normed vectorspace
	$X$, $a \in U$, and $m$ is a positive integer.

	Then $\Tan^m ( \mu, a )$ denotes the closed cone of \emph{$(\mu,m)$
	approximate tangent vectors} at $a$ consisting of all $v \in X$ such
	that
	\begin{gather*}
		\density^{\ast m} ( \mu \restrict \mathbf{E}
		(a,v,\varepsilon), a ) > 0 \quad \text{for every $\varepsilon
		> 0$}, \\
		\text{where $\mathbf{E}(a,v,\varepsilon) = X \cap \{ x
		\with \text{$|r(x-a) -v| < \varepsilon$ for some $r > 0$} \}$}.
	\end{gather*}
	Moreover, if $X$ is an inner product space, then the cone of
	\emph{$(\mu,m)$ approximate normal vectors} at $a$ is defined to be
	\begin{gather*}
		\Nor^m ( \mu, a ) = X \cap \{ u \with \text{$u \bullet v \leq
		0$ for $v \in \Tan^m ( \mu, a )$} \}.
	\end{gather*}
\end{intro_definition}
If $V$ is an $\vdim$ dimensional rectifiable varifold in an open subset $U$ of
$\rel^\adim$, then at $\| V \|$ almost all $a$, $T = \Tan^\vdim ( \| V \|, a
)$ is an $\vdim$ dimensional plane such that
\begin{gather*}
	r^{-\vdim} \tint{}{} f(r^{-1}(x-a)) \ud \| V \| x \to \density^\vdim (
	\| V \|, a ) \tint{T}{} f \ud \mathscr{H}^\vdim \quad \text{as
	$r \to 0+$}
\end{gather*}
whenever $f : U \to \rel$ is a continuous function with compact support.
However, at individual points the requirement that $\Tan^\vdim ( \|V \|, a )$
forms an $\vdim$ dimensional plane differs from the condition that $\|V\|$
admits an $\vdim$ dimensional approximate tangent plane in the sense of Simon
\cite[11.8]{MR756417}.
\begin{intro_definition}
	Suppose $\mu$, $U$, $a$ and $m$ are as in the preceding definition
	and $f$ maps a subset of $X$ into another normed vectorspace $Y$.

	Then $f$ is called \emph{$(\mu,m)$ approximately differentiable} at
	$a$ if and only if there exist $b \in Y$ and a continuous linear
	map $L : X \to Y$ such that
	\begin{gather*}
		\density^\vdim ( \mu \restrict X \without \{ x \with
		|f(x)-b-L (x-a)| \leq \varepsilon |x-a| \}, a ) = 0
		\quad \text{for every $\varepsilon > 0$}.
	\end{gather*}
	In this case $L | \Tan^\vdim ( \mu, a )$ is unique and it is
	called the \emph{$(\mu,m)$ approximate differential} of $f$ at $a$,
	denoted
	\begin{gather*}
		(\mu,m) \ap Df(a).
	\end{gather*}
\end{intro_definition}
Also, the following notation will be convenient, see Almgren
\cite[T.1\,(9)]{MR1777737}.
\begin{intro_definition}
	Whenever $P$ is an $\vdim$ dimensional plane in $\rel^\adim$, the
	orthogonal projection of $\rel^\adim$ onto $P$ will be denoted by
	$\project{P}$.
\end{intro_definition}

The Sobolev Poincar{\'e} inequality with zero median and a suitably chosen
number $Q$ is the key to prove the following structural result for weakly
differentiable functions.
\begin{citing} [Approximate differentiability, see \ref{thm:approx_diff}]
	Suppose $\vdim$, $\adim$, $U$, and $V$ satisfy the general hypothesis
	and the density hypothesis, $Y$ is a finite dimensional normed
	vectorspace, and $f \in \trunc (V,Y)$.

	Then $f$ is $( \| V \|, \vdim )$ approximately differentiable with
	\begin{gather*}
		\derivative{V}{f} (a) = ( \| V \|, \vdim ) \ap Df(a) \circ
		\project{\Tan^\vdim ( \| V \|, a )}
	\end{gather*}
	at $\| V \|$ almost all $a$.
\end{citing}
The preceding assertion consists of two parts. Firstly, it yields that
\begin{gather*}
	\derivative{V}{f}(a) | \Nor^\vdim ( \| V \|, a ) = 0 \quad \text{for
	$\| V \|$ almost all $a$};
\end{gather*}
a property that is not required by the definition. Secondly, it asserts that
\begin{gather*}
	\derivative{V}{f}(a) | \Tan^\vdim ( \| V \|, a ) = ( \| V \|, \vdim )
	\ap D f(a) \quad \text{for $\|V \|$ almost all $a$}.
\end{gather*}
This is somewhat similar to the situation for the generalised mean curvature
of an integral $\vdim$ varifold where it was first obtained by Brakke in
\cite[5.8]{MR485012} that its tangential component vanishes and secondly by
the author in \cite[4.8]{snulmenn.c2} that its normal component is induced by
approximate quantities.

\begin{citing} [Differentiability in Lebesgue spaces, see
\ref{thm:diff_lebesgue_spaces}\,\eqref{item:diff_lebesgue_spaces:m>1=p}]
	Suppose $\vdim$, $\adim$, $U$, and $V$ satisfy the general hypothesis
	and the density hypothesis, $Y$ is a finite dimensional normed
	vectorspace, $f \in \trunc (V,Y) \cap \Lploc{1} ( \| \delta V \|, Y
	)$, and $\derivative Vf \in \Lploc 1 ( \| V \|, \Hom ( \rel^\adim, Y )
	)$.

	If $\vdim > 1$ and $\beta = \vdim/(\vdim-1)$, then
	\begin{gather*}
		\lim_{r \to 0+} r^{-\vdim} \tint{\cball ar}{} ( | f(x)-f(a)-
		\derivative Vf (a) (x-a) | / | x-a| )^\beta \ud \| V \| x = 0
	\end{gather*}
	for $\| V \|$ almost all $a$.
\end{citing}
The result is derived from the approximate differentiability result mainly by
means of the zero median version of the Sobolev Poincar{\'e} inequality.
Another consequence of the approximate differentiability result is the
rectifiability of the distributional boundary of almost all superlevel sets of
weakly differentiable functions which supplements the functional analytic form
of the coarea formula.
\begin{citing} [Coarea formula, measure theoretic form, see \ref{corollary:coarea}]
	Suppose $\vdim$, $\adim$, $U$, and $V$ satisfy the general hypothesis
	and the density hypothesis, $f \in \trunc (V,\rel)$, and $E(y) = \{ x
	\with f(x) > y \}$ for $y \in \rel$.

	Then there exists an $\mathscr{L}^1$ measurable function $W$ with
	values in the weakly topologised space of $\vdim-1$ dimensional
	rectifiable varifolds in $U$ such that for $\mathscr{L}^1$ almost all
	$y$ there holds
	\begin{gather*}
		\Tan^{\vdim-1} ( \| W(y) \|, x ) = \Tan^\vdim ( \| V \|, x )
		\cap \ker \derivative{V}{f} (x) \in \grass{\adim}{\vdim-1}, \\
		\density^{\vdim-1} ( \| W(y) \|, x ) = \density^\vdim ( \| V
		\|, x )
	\end{gather*}
	for $\| W(y) \|$ almost all $x$ and
	\begin{gather*}
		\boundary{V}{E(y)}(\theta) = \tint{}{}
		|\derivative{V}{f}(x)|^{-1} \derivative{V}{f} (x)(\theta(x))
		\ud \| W(y) \| x \quad \text{for $\theta \in \mathscr{D}
		(U,\rel^\adim )$},
	\end{gather*}
	in particular $\| \boundary{V}{E(y)} \| = \| W (y) \|$ for such $y$.
\end{citing}
The proof of an appropriate form of the coarea formula was the original
motivation for the author to establish the new relative isoperimetric
inequality and its corresponding Sobolev Poincar\'e inequalities as well as the
approximate differentiability result. In this respect notice that the proof of
the measure theoretic form of the coarea formula could not be based on the
more elementary functional analytic form of the coarea formula in conjunction
with an extension of the Gauss Green theorem (see \cite[4.5.6]{MR41:1976}) to
sets whose distributional $V$ boundary is representable by integration. In
fact, for general sets $E$ whose distributional $V$ boundary is representable
by integration it may happen that there is no $\vdim-1$ dimensional
rectifiable varifold whose weight measure equals $\| \boundary{V}{E} \|$, see
\ref{remark:no_rectifiable_structure}. This is in contrast to the behaviour of
sets of locally finite perimeter in Euclidean space.

\subsubsection*{Critical mean curvature} Several of the preceding estimates
and structural results may be sharpened in case the generalised mean curvature
satisfies an appropriate summability condition.
\begin{citing} [Mean curvature hypothesis]
	Suppose $\vdim$, $\adim$, $U$ and $V$ are as in the general
	hypothesis and satisfies the following condition.

	If $\vdim > 1$ then for some $h \in \Lploc{\vdim} ( \| V \|,
	\rel^\adim )$ there holds
	\begin{gather*}
		( \delta V ) ( \theta ) = - \tint{}{} h \bullet \theta \ud \|
		V \| \quad \text{for $\theta \in \mathscr{D}(U,\rel^\adim)$}.
	\end{gather*}
	In this case $\psi$ will denote the Radon measure over $U$
	characterised by the condition $\psi (X) = \tint{X}{} |h|^\vdim \ud \|
	V \|$ whenever $X$ is a Borel subset of $U$.
\end{citing}
Clearly, if this condition is satisfied, then the function $h$ is $\| V \|$
almost equal to the generalised mean curvature vector $\mathbf{h}(V,\cdot)$ of
$V$. The exponent $\vdim$ is ``critical'' with respect to homothetic rescaling
of the varifold. Replacing it by a slightly larger number would entail upper
semicontinuity of $\density^\vdim ( \| V \|, \cdot)$ and the applicability of
Allard's regularity theory, see Allard \cite[\S 8]{MR0307015}. In contrast,
replacing the exponent by a slightly smaller number would allow for examples
in which the varifold is locally highly disconnected, see the author
\cite[1.2]{snulmenn.isoperimetric}.

The importance of this hypothesis for the present considerations lies in the
fact that in the relative isoperimetric inequality, the summand ``$\| \delta V
\| ( U )$'' may be dropped provided $V$ satisfies additionally the mean
curvature hypothesis with $\psi (E)^{1/\vdim} \leq \Lambda^{-1}$, see
\ref{corollary:true_rel_iso_ineq}. As a first consequence, one obtains the
following version of the Sobolev Poincar{\'e} inequality.
\begin{citing} [Sobolev Poincar{\'e} inequality, zero median, critical mean
curvature, see \ref{remark:mod_tv}, \ref{remark:trunc},
\ref{thm:sob_poin_summary}\,\eqref{item:sob_poin_summary:interior:q<m=p}\,\eqref{item:sob_poin_summary:interior:p=m<q}]
	Suppose $1 \leq M < \infty$.

	Then there exists a positive, finite number $\Gamma$ with the
	following property.

	If $\vdim$, $\adim$, $U$, $V$, and $\psi$ satisfy the general
	hypothesis, the density hypothesis and the mean curvature hypothesis,
	$\adim \leq M$, $Y$ is a finite dimensional normed vectorspace, $f \in
	\trunc ( V, Y )$, $1 \leq Q \leq M$, $0 < r < \infty$, $E = U \cap \{
	x \with f(x) \neq 0 \}$,
	\begin{gather*}
		\| V \| ( E ) \leq ( Q-M^{-1} ) \unitmeasure{\vdim} r^\vdim,
		\quad \psi ( E ) \leq \Gamma^{-1}, \\
		\| V \| ( E \cap \{ x \with \density^\vdim ( \| V \|, x ) < Q
		\} ) \leq \Gamma^{-1} r^\vdim,
	\end{gather*}
	and $A = \{ x \with \oball{x}{r} \subset U \}$, then the following two
	statements hold:
	\begin{enumerate}
		\item If $1 \leq q < \vdim$, then
		\begin{gather*}
			\eqLpnorm{\| V \| \restrict A}{\vdim q/(\vdim-q)}{f}
			\leq \Gamma (\vdim-q)^{-1} \Lpnorm{\| V \|}{q}{
			\derivative{V}{f} }.
		\end{gather*}
		\item If $1 < \vdim < q \leq \infty$, then
		\begin{gather*}
			\eqLpnorm{\| V \| \restrict A}{\infty}{f} \leq
			\Gamma^{1/(1/\vdim-1/q)} \| V \| ( E)^{1/\vdim-1/q}
			\Lpnorm{\| V \|}{q}{ \derivative{V}{f} }.
		\end{gather*}
	\end{enumerate}
\end{citing}
Even if $Q = q = 1$ and $f$ and $V$ correspond to smooth objects, this
estimate appears to be new; at least, it seems not to be straightforward to
derive from Hutchinson \cite[Theorem 1]{MR1066398}.

A similar result holds for $\vdim = 1$, see
\ref{thm:sob_poin_summary}\,\eqref{item:sob_poin_summary:interior:p=m=1}.
Again, if $f$ belongs to $\trunc_G(V)$, then $A$ may be replaced by $A'$.
Also, appropriate versions of the Sobolev inequality may be furnished, see
\ref{thm:sob_poin_summary}\,\eqref{item:sob_poin_summary:global:p=m=1}--\eqref{item:sob_poin_summary:global:p=m<q}.
The same is true with respect to the Sobolev Poincar{\'e} inequalities with
several medians, see
\ref{thm:sob_poincare_q_medians}\,\eqref{item:sob_poincare_q_medians:p=m=1}--\eqref{item:sob_poincare_q_medians:p=m<q}
and
\ref{thm:sob_poin_several_med}\,\eqref{item:sob_poin_several_med:p=m=1}--\eqref{item:sob_poin_several_med:p=m<q}.

Before stating the stronger differentiability properties that result from the
mean curvature hypothesis, recall the definition of relative differential from
\cite[3.1.21, 3.1.22]{MR41:1976}.
\begin{intro_definition}
	Suppose $X$ and $Y$ are normed vectorspaces, $A \subset X$, and $a \in
	\Clos A$, and $f : A \to Y$.

	Then the \emph{tangent cone} of $A$ at $a$, denoted $\Tan(A,a)$, is
	the set of all $v \in X$ such that for every $\varepsilon > 0$ there
	exist $x \in A$ and $0 < r \in \rel$ with $|x-a| < \varepsilon$ and $|
	r(x-a)-v| < \varepsilon$. Moreover, $f$ is called \emph{differentiable
	relative to $A$ at $a$} if and only if there exist $b \in Y$ and a
	continuous linear map $L : X \to Y$ such that
	\begin{gather*}
		|f(x)-b-L(x-a) |/|x-a| \to 0 \quad \text{as $A \owns x
		\to a$}.
	\end{gather*}
	In this case $L|\Tan(A,a)$ is unique and denoted $Df(a)$.
\end{intro_definition}
\begin{citing} [Differentiability in Lebesgue spaces, critical mean curvature,
see
\protect{\ref{thm:diff_lebesgue_spaces}\allowbreak\eqref{item:diff_lebesgue_spaces:m=p=1}--\eqref{item:diff_lebesgue_spaces:p=m<q}}]
	Suppose $\vdim$, $\adim$, $U$, and $V$ satisfy the general hypothesis,
	the density hypothesis and the mean curvature hypothesis, $Y$ a finite
	dimensional normed vectorspace, $f \in \trunc (V,Y)$, $1 \leq q \leq
	\infty$, and $\derivative Vf \in \Lploc{q} ( \| V \|, \Hom
	(\rel^\adim, Y) )$.

	Then the following two statements hold.
	\begin{enumerate}
		\item If $q < \vdim$ and $\iota = \vdim q/(\vdim-q)$, then
		\begin{gather*}
			\lim_{r \to 0+} r^{-\vdim} \tint{\cball ar}{} ( |
			f(x)-f(a)- \derivative Vf (a)(x-a) | / | x-a| )^\iota
			\ud \| V \| x = 0
		\end{gather*}
		for $\| V \|$ almost all $a$.
		\item If $\vdim =1$ or $\vdim < q$, then there exists a
		subset $A$ of $U$ with $\| V \| ( U \without A ) = 0$ such
		that $f|A$ is differentiable relative to $A$ at $a$ with
		\begin{gather*}
			D (f|A) (a) = \derivative Vf (a) | \Tan^\vdim ( \| V
			\|, a) \quad \text{for $\| V \|$ almost all $a$},
		\end{gather*}
		in particular $\Tan(A,a) = \Tan^\vdim ( \| V \|, a )$ for such
		$a$.
	\end{enumerate}
\end{citing}
Considering $V$ associated to two crossing lines, it is evident that one
cannot expect a function in $\trunc(V,\rel)$ to be $\| V \|$ almost equal to a
continuous function even if $\delta V = 0$ and $\derivative{V}{f} = 0$. Yet,
in case $f$ is continuous, its modulus of continuity may be locally estimated
by its weak derivative. This estimate depends on $V$ but not on $f$ as
formulated in the next theorem.
\begin{citing} [Oscillation estimate, see \ref{thm:mod_continuity}\,\eqref{item:mod_continuity:m>1}]
	Suppose $\vdim$, $\adim$, $U$, and $V$ satisfy the general hypothesis,
	the density hypothesis and the mean curvature hypothesis, $K$ is a
	compact subset of $U$, $0 < \varepsilon \leq \dist ( K, \rel^\adim
	\without U )$, $\varepsilon < \infty$, and $1 < \vdim < q$.

	Then there exists a positive, finite number $\Gamma$ with the
	following property.

	If $Y$ is a finite dimensional normed vectorspace, $f : \spt \| V \|
	\to Y$ is a continuous function, $f \in \trunc (V,Y)$, and $\kappa =
	\sup \{ \eqLpnorm{\| V \| \restrict \oball{a}{\varepsilon}}{q}{
	\derivative{V}{f} } \with a \in K \}$, then
	\begin{gather*}
		|f(x)-f(\chi)| \leq \varepsilon \kappa \quad \text{whenever
		$x,\chi \in K \cap \spt \| V \|$ and $|x-\chi| \leq
		\Gamma^{-1}$}.
	\end{gather*}
\end{citing}
A similar result holds for $\vdim = 1$, see
\ref{thm:mod_continuity}\,\eqref{item:mod_continuity:m=1}. This theorem rests
on the fact that connected components of $\spt \| V \|$ are relatively open in
$\spt \| V \|$, see
\ref{corollary:conn_structure}\,\eqref{item:conn_structure:open}. If a
varifold $V$ satisfies the general hypothesis, the density hypothesis and the
mean curvature hypothesis then the decomposition of $\spt \| V \|$ into
connected components yields a locally finite decomposition into relatively
open and closed subsets whose distributional $V$ boundary vanishes, see
\ref{corollary:conn_structure}\,\eqref{item:conn_structure:finite}--\eqref{item:conn_structure:piece}.
Moreover, any decomposition of the varifold $V$ will refine the decomposition
of the topological space $\spt \| V \|$ into connected components, see
\ref{corollary:conn_structure}\,\eqref{item:conn_structure:union}.

When the amount of total weight measure (``mass'') available excludes the
possibility of two separate sheets, the oscillation estimate may be sharpened
to yield H{\"o}lder continuity even without assuming a priori the continuity
of the function, see \ref{thm:hoelder_continuity}.

\subsubsection*{Two applications} Despite the necessarily rather weak
oscillation estimate in the general case, this estimate is still sufficient to
prove that the geodesic distance between any two points in the same connected
component of the support is finite.
\begin{citing} [Geodesic distance, see \ref{thm:conn_path_finite_length}]
	Suppose $\vdim$, $\adim$, $U$, and $V$ satisfy the general hypothesis,
	the density hypothesis and the mean curvature hypothesis, $C$ is a
	connected component of $\spt \| V \|$, and $a,x \in C$.

	Then there exist $- \infty < b \leq y < \infty$ and a Lipschitzian
	function $g : \{ \upsilon \with b \leq \upsilon \leq y \} \to \spt \|
	V \|$ such that $g(b) = a$ and $g(y) = x$.
\end{citing}
The proof follows a pattern common in theory of metric spaces, see
\ref{remark:metric_spaces}.

Finally, the presently introduced notion of weak differentiability may also be
used to reformulate the defining condition for curvature varifolds, see
\ref{def:curvature_varifold} and \ref{remark:hutchinson_reformulations},
introduced by Hutchinson in \cite[5.2.1]{MR825628}.
\begin{citing} [Characterisation of curvature varifolds, see
\ref{thm:curvature_varifolds}]
	Suppose $\vdim$ and $\adim$ are positive integers, $\vdim \leq \adim$,
	$U$ is an open subset of $\rel^\adim$, $V$ is an $\vdim$ dimensional
	integral varifold in $U$,
	\begin{gather*}
		X = U \cap \{ x \with \Tan^\vdim ( \| V \|, x ) \in
		\grass{\adim}{\vdim} \}, \quad Y = \Hom (\rel^\adim,
		\rel^\adim ) \cap \{ \sigma \with \sigma = \sigma^\ast \},
	\end{gather*}
	and $\tau : X \to Y$ satisfies
	\begin{gather*}
		\tau(x) = \project{\Tan^\vdim ( \| V \|, x )} \quad
		\text{whenever $x \in X$}.
	\end{gather*}

	Then $V$ is a curvature varifold if and only if $\| \delta V \|$ is a
	Radon measure absolutely continuous with respect to $\| V \|$ and $\tau
	\in \trunc(V, Y )$.
\end{citing}
The condition is readily verified to be a necessary one for curvature
varifolds. The key to prove its sufficiency is to relate the mean curvature
vector of $V$ to the weak differential of the tangent plane map $\tau$. This
may be accomplished, for instance, by means of the approximate
differentiability result, see \ref{thm:approx_diff}, applied to $\tau$, in
conjunction with the previously obtained second order rectifiability result
for such varifolds, see \cite[4.8]{snulmenn.c2}.  \subsection*{Possible lines
of further study}
\paragraph{Sobolev spaces.} The results obtained by the author on the area
formula for the Gauss map, see \cite[Theorem 3]{snulmenn.mfo1230}, rest on
several estimates for Lipschitzian solutions of certain linear elliptic
equations on varifolds satisfying the mean curvature hypothesis. The
formulation of these estimates necessitated several ad hoc formulations of
concepts such as zero boundary values which would not seem natural from
the point of view of partial differential equations. In order to avoid
repetition, the author decided to directly build an adequate framework for
these results rather than first publish the Lipschitzian case along with a
proof of the results announced in \cite{snulmenn.mfo1230}.\footnote{Would the
author have fully anticipated the effort needed for such a construction he
might have reconciled himself with some repetition.} The first part of such
framework is provided in the present paper. Continuing this programme, a
notion of Sobolev spaces, complete vectorspaces contained in
$\trunc(V,Y)$ in which locally Lipschitzian functions are suitably dense,
has already been developed but is not included here for length considerations.
\paragraph{Functions of locally bounded variation.} It seems worthwhile to
investigate to which extent results of the present paper for weakly
differentiable functions extend to a class of ``functions of locally bounded
variation.'' A possible definition is suggested in \ref{remark:bv}.
\paragraph{Intermediate conditions on the mean curvature.} The mean curvature
hypothesis may be weakened by replacing $\Lploc{\vdim}$ by $\Lploc{p}$ for
some $1 < p < \vdim$. In view of the Sobolev Poincar{\'e} type inequalities
obtained for height functions by the author in \cite[Theorem
4.4]{snulmenn.poincare}, one might seek for adequate formulations of the
Sobolev Poincar{\'e} inequalities and the resulting differentiability results
for functions belonging to $\trunc(V,Y)$ in these intermediate cases.
This could potentially have applications to structural results for curvature
varifolds, see \ref{remark:curv_flatness} and \ref{remark:curv_to_do}.
Additionally, the case $p = \vdim-1$ seems to be related to the study of the
geodesic distance for indecomposable varifolds, see \ref{remark:topping}.
\paragraph{Multiple valued weakly differentiable functions.} For convergence
considerations it appears useful to extend the concept of weakly
differentiable functions to a more general class of ``multiple-valued''
functions in the spirit of Moser \cite[\S 4]{62659}, see \ref{remark:moser}.
\subsection*{Organisation of the paper} In Section \ref{sec:notation} the
notation is introduced. In Section \ref{sec:tvs} some basic terminology and
results on topological vector spaces are collected. Sections
\ref{sec:distributions} and \ref{sec:monotonicity} contain preliminary results
concerning respectively certain distributions representable by integration and
consequences of the monotonicity identity. Section \ref{sec:distrib_boundary}
introduces the concept of distributional boundary of a set with respect to a
varifold. In Section \ref{sec:decomposition} the decomposition properties for
varifolds are established and Section \ref{sec:rel_iso} contains the relative
isoperimetric inequality. In Sections \ref{sec:basic}--\ref{sec:oscillation}
the theory of weakly differentiable functions is presented. Finally, in
Sections \ref{sec:geodesic_distance} and \ref{sec:curvature_varifolds} the
applications to the study of the geodesic distance associated to certain
varifolds and to curvature varifolds are discussed briefly.

\subsection*{Acknowledgements} The author would like to thank Dr.~Theodora
Bourni, Dr.~Leobardo Rosales and in particular Dr.~S{\l}awomir Kolasi{\'n}ski
for discussions and suggestions concerning this paper.

\section{Notation} \label{sec:notation}
The notation of Federer \cite{MR41:1976} and Allard \cite{MR0307015} will be
used throughout.

\paragraph{Less common symbols.} The set of positive integers is denoted by
$\nat$, see \cite[2.2.6]{MR41:1976}. For the open and closed ball with centre
$a$ and radius $r$ the symbols $\oball{a}{r}$ and $\cball{a}{r}$ are employed,
see \cite[2.8.1]{MR41:1976}. Whenever $f$ is a linear map and $v$ belongs to
its domain the alternate notation $\left <v,f \right>$ for $f(v)$ are used,
see \cite[1.10.1]{MR41:1976}. Inner products, in contrast, are denoted by
``$\bullet$'', see \cite[1.7.1]{MR41:1976}. For integration the alternate
notations $\tint{}{} f \ud \mu$, $\tint{}{} f(x) \ud \mu x$ and $\mu(f)$
are employed, see \cite[2.4.2]{MR41:1976}. Moreover, for evaluating a
distribution $T$ at $\phi$ are alternately denoted by $T(\phi)$ and $T_x (
\phi (x))$, see \cite[4.1.1]{MR41:1976}.

\paragraph{Modifications.} If $f$ is a relation, then $f \lIm A \rIm = \{ y
\with \text{$(x,y) \in f$ for some $x \in A $} \}$ whenever $A$ is a set, see
Kelley \cite[p.~8]{MR0370454}. Following Almgren \cite[T.1\,(9)]{MR1777737},
the symbol $\project P$ will denote the orthogonal projection of $\rel^\adim$
onto $P$ whenever $P$ is a plane in $\rel^\adim$. Extending Federer
\cite[3.2.16]{MR41:1976}, whenever $\mu$ measures an open subset of a normed
vectorspace $X$, $\iota : U \to X$ is the inclusion map, $a \in U$ and $\vdim$
is a positive integer notions of tangent vectors, normal vectors and
differentials of $(\mu,\vdim)$ approximate type will refer to the
corresponding $(\iota_\# \mu,\vdim)$ approximate notion, for instance
$\Tan^\vdim ( \mu, a )$ will denote $\Tan^\vdim ( \iota_\# \mu, a)$.
Following Schwartz, see \cite[Chapitre \printRoman{3}, \S 1]{MR0209834}, the
vectorspace $\mathscr{D}(U,Y)$ is given the (usual) locally convex topology,
see \ref{def:duy}, which differs from the topology employed by Federer in
\cite[4.1.1]{MR41:1976}, see \ref{remark:top_comp}. Moreover, the usage of $\|
S \|$ for distributions $S$ is chosen to be in accordance with Allard
\cite[4.2]{MR0307015} which agrees with Federer's usage in
\cite[4.1.5]{MR41:1976} in most but not in all cases, see
\ref{def:variation_measure} and \ref{remark:allard_vs_federer}.

Extending Allard \cite[2.5\,(2)]{MR0307015}, whenever $M$ is a submanifold of
$\rel^\adim$ of class $2$ and $a \in M$ the \emph{mean curvature vector of $M$
at $a$} is the unique $\mathbf{h} (M,a) \in \Nor (M,a)$ such that
\begin{gather*}
	\project{\Tan(M,a)} \bullet ( Dg(a) \circ \project{\Tan (M,a)}) = - g
	(a) \bullet \mathbf{h}(M,a)
\end{gather*}
whenever $g : M \to \rel^\adim$ is of class $1$ and $g(x) \in \Nor (M,x)$ for
$x \in M$.  If $V$ is an $\vdim$ dimensional varifold in $U$ and $\| \delta V
\|$ is a Radon measure, then the \emph{generalised mean curvature vector of
$V$ at $x$} is the unique $\mathbf{h} (V,x) \in \rel^\adim$ such that
\begin{gather*}
	\mathbf{h} (V,x) \bullet u = - \lim_{r \to 0+} \frac{( \delta V ) (
	b_{x,r} \cdot u)}{\measureball{\| V \|}{\cball{x}{r}}} \quad \text{for
	$u \in \rel^\adim$},
\end{gather*}
where $b_{x,r}$ is the characteristic function of $\cball{x}{r}$; hence $x \in
\dmn \mathbf{h} ( V,\cdot)$ if and only if the above limit exists for every $u
\in \rel^\adim$, see \cite[p.~9]{snulmenn.decay}.

\paragraph{Additional notation.} If $1 \leq p \leq \infty$, $\mu$ is a Radon
measure over a locally compact Hausdorff space $X$, and $Y$ is a Banach space,
then $\Lploc{p} ( \mu, Y )$ consists of all $f$ such that $f \in \Lp{p} (
\mu \restrict K, Y )$ whenever $K$ is a compact subset of $X$. Concerning the
Besicovitch Federer covering theorem, whenever $\adim \in \nat$ the number
$\besicovitch{\adim}$ denotes the least positive integer with the following
property, see Almgren \cite[p.~464]{MR855173}: If $G$ is a family of closed
balls in $\rel^\adim$ with $\sup \{ \diam B \with B \in G \} < \infty$, then
there exist disjointed subfamilies $G_1, \ldots, G_{\besicovitch{\adim}}$ such
that
\begin{gather*}
	\{ x \with \text{$\cball{x}{r} \in G$ for some $0 < r < \infty$} \}
	\subset {\textstyle\bigcup\bigcup \{ G_i \with i = 1, \ldots,
	\besicovitch{\adim} \}}.
\end{gather*}
Concerning the isoperimetric inequality, whenever $\vdim \in \nat$ the
smallest number with the following property is denoted by
$\isoperimetric{\vdim}$, see \cite[2.2--2.4]{snulmenn.isoperimetric}: If
$\adim \in \nat$, $\vdim \leq \adim$, $V \in \RVar_\vdim ( \rel^\adim )$, $\|
V \| ( \rel^\adim ) < \infty$, and $\| \delta V \| ( \rel^\adim ) < \infty$,
then
\begin{gather*}
	\| V \| ( \rel^\adim \cap \{ x \with \density^\vdim ( \| V \|, x )
	\geq 1 \} ) \leq \isoperimetric{\vdim} \| V \| ( \rel^\adim
	)^{1/\vdim} \| \delta V \| ( \rel^\adim ).
\end{gather*}

\paragraph{Definitions in the text.} The notions of \emph{Lusin space},
\emph{locally convex space} and \emph{locally convex topology}, \emph{strict
inductive limit} as well as \emph{final topology} and \emph{locally convex
final topology} are defined in \ref{def:lusin}, \ref{def:locally_convex},
\ref{def:strict_inductive_limit} and \ref{def:final_topologies} respectively.
The topologies on $\mathscr{K} (X)$ and $\mathscr{D}(U,Y)$ are defined in
\ref{def:kx} and \ref{def:duy}. The \emph{restriction} of a distribution
representable by integration to a set will be defined in
\ref{def:restriction_distribution}. In \ref{def:v_boundary}, the notion of
\emph{distributional boundary} of a set with respect to certain varifolds is
defined. The notions of \emph{indecomposability}, \emph{component}, and
\emph{decomposition} for certain varifolds are defined in
\ref{def:indecomposable}, \ref{def:component} and \ref{def:decomposition}
respectively. The notion of \emph{generalised weakly differentiable functions}
with respect to certain varifolds and the corresponding \emph{generalised weak
derivatives}, $\derivative{V}{f}$, and the associated space $\trunc (V,Y)$ are
introduced in \ref{def:v_weakly_diff}. The space $\trunc_G ( V )$ is defined
in \ref{def:trunc_g}. Finally, the notion of \emph{curvature varifold} is
explained in \ref{def:curvature_varifold}.

\paragraph{A convention.} Each statement asserting the existence of a
positive, finite number $\Gamma$ will give rise to a function depending on the
listed parameters whose name is $\Gamma_{\mathrm{x.y}}$, where $\mathrm{x.y}$
denotes the number of the statement. This is a refinement of a concept
employed by Almgren in \cite{MR1777737}.

\section{Topological vector spaces} \label{sec:tvs}
Some basic results on topological vector spaces and Lusin spaces are gathered
here mainly from Schwartz \cite{MR0426084} and Bourbaki \cite{MR910295}.
\begin{definition} [see \protect{\cite[Chapter \printRoman{2}, Definition 2,
p.~94]{MR0426084}}] \label{def:lusin}
	Suppose $X$ is a topological space.

	Then $X$ is called a \emph{Lusin space} if and only if $X$ is a
	Hausdorff topological space and there exists a complete, separable
	metric space $W$ and a continuous univalent map $f : W \to X$ whose
	image is $X$.
\end{definition}
\begin{remark} \label{remark:lusin}
	Any subset of a Lusin space is sequentially separable.
\end{remark}
\begin{definition} [see \protect{\cite[\printRoman{2}, p.~23,
def.~1]{MR910295}}] \label{def:locally_convex}
	A topological vector space is called a \emph{locally convex space} if
	and only if there exists a fundamental system of neighbourhoods of $0$
	that are convex sets; its topology is called \emph{locally convex
	topology}.
\end{definition}
\begin{definition} \label{def:strict_inductive_limit}
	The locally convex spaces form a category; its morphisms are the
	continuous linear maps. An inductive limit in this category is called
	\emph{strict} if and only if the morphisms of the corresponding
	inductive system are homeomorphic embeddings.
\end{definition}
\begin{remark}
	The notion of inductive limit is employed in accordance with
	\cite[p.~67--68]{MR1712872}.
\end{remark}
\begin{definition} [see \protect{\cite[\printRoman{1}, \S 2.4,
prop.~6]{MR1726779}, \cite[\printRoman{2}, p.~27,
prop.~5]{MR910295}}] \label{def:final_topologies}
	Suppose $E$ is a set [vector space], $E_i$ are topological spaces
	[topological vector spaces] and $f_i : E_i \to E$ are functions
	[linear maps] for $i \in I$.

	Then there exists a unique topology [unique locally convex topology]
	on $E$ such that a function [linear map] $g : E \to F$ into a
	topological space [locally convex space] $F$ is continuous if and only
	if $g \circ f_i$ are continuous for $i \in I$. This topology is called
	\emph{final topology} [\emph{locally convex final topology}] on $E$
	with respect to the family $f_i$.
\end{definition}
\begin{definition} [see \protect{\cite[\printRoman{2}, p.~29, Example
\printRoman{2}]{MR910295}}\footnote{In the terminology of
\cite[\printRoman{2}, p.~29, Example \printRoman{2}]{MR910295} a locally
compact Hausdorff space is called a locally compact space.}] \label{def:kx}

	Suppose $X$ is a locally compact Hausdorff space. Consider the
	inductive system consisting of the locally convex spaces
	$\mathscr{K}(X) \cap \{ f \with \spt f \subset K \}$ with the topology
	of uniform convergence corresponding to all compact subsets $K$ of $U$
	and its inclusion maps.

	Then $\mathscr{K}(X)$ endowed with the locally convex final topology
	with respect to the inclusions of the topological vector spaces
	$\mathscr{K} (X) \cap \{ f \with \spt f \subset K \}$ is the inductive
	limit of the above system in the category of locally convex spaces.
\end{definition}
\begin{remark} [see \protect{\cite[\printRoman{2}, p.~29, Example
\printRoman{2}]{MR910295}}]
	The locally convex topology on $\mathscr{K} (X)$ is Hausdorff and
	induces the given topology on each closed subspace $\mathscr{K} ( X)
	\cap \{ f \with \spt f \subset K \}$. Moreover, the space
	$\mathscr{K}(X)^\ast$ of Daniell integrals on $\mathscr{K}(X)$ agrees
	with the space of continuous linear functionals on $\mathscr{K} (X)$.
\end{remark}
\begin{remark}
	If $K(i)$ is a sequence of compact subsets of $X$ with $K(i) \subset
	\Int K(i+1)$ for $i \in \nat$ and $X = \bigcup_{i=1}^\infty K(i)$,
	then $\mathscr{K}(X)$ is the strict inductive limit of the sequence of
	locally convex spaces $\mathscr{K}(X) \cap \{ f \with \spt f \subset
	K(i) \}$.
\end{remark}
\begin{definition} \label{def:duy}
	Suppose $U$ is an open subset of $\rel^\adim$ and $Z$ is a Banach
	space. Consider the inductive system consisting of the locally convex
	spaces $\mathscr{D}_K (U,Z)$ with the topology induced from
	$\mathscr{E}(U,Z)$ corresponding to all compact subsets $K$ of $U$ and
	its inclusion maps.

	Then $\mathscr{D}(U,Z)$ endowed with the locally convex final topology
	with respect to the inclusions of the topological vector spaces
	$\mathscr{D}_K (U,Z)$ is the inductive limit of the above inductive
	system in the category of locally convex spaces.
\end{definition}
\begin{remark} \label{remark:basic}
	The locally convex topology on $\mathscr{D} (U,Z)$ is Hausdorff and
	induces the given topology on each closed subspace
	$\mathscr{D}_K(U,Z)$.
\end{remark}
\begin{remark} \label{remark:ind-limit}
	If $K(i)$ is a sequence of compact subsets of $U$ with $K(i) \subset
	\Int K(i+1)$ for $i \in \nat$ and $U = \bigcup_{i=1}^\infty K(i)$,
	then $\mathscr{D} (U,Z)$ is the strict inductive limit of the sequence
	of locally convex spaces $\mathscr{D}_{K(i)} ( U,Z )$. In particular,
	the convergent sequences in $\mathscr{D}(U,Z)$ are precisely the
	convergent sequences in some $\mathscr{D}_K(U,Z)$, see
	\cite[\printRoman{3}, p.~3, prop.~2; \printRoman{2}, p.~32,
	prop.~9\,(i)\,(ii); \printRoman{3}, p.~5, prop.~6]{MR910295}.
\end{remark}
\begin{remark} \label{remark:tvs}
	The locally convex topology on $\mathscr{D} (U,Z)$ is also the
	inductive limit topology in the category of topological vector spaces,
	see \cite[\printRoman{1}, p.~9, Lemma to prop.~7; \printRoman{2},
	p.~75, exerc.~14]{MR910295}.
\end{remark}
\begin{remark} \label{remark:top_comp}
	Consider the inductive system consisting of the topological spaces
	$\mathscr{D}_K (U,Z)$ corresponding to all compact subsets $K$ of $U$
	and its inclusion maps. Then $\mathscr{D}(U,Z)$ endowed with the final
	topology with respect to the inclusions of the topological vector
	spaces $\mathscr{D}_K (U,Z)$ is the inductive limit of this inductive
	system in the category of topological spaces.

	Denoting this topology by $S$ and the topology described in
	\ref{def:duy} by $T$, the following three statements hold.
	\begin{enumerate}
		\item Amongst all locally convex topologies on
		$\mathscr{D}(U,Z)$, the topology $T$ is characterised by the
		following property: A linear map from $\mathscr{D}(U,Z)$ into
		some locally convex space is $T$ continuous if and only if it
		is $S$ continuous.
		\item The $S$ closed sets are precisely the $T$ sequentially
		closed sets, see \ref{remark:ind-limit}.
		\item If $U$ is nonempty and $\dim Z > 0$, then $S$ is
		strictly finer than $T$, compare Shirai \cite[Th{\'e}or{\`e}me
		5]{MR0105613}.\footnote{The topological preliminaries of
		\cite{MR0105613} are quite general. For the purpose of
		verifying only the cited result, one may replace the terms
		\guillemotleft topologie\guillemotright{} [respectively
		\guillemotleft v{\'e}ritable-topologie\guillemotright{}] with
		operators satisfying conditions (a), (b) and (d) [respectively
		all] of the Kuratowski closure axioms, see
		\cite[p.~43]{MR0370454}. (Extending \cite[Th{\'e}or{\`e}me
		5]{MR0105613} which treats the case $U = \rel^\adim$ and $Z =
		\rel$ to the case stated poses no difficulty.)} In particular,
		in this case $S$ is not compatible with the vector space
		structure of $\mathscr{D}(U,Z)$ by \ref{remark:tvs}.
	\end{enumerate}
\end{remark}
\begin{definition} \label{def:variation_measure} % \label{miniremark:extension}
	Suppose $U$ is an open subset of $\rel^\adim$, $Z$ is a separable
	Banach space, and $S : \mathscr{D} ( U,Z ) \to \rel$ is linear.

	Then $\| S \|$ is defined to be the largest Borel regular measure over
	$U$ such that
	\begin{gather*}
		\| S \| (G) = \sup \{ S ( \theta ) \with \text{$\theta \in
		\mathscr{D} (U,Z)$ with $\spt \theta \subset G$ and $| \theta
		(x) | \leq 1$ for $x \in U$} \}
	\end{gather*}
	whenever $G$ is an open subset of $U$.\footnote{Existence may be shown
	by use of the techniques occurring in \cite[2.5.14]{MR41:1976}.}
\end{definition}
\begin{remark} \label{remark:allard_vs_federer}
	This concept agrees with \cite[4.1.5]{MR41:1976} in case $\| S \|$ is
	a Radon measure (which is equivalent to $S$ being a distribution in
	$U$ of type $Z$ representable by integration) in which case
	$S(\theta)$ continuous to denote the value of the unique $\|
	S\|_{(1)}$ continuous extension of $S$ to $\Lp{1} ( \| S \|, Z )$ at
	$\theta \in \Lp{1} ( \| S \|, Z )$. In general the concepts differ
	since it may happen that $| S ( \theta ) |
	> \| S \| ( |\theta|)$ for some $\theta \in \mathscr{D} ( U, Z )$; an
	example is provided by taking $U = Z = \rel$ and $S = \partial (
	\boldsymbol{\delta}_0 \wedge \mathbf{e}_1 )$, see
	\cite[4.1.16]{MR41:1976}.
\end{remark}
\begin{definition} \label{def:restriction_distribution}
	Suppose $U$ is an open subset of $\rel^\adim$, $Z$ is a separable
	Banach space, $S \in \mathscr{D}'(U,Z)$ is representable by
	integration, and $A$ is $\| S \|$ measurable.

	Then the \emph{restriction of $S$ to $A$}, $S \restrict A \in
	\mathscr{D}'( U,Z )$, is defined by
	\begin{gather*}
		( S \restrict A ) ( \theta ) = S ( \theta_A ) \quad
		\text{whenever $\theta \in \mathscr{D} (U,Z)$},
	\end{gather*}
	where $\theta_A (x) = \theta (x)$ for $x \in A$ and $\theta_A (x) = 0$
	for $x \in U \without A$.
\end{definition}
\begin{remark}
	This extends the notion for currents in \cite[4.1.7]{MR41:1976}.
\end{remark}
\begin{theorem}% \label{thm:strict-inductive-limits}
	Suppose $E$ is the strict inductive limit of an increasing sequence of
	locally convex spaces $E_i$ and the spaces $E_i$ are separable,
	complete (with respect to its topological vector space structure) and
	metrisable for $i \in \nat$.\footnote{That is, $E_i$ are separable
	Fr{\'e}chet spaces in the terminology of \cite[\printRoman{2},
	p.~24]{MR910295}.}

	Then $E$ and the dual $E'$ of $E$ endowed with the compact open
	topology are Lusin spaces and whenever $D$ is a dense subset of $E$
	the Borel family of $E'$ is generated by the family
	\begin{gather*}
		\big \{ E' \cap \{ u \with u(x) < t \} \with \text{$x \in D$
		and $t \in \rel$} \big \}.
	\end{gather*}
\end{theorem}
\begin{proof}
	First, note that $E_i$ is a Lusin space for $i \in \nat$, since it may
	be metrised by a translation invariant metric by \cite[\printRoman{2},
	p.~34, prop.~11]{MR910295}. Second, note that $E$ is Hausdorff and
	induces the given topology on $E_i$ by \cite[\printRoman{2}, p.~32,
	prop.~9\,(i)]{MR910295}. Therefore $E$ is a Lusin space by
	\cite[Chapter \printRoman{2}, Corollary 2 to Theorem 5,
	p.~102]{MR0426084}. Since every compact subset $K$ of $E$ is a compact
	subset of $E_i$ for some $i$ by \cite[\printRoman{3}, p.~3, prop.~2;
	\printRoman{2}, p.~32, prop.~9\,(ii); \printRoman{3}, p.~5,
	prop.~6]{MR910295}, it follows that $E'$ is a Lusin space from
	\cite[Chapter 2, Theorem 7, p.~112]{MR0426084}. Defining the
	continuous, univalent map $\iota : E' \to \rel^D$ by $\iota (u) = u|D$
	for $u \in E'$, the initial topology on $E'$ induced by $\iota$ is a
	Hausdorff topology coarser than the compact open topology, hence their
	Borel families agree by \cite[Chapter \printRoman{2}, Corollary 2 to
	Theorem 4, p.~101]{MR0426084}.
\end{proof}
\begin{example} \label{example:kx_lusin}
	Suppose $X$ is a locally compact Hausdorff space which admits a
	countable base of its topology.

	Then $\mathscr{K}(X)$ with its locally convex topology (see
	\ref{def:kx}) and $\mathscr{K}(X)^\ast$ with its weak topology are
	Lusin spaces and the Borel family of $\mathscr{K}(X)^\ast$ is
	generated by the sets $\mathscr{K}(X)^\ast \cap \{ \mu \with \mu(f) <
	t \}$ corresponding to $f \in \mathscr{K}(X)$ and $t \in \rel$.
	Moreover, the topological space
	\begin{gather*}
		\mathscr{K}(X)^\ast \cap \{ \mu \with \text{$\mu$ is monotone}
		\}
	\end{gather*}
	is homeomorphic to a complete, separable metric space;\footnote{Such
	spaces are termed polish spaces in \cite[Chapter \printRoman{2},
	Definition 1]{MR0426084}.} in fact, choosing a countable dense subset
	$D$ of $\mathscr{K}(X)^+$, the image of its homeomorphic embedding
	into $\rel^D$ is closed.
\end{example}
\begin{example} \label{example:distrib_lusin}
	Suppose $U$ is an open subset of $\rel^\adim$ and $Z$ is a separable
	Banach space.

	Then $\mathscr{D}(U,Z)$ with its locally convex topology (see
	\ref{def:duy}) and $\mathscr{D}'(U,Z)$ with its weak topology are
	Lusin spaces and the Borel family of $\mathscr{D}'(U,Z)$ is generated
	by the sets $\mathscr{D}'(U,Z) \cap \{ S \with S (\theta) < t \}$
	corresponding to $\theta \in \mathscr{D}(U,Z)$ and $t \in \rel$.
	Moreover, recalling \ref{remark:lusin} and \cite[4.1.5]{MR41:1976}, it
	follows that
	\begin{gather*}
		\big ( \mathscr{D}'(U,Z) \times \mathscr{K}(U)^\ast \big )
		\cap \{ ( S, \| S \| ) \with \text{$S$ is representable by
		integration} \}
	\end{gather*}
	is a Borel function whose domain is a Borel set (with respect to the
	weak topology on both spaces).
\end{example}
\section{Distributions on products} \label{sec:distributions}
The purpose of the present section is to separate functional analytic
considerations from those employing properties of the varifold or the weakly
differentiable function in deriving the various coarea type formulae in
\ref{lemma:meas_fct}, \ref{thm:tv_coarea}, \ref{thm:ap_coarea}, and
\ref{corollary:coarea}.
\begin{miniremark} \label{miniremark:distrib_on_products}
	Suppose $U$ and $V$ are open subsets of Euclidean spaces and
	$Z$ is a Banach space. Then the image of $\mathscr{D} (U,\rel) \otimes
	\mathscr{D} ( V,\rel ) \otimes Z$ in $\mathscr{D} ( U \times V, Z )$
	is sequentially dense, compare \cite[1.1.3, 4.1.2, 4.1.3]{MR41:1976}.
\end{miniremark}
\begin{miniremark} \label{miniremark:distrib_on_r}
	If $J$ is an open subset of $\rel$, $R \in \mathscr{D}' ( J, \rel )$
	is representable by integration, and $\| R\|$ is absolutely continuous
	with respect to $\mathscr{L}^1 | \mathbf{2}^J$, then there exists $k
	\in \Lploc{1} ( \mathscr{L}^1 | \mathbf{2}^J )$ such that
	\begin{gather*}
		R ( \omega ) = \tint{J}{} \omega k \ud \mathscr{L}^1 \quad
		\text{whenever $\omega \in \Lp{1} ( \| R \| )$}, \\
		k(y) = \lim_{\varepsilon \to 0+} \varepsilon^{-1} R
		(i_{y,\varepsilon}) \quad\text{for $\mathscr{L}^1$ almost all
		$y \in J$},
	\end{gather*}
	where $i_{y,\varepsilon}$ is the characteristic function of the
	interval $\{ \upsilon \with y < \upsilon \leq y + \varepsilon \}$; in
	fact, localising the problem, one employs \cite[2.5.8]{MR41:1976} to
	construct $k$ satisfying the first condition which implies the second
	one by \cite[2.8.17, 2.9.8]{MR41:1976}.
\end{miniremark}
\begin{lemma} \label{lemma:push_on_product}
	Suppose $\adim \in \nat$, $\mu$ is a Radon measure over an open subset
	$U$ of $\rel^\adim$, $J$ is an open subset of $\rel$, $f$ is a $\mu$
	measurable real valued function, $A = f^{-1} \lIm J \rIm$, $F$ is a
	$\mu$ measurable $\Hom ( \rel^\adim, \rel )$ valued function with
	\begin{gather*}
		\tint{K \cap \{ x \with f (x) \in I \}}{} |F| \ud \mu <
		\infty
	\end{gather*}
	whenever $K$ is a compact subset of $U$ and $I$ is a compact subset of
	$J$, and $T \in \mathscr{D}' ( U \times J, \rel^\adim )$ satisfies
	\begin{gather*}
		T ( \phi ) = \tint A{} \left < \phi (x,f(x)), F(x) \right >
		\ud \mu x \quad \text{for $\phi \in \mathscr{D} ( U \times
		J, \rel^\adim )$}
	\end{gather*}

	Then $T$ is representable by integration and
	\begin{gather*}
		T ( \phi ) = \tint A{} \left < \phi (x,f(x)), F(x) \right >
		\ud \mu x, \quad \tint{}{} g \ud \| T \| = \tint A{} g(x,f(x))
		| F(x)| \ud \mu x
	\end{gather*}
	whenever $\phi \in \Lp{1} ( \| T \|, \rel^\adim )$ and $g$ is an
	$\overline{\rel}$ valued $\| T \|$ integrable function.
\end{lemma}
\begin{proof}
	Define $p : U \times J \to U$ by
	\begin{gather*}
		p (x,y) = x \quad \text{for $(x,y) \in U \times J$}.
	\end{gather*}
	Define the measure $\nu$ over $U$ by $\nu (B) = \tint{A \cap B}{\ast}
	|F| \ud \mu$ for $B \subset U$. Let $G : A \to U \times J$ be defined
	by $G(x) = (x,f(x))$ for $x \in A$. Employing \cite[2.2.2, 2.2.3,
	2.4.10]{MR41:1976} yields that $\nu \restrict \{ x \with f(x) \in I
	\}$ is a Radon measure whenever $I$ is a compact subset of $J$, hence
	so is the measure $G_\# (\nu|\mathbf{2}^A)$ over $U \times J$ by
	\cite[2.2.2, 2.2.3, 2.2.17, 2.3.5]{MR41:1976}. One deduces that
	\begin{gather*}
		T(\phi) = \tint{}{} \left < \phi, |F \circ p |^{-1} F \circ p
		\right > \ud G_\# (\nu|\mathbf{2}^A) \quad \text{for $\phi \in
		\mathscr{D} (U \times J, \rel^\adim)$};
	\end{gather*}
	in fact, $F | A = F \circ p \circ G$, hence both sides equal
	$\tint{}{} \left < \phi \circ G, |F|^{-1} F \right > \ud \nu$ by
	\cite[2.4.10, 2.4.18\,(1)]{MR41:1976}. Consequently, $\| T \| = G_\# (
	\nu|\mathbf{2}^A)$ and the conclusion follows by means of \cite[2.2.2,
	2.2.3, 2.4.10, 2.4.18\,(1)]{MR41:1976}.
\end{proof}
\begin{theorem} \label{thm:distribution_on_product}
	Suppose $U$ is an open subset of $\rel^\adim$, $J$ is an open subset
	of $\rel$, $Z$ is a separable Banach space, $T \in \mathscr{D}' ( U
	\times J, Z)$ is representable by integration, $R_\theta \in
	\mathscr{D}' ( J, \rel)$ satisfy
	\begin{gather*}
		R_\theta ( \omega ) = T_{(x,y)} ( \omega (y) \theta (x) )
		\quad \text{whenever $\omega \in \mathscr{D} ( J, \rel )$ and
		$\theta \in \mathscr{D} (U,Z )$}.
	\end{gather*}
	and $S(y) : \mathscr{D} ( U, Z ) \to \rel$ satisfy, see
	\ref{miniremark:distrib_on_r},
	\begin{gather*}
		S(y)( \theta ) = \lim_{\varepsilon \to 0+} \varepsilon^{-1}
		R_\theta ( i_{y,\varepsilon} ) \in \rel \quad \text{for
		$\theta \in \mathscr{D} (U, Z )$}
	\end{gather*}
	whenever $y \in J$, that is $y \in \dmn S$ if and only if the limit
	exists and belongs to $\rel$ for $\theta \in \mathscr{D} ( U, Z )$.

	Then $S$ is an $\mathscr{L}^1 \restrict J$ measurable function with
	respect to the weak topology on $\mathscr{D}'(U,Z)$ and the following
	two statements hold.
	\begin{enumerate}
		\item \label{item:distribution_on_product:inequality} If $g$
		is an $\{ u \with 0 \leq u \leq \infty \}$ valued $\| T \|$
		measurable function, then
		\begin{gather*}
			\tint{J}{} \tint{}{} g(x,y) \ud \| S(y) \| x \ud
			\mathscr{L}^1 y \leq \tint{}{} g \ud \| T \|.
		\end{gather*}
		\item \label{item:distribution_on_product:absolute} If $\|
		R_\theta \|$ is absolutely continuous with respect to
		$\mathscr{L}^1 | \mathbf{2}^J$ for $\theta \in
		\mathscr{D}(U,Z)$, then
		\begin{gather*}
			T ( \phi ) = \tint{J}{} S(y) ( \phi (\cdot, y ) ) \ud
			\mathscr{L}^1 y, \quad \tint{J}{} g \ud \| T \| =
			\tint{J}{} \tint{}{} g (x,y) \ud \| S(y) \| x \ud
			\mathscr{L}^1 y.
		\end{gather*}
		whenever $\phi \in \Lp{1} ( \| T \|, Z )$ and $g$ is
		an $\overline{\rel}$ valued $\| T \|$ integrable function.
	\end{enumerate}
\end{theorem}
\begin{proof}
	Define $p : U \times J \to \rel$ and $q : U \times J \to U$ by
	\begin{gather*}
		p (x,y) = x \quad \text{and} \quad q (x,y) = y \quad \text{for
		$(x,y) \in U \times J$}.
	\end{gather*}
	First, one derives that $S$ is an $\mathscr{L}^1 \restrict J$
	measurable function with values in
	\begin{gather*}
		\mathscr{D}' (U,Z) \cap \{ \Sigma \with \text{$\Sigma$ is
		representable by integration} \},
	\end{gather*}
	where the weak topology on $\mathscr{D}'(U,Z)$ is employed; in fact,
	choosing countable, sequentially dense subsets $C$ and $D$ of
	$\mathscr{K} (U)^+$ and $\mathscr{D}(U,Z)$ respectively (see
	\ref{remark:lusin}, \ref{example:kx_lusin}, and
	\ref{example:distrib_lusin}) and noting that $S(y)$ belongs to set in
	question whenever
	\begin{gather*}
		\lim_{\varepsilon \to 0+} \varepsilon^{-1} \| T \| ( ( f \circ
		p ) ( i_{y,\varepsilon} \circ q ) ) \in \rel, \quad
		\lim_{\varepsilon \to 0+} \varepsilon^{-1} R_\theta
		(i_{y,\varepsilon}) \in \rel
	\end{gather*}
	for $f \in C$ and $\theta \in D$ by means of \ref{remark:ind-limit},
	the assertion follows from \cite[2.9.19]{MR41:1976} and
	\ref{example:distrib_lusin}.

	In order to prove \eqref{item:distribution_on_product:inequality}, one
	may assume $g \in \mathscr{K} (U \times J)^+$. Suppose $f \in
	\mathscr{K} (U)^+$, $h \in \mathscr{K} ( J)^+$, and $g = ( f \circ
	p ) ( h \circ q )$ and $\beta$ denotes the Radon measure over $U \times
	J$ defined by $\beta (B) = \tint{B}{\ast} f\circ p \ud \| T \|$ for
	$B \subset U \times J$. Noting
	\begin{align*}
		& \| S (y) \| ( f ) \leq \mathbf{D} ( q_\# \beta,
		\mathscr{L}^1|\mathbf{2}^J, V, y ) \quad \text{for
		$\mathscr{L}^1$ almost all $y \in J$}, \\
		& \qquad \text{where $V = \{ (b,I) \with \text{$b \in I
		\subset J$ and $I$ is a compact interval} \}$}
	\end{align*}
	and employing \cite[2.4.10, 2.8.17, 2.9.7]{MR41:1976} and
	\ref{example:distrib_lusin}, one infers that
	\begin{align*}
		\tint{J}{} \| S (y) \| ( f ) h(y) \ud \mathscr{L}^1 y & \leq
		\tint{J}{} \mathbf{D} ( q_\# \beta, \mathscr{L}^1|\mathbf{2}^J,
		V, y ) h (y) \ud \mathscr{L}^1 y \\
		& \leq ( q_\# \beta ) ( h ) = \| T \| (g).
	\end{align*}
	An arbitrary $g \in \mathscr{K} (U \times J)^+$ may be approximated
	by a sequence of functions which are nonnegative linear combinations
	of functions of the previously considered type, compare \cite[4.1.2,
	4.1.3]{MR41:1976}.

	To prove the first equation in
	\eqref{item:distribution_on_product:absolute}, it is sufficient to
	exhibit a sequentially dense subset $E$ of $\mathscr{D} (U \times
	J, Z )$ such that
	\begin{gather*}
		T ( \phi ) = \tint{J}{} S(y) (\phi(\cdot,y)) \ud \mathscr{L}^1
		y \quad \text{whenever $\phi \in E$}
	\end{gather*}
	by \eqref{item:distribution_on_product:inequality} and
	\ref{remark:ind-limit}. If $\theta \in \mathscr{D} (U,Z)$, then
	\begin{gather*}
		T_{(x,y)} ( \omega (y) \theta (x) ) = R_\theta ( \omega ) =
		\tint{J}{} S(y)_x ( \omega(y) \theta (x) ) \ud \mathscr{L}^1 y
	\end{gather*}
	for $\omega \in \mathscr{D} ( J, \rel )$ by
	\ref{miniremark:distrib_on_r}. One may now take $E$ to be the image of
	$\mathscr{D} ( J, \rel ) \otimes \mathscr{D} (U,Z)$ in $\mathscr{D}
	( U \times J, Z )$, see \ref{miniremark:distrib_on_products}.

	The second equation in \eqref{item:distribution_on_product:absolute}
	follows from \eqref{item:distribution_on_product:inequality} and the
	first equation in \eqref{item:distribution_on_product:absolute}.
\end{proof}
\section{Monotonicity identity} \label{sec:monotonicity}
The purpose of this section is to derive the modifications and consequences of
the monotonicity identity which will be employed in \ref{thm:decomposition},
\ref{thm:zero_derivative}, \ref{thm:mod_continuity}, and
\ref{thm:conn_path_finite_length}.
\begin{lemma} \label{lemma:calculus}
	Suppose $1 \leq \vdim < \infty$, $0 \leq s < r < \infty$, $0 \leq
	\kappa < \infty$, $I = \{ t \with s < t \leq r \}$, and $f : I \to \{
	y : 0 < y < \infty \}$ is a function satisfying
	\begin{gather*}
		\limsup_{u \to t-} f(u) \leq f (t) \leq f(r) + \kappa
		f(r)^{1/\vdim} \tint{t}{r} u^{-1} f (u)^{1-1/\vdim} \ud
		\mathscr{L}^1 u \quad \text{for $t \in I$}.
	\end{gather*}

	Then there holds
	\begin{gather*}
		f (t) \leq \big ( 1 + \vdim^{-1} \kappa \log ( r/t ) \big
		)^\vdim f (r) \quad \text{for $t \in I$}.
	\end{gather*}
\end{lemma}
\begin{proof}
	Suppose $f (r) < y < \infty$ and $\upsilon = \vdim^{-1} \kappa
	f(r)^{1/\vdim} y^{-1/\vdim}$ and consider the set $J$ of all $t \in I$
	such that
	\begin{gather*}
		f (u) \leq ( 1 + \upsilon \log ( r/u ) )^\vdim y \quad
		\text{whenever $t \leq u \leq r$}.
	\end{gather*}
	Clearly, $J$ is an interval and $r$ belongs to the interior of $J$
	relative to $I$. The same holds for $t$ with $s < t \in \Clos J$
	since
	\begin{gather*}
		\begin{aligned}
			f ( t ) & \leq f ( r ) + \kappa f(r)^{1/\vdim}
			y^{1-1/\vdim} \tint{t}{r} u^{-1} ( 1 + \upsilon \log (
			r/u))^{\vdim-1} \ud \mathscr{L}^1 u \\
			& = f ( r ) + \big ( ( 1 + \upsilon \log ( r/t))^\vdim
			- 1 \big ) y < ( 1 + \upsilon \log ( r/t ))^\vdim y.
		\end{aligned}
	\end{gather*}
	Therefore $I$ equals $J$.
\end{proof}
\begin{theorem} \label{thm:monotonicity}
	Suppose $\vdim, \adim \in \nat$, $\vdim \leq \adim$, $a \in
	\rel^\adim$, $0 < r < \infty$, $V \in \Var_\vdim ( \oball{a}{r}
	)$, and $\varrho \in \mathscr{D} ( \{ s \with 0 < s < r \}, \rel )$.

	Then there holds
	\begin{multline*}
		- \tint{0}{r} \varrho'(s) s^{-\vdim} \measureball{\| V \|}{
		\cball{a}{s} } \ud \mathscr{L}^1 s \\
		\begin{aligned}
			& = \tint{( \oball{a}{r} \without \{ a \} ) \times
			\grass{\adim}{\vdim} }{} \varrho ( |x-a| )
			|x-a|^{-\vdim-2} | \perpproject{P} (x-a) |^2 \ud V
			(x,P) \\
			& \quad - ( \delta V )_x \left ( \big (
			\tint{|x-a|}{r} s^{-\vdim-1} \varrho (s) \ud
			\mathscr{L}^1 s \big ) (x-a) \right ).
		\end{aligned}
	\end{multline*}
\end{theorem}
\begin{proof}
	Assume $a=0$ and let $I = \{ s \with - \infty < s < r \}$ and $J = \{
	s \with 0 < s < r \}$.
	
	If $\omega \in \mathscr{E} (I,\rel)$, $\sup \spt \omega < r$, $0
	\notin \spt \omega'$ and $\theta : \oball{a}{r} \to \rel^\adim$ is
	associated to $\omega$ by $\theta(x) = \omega ( |x| ) x$ for $x \in
	\oball{a}{r}$, then $D\theta(0) \bullet \project{P} = \vdim \omega(0)$
	and
	\begin{align*}
		D\theta(x) \bullet \project{P} & = | \project{P} (x) |^2
		|x|^{-1} \omega' (|x|) + \vdim \omega (|x|) \\
		& = - | \perpproject{P} (x) |^2 |x|^{-1} \omega'(|x|) + |x|
		\omega' (|x|) + \vdim \omega ( |x| )
	\end{align*}
	whenever $x \in \rel^\adim$, $0 < |x| < r$, and $P \in
	\grass{\adim}{\vdim}$. Define $\omega, \gamma \in \mathscr{E} (I,\rel)$
	by
	\begin{gather*}
		\omega (s) = - \tint{\sup \{ s, 0 \}}{r} u^{-\vdim-1} \varrho(u)
		\ud \mathscr{L}^1 u \quad \text{and} \quad \gamma (s) = s
		\omega'(s) + \vdim \omega(s)
	\end{gather*}
	for $s \in I$, hence $\sup \spt \gamma \leq \sup \spt \omega < r$, $0
	\notin \spt \omega'$, and
	\begin{gather*}
		\omega'(s) = s^{-\vdim-1} \varrho(s), \quad \omega''(s) = -
		(\vdim+1) s^{-\vdim-2} \varrho(s) + s^{-\vdim-1} \varrho'(s)
	\end{gather*}
	for $s \in J$. Using Fubini's theorem, one computes with $\theta$ as
	before that
	\begin{gather*}
		\begin{aligned}
			& \delta V (\theta) + \tint{( \rel^\adim \times
			\grass{\adim}{\vdim}) \cap \{ (x,P) \with 0 < |x| < r
			\}}{} \varrho (|x|) |x|^{-\vdim-2} | \perpproject{P} (x)
			|^2 \ud V (x,P) \\
			& \quad = \tint{}{} \gamma (|x|) \ud \| V \| x = -
			\tint{}{} \tint{|x|}{r} \gamma '(s) \ud \mathscr{L}^1 s
			\ud \| V \| x = - \tint{0}{r} \gamma'(s) \measureball{\|
			V \|}{\cball{a}{s}} \ud \mathscr{L}^1 s.
		\end{aligned}
	\end{gather*}
	Finally, notice that $\gamma'(s) = s \omega''(s) + ( \vdim+1) \omega'
	(s) = s^{-\vdim} \varrho'(s)$ for $s \in J$.
\end{proof}
\begin{remark}
	This is a slight generalisation of Simon's version of the monotonicity
	identity, see \cite[17.3]{MR756417}, included here for the convenience
	of the reader.
\end{remark}
\begin{miniremark} \label{miniremark:situation_general_varifold}
	Suppose $\vdim, \adim \in \nat$, $\vdim \leq \adim$, $U$ is an
	open subset of $\rel^\adim$, $V \in \Var_\vdim ( U )$, $\| \delta V
	\|$ is a Radon measure, and $\eta ( V, \cdot )$ is a $\| \delta V \|$
	measurable $\mathbf{S}^{\adim-1}$ valued function satisfying
	\begin{gather*}
		( \delta V ) (\theta) = \tint{}{} \eta ( V, x ) \bullet
		\theta(x) \ud \| \delta V \| x \quad \text{for $\theta \in
		\mathscr{D} ( U, \rel^\adim )$},
	\end{gather*}
	see Allard \cite[4.3]{MR0307015}.
\end{miniremark}
\begin{corollary} \label{corollary:monotonicity}
	Suppose $\vdim$, $\adim$, $U$, $V$, and $\eta$ are as in
	\ref{miniremark:situation_general_varifold}.

	Then there holds
	\begin{multline*}
		s^{-\vdim} \measureball{\| V \|}{ \cball as } + \tint{( \cball
		ar \without \cball as )\times \grass{\adim}{\vdim}}{} | x-a
		|^{-\vdim-2} | \perpproject{P} (x-a) |^2 \ud V(x,P) \\
		\begin{aligned}
			& = r^{-\vdim} \measureball{\| V \|}{ \cball ar } \\
			& \quad + \vdim^{-1} \tint{\cball{a}{r}}{} ( \sup \{
			|x-a|,s \}^{-\vdim}- r^{-\vdim}) (x-a) \bullet \eta
			(V,x) \ud \| \delta V \| x
		\end{aligned}
	\end{multline*}
	whenever $a \in \rel^\adim$, $0 < s \leq r < \infty$, and $\cball ar
	\subset U$.
\end{corollary}
\begin{proof}
	Letting $\zeta$ approximate the characteristic function of $\{ t \with
	s < t \leq r \}$, the assertion is a consequence of
	\ref{thm:monotonicity}.
\end{proof}
\begin{remark} \label{remark:monotonicity}
	Using Fubini's theorem, the last summand can be expressed as
	\begin{gather*}
		\tint{s}{r} t^{-\vdim-1} \tint{\cball at}{} (x-a) \bullet \eta
		(V,x) \ud \| \delta V \| x \ud \mathscr{L}^1 t.
	\end{gather*}
\end{remark}
\begin{remark}
	\ref{corollary:monotonicity} and \ref{remark:monotonicity} are
	a minor variations of Simon \cite[17.3, 17.4]{MR756417}.
\end{remark}
\begin{corollary} \label{corollary:density_1d}
	Suppose $\vdim$, $\adim$, $U$, and $V$ are as in
	\ref{miniremark:situation_general_varifold}, $\vdim = 1$, and $X = U
	\cap \{ a \with \| \delta V \| ( \{ a \} ) > 0 \}$.

	Then the following three statements hold.
	\begin{enumerate}
		\item \label{item:density_1d:upper_bound} If $a \in
		\rel^\adim$, $0 < s \leq r < \infty$, and $\cball{a}{r} \subset
		U$, then
		\begin{gather*}
			s^{-1} \measureball{\| V \|}{ \cball as } +
			\tint{(\cball{a}{r} \without \cball as ) \times
			\grass{\adim}{\vdim}}{} |x-a|^{-3} | \perpproject{P}
			(x-a) |^2 \ud V (x,P) \\
			\leq r^{-1} \measureball{\| V \|}{\cball{a}{r}} + \|
			\delta V \| ( \cball{a}{r} \without \{ a \} )
		\end{gather*}
		\item \label{item:density_1d:real} $\density^1 ( \| V
		\|, \cdot )$ is a real valued function whose domain is $U$.
		\item \label{item:density_1d:upper_semi} $\density^1 ( \| V
		\|, \cdot )$ is upper semicontinuous at $a$ whenever $a \in U
		\without X$.
	\intertextenum{If additionally $\density^1 ( \| V \|, x ) \geq 1$ for
	$\| V \|$ almost all $x$, then the following two statements hold.}
		\item \label{item:density_1d:lower_bound} If $a \in \spt \| V
		\|$, then
		\begin{gather*}
			\density^1 ( \| V \|, a ) \geq 1 \quad \text{if $a
			\not \in X$} \qquad \text{and} \qquad \density^1 ( \| V
			\|, a) \geq 1/2 \quad \text{if $a \in X$}.
		\end{gather*}
		\item \label{item:density_1d:local_lower_bound} If $a \in \spt
		\| V \|$, $0 < s \leq r < \infty$, $\oball ar \subset U$, and
		$\| \delta V \| ( \oball ar \without \{ a \} ) \leq
		\varepsilon$, then
		\begin{gather*}
			\measureball{ \| V \|}{ \oball xs } \geq 2^{-1}
			(1-\varepsilon) s \quad \text{whenever $x \in \spt \|
			V \|$ and $|x-a| + s \leq r$}.
		\end{gather*}
	\end{enumerate}
\end{corollary}
\begin{proof}
	If $a \in U$, $0 < s < r < \infty$ and $\cball ar \subset U$, then
	\begin{gather*}
		\big | \sup \{ |x-a|, s \}^{-1} - r^{-1} \big | |x-a| \leq 1
		\quad \text{whenever $x \in \cball ar$}.
	\end{gather*}
	Therefore \eqref{item:density_1d:upper_bound} follows from
	\ref{corollary:monotonicity}. \eqref{item:density_1d:upper_bound}
	readily implies \eqref{item:density_1d:real} and
	\eqref{item:density_1d:upper_semi} and the first half of
	\eqref{item:density_1d:lower_bound}. To prove the second half, choose
	$\eta$ as in \ref{miniremark:situation_general_varifold} and consider
	$a \in X$. One may assume $a = 0$ and in view of Allard
	\cite[4.10\,(2)]{MR0307015} also $U = \rel^\adim$. Abbreviating $v =
	\eta (V,0) \in \mathbf{S}^{\adim-1}$ and defining the reflection $f :
	\rel^\adim \to \rel^\adim$ by $f(x) = x- 2 (x \bullet v) v$ for $x \in
	\rel^\adim$, one infers the second half of
	\eqref{item:density_1d:lower_bound} by applying the first half of
	\eqref{item:density_1d:lower_bound} to the varifold $V + f_\# V$.

	If $a$, $s$, $r$, $\varepsilon$ and $x$ satisfy the conditions of
	\eqref{item:density_1d:local_lower_bound}, then
	\begin{gather*}
		\measureball{ \| V \|}{\oball xs } \geq (1-\varepsilon) s
		\quad \text{if either $s \leq |x-a|$ or $x=a$}, \\
		\measureball{ \| V \|}{ \oball xs } \geq \measureball{ \| V
		\|}{ \oball a{s/2} } \geq 2^{-1} (1-\varepsilon) s \quad
		\text{if $2 |x-a| \leq s$}
	\end{gather*}
	by \eqref{item:density_1d:upper_bound} and
	\eqref{item:density_1d:lower_bound}, the case $|x-a| < s < 2|x-a|$
	then follows.
\end{proof}
\begin{corollary} \label{corollary:density_ratio_estimate}
	Suppose $\vdim$, $\adim$, $U$, and $V$ are as in
	\ref{miniremark:situation_general_varifold}, $a \in \rel^\adim$, $0 <
	s < r < \infty$, $\cball{a}{r} \subset U$, $0 \leq \kappa < \infty$,
	and
	\begin{gather*}
		\measureball{\| \delta V \|}{\cball{a}{t}} \leq \kappa r^{-1}
		\| V \| ( \cball{a}{r} )^{1/\vdim} \| V \| ( \cball {a}{t}
		)^{1-1/\vdim} \quad \text{for $s < t < r$},
	\end{gather*}
	where $0^0=1$.

	Then there holds
	\begin{gather*}
		s^{-\vdim} \measureball{\| V \|}{\cball{a}{s}} \leq \big ( 1 +
		\vdim^{-1} \kappa \log ( r/s) \big )^\vdim r^{-\vdim}
		\measureball{\| V \|}{\cball{a}{r}}.
	\end{gather*}
\end{corollary}
\begin{proof}
	Assume $\measureball{\| V \|}{ \cball ar} > 0$ and define $t = \inf
	\{ u \with \measureball{\| V \|}{ \cball au} > 0 \}$ and $f (u) =
	u^{-\vdim} \measureball{\| V \|}{\cball au}$ for $\sup \{s,t\} < u
	\leq r$. Then, in view of \ref{corollary:monotonicity} and
	\ref{remark:monotonicity}, one may apply \ref{lemma:calculus} with $s$
	replaced by $\sup \{ s,t \}$ to infer the conclusion.
\end{proof}
\section{Distributional boundary} \label{sec:distrib_boundary}
In this section the notion of distributional boundary of a set with respect to
certain varifolds is introduced, see \ref{def:v_boundary}. Moreover, a basic
structural theorem is proven, see \ref{thm:basic_structure_v_boundary}, which
allows to compare this notion to a similar one employed by Bombieri and Giusti
in the context of area minimising currents in \cite[Theorem 2]{MR0308945}, see
\ref{remark:bombieri_giusti}.
\begin{definition} \label{def:v_boundary}
	Suppose $\vdim, \adim \in \nat$, $\vdim \leq \adim$, $U$ is an open
	subset of $\rel^\adim$, $V \in \Var_\vdim ( U )$, $\| \delta V \|$ is
	a Radon measure, and $E$ is $\| V \| + \| \delta V \|$ measurable.

	Then the \emph{distributional $V$ boundary of $E$} is given by (see
	\ref{def:restriction_distribution})
	\begin{gather*}
		\boundary{V}{E} = ( \delta V ) \restrict E - \delta ( V
		\restrict E \times \grass{\adim}{\vdim} ) \in \mathscr{D}' (
		U, \rel^\adim ).
	\end{gather*}
\end{definition}
\begin{remark} \label{remark:boundary_of_sum}
	If $W \in \Var_\vdim ( U )$, $\| \delta W \|$ is a Radon measure and
	$E$ is additionally $\| W \| + \| \delta W \|$ measurable, then
	\begin{gather*}
		\boundary {(V+W)}E = \boundary VE + \boundary WE.
	\end{gather*}
\end{remark}
\begin{remark} \label{remark:v_boundary}
	If $E$ and $F$ are $\| V \| + \| \delta V \|$ measurable sets and $E
	\subset F$, then
	\begin{gather*}
		\boundary{V}{(F \without E)} = ( \boundary{V}{F} ) - (
		\boundary{V}{E} ).
	\end{gather*}
\end{remark}
\begin{remark} \label{remark:iterated_boundaries}
	If $\boundary VE$ is representable by integration, $W = V \restrict E
	\times \grass \adim \vdim$, and $F$ is a Borel set, then
	\begin{gather*}
		\boundary V{(E \cap F)} = \boundary WF + ( \boundary VE )
		\restrict F.
	\end{gather*}
\end{remark}
\begin{remark} \label{remark:partition}
	If $G$ is a countable, disjointed collection of $\| V \| + \| \delta V
	\|$ measurable sets with $\boundary{V}{E} = 0$ for $E \in G$ and $\| V
	\| ( U \without \bigcup G ) = 0$, then
	\begin{gather*}
		{\textstyle \| \delta V \| ( U \without \bigcup G ) = 0};
	\end{gather*}
	in fact $\delta V = \sum_{E \in G} \delta
	( V \restrict E \times \grass{\adim}{\vdim} ) = \sum_{E \in G} (
	\delta V) \restrict E = ( \delta V ) \restrict \bigcup G$.
\end{remark}
\begin{example} \label{example:two_crossing_lines}
	Suppose $E$ and $P$ are distinct members of $\grass 2 1$ and $V \in
	\IVar_1 ( \rel^2 )$ is characterised by $\| V \| = \mathscr{H}^1
	\restrict ( E \cup P )$.

	Then $\delta V =0$ and $\boundary VE = 0$ but there exists no sequence
	of locally Lipschitzian functions $f_i : \rel^2 \to \rel$ satisfying
	\begin{gather*}
		\tint{}{} |f_i-f| + | ( \| V \|, 1 ) \ap D f_i | \ud
		\| V \| \to 0 \quad \text{as $i \to \infty$},
	\end{gather*}
	where $f$ is the characteristic function of $E$.
\end{example}
\begin{remark}
	Since it will follow from
	\ref{thm:tv_on_decompositions}\,\eqref{item:tv_on_decompositions:tv}
	that $f$ is a generalised weakly differentiable function with
	vanishing generalised weak derivative, the preceding example shows
	that the theory of generalised weakly differentiable functions cannot
	be developed using approximation by locally Lipschitzian functions.
	Instead, the relevant properties of sets are studied first in Sections
	\ref{sec:distrib_boundary}--\ref{sec:rel_iso} before proceeding to the
	theory of generalised weakly differentiable functions in Sections
	\ref{sec:basic}--\ref{sec:oscillation}.
\end{remark}
\begin{lemma} \label{lemma:capacity}
	Suppose $\adim \in \nat$, $1 \leq \vdim \leq \adim$, $U$ is an open
	subset of $\rel^\adim$, $\mu$ is a Radon measure over $U$, $K$ is a
	compact subset of $U$, $A$ is compact subset of $\Int K$,
	$\mathscr{H}^{\vdim-1} (A) = 0$, and
	\begin{gather*}
		\lim_{r \to 0+} \sup \{ s^{-\vdim} \measureball \mu {\cball
		xs} \with \text{$x \in A$ and $0 < s \leq r$} \} < \infty.
	\end{gather*}

	Then there exists a sequence $f_i \in \mathscr{E} ( U, \rel )$
	satisfying
	\begin{gather*}
		0 \leq f_i \leq 1, \quad A \subset \Int \{ x \with f_i (x) = 0
		\}, \quad \{ x \with f_i (x) < 1 \} \subset K \qquad \text{for
		$i \in \nat$}, \\
		\lim_{i \to \infty} f_i (x) = 1 \quad \text{for
		$\mathscr{H}^{\vdim-1}$ almost all $x \in U$}, \qquad \lim_{i
		\to \infty} \tint{}{} | Df_i | \ud \mu = 0.
	\end{gather*}
\end{lemma}
\begin{proof}
	Let $\phi_\infty$ denote the size $\infty$ approximating measure for
	$\mathscr{H}^{\vdim-1}$ over $\rel^\adim$. Observe that it is
	sufficient to prove that for $\varepsilon > 0$ there exists $f \in
	\mathscr{E} (U,\rel)$ with
	\begin{gather*}
		0 \leq f \leq 1, \quad A \subset \Int \{ x \with f(x)=0 \},
		\quad \{ x \with f(x) < 1 \} \subset K, \\
		\phi_\infty ( \{ x \with f(x) < 1 \} ) < \varepsilon, \quad
		\tint{}{} |Df| \ud \mu < \varepsilon.
	\end{gather*}
	For this purpose assume $A \neq \varnothing$, denote the limit in the
	hypotheses of the lemma by $Q$, and choose $j \in \nat$ and $x_i \in
	A$, $0 < r_i < \infty$ for $i \in \{ 1, \ldots, j \}$ with
	\begin{gather*}
		\measureball \mu {\cball{x_i}{r_i}} \leq (Q+1) r_i^\vdim \quad
		\text{and} \quad \cball{x_i}{r_i} \subset K \quad \text{for $i
		\in \{1, \ldots,j \}$}, \\
		{\textstyle A \subset \bigcup_{i=1}^j \oball{x_i}{r_i/2},
		\quad \sum_{i=1}^j r_i^{\vdim-1} < \varepsilon/\Delta},
	\end{gather*}
	where $\Delta = \sup \{ \unitmeasure{\vdim-1}, 4 (Q+1) \}$. Selecting
	$f_i \in \mathscr{E} ( U, \rel )$ with $0 \leq f_i \leq 1$, $|Df_i|
	\leq 4 r_i^{-1}$, and
	\begin{gather*}
		\oball{x_i}{r_i/2} \subset \{ x \with f_i (x) = 0 \}, \quad \{
		x \with f_i (x) < 1 \} \subset \cball{x_i}{r_i}
	\end{gather*}
	whenever $i \in \{ 1, \ldots, j \}$, one may take $f = \prod_{i=1}^j
	f_i$.
\end{proof}
\begin{miniremark} [see \protect{\cite[1.7.5]{MR41:1976}}]
\label{miniremark:gamma_m}
	Suppose $\vdim, \adim \in \nat$ and $\vdim \leq \adim$. Then one
	defines the linear map $\gamma_\vdim : \Lambda_\vdim \rel^\adim \to
	\Lambda^\vdim \rel^\adim$ by the equation
	\begin{gather*}
		\left < \xi, \gamma_\vdim ( \eta ) \right > = \xi \bullet \eta
		\quad \text{whenever $\xi, \eta \in \Lambda_\vdim
		\rel^\adim$}.
	\end{gather*}
\end{miniremark}
\begin{theorem} \label{thm:basic_structure_v_boundary}
	Suppose $\vdim, \adim \in \nat$, $\vdim \leq \adim$, $U$ is an open
	subset of $\rel^\adim$, $V \in \IVar_\vdim ( U )$, $0 \leq \kappa <
	\infty$, $\| \delta V \| \leq \kappa \| V \|$, $M$ is a relatively
	open subset of $\spt \| V \|$, $\mathscr{H}^{\vdim-1} ( ( \spt \| V \|
	) \without M ) = 0$, $M$ is an $\vdim$ dimensional submanifold of
	class $2$, $\tau : M \to \Hom ( \rel^\adim, \rel^\adim )$ satisfies
	$\tau (x) = \project{\Tan ( M,x ) }$ for $x \in M$, $E$ is a $\| V \|$
	measurable set, and
	\begin{gather*}
		B = M \without \{ x \with \text{$\density^\vdim (
		\mathscr{H}^\vdim \restrict M \cap E, x ) = 0$ or
		$\density^\vdim ( \mathscr{H}^\vdim \restrict M \without E, x
		) = 0$} \}.
	\end{gather*}

	Then the following three statements hold.
	\begin{enumerate}
		\item \label{item:basic_structure_v_boundary:tang_var} If
		$\theta \in \mathscr{D} (U,\rel^\adim)$ then
		\begin{gather*}
			\boundary VE ( \theta ) = - \tint{E \cap M}{} \tau (x)
			\bullet D \left < \theta, \tau \right > (x)
			\density^\vdim ( \| V\|, x ) \ud \mathscr{H}^\vdim x.
		\end{gather*}
		\item \label{item:basic_structure_v_boundary:oriented} If
		$\xi$ is an $\vdim$ vectorfield orienting $M$ and $\gamma_\vdim$
		is as in \ref{miniremark:gamma_m}, then
		\begin{gather*}
			\boundary VE ( \theta ) = (-1)^\vdim \tint{E \cap M}{}
			\left < \xi (x), d ( \theta \mathop{\lrcorner} (
			\gamma_\vdim \circ \xi ) ) (x) \right > \density^\vdim
			( \| V \|, x ) \ud \mathscr{H}^\vdim x
		\end{gather*}
		for $\theta \in \mathscr{D} (U, \rel^\adim )$.
		\item \label{item:basic_structure_v_boundary:representable}
		The distribution $\boundary VE$ is representable by
		integration if and only if
		\begin{gather*}
			\mathscr{H}^{\vdim-1} ( K \cap B ) < \infty \quad
			\text{whenever $K$ is a compact subset of $U$};
		\end{gather*}
		in this case $B$ is $\mathscr{H}^{\vdim-1}$ almost equal to
		$\{ x \with \mathbf{n} (M;E,x) \in \mathbf{S}^{\adim-1} \}$
		and meets every compact subset of $U$ is a
		$(\mathscr{H}^{\vdim-1},\vdim-1)$ rectifiable set and
		\begin{gather*}
			\boundary VE ( \theta ) = - \tint{}{} \mathbf{n}
			(M;E,x) \bullet \theta (x) \density^\vdim ( \| V \|, x
			) \ud \mathscr{H}^{\vdim-1} x
		\end{gather*}
		for $\theta \in \mathscr{D}(U,\rel^\adim)$.
	\end{enumerate}
\end{theorem}
\begin{proof}
	Notice that the results of Allard \cite[2.5, 4.5--4.7]{MR0307015}
	stated for submanifolds of class $\infty$ have analogous formulations
	for submanifolds of class $2$. The meaning of $\Tan (M,\theta)$, $\Nor
	(M,\theta)$, $\mathbf{G}_\vdim ( M )$, and $\IVar_\vdim ( M )$ defined
	in \cite[2.5, 3.1]{MR0307015} will be extended accordingly.
	
	The following assertion will be proven. \emph{If $\xi$ is as in
	\eqref{item:basic_structure_v_boundary:oriented}, then $\phi =
	\gamma_\vdim \circ \xi$ satisfies
	\begin{gather*}
		\left < \xi, d ( \theta \mathop{\lrcorner} \phi ) (x) \right >
		= (-1)^{\vdim-1} \tau (x) \bullet D \left < \theta, \tau
		\right > (x) \quad \text{for $x \in M$}
	\end{gather*}
	whenever $\theta : U \to \rel^\adim$ is a vectorfield of class $1$;}
	in fact, assuming $\theta | M = \Tan (M,\theta)$ and noting $| \xi (x)
	| = 1$ for $x \in M$, one infers
	\begin{gather*}
		\left < \xi (x), \left < u, D \phi (x)
		\right > \right > = \xi (x) \bullet \left < u, D \xi (x)
		\right > = 0 \quad \text{for $x \in M$, $u \in \Tan (M,x)$},
	\end{gather*}
	hence for $x \in M$ one expresses $\xi (x) = u_1 \wedge \cdots \wedge
	u_\vdim$ for some orthonormal basis $u_1, \ldots, u_\vdim$ of $\Tan (
	M,x )$ and computes
	\begin{gather*}
		\begin{aligned}
			& \left < \xi(x), d ( \theta \mathop{\lrcorner} \phi )
			(x) \right > \\
			& \qquad = \tsum{i=1}{\vdim} (-1)^{i-1} \left < u_1
			\wedge \cdots \wedge u_{i-1} \wedge u_{i+1} \wedge
			\cdots \wedge u_\vdim, \left < u_i, D \theta (x)
			\right > \mathop{\lrcorner} \phi (x) \right > \\
			& \qquad = (-1)^{\vdim-1} \tsum{i=1}{\vdim} \left <
			u_i, D \theta (x) \right > \bullet u_i =
			(-1)^{\vdim-1} \tau (x) \bullet D \theta (x).
		\end{aligned}
	\end{gather*}

	Next, the case $M = \spt \| V \|$ will be considered. In this case one
	may assume $M$ to be connected. Since
	\begin{gather*}
		\delta V ( \theta ) = 0 \quad \text{whenever $\theta \in
		\mathscr{D} (U,\rel^\adim)$ and $\Nor (M,\theta) = 0$}
	\end{gather*}
	for instance by Brakke \cite[5.8]{MR485012} (or
	\cite[4.8]{snulmenn.c2}), Allard \cite[4.6\,(3)]{MR0307015} then
	implies for some $0 < \lambda < \infty$ that
	\begin{gather*}
		V(k) = \lambda \tint M{} k(x,\Tan(M,x)) \ud \mathscr{H}^\vdim
		x \quad \text{for $k \in \mathscr{K} ( U \times \grass \adim
		\vdim )$}.
	\end{gather*}
	Define $W \in \IVar_\vdim ( M )$ by $W(k) = \tint{E \times \grass
	\adim \vdim}{} k \ud V$ for $k \in \mathscr{K} ( \mathbf{G}_\vdim ( M
	) )$ and notice that
	\begin{gather*}
		\begin{aligned}
			& \boundary VE ( \theta ) = - \tint{E}{} \mathbf{h}
			(V,x) \bullet \theta (x) \ud \| V \| x - \tint{E
			\times \grass \adim \vdim}{} \project P \bullet D
			\theta (x) \ud V (x,P) \\
			& \qquad = - \tint{E \times \grass \adim \vdim}{}
			\project P \bullet ( D \Tan (M,\theta)(x) \circ
			\project P ) \ud V(x,P) = - \delta W ( \Tan (M,\theta)
			)
		\end{aligned}
	\end{gather*}
	whenever $\theta \in \mathscr{D} (U,\rel^\adim )$ by Allard
	\cite[2.5\,(2)]{MR0307015}.
	\eqref{item:basic_structure_v_boundary:tang_var} is now evident and
	implies \eqref{item:basic_structure_v_boundary:oriented} by the
	assertion of the preceding paragraph. Concerning
	\eqref{item:basic_structure_v_boundary:representable}, the subcase $M
	= U$ follows from \cite[4.5.6, 4.5.11]{MR41:1976} (recalling also
	\cite[4.1.28\,(5), 4.2.1]{MR41:1976}) to which the case $M = \spt \| V
	\|$ may be reduced by applying Allard \cite[4.5]{MR0307015} to $W$,
	see also Allard \cite[4.7]{MR0307015}.

	To treat the general case, first notice that, in view of Allard
	\cite[5.1\,(3)]{MR0307015}, \ref{lemma:capacity} is applicable with
	$\mu = \| V \|$ and $A = ( \spt \theta ) \cap ( \spt \| V \| )
	\without M$ whenever $\theta \in \mathscr{D} (U, \rel^\adim )$, $K$ is
	a compact subset of $U$, and $\spt \theta \subset \Int K$. The
	resulting functions $f_i \in \mathscr{E} (U,\rel)$ satisfy
	\begin{gather*}
		\tint{}{} | \theta - f_i \theta | + | D \theta - D (f_i\theta)
		| \ud \| V \| + \tint{X}{} | \theta - f_i \theta | \ud
		\mathscr{H}^{\vdim-1} \to 0
	\end{gather*}
	as $i \to \infty$ whenever $X$ is a $\mathscr{H}^{\vdim-1}$ measurable
	subset of $U$ with $\mathscr{H}^{\vdim-1} (X) < \infty$. It follows
	that $\boundary V E ( \theta ) = \lim_{i \to \infty} \boundary VE (f_i
	\theta)$ and, if $\mathscr{H}^{\vdim-1} (K \cap B) < \infty$, then
	\begin{gather*}
		\lim_{i \to \infty} \tint{}{} \mathbf{n} (M;E,x) \bullet (
		\theta (x) - (f_i\theta)(x)) \density^\vdim ( \| V \|, x) \ud
		\mathscr{H}^{\vdim-1} x = 0.
	\end{gather*}
	Therefore one now readily verifies the assertion.
\end{proof}
\begin{remark} \label{remark:link}
	If $S \in \mathscr{R}_\vdim^\mathrm{loc} ( \rel^\adim )$ is absolutely
	area minimising with respect to $\rel^\adim$, $\partial S = 0$, and
	$\codim = 1$, then $V \in \IVar_\vdim ( \rel^\adim )$ characterised by
	$\| S \| = \| V \|$ satisfies the hypotheses of
	\ref{thm:basic_structure_v_boundary}\,\eqref{item:basic_structure_v_boundary:oriented}
	with $U = \rel^\adim$, $\kappa = 0$ for some $M$ and $\xi$ by Allard
	\cite[4.8\,(4)]{MR0307015}, \cite[5.4.15]{MR41:1976}, and Federer
	\cite[Theorem 1]{MR0260981}, and in this case $\| \partial ( S
	\restrict E ) \| = \| \boundary VE \|$ as may be verified using
	\ref{thm:basic_structure_v_boundary}\,\eqref{item:basic_structure_v_boundary:oriented}\,\eqref{item:basic_structure_v_boundary:representable}
	in conjunction with \cite[3.1.19, 4.1.14, 4.1.20, 4.1.30,
	4.5.6]{MR41:1976}.\footnote{Referring additionally to
	\cite[5.3.20]{MR41:1976} and Almgren \cite[5.22]{MR1777737}, the
	hypothesis $\codim = 1$ could have been omitted. However, the author
	has not checked Almgren's result and its consequences will not be used
	in the present paper.}
\end{remark}
\begin{remark}
	Considering the situation $\vdim, \adim \in \nat$, $1 < \vdim <
	\adim$, $U$ is an open subset of $\rel^\adim$, $V \in \IVar_\vdim (
	U)$ and $\delta V = 0$, few properties of $V$ are known to hold near
	$\mathscr{H}^{\vdim-1}$ almost all $x \in \spt \| V \|$.
	Consequently, it appears difficult to obtain a structural description
	similar to
	\ref{thm:basic_structure_v_boundary}\,\eqref{item:basic_structure_v_boundary:representable}
	for $\| V\|$ measurable sets whose distributional $V$ boundary is
	representable by integration in this more general situation. However,
	for $\mathscr{L}^1$ almost all superlevel sets of a real valued
	generalised weakly differential function such a description will be
	proven under even milder hypotheses on $V$ in \ref{corollary:coarea}.
\end{remark}
\section{Decompositions of varifolds} \label{sec:decomposition}
In this section the existence of a decomposition of rectifiable varifolds
whose first variation is representable by integration is established in
\ref{thm:decomposition}. If the first variation is sufficiently well behaved,
this decomposition may be linked to the decomposition of the support of the
weight measure into connected components, see \ref{corollary:conn_structure}.
\begin{miniremark} \label{miniremark:situation_general}
	A useful set of hypotheses is gathered here for later reference.

	Suppose $\vdim, \adim \in \nat$, $\vdim \leq \adim$, $1 \leq p \leq
	\vdim$, $U$ is an open subset of $\rel^\adim$, $V \in \Var_\vdim ( U
	)$, $\| \delta V \|$ is a Radon measure, $\density^\vdim ( \| V \|, x
	) \geq 1$ for $\| V \|$ almost all $x$. If $p > 1$, then suppose
	additionally that $\mathbf{h} ( V, \cdot ) \in \Lploc{p} ( \| V \|,
	\rel^\adim)$ and
	\begin{gather*}
		\delta V (\theta) = - \tint{}{} \mathbf{h} (V,x) \bullet
		\theta(x) \ud \| V \| x \quad \text{for $\theta \in
		\mathscr{D} ( U, \rel^\adim)$}.
	\end{gather*}
	Therefore $V \in \RVar_\vdim (U)$ by Allard
	\cite[5.5\,(1)]{MR0307015}. If $p = 1$ let $\psi = \| \delta V \|$. If
	$p>1$ define a Radon measure $\psi$ over $U$ by $\psi (A) =
	\tint{A}{\ast} | \mathbf{h} ( V, x ) |^p \ud \| V \| x$ for $A \subset
	U$.
\end{miniremark}
\begin{definition} \label{def:indecomposable}
	Suppose $\vdim, \adim \in \nat$, $\vdim \leq \adim$, $U$ is an open
	subset of $\rel^\adim$, $V \in \Var_\vdim ( U )$ and $\| \delta V \|$
	is a Radon measure.

	Then $V$ is called \emph{indecomposable} if there exists no $\| V \| +
	\| \delta V \|$ measurable set $E$ such that
	\begin{gather*}
		\| V \| (E) > 0, \quad \| V \| ( U \without E ) > 0, \quad
		\boundary{V}{E} = 0.
	\end{gather*}
\end{definition}
\begin{remark} \label{remark:indecomposable}
	The same definition results if $E$ is required to be a Borel set.
\end{remark}
\begin{remark}
	If $V$ is indecomposable then so is $\lambda V$ for $0 < \lambda <
	\infty$. This is in contrast to a similar notion employed by Mondino
	in \cite[2.15]{MR3148123}.
\end{remark}
\begin{lemma} \label{lemma:connected}
	Suppose $\vdim, \adim \in \nat$, $\vdim \leq \adim$, $U$ is an open
	subset of $\rel^\adim$, $V \in \Var_\vdim ( U )$, $\| \delta V \|$
	is a Radon measure, $E_0$ and $E_1$ are nonempty, disjoint, relatively
	closed subset of $\spt \| V \|$, and $E_0 \cup E_1 = \spt \| V \|$.

	Then there holds
	\begin{gather*}
		\| V \| ( E_i ) > 0 \quad \text{and} \quad \boundary V{E_i} =
		0
	\end{gather*}
	for $i \in \{ 0,1 \}$. In particular, if $V$ is indecomposable then
	$\spt \| V \|$ is connected.
\end{lemma}
\begin{proof}
	Notice that $U \without E_0$ and $U \without E_1$ are open, hence
	\begin{gather*}
		\| V \| ( E_0 ) = \| V \| ( U \without E_1 ) > 0, \quad
		\| V \| ( E_1 ) = \| V \| ( U \without E_0 ) > 0.
	\end{gather*}
	Next, one constructs $w \in \mathscr{E} ( U,\rel )$ such that
	\begin{gather*}
		E_i \subset \Int \{ x \with w(x) = i \} \quad
		\text{for $i = \{ 0, 1 \}$};
	\end{gather*}
	in fact, applying \cite[3.1.13]{MR41:1976} with $\Phi = \{ U \without
	E_0, U \without E_1 \}$, one obtains $h$, $S$ and $v_s$, notices that
	either $\cball s{10h(s)} \subset \rel^\adim \without E_0$ or $\cball
	s{10h(s)} \subset \rel^\adim \without E_1$ whenever $s \in S$, lets $T
	= S \cap \big \{ s \with \cball s{10h(s)} \subset \rel^\adim \without
	E_0 \}$ and takes
	\begin{gather*}
		w (x) = \sum_{t \in T} v_t (x) \quad \text{for $x \in U$}.
	\end{gather*}
	This yields $\boundary{V}{E_i} = 0$ for $i = \{ 0, 1
	\}$ by Allard \cite[4.10\,(1)]{MR0307015}.
\end{proof}
\begin{definition} \label{def:component}
	Suppose $\vdim, \adim \in \nat$, $\vdim \leq \adim$, $U$ is an open
	subset of $\rel^\adim$, $V \in \Var_\vdim ( U )$, and $\| \delta V \|$
	is a Radon measure.

	Then $W$ is called a \emph{component of $V$} if and only if $0 \neq W
	\in \Var_\vdim ( U )$ is indecomposable and there exists a $\| V \| +
	\| \delta V \|$ measurable set $E$ such that
	\begin{gather*}
		W = V \restrict E \times \grass{\adim}{\vdim}, \quad
		\boundary{V}{E} = 0.
	\end{gather*}
\end{definition}
\begin{remark} \label{remark:unique_component}
	Suppose $F$ is a $\| V \| + \| \delta V \|$ measurable set. Then $E$
	is $\| V \| + \| \delta V \|$ almost equal to $F$ if and only if 
	\begin{gather*}
		W = V \restrict F \times \grass{\adim}{\vdim}, \quad
		\boundary{V}{F} = 0.
	\end{gather*}
\end{remark}
\begin{remark} \label{remark:component}
	If $C$ is a connected component of $\spt \| V \|$ and $W$ is a
	component of $V$ with $C \cap \spt \| W \| \neq \varnothing$, then
	$\spt \| W \| \subset C$ by \ref{lemma:connected}.
\end{remark}
\begin{definition} \label{def:decomposition}
	Suppose $\vdim, \adim \in \nat$, $\vdim \leq \adim$, $U$ is an open
	subset of $\rel^\adim$, $V \in \Var_\vdim ( U )$,  $\| \delta V \|$ is
	a Radon measure, and $\Xi \subset \Var_\vdim ( U )$.

	Then $\Xi$ is called a \emph{decomposition of $V$} if and only if the
	following three conditions are satisfied:
	\begin{enumerate}
		\item Each member of $\Xi$ is a component of $V$.
		\item Whenever $W$ and $X$ are distinct members of $\Xi$ there
		exist disjoint $\| V \| + \| \delta V \|$ measurable sets $E$
		and $F$ with $\boundary{V}{E} = 0 = \boundary{V}{F}$ and
		\begin{gather*}
			W = V \restrict E \times \grass{\adim}{\vdim}, \quad X
			= V \restrict F \times \grass{\adim}{\vdim}.
		\end{gather*}
		\item $V ( k ) = \tsum{W \in \Xi}{} W (k)$ whenever $k \in
		\mathscr{K} ( U \times \grass{\adim}{\vdim} )$.
	\end{enumerate}
\end{definition}
\begin{remark} \label{remark:decomp_rep}
	Clearly, $\Xi$ is countable.

	Moreover, using \ref{remark:unique_component} one constructs a
	function $\xi$ mapping $\Xi$ into the class of all Borel subsets of
	$U$ such that distinct members of $\Xi$ are mapped onto disjoint sets
	and
	\begin{gather*}
		W = V \restrict \xi(W) \times \grass{\adim}{\vdim},
		\quad \boundary{V}{\xi(W)} = 0
	\end{gather*}
	whenever $W \in \Xi$. Consequently, in view of \ref{remark:partition},
	one infers
	\begin{gather*}
		{\textstyle ( \| V \| + \| \delta V \| ) \big ( U \without
		\bigcup \im \xi \big ) = 0}.
	\end{gather*}
	Also notice that $\boundary{V}{\left ( \bigcup \xi \lIm N \rIm \right
	)} = 0$ whenever $N \subset \Xi$.
\end{remark}
\begin{remark} \label{remark:decomposition}
	Suppose $\vdim$, $\adim$, $p$, $U$ and $V$ are as in
	\ref{miniremark:situation_general}, $p = \vdim$, and $\Xi$ is a
	decomposition of $W$. Observe that
	\ref{corollary:density_1d}\,\eqref{item:density_1d:local_lower_bound}
	and \cite[2.5]{snulmenn.isoperimetric} imply
	\begin{gather*}
		\card (\Xi \cap \{ W \with K \cap \spt \| W \| \neq
		\varnothing \} ) < \infty
	\end{gather*}
	whenever $K$ is a compact subset of $U$, hence
	\begin{gather*}
		\spt \| V \| = {\textstyle \bigcup \{ \spt \| W \| \with W \in
		\Xi \}}.
	\end{gather*}
	Notice that \cite[1.2]{snulmenn.isoperimetric} readily shows that both
	assertions need not to hold in case $p < \vdim$.
\end{remark}
\begin{theorem} \label{thm:decomposition}
	Suppose $\vdim, \adim \in \nat$, $\vdim \leq \adim$, $U$ is an open
	subset of $\rel^\adim$, $V \in \RVar_\vdim ( U )$, and $\| \delta V
	\|$ is a Radon measure.

	Then there exists a decomposition of $V$.
\end{theorem}
\begin{proof}
	Assume $V \neq 0$.

	Denote by $R$ the family of Borel subsets $E$ of $U$ such that
	$\boundary{V}{E} = 0$. Notice that
	\begin{gather*}
		{\textstyle \bigcap_{i=1}^\infty E_i \in R} \qquad
		\text{whenever $E_i$ is a sequence in $R$ with $E_{i+1}
		\subset E_i$ for $i \in \nat$}, \\
		E \in R \quad \text{if and only if} \quad E \without F \in R
		\qquad \text{whenever $E \supset F \in R$}
	\end{gather*}
	by \ref{remark:v_boundary}. Let $P = R \cap \{ E \with \| V \| ( E ) >
	0 \}$. Next, define
	\begin{gather*}
		\delta_i = \unitmeasure{\vdim} 2^{-\vdim-1} i^{-1-2\vdim},
		\quad \varepsilon_i = 2^{-1} i^{-2}
	\end{gather*}
	for $i \in \nat$ and let $A_i$ denote the Borel set of $a \in
	\rel^\adim$ satisfying
	\begin{gather*}
		|a| \leq i, \quad \oball{a}{2\varepsilon_i} \subset U, \quad
		\density^\vdim ( \| V \|, a ) \geq 1/i, \\
		\measureball{\| \delta V \|}{ \cball{a}{r} } \leq
		\unitmeasure{\vdim} i r^\vdim \quad \text{for $0 < r <
		\varepsilon_i$}
	\end{gather*}
	whenever $i \in \nat$. Clearly, $A_i \subset A_{i+1}$ for $i \in \nat$
	and $\| V \| ( U \without \bigcup_{i=1}^\infty A_i ) = 0$ by
	Allard \cite[3.5\,(1a)]{MR0307015} and \cite[2.8.18,
	2.9.5]{MR41:1976}. Moreover, define
	\begin{gather*}
		P_i = R \cap \{ E \with \| V \| ( E \cap A_i ) > 0 \}
	\end{gather*}
	and notice that $P_i \subset P_{i+1}$ for $i \in \nat$ and $P =
	\bigcup_{i=1}^\infty P_i$. One observes the lower bound given by
	\begin{gather*}
		\| V \| ( E \cap \cball{a}{\varepsilon_i} ) \geq \delta_i
	\end{gather*}
	whenever $E \in R$, $i \in \nat$, $a \in A_i$ and $\density^{\ast
	\vdim} ( \| V \| \restrict E, a ) \geq 1/i$; in fact, noting
	\begin{gather*}
		\tint{0}{\varepsilon_i} r^{-\vdim} \measureball{\| \delta ( V
		\restrict E \times \grass{\adim}{\vdim} ) \|}{ \cball{a}{r} }
		\ud \mathscr{L}^1 r \leq \unitmeasure{\vdim} i \varepsilon_i,
	\end{gather*}
	the inequality follows from \ref{corollary:monotonicity} and
	\ref{remark:monotonicity}. Let $Q_i$ denote the set of $E \in P$ such
	that there is no $F$ satisfying
	\begin{gather*}
		F \subset E, \quad F \in P_i, \quad E \without F \in P_i.
	\end{gather*}

	Denote by $\Omega$ the class of Borel partitions $H$ of $U$ with $H
	\subset P$ and let $G_0 = \{ U \} \in \Omega$. The previously observed
	lower bound implies
	\begin{gather*}
		\delta_i \card ( H \cap P_i ) \leq \| V \| ( U \cap \{ x \with
		\dist (x,A_i) \leq \varepsilon_i \} ) < \infty
	\end{gather*}
	whenever $H$ is a disjointed subfamily of $P$, since for each $E \in
	H \cap P_i$ there exists $a \in A_i$ with $\density^\vdim ( \| V \|
	\restrict E, a ) = \density^\vdim ( \| V \|, a ) \geq 1/i$ by
	\cite[2.8.18, 2.9.11]{MR41:1976}, hence
	\begin{gather*}
		\| V \| ( E \cap \{ x \with \dist (x,A_i) \leq \varepsilon_i
		\} ) \geq \| V \| ( E \cap \cball{a}{\varepsilon_i} ) \geq
		\delta_i.
	\end{gather*}
	In particular, such $H$ is countable.

	Next, one inductively (for $i \in \nat$) defines $\Omega_i$ to be the
	class of all $H \in \Omega$ such that every $E \in G_{i-1}$ is the
	union of some subfamily of $H$ and chooses $G_i \in \Omega_i$ such
	that
	\begin{gather*}
		\card ( G_i \cap P_i ) \geq \card ( H \cap P_i ) \quad
		\text{whenever $H \in \Omega_i$}.
	\end{gather*}
	The maximality of $G_i$ implies $G_i \subset Q_i$; in fact, if there
	would exist $E \in G_i \without Q_i$ there would exist $F$ satisfying
	\begin{gather*}
		F \subset E, \quad F \in P_i, \quad E \without F \in P_i
	\end{gather*}
	and $H = ( G_i \without \{ E \} ) \cup \{ F, E \without F \}$ would
	belong to $\Omega_i$ with
	\begin{gather*}
		\card ( H \cap P_i ) > \card ( G_i \cap P_i ).
	\end{gather*}
	Moreover, it is evident that to each $x \in U$ there corresponds a
	sequence $E_i$ uniquely characterised by the requirements $x \in
	\bigcap_{i=1}^\infty E_i$ and $E_{i+1} \subset E_i \in G_i$ for $i \in
	\nat$.

	Define $G = \bigcup_{i=1}^\infty G_i$ and notice that $G$ is
	countable. Define $C$ to be the collection of sets
	$\bigcap_{i=1}^\infty E_i$ with positive $\| V \|$ measure
	corresponding to all sequences $E_i$ with $E_{i+1} \subset E_i \in
	G_i$ for $i \in \nat$. Clearly, $C$ is a disjointed subfamily of $P$,
	hence $C$ is countable. Next, it will be shown that
	\begin{gather*}
		\| V \| ( U \without {\textstyle \bigcup C } ) = 0.
	\end{gather*}
	In view of \cite[2.8.18, 2.9.11]{MR41:1976} it is sufficient to prove
	\begin{gather*}
		A_i \without {\textstyle \bigcup C} \subset {\textstyle
		\bigcup } \big \{ E \cap \{ x \with \density^{\ast \vdim} ( \|
		V \| \restrict E, x ) < \density^{\ast \vdim} ( \| V \|, x )
		\} \with E \in G \big \}
	\end{gather*}
	for $i \in \nat$. For this purpose consider $a \in A_i \without
	\bigcup C$ with corresponding sequence $E_j$. It follows that $\| V \|
	( \bigcap_{j=1}^\infty E_j ) = 0$, hence there exists $j$ with $\| V
	\| ( E_j \cap \cball{a}{\varepsilon_i} ) < \delta_i$ and the lower
	bound implies
	\begin{gather*}
		\density^{\ast \vdim} ( \| V \| \restrict E_j, a ) < 1/i \leq
		\density^\vdim ( \| V \|, a ).
	\end{gather*}

	It remains to prove that each varifold $V \restrict E \times
	\grass{\adim}{\vdim}$ corresponding to $E \in C$ is indecomposable. If
	this were not the case, then there would exist $E = \bigcap_{i=1} E_i
	\in C$ with $E_{i+1} \subset E_i \in G_i$ for $i \in \nat$ and a Borel
	set $F$ such that
	\begin{gather*}
		\| V \| ( E \cap F ) > 0, \quad \| V \| ( E \without F ) > 0,
		\quad \boundary V (E \cap F) = 0
	\end{gather*}
	by \ref{remark:iterated_boundaries} and \ref{remark:indecomposable}.
	This would imply $E \cap F \in P$ and $E \without F \in P$, hence for
	some $i$ also $E \cap F \in P_i$ and $E \without F \in P_i$ which
	would yield
	\begin{gather*}
		E \cap F \subset E_i, \quad E_i \without ( E \cap F ) \in P_i,
	\end{gather*}
	since $E \without F \subset E_i \without ( E \cap F ) \in R$; a
	contradiction to $E_i \in Q_i$.
\end{proof}
\begin{remark} \label{remark:nonunique_decomposition}
	The decomposition of $V$ may be nonunique. In fact,
	considering the six rays
	\begin{gather*}
		R_j = \{ t \exp ( \pi \mathbf{i} j/3) \with 0 < t < \infty
		\} \subset \complex = \rel^2, \quad \text{where $\pi =
		\boldsymbol{\Gamma} (1/2)^2$},
	\end{gather*}
	corresponding to $j \in \{ 0, 1, 2, 3, 4, 5 \}$ and their associated
	varifolds $V_j \in \IVar_1 ( \rel^2 )$ with $\| V_j \| = \mathscr{H}^1
	\restrict R_j$, one notices that $V = \sum_{j=0}^5 V_j \in \IVar_1 (
	\rel^2 )$ is a stationary varifold such that
	\begin{gather*}
		\{ V_0 + V_2 + V_4, V_1 + V_3 + V_5 \} \quad \text{and} \quad
		\{ V_0 + V_3, V_1 + V_4, V_2 + V_5 \}
	\end{gather*}
	are distinct decompositions of $V$.
\end{remark}
\begin{corollary} \label{corollary:conn_structure}
	Suppose $\vdim$, $\adim$, $p$, $U$ and $V$ are as in
	\ref{miniremark:situation_general}, $p = \vdim$, and $\Phi$ is the
	family of all connected components of $\spt \| V \|$.

	Then the following four statements hold.
	\begin{enumerate}
		\item \label{item:conn_structure:union} If $C \in \Phi$,
		then
		\begin{gather*}
			C = {\textstyle \bigcup \{ \spt \| W \| \with W \in \Xi,
			C \cap \spt \| W \| \neq \varnothing \}}
		\end{gather*}
		whenever $\Xi$ is decomposition of $V$.
		\item \label{item:conn_structure:finite} $\card ( \Phi \cap
		\{ C \with C \cap K \neq \varnothing \}) < \infty$ whenever
		$K$ is a compact subset of $U$.
		\item \label{item:conn_structure:open} If $C \in \Phi$, then
		$C$ is open relative to $\spt \| V \|$.
		\item \label{item:conn_structure:piece} If $C \in \Phi$, then
		$\spt ( \| V \| \restrict C ) = C$ and $\boundary{V}{C} = 0$.
	\end{enumerate}
\end{corollary}
\begin{proof}
	\eqref{item:conn_structure:union} is a consequence of
	\ref{remark:component} and \ref{remark:decomposition}. In view of
	\eqref{item:conn_structure:union} and \ref{thm:decomposition},
	\eqref{item:conn_structure:finite} is a consequence of
	\ref{remark:decomposition}. Next, \eqref{item:conn_structure:finite}
	implies \eqref{item:conn_structure:open}. Finally,
	\eqref{item:conn_structure:open} and \ref{lemma:connected} yield
	\eqref{item:conn_structure:piece}.
\end{proof}
\begin{remark}
	If $V$ is stationary, then \eqref{item:conn_structure:piece} implies
	that $V \restrict C \times \grass{\adim}{\vdim}$ is stationary. This
	fact might prove useful in considerations involving a strong maximum
	principle such as Wickramasekera \cite[Theorem
	1.1]{10.1007/s00526-013-0695-4}.
\end{remark}
\section{Relative isoperimetric inequality} \label{sec:rel_iso}
In this section a general isoperimetric inequality for varifolds satisfying a
lower density bound is established, see \ref{thm:rel_iso_ineq}.  As
corollaries one obtains two relative isoperimetric inequalities under the
relevant conditions on the first variation of the varifold, see
\ref{corollary:rel_iso_ineq} and \ref{corollary:true_rel_iso_ineq}.
\begin{miniremark} \label{miniremark:zero_boundary}
	Suppose $\vdim, \adim \in \nat$, $\vdim \leq \adim$, $U$ is an open
	subset of $\rel^\adim$, $V \in \Var_\vdim (U)$, $\| \delta V \|$ is a
	Radon measure, $\density^\vdim ( \| V \|, x ) \geq 1$ for $\| V \|$
	almost all $x$, $E$ is $\| V \| + \| \delta V \|$ measurable, $B$ is a
	closed subset of $\Bdry U$, and, see
	\ref{def:variation_measure}--\ref{def:restriction_distribution},
	\begin{gather*}
		( \| V \| + \| \delta V \| ) ( E \cap K ) + \| \boundary{V}{E}
		\| ( U \cap K ) < \infty, \\
		\tint{E \times \grass{\adim}{\vdim}}{} \project{P} \bullet D
		\theta (x) \ud V (x,P) =  ( ( \delta V ) \restrict E ) (
		\theta |U ) - ( \boundary{V}{E} ) ( \theta | U )
	\end{gather*}
	whenever $K$ is compact subset of $\rel^\adim \without B$ and $\theta
	\in \mathscr{D} ( \rel^\adim \without B, \rel^\adim )$. Defining $W
	\in \Var_\vdim ( \rel^\adim \without B )$ by
	\begin{gather*}
		W (A) = V ( A \cap ( E \times \grass{\adim}{\vdim} ) ) \quad
		\text{for $A \subset ( \rel^\adim \without B ) \times
		\grass{\adim}{\vdim}$},
	\end{gather*}
	this implies
	\begin{gather*}
		\| \delta W \| (A) \leq \| \delta V \| ( E \cap A ) + \|
		\boundary{V}{E} \| ( U \cap A ) \quad \text{for $A \subset
		\rel^\adim \without B$}.
	\end{gather*}
\end{miniremark}
\begin{example} \label{example:smooth_case_zero_boundary}
	Using
	\ref{thm:basic_structure_v_boundary}\,\eqref{item:basic_structure_v_boundary:tang_var}\,\eqref{item:basic_structure_v_boundary:representable}
	one verifies the following statement. \emph{If $\vdim, \adim\in \nat$,
	$\vdim \leq \adim$, $U$ is an open subset of $\rel^\adim$, $B$ is a
	closed subset of $\Bdry U$, $M$ is an $\vdim$ dimensional submanifold
	of $\rel^\adim$ of class $2$, $M \subset \rel^\adim \without B$, $(
	\Clos M ) \without M \subset B$, $V \in \Var_\vdim ( U )$ and $V' \in
	\Var_\vdim ( \rel^\adim \without B )$ satisfy
	\begin{gather*}
		V(k) = \tint{M \cap U}{} k(x,\Tan(M,x)) \ud \mathscr{H}^\vdim
		x \quad \text{for $k \in \mathscr{K} (U \times \grass \adim
		\vdim )$}, \\
		V'(k) = \tint{M}{} k (x,\Tan(M,x)) \ud \mathscr{H}^\vdim x
		\quad \text{for $k \in \mathscr{K} (( \rel^\adim \without B)
		\times \grass \adim \vdim)$},
	\end{gather*}
	and $E$ is an $\mathscr{H}^\vdim$ measurable subset of $M \cap U$,
	then $E$ satisfies the conditions of \ref{miniremark:zero_boundary} if
	and only if $\boundary {V'} E$ is representable by integration and
	$\mathbf{n} (M;E,x) = 0$ for $\mathscr{H}^{\vdim-1}$ almost all $x \in
	( \rel^\adim \without B ) \without U$; in this case
	\begin{gather*}
		\boundary VE ( \theta | U ) = - \tint U{} \mathbf{n} (M;E,x)
		\bullet \theta (x) \ud \mathscr{H}^{\vdim-1} x \quad \text{for
		$\theta \in \mathscr{D} (\rel^\adim \without B,\rel^\adim)$}.
	\end{gather*}}
	Since in the situation of \ref{miniremark:zero_boundary} there no
	varifold $V'$ available which extends $V$ in a canonical way, the
	condition on $E$ is formulated in terms of the behaviour $W$.
\end{example}
\begin{lemma} \label{lemma:lower_mass_bound_Q}
	Suppose $1 \leq M < \infty$.

	Then there exists a positive, finite number $\Gamma$ with the
	following property.

	If $\vdim, \adim \in \nat$, $\vdim \leq \adim \leq M$, $1 \leq Q \leq
	M$, $a \in \rel^\adim$, $0 < r < \infty$, $W \in \Var_\vdim (
	\oball{a}{r} )$, $\| \delta W \|$ is a Radon measure, $\density^\vdim
	( \| W \|, x ) \geq 1$ for $\| W \|$ almost all $x$, $a \in \spt \| W
	\|$, and
	\begin{gather*}
		\measureball{\| \delta W \|}{\cball{a}{s}} \leq \Gamma^{-1} \|
		W \| ( \cball{a}{s} )^{1-1/\vdim} \quad \text{for $0 < s <
		r$}, \\
		\| W \| ( \{ x \with \density^\vdim ( \| W \|, x ) < Q \} )
		\leq \Gamma^{-1} \measureball{\| W \|}{\oball{a}{r}},
	\end{gather*}
	then there holds
	\begin{gather*}
		\measureball{\| W \|}{\oball{a}{r}} \geq ( Q-M^{-1} )
		\unitmeasure{\vdim} r^\vdim.
	\end{gather*}
\end{lemma}
\begin{proof}
	If the lemma were false for some $M$, there would exist a sequence
	$\Gamma_i$ with $\Gamma_i \to \infty$ as $i \to \infty$, and sequences
	$\vdim_i$, $\adim_i$, $Q_i$, $a_i$, $r_i$, and $W_i$ showing that
	$\Gamma = \Gamma_i$ does not have the asserted property.

	One could assume for some $\vdim, \adim \in \nat$, $1 \leq Q \leq M$
	that $\vdim \leq \adim \leq M$,
	\begin{gather*}
		\vdim = \vdim_i, \quad \adim = \adim_i, \quad  a_i = 0, \quad
		r_i = 1
	\end{gather*}
	for $i \in \nat$ and $Q_i \to Q$ as $i \to \infty$ and, by
	\cite[2.5]{snulmenn.isoperimetric},
	\begin{gather*}
		\measureball{\| W_i \|}{\cball{0}{s}} \geq ( 2 \vdim
		\isoperimetric{\vdim} )^{-\vdim} s^\vdim \quad \text{whenever
		$0 < s < 1$ and $i \in \nat$}.
	\end{gather*}
	Defining $W \in \Var_\vdim ( \rel^\adim \cap \oball{0}{1} )$ to be the
	limit of some subsequence of $W_i$, one would obtain
	\begin{gather*}
		\measureball{\| W \|}{\oball{0}{1}} \leq ( Q - M^{-1} )
		\unitmeasure{\vdim}, \quad 0 \in \spt \| W \|, \quad \delta W
		= 0.
	\end{gather*}
	Finally, using Allard \cite[5.1\,(2), 5.4, 8.6]{MR0307015}, one would
	then conclude that
	\begin{gather*}
		\density^\vdim ( \| W \|, x ) \geq Q \quad \text{for $\| W \|$
		almost all $x$}, \\
		\density^\vdim ( \| W \|, 0 ) \geq Q, \quad \measureball{\| W
		\|}{\oball{0}{1}} \geq Q \unitmeasure{\vdim},
	\end{gather*}
	a contradiction.
\end{proof}
\begin{remark}
	Considering stationary varifolds whose support is contained in two
	affine planes with $\density^\vdim ( \| W \| , a )$ a small positive
	number, shows that the hypotheses ``$\density^\vdim ( \| W \|, x )
	\geq 1$ for $\| W \|$ almost all $x$'' cannot be omitted.
\end{remark}
\begin{remark}
	Even for smooth functions, \ref{lemma:lower_mass_bound_Q} is the key
	observation which -- through the relative isoperimetric inequalities
	\ref{corollary:rel_iso_ineq} and \ref{corollary:true_rel_iso_ineq} --
	leads to Sobolev Poincar{\'e} type estimates which are applicable near
	$\| V \|$ almost all points of $\{ x \with \density^\vdim ( \| V \|, x
	) \geq 2 \}$, see \ref{thm:sob_poin_summary}. For generalised weakly
	differentiable functions, these estimates in turn provide an important
	ingredient for the differentiability results obtained in
	\ref{thm:approx_diff} and \ref{thm:diff_lebesgue_spaces} and the
	coarea formula in \ref{corollary:coarea}.
\end{remark}
\begin{remark} \label{remark:kuwert_schaetzle}
	Taking $Q=1$ in \ref{lemma:lower_mass_bound_Q} (or applying
	\cite[2.6]{snulmenn.isoperimetric}) yields the following proposition:
	\emph{If $\vdim$, $\adim$, $p$, $U$, $V$, and $\psi$ are as in
	\ref{miniremark:situation_general}, $p = \vdim$, $a \in \spt \| V \|$,
	and $\psi ( \{ a \} ) = 0$, then $\density_\ast^\vdim ( \| V \|, a )
	\geq 1$.} If $\| \delta V \|$ is absolutely continuous with respect to
	$\| V \|$ then the condition $\psi ( \{ a \} ) = 0$ is redundant.

	If $\vdim = \adim$ and $f : \rel^\adim \to \{ y \with 1 \leq y <
	\infty \}$ is a weakly differentiable function with $\weakD f \in
	\Lploc{\adim} ( \mathscr{L}^\adim, \Hom ( \rel^\adim, \rel ) )$, then
	the varifold $V \in \RVar_\adim ( \rel^\adim )$ defined by the
	requirement $\| V \| ( B ) = \tint{B}{} f \ud \mathscr{L}^\adim$
	whenever $B$ is a Borel subset of $\rel^\adim$ satisfies the
	conditions of \ref{miniremark:situation_general} with $p = \adim$
	since
	\begin{gather*}
		\delta V ( g) = - \tint{}{} \left < g (x), \weakD (\log \circ
		f ) (x) \right > \ud \| V \| x \quad \text{for $g \in
		\mathscr{D} ( \rel^\adim, \rel^\adim )$}.
	\end{gather*}
	If $\adim > 1$ and $C$ is a countable subset of $\rel^\adim$, one may
	use well known properties of Sobolev functions, in particular example
	\cite[4.43]{MR2424078}, to construct $f$ such that $C \subset \{ x
	\with \density^\adim ( \| V \|, x ) = \infty \}$. It is therefore
	evident that the conditions of \ref{miniremark:situation_general} with
	$p = \vdim>1$ are insufficient to guarantee finiteness or upper
	semicontinuity of $\density^\vdim ( \| V \|, \cdot )$ at each point of
	$U$. However, the following proposition was obtained by Kuwert and
	Sch\"atzle in \cite[Appendix A]{MR2119722}: \emph{If $\vdim$, $\adim$,
	$p$, $U$, and $V$ are as in \ref{miniremark:situation_general}, $p =
	\vdim = 2$, and $V \in \IVar_2 ( U )$, then $\density^2 ( \| V \|,
	\cdot )$ is a real valued, upper semicontinuous function whose domain
	is $U$.}
\end{remark}
\begin{remark}
	The preceding remark is a corrected and extended version of the
	author's remark in \cite[2.7]{snulmenn.isoperimetric} where the last
	two sentences should have referred to integral varifolds.
\end{remark}
\begin{theorem} \label{thm:rel_iso_ineq}
	Suppose $1 \leq M < \infty$.

	Then there exists a positive, finite number $\Gamma$ with the
	following property.

	If $\vdim, \adim \in \nat$, $\vdim \leq \adim \leq M$, $1 \leq Q \leq
	M$, $U$ is an open subset of $\rel^\adim$, $W \in \Var_\vdim ( U )$,
	$S, \Sigma \in \mathscr{D}' ( U, \rel^\adim )$ are representable by
	integration, $\delta W = S + \Sigma$, $\density^\vdim ( \| W \|, x )
	\geq 1$ for $\| W \|$ almost all $x$, $0 < r < \infty$, and
	\begin{gather*}
		\| S \| ( U ) \leq \Gamma^{-1} \qquad \text{if $\vdim = 1$},
		\\
		S(\theta) \leq \Gamma^{-1} \Lpnorm{\|
		W\|}{\vdim/(\vdim-1)}{\theta} \quad \text{for $\theta \in
		\mathscr{D} ( U, \rel^\adim )$} \qquad \text{if $\vdim > 1$},
		\\
		\| W \| ( U ) \leq ( Q-M^{-1} ) \unitmeasure{\vdim} r^\vdim,
		\quad
		\| W \| ( \{ x \with \density^\vdim ( \| W \|, x )
		< Q \} ) \leq \Gamma^{-1} r^\vdim,
	\end{gather*}
	then there holds
	\begin{gather*}
		\| W \| ( \{ x \with \oball{x}{r} \subset U \})^{1-1/\vdim}
		\leq \Gamma \, \| \Sigma \| ( U ),
	\end{gather*}
	where $0^0=0$.
\end{theorem}
\begin{proof}
	Define
	\begin{gather*}
		\Delta_1 = \Gamma_{\ref{lemma:lower_mass_bound_Q}} ( 2 M
		)^{-1}, \quad \Delta_2 = \inf \{ ( 2
		\isoperimetric{\vdim} )^{-1} \with M \geq \vdim \in \nat
		\}, \\
		\Delta_3 = \Delta_1 \inf \{ ( 2 \vdim
		\isoperimetric{\vdim})^{-\vdim} \with M \geq \vdim \in \nat
		\}, \quad \Delta_4 = \sup \{ \besicovitch{\adim} \with M \geq
		\adim \in \nat \}, \\
		\Delta_5 = (1/2) \inf \{ \Delta_2, \Delta_3 \}, \quad \Gamma =
		\Delta_4 \Delta_5^{-1}.
	\end{gather*}
	Notice that $\isoperimetric{1} \geq 1/2$, hence $2 \Delta_5 \leq
	\Gamma_{\ref{lemma:lower_mass_bound_Q}} ( 2M )^{-1}$.

	Suppose $\vdim$, $\adim$, $Q$, $U$, $W$, $S$, $\Sigma$, and $r$
	satisfy the hypotheses in the body of the theorem with $\Gamma$.

	Abbreviate $A = \{ x \with \oball{x}{r} \subset U \}$. Clearly, if
	$\vdim=1$ then $\| S \| ( U ) \leq \Delta_5$.  Observe, if $\vdim>1$
	then
	\begin{gather*}
		\| S \| (X) \leq \Delta_5 \| W \| ( X )^{1-1/\vdim} \quad
		\text{whenever $X \subset U$}.
	\end{gather*}

	Next, the following assertion will be shown: \emph{If $a \in A \cap
	\spt \| W \|$ then there exists $0 < s < r$ such that
	\begin{gather*}
		\Delta_5 \| W \| ( \cball{a}{s} )^{1-1/\vdim} <
		\measureball{\| \Sigma \|}{ \cball{a}{s} }.
	\end{gather*}}
	Since $\measureball{\| \delta W \|}{\cball{a}{s}} \leq \Delta_5 \|
	W \| ( \cball{a}{s} )^{1-1/\vdim} + \measureball{\| \Sigma
	\|}{\cball{a}{s}}$, it is sufficient to exhibit $0 < s < r$ with
	\begin{gather*}
		2 \Delta_5 \| W \| ( \cball{a}{s} )^{1-1/\vdim} <
		\measureball{\| \delta W \|}{ \cball{a}{s} }.
	\end{gather*}
	As $\measureball{\| W \|}{\oball{a}{r}} \leq \| W \| ( U ) < \big (
	Q-(2M)^{-1} \big ) \unitmeasure{\vdim} r^\vdim$ the nonexistence of
	such $s$ would imply by use of \cite[2.5]{snulmenn.isoperimetric} and
	\ref{lemma:lower_mass_bound_Q}
	\begin{gather*}
		\begin{aligned}
			\Delta_1 ( 2 \vdim \isoperimetric{\vdim} )^{-\vdim}
			r^\vdim & \leq \Delta_1 \measureball{\| W
			\|}{\oball{a}{r}} \\
			& < \| W \| ( \classification{\oball{a}{r}}{x}{
			\density^\vdim ( \| W \|, x ) < Q } ) \leq \Delta_5
			r^\vdim,
		\end{aligned}
	\end{gather*}
	a contradiction.

	By the assertion of the preceding paragraph there exist countable
	disjointed families of closed balls $G_1, \ldots,
	G_{\besicovitch{\adim}}$ such that
	\begin{gather*}
		A \cap \spt \| W \| \subset {\textstyle\bigcup\bigcup} \{ G_i
		\with i=1, \ldots, \besicovitch{\adim} \} \subset U, \\
		\| W \| ( B )^{1-1/\vdim} \leq \Delta_5^{-1} \| \Sigma \| ( B )
		\quad \text{whenever $B \in G_i$ and $i \in \{ 1, \ldots,
		\besicovitch{\adim} \}$}.
	\end{gather*}
	If $\vdim > 1$ then, defining $\beta = \vdim / (\vdim-1)$, one
	estimates
	\begin{gather*}
		\begin{aligned}
			\| W \| (A) & \leq \tsum{i=1}{\besicovitch{\adim}}
			\tsum{B \in G_i}{} \| W \| ( B ) \leq
			\Delta_5^{-\beta} \tsum{i=1}{\besicovitch{\adim}}
			\tsum{B \in G_i}{} \| \Sigma \| ( B )^\beta \\
			& \leq \Delta_5^{-\beta}
			\tsum{i=1}{\besicovitch{\adim}} \big ( \tsum{B \in
			G_i}{} \| \Sigma \| (B) \big )^\beta \leq
			\Delta_5^{-\beta} \besicovitch{\adim} \| \Sigma \| ( U
			)^\beta.
		\end{aligned}
	\end{gather*}
	If $\vdim = 1$ and $\| W \| (A) > 0$ then $\Delta_5 \leq \| \Sigma \|
	(B) \leq \| \Sigma \| ( U )$ for some $B \in \bigcup \{ G_i \with i =
	1, \ldots, \besicovitch{\adim} \}$.
\end{proof}
\begin{corollary} \label{corollary:rel_iso_ineq}
	Suppose $\vdim$, $\adim$, $U$, $V$, $E$, and $B$ are as in
	\ref{miniremark:zero_boundary}, $1 \leq Q \leq M < \infty$, $\adim
	\leq M$, $\Lambda = \Gamma_{\ref{thm:rel_iso_ineq}} ( M )$, $0 < r <
	\infty$, and
	\begin{gather*}
		\| V \| (E) \leq (Q-M^{-1}) \unitmeasure{\vdim} r^\vdim, \quad
		\| V \| ( \classification{E}{x}{ \density^\vdim ( \| V \|, x )
		< Q } ) \leq \Lambda^{-1} r^\vdim.
	\end{gather*}
	
	Then there holds
	\begin{gather*}
		\| V \| ( \classification{E}{x}{ \oball{x}{r} \subset
		\rel^\adim \without B } )^{1-1/\vdim} \leq \Lambda \big ( \|
		\boundary{V}{E} \| (U) + \| \delta V \| ( E ) \big ),
	\end{gather*}
	where $0^0 = 0$.
\end{corollary}
\begin{proof}
	Define $W$ as in \ref{miniremark:zero_boundary} and note
	$\density^\vdim ( \| W \|, x ) = \density^\vdim ( \| V \|, x ) \geq 1$
	for $\| W \|$ almost all $x$ by \cite[2.8.18, 2.9.11]{MR41:1976}.
	Therefore applying \ref{thm:rel_iso_ineq} with $U$, $S$, and $\Sigma$
	replaced by $\rel^\adim \without B$, $0$, and $\delta W$ yields the
	conclusion.
\end{proof}
\begin{remark}
	The case $B = \Bdry U$ and $Q = 1$ corresponds to Hutchinson
	\cite[Theorem 1]{MR1066398} which is formulated in the context of
	functions rather than sets.
\end{remark}
\begin{corollary} \label{corollary:true_rel_iso_ineq}
	Suppose $\vdim$, $\adim$, $p$, $U$, $V$, and $\psi$ are as in
	\ref{miniremark:situation_general}, $p = \vdim$, $\adim \leq M$, $E$
	and $B$ are related to $\vdim$, $\adim$, $U$, and $V$, as in
	\ref{miniremark:zero_boundary}, $1 \leq Q \leq M < \infty$, $\Lambda =
	\Gamma_{\ref{thm:rel_iso_ineq}} ( M )$, $0 < r < \infty$, and
	\begin{gather*}
		\| V \| ( E ) \leq (Q-M^{-1}) \unitmeasure{\vdim} r^\vdim,
		\quad \psi (E)^{1/\vdim} \leq \Lambda^{-1} , \\
		\| V \| ( \classification{E}{x}{ \density^\vdim ( \| V \|, x )
		< Q } ) \leq \Lambda^{-1} r^\vdim.
	\end{gather*}

	Then there holds
	\begin{gather*}
		\| V \| ( \classification{E}{x}{ \oball{x}{r} \subset
		\rel^\adim \without B } )^{1-1/\vdim} \leq \Lambda \, \|
		\boundary{V}{E} \| ( U ),
	\end{gather*}
	where $0^0=0$.
\end{corollary}
\begin{proof}
	Define $W$ as in \ref{miniremark:zero_boundary} and note
	$\density^\vdim ( \| W \|, x ) = \density^\vdim ( \| V \|, x ) \geq 1$
	for $\| W \|$ almost all $x$ by \cite[2.8.18, 2.9.11]{MR41:1976}. If
	$\vdim > 1$ then approximation shows
	\begin{gather*}
		( ( \delta V ) \restrict E ) ( \theta | U ) = - \tint{E}{}
		\mathbf{h} (V,x) \bullet \theta (x) \ud \| V \| x \quad
		\text{for $\theta \in \mathscr{D} ( \rel^\adim \without B,
		\rel^\adim )$};
	\end{gather*}
	in fact since $\| ( \delta V ) \restrict E \| = \| \delta V \|
	\restrict E$ and $\theta |U \in \Lp{1} ( \| \delta V \| \restrict E,
	\rel^\adim )$ the problem reduces to the case $\spt \theta \subset U$
	which is readily treated. Therefore applying \ref{thm:rel_iso_ineq}
	with $U$, $S(\theta)$ and $\Sigma( \theta )$ replaced by $\rel^\adim
	\without B$, $(( \delta V ) \restrict E ) ( \theta | U )$ and $- (
	\boundary{V}{E} ) ( \theta | U )$ yields the conclusion.
\end{proof}
\begin{remark}
	Evidently, considering $E=U$, $B = \Bdry U$, $Q=1$, and stationary
	varifolds $V$ such that $\| V \| ( U ) $ is a small positive number
	shows that the intersection with $\{ x \with \oball{x}{r} \subset
	\rel^\adim \without B \}$ cannot be omitted in the conclusion.
\end{remark}
\begin{remark}
	Since estimating $\| \delta V \| ( E )$ by means of H{\"o}lder's
	inequality seems insufficient to derive
	\ref{corollary:true_rel_iso_ineq} from \ref{corollary:rel_iso_ineq} by
	means of ``absorption'', this procedure was implemented at an earlier
	stage and led to the formulation of \ref{thm:rel_iso_ineq}.
\end{remark}
\begin{remark} \label{remark:bombieri_giusti}
	It is instructive to consider the following situation. Suppose $\vdim,
	\adim \in \nat$, $1 \leq \vdim \leq \adim$, $1 \leq M < \infty$, $a
	\in \rel^\adim$, $0 < r < \infty$, $0 \leq \kappa < \infty$, $1 <
	\lambda < \infty$, $U = \oball a{\lambda r}$, $B = \Bdry U$, $A = \{ x
	\with \oball xr \subset \rel^\adim \without B \}$, hence $A = \cball
	a{(\lambda-1)r}$, $V \in \IVar_\vdim (U)$ is a stationary varifold,
	and $E$ is a $\| V \|$ measurable set satisfying the relative
	isoperimetric estimate
	\begin{gather*}
		\inf \{ \| V \| ( A \cap E ), \| V \| ( A \without E )
		\}^{1-1/\vdim} \leq \kappa \| \boundary VE \| (U),
	\end{gather*}
	where $0^0=0$. Then $\| V \| ( A \cap E ) \leq (1-M^{-1}) \| V\| (A)$
	implies
	\begin{gather*}
		\| V \| ( A \cap E )^{1-1/\vdim} \leq M \kappa \| \boundary VE
		\| ( U ).
	\end{gather*}

	Exhibiting a suitable class of $V$ and $E$ such that the
	relative isoperimetric estimate holds with a uniform number $\kappa$
	is complicated by the absence of such an estimate on the
	catenoid.\footnote{See for instance \cite[p.~18]{MR852409} for a
	description of the catenoid.} If $V$ corresponds to an absolutely area
	minimising locally rectifiable current in codimension one, then such
	uniform control was obtained for some $\lambda$ in Bombieri and Giusti
	\cite[Theorem 2]{MR0308945} (and attributed by the authors to
	De~Giorgi); see \ref{remark:link} for the link of the concept of
	distributional boundary employed by Bombieri and Giusti and the one of
	the present paper.

	Finally, the term ``relative isoperimetric inequality'' (or estimate)
	is chosen in accordance with the usage of that term in the case $\| V
	\| = \mathscr{L}^\adim$ by Ambrosio, Fusco and Pallara, see
	\cite[(3.43), p.~152]{MR2003a:49002}.
\end{remark}
\section{Basic properties of weakly differentiable functions}
\label{sec:basic} In this section generalised weakly differentiable functions
are defined in \ref{def:v_weakly_diff}. Properties studied include behaviour
under composition, see \ref{lemma:basic_v_weakly_diff},
\ref{example:composite} and \ref{lemma:comp_lip}, addition and multiplication,
see \ref{thm:addition}\,\eqref{item:addition:add}\,\eqref{item:addition:mult},
and decomposition of the varifold, see \ref{thm:tv_on_decompositions} and
\ref{thm:zero_derivative}. Moreover, coarea formulae in terms of the
distributional boundary of superlevel sets are established, see
\ref{lemma:meas_fct} and \ref{thm:tv_coarea}. A measure theoretic description of
the boundary of the superlevel sets will appear in \ref{corollary:coarea}. The
theory is illustrated by examples in \ref{example:star} and
\ref{example:axioms}.
\begin{lemma} \label{lemma:meas_fct}
	Suppose $\vdim, \adim \in \nat$, $U$ is an open subset of
	$\rel^\adim$, $V \in \Var_\vdim (U)$, $\| \delta V \|$ is a Radon
	measure, $f$ is a real valued $\| V \| + \| \delta V \|$ measurable
	function with $\dmn f \subset U$, and $E(y) = \{ x \with f(x) > y \}$
	for $y \in \rel$.

	Then there exists a unique $T \in \mathscr{D}' (U \times \rel,
	\rel^\adim )$ such that, see \ref{remark:allard_vs_federer},
	\begin{gather*}
		( \delta V ) ( (\omega \circ f ) \theta ) = \tint{}{} \omega (
		f(x)) \project{P} \bullet D \theta (x) \ud V(x,P) + T_{(x,y)}
		( \omega'(y) \theta (x) )
	\end{gather*}
	whenever $\theta \in \mathscr{D} (U,\rel^\adim)$, $\omega \in
	\mathscr{E} ( \rel, \rel )$, $\spt \omega'$ is compact and $\inf \spt
	\omega > - \infty$. Moreover, there holds
	\begin{gather*}
		T ( \phi ) = \tint{}{} \boundary{V}{E(y)} ( \phi ( \cdot, y )
		) \ud \mathscr{L}^1y \quad \text{for $\phi \in \mathscr{D} (U
		\times \rel, \rel^\adim )$}.
	\end{gather*}
\end{lemma}
\begin{proof}
	Define $C = \{ (x,y) \with x \in E(y) \}$ and $g : (U \times
	\grass{\adim}{\vdim} ) \times \rel \to U \times \rel$ by $g((x,P),y) =
	(x,y)$ for $x \in U$, $P \in \grass{\adim}{\vdim}$ and $y \in \rel$.
	Using \cite[2.2.2, 2.2.3, 2.2.17, 2.6.2]{MR41:1976}, one obtains that
	$C$ is $\| V \| \times \mathscr{L}^1$ and $\| \delta V \| \times
	\mathscr{L}^1$ measurable, hence that $g^{-1} \lIm C \rIm$ is $V
	\times \mathscr{L}^1$ measurable since $\| V \| \times \mathscr{L}^1 =
	g_\# ( V \times \mathscr{L}^1 )$.

	Define $p : \rel^\adim \times \rel \to \rel^\adim$ by $p(x,y) = x$ for
	$(x,y) \in \rel^\adim \times \rel$ and let $T \in \mathscr{D} ' ( U
	\times \rel, \rel^\adim )$ be defined by
	\begin{align*}
		T ( \phi ) & = \tint{C}{} \eta (V,x) \bullet \phi (x,y) \ud (
		\| \delta V \| \times \mathscr{L}^1 ) (x,y) \\
		& \phantom{=} \ - \tint{g^{-1} \lIm C \rIm}{} \project{P}
		\bullet ( D \phi (x,y) \circ p^\ast ) \ud (V \times
		\mathscr{L}^1) ((x,P),y)
	\end{align*}
	whenever $\phi \in \mathscr{D} ( U \times \rel, \rel^\adim )$, see
	\ref{miniremark:situation_general_varifold}. Fubini's theorem then
	yields the two equations. The uniqueness of $T$ follows from
	\ref{miniremark:distrib_on_products}.
\end{proof}
\begin{remark}
	Notice that characterising equation for $T$ also holds if the
	requirement $\inf \spt \omega > - \infty$ is dropped.
\end{remark}
\begin{definition} \label{def:v_weakly_diff}
	Suppose $\vdim, \adim \in \nat$, $\vdim \leq \adim$, $U$ is an open
	subset of $\rel^\adim$, $V \in \RVar_\vdim ( U )$, $\| \delta V
	\|$ is a Radon measure, and $Y$ is a finite dimensional normed
	vectorspace.

	Then a $Y$ valued $\| V \| + \| \delta V \|$ measurable function $f$
	with $\dmn f \subset U$ is called \emph{generalised $V$ weakly
	differentiable} if and only if for some $\| V \|$ measurable $\Hom
	(\rel^\adim, Y )$ valued function $F$ the following two conditions
	hold:
	\begin{enumerate}
		\item \label{item:v_weakly_diff:int} If $K$ is a compact
		subset of $U$ and $0 \leq s < \infty$, then
		\begin{gather*}
			\tint{\classification{K}{x}{|f(x)|\leq s}}{} \|F\| \ud
			\| V \| < \infty.
		\end{gather*}
		\item \label{item:v_weakly_diff:partial} If $\theta \in
		\mathscr{D} ( U, \rel^\adim )$, $\gamma \in \mathscr{E} ( Y,
		\rel)$ and $\spt D \gamma$ is compact then
		\begin{gather*}
			\begin{aligned}
				& ( \delta V ) ( ( \gamma \circ f ) \theta ) \\
				& \quad = \tint{}{} \gamma(f(x)) \project{P}
				\bullet D \theta (x) \ud V (x,P) + \tint{}{}
				\left < \theta(x), D\gamma (f(x)) \circ F (x)
				\right > \ud \| V \| x.
			\end{aligned}
		\end{gather*}
	\end{enumerate}
	The function $F$ is $ \| V \|$ almost unique. Therefore, one may
	define the \emph{generalised $V$ weak derivative of $f$} to be the
	function $\derivative{V}{f}$ characterised by $a \in \dmn
	\derivative{V}{f}$ if and only if
	\begin{gather*}
		(\| V \|, C ) \aplim_{x\to a} F (x) = \sigma \quad \text{for
		some $\sigma \in \Hom ( \rel^\adim, Y)$}
	\end{gather*}
	and in this case $\derivative{V}{f}(a) = \sigma$, where $C = \{
	(a,\cball{a}{r}) \with \cball{a}{r} \subset U \}$. Moreover, the set
	of all $Y$ valued generalised $V$ weakly differentiable functions will
	be denoted by $\trunc ( V, Y )$ and $\trunc (V) = \trunc (V,\rel)$.
\end{definition}
\begin{remark} \label{remark:eq_condition_weak_diff}
	Condition \eqref{item:v_weakly_diff:partial} is equivalent to the
	following condition:
	\begin{enumerate}
		\item [\eqref{item:v_weakly_diff:partial}$'$] If $\zeta \in
		\mathscr{D} ( U, \rel )$, $u \in \rel^\adim$, $\gamma \in
		\mathscr{E} ( Y, \rel )$ and $\spt D \gamma$ is compact, then
		\begin{gather*}
			\begin{aligned}
				& ( \delta V ) ( ( \gamma \circ f ) \zeta
				\cdot u ) \\
				& \qquad = \tint{}{} \gamma(f(x)) \left <
				\project{P}(u), D \zeta (x) \right > + \zeta
				(x) \left < u , D\gamma (f(x)) \circ F (x)
				\right > \ud V(x,P).
			\end{aligned}
		\end{gather*}
	\end{enumerate}
	If $\dim Y \geq 2$ or $f$ is locally bounded then one may require
	$\spt \gamma$ to be compact in \eqref{item:v_weakly_diff:partial} or
	\eqref{item:v_weakly_diff:partial}$'$ as $\dim Y \geq 2$ implies that
	$Y \without \cball 0s$ is connected for $0 < s < \infty$.
\end{remark}
\begin{remark} \label{remark:associated_distribution}
	If $f \in \trunc(V)$ then the distribution $T$ associated to $f$ in
	\ref{lemma:meas_fct} is representable by integration and, see
	\ref{remark:ind-limit}, \ref{miniremark:distrib_on_products}, and
	\ref{lemma:push_on_product} with $J = \rel$,
	\begin{gather*}
		\begin{aligned}
			T ( \phi ) & = \tint{}{} \left < \phi (x,f(x)),
			\derivative{V}{f} (x) \right > \ud \| V \| x, \\
			\tint{}{} g \ud \| T \| & = \tint{}{} g(x,f(x)) |
			\derivative{V}{f} (x) | \ud \| V \| x
		\end{aligned}
	\end{gather*}
	whenever $\phi \in \Lp{1} ( \| T \|, \rel^\adim )$ and $g$ is an
	$\overline{\rel}$ valued $\| T \|$ integrable function.
\end{remark}
\begin{remark} \label{remark:lipschitzian_theta}
	If $f \in \trunc (V,Y)$, $\theta : U \to \rel^\adim$ is Lipschitzian
	with compact support, $\gamma : Y \to \rel$ is of class $1$, and
	either $\spt D \gamma$ is compact of $f$ is locally bounded, then
	\begin{multline*}
		( \delta V ) ( ( \gamma \circ f ) \theta ) = \tint{}{} \gamma
		(f(x)) \project P \bullet ( ( \| V \|, \vdim ) \ap D \theta
		(x) \circ \project P ) \ud V(x,P) \\
		+ \tint{}{} \left < \theta (x), D \gamma (f(x)) \circ
		\derivative Vf(x) \right > \ud \| V \| x
	\end{multline*}
	as may be verified by means of approximation and
	\cite[4.5\,(3)]{snulmenn.decay}. Consequently, if $f \in \trunc (V,Y)$
	is locally bounded, $Z$ is a finite dimensional normed vectorspace,
	and $g : Y \to Z$ is of class $1$, then $g \circ f \in \trunc (V,Z)$
	with
	\begin{gather*}
		\derivative V{(g \circ f)} (x) = Dg(f(x)) \circ \derivative
		Vf(x) \quad \text{for $\| V \|$ almost all $x$}.
	\end{gather*}
\end{remark}
\begin{example} \label{example:lipschitzian}
	If $f : U \to Y$ is a locally Lipschitzian function then $f$ is
	generalised $V$ weakly differentiable with
	\begin{gather*}
		\derivative{V}{f}(x) = ( \| V \|, \vdim ) \ap Df (x) \circ
		\project{\Tan^\vdim ( \| V \|, x )} \quad \text{for $\| V\|$
		almost all $x$},
	\end{gather*}
	as may be verified by means of \cite[4.5\,(4)]{snulmenn.decay}.
	Moreover, if $\density^\vdim ( \| V \|, x ) \geq 1$ for $\| V \|$
	almost all $x$, then the equality holds for any $f \in \trunc (V,Y)$
	as will be shown in \ref{thm:approx_diff}.
\end{example}
\begin{remark}
	The prefix ``generalised'' has been chosen in analogy with the notion
	of ``generalised function of bounded variation'' treated in
	\cite[\S 4.5]{MR2003a:49002} originating from De~Giorgi and Ambrosio
	\cite{MR1152641}.
\end{remark}
\begin{remark}
	The usefulness of partial integration identities involving the first
	variation in defining a concept of weakly differentiable functions on
	varifolds has already been ``expected'' by Anzellotti, Delladio and
	Scianna who developed two notions of functions of bounded variation on
	integral currents, see \cite[p.~261]{MR1441622},
\end{remark}
\begin{remark} \label{remark:bv}
	In order to define a concept of ``generalised (real valued) function
	of bounded variation'' with respect to a varifold, it could be of
	interest to study the class of those functions $f$ satisfying the
	hypotheses of \ref{lemma:meas_fct} such that the associated function
	$T$ is representable by integration.
\end{remark}
\begin{remark} \label{remark:moser}
	A concept related to the present one has been proposed by Moser in
	\cite[Definition 4.1]{62659} in the context of curvature varifolds
	(see \ref{def:curvature_varifold} and
	\ref{remark:hutchinson_reformulations}); in fact, it allows for
	certain ``multiple-valued'' functions. In studying convergence of
	pairs of varifolds and weakly differentiable functions, it would seem
	natural to investigate the extension of the present concept to such
	functions. (Notice that the usage of the term ``multiple-valued'' here
	is different but related to the one of Almgren in \cite[\S
	1]{MR1777737}).
\end{remark}
\begin{lemma} \label{lemma:basic_v_weakly_diff}
	Suppose $\vdim$, $\adim$, $U$, $V$, and $Y$ are as in
	\ref{def:v_weakly_diff}, $f \in \trunc (V,Y)$, $Z$ is a finite
	dimensional normed vectorspace, $0 \leq \kappa < \infty$, $\Upsilon$
	is a closed subset of $Y$, $g : Y \to Z$, $H : Y \to \Hom ( Y, Z )$,
	$g_i : Y \to Z$ is a sequence of functions of class $1$, $\spt D g_i
	\subset \Upsilon$, $\Lip g_i \leq \kappa$, $g | \Upsilon$ is proper,
	and
	\begin{gather*}
		g(y) = \lim_{i \to \infty} g_i (y) \quad \text{uniformly
		in $y \in Y$}, \\
		H (y) = \lim_{i \to \infty} D g_i ( y ) \quad \text{for
		$y \in Y$}.
	\end{gather*}
	
	Then $g \circ f \in \trunc (V,Z)$ and
	\begin{gather*}
		\derivative{V}{( g \circ f )} (x) = H ( f (x) ) \circ
		\derivative{V}{f} (x) \quad \text{for $\| V \|$ almost all
		$x$}.
	\end{gather*}
\end{lemma}
\begin{proof}
	Note $\Lip g \leq \kappa$, $\| H(y) \| \leq \kappa$ for $y \in Y$, and
	$H| Y \without \Upsilon = 0$, hence
	\begin{gather*}
		\{ y \with | g(y) | \leq s \} \subset \{ y \with H(y) = 0 \}
		\cup (g|\Upsilon)^{-1} \lIm \cball{0}{s} \rIm, \\
		\tint{K \cap \{ x \with |g(f(x))| \leq s \}}{} \|H(f(x)) \circ
		\derivative{V}{f}(x) \| \ud \| V \| x < \infty
	\end{gather*}
	whenever $K$ is a compact subset of $U$ and $0 \leq s < \infty$.

	Suppose $\gamma \in \mathscr{E} ( Z,\rel )$ and $0 \leq s < \infty$
	with $\spt D \gamma \subset \oball{0}{s}$ and $C = ( g |
	\Upsilon)^{-1} \lIm \cball{0}{s} \rIm$. Then $\im \gamma$ is bounded,
	$C$ is compact and
	\begin{gather*}
		Y \cap \{ y \with D \gamma ( g_i (y ) ) \circ D g_i (y) \neq 0
		\} \subset ( g_i | \Upsilon )^{-1} \lIm \spt D \gamma \rIm
		\subset C \quad \text{for large $i$},
	\end{gather*}
	in particular $\spt D ( \gamma \circ g_i ) \subset C$ for such $i$.
	Using \ref{remark:lipschitzian_theta}, it follows
	\begin{multline*}
		( \delta V ) ( ( \gamma \circ g_i \circ f ) \theta) =
		\tint{}{} \gamma(g_i(f(x))) \project{P} \bullet D \theta (x)
		\ud V (x,P) \\
		+ \tint{}{} \left < \theta(x), D\gamma (g_i(f(x))) \circ D g_i
		(f(x)) \circ \derivative{V}{f} (x) \right > \ud \| V \| x
	\end{multline*}
	for $\theta \in \mathscr{D} (U,\rel^\adim)$ and considering the limit
	$i \to \infty$ yields the conclusion.
\end{proof}
\begin{example} \label{example:composite}
	Amongst the functions $g$ and $H$ admitting an approximation as in
	\ref{lemma:basic_v_weakly_diff} are the following:
	\begin{enumerate}
		\item \label{item:composite:scalar_mult} If $L : Y \to Y$ is
		a linear automorphism of $Y$, then $g = L$ and $H = DL$ is
		admissible.
		\item \label{item:composite:add_constant} If $b \in Y$ then $g
		= \boldsymbol{\tau}_b$ with $H = D\boldsymbol{\tau}_b$ is
		admissible.
		\item \label{item:composite:mod} If $Y$ is an inner product
		space and $b \in Y$ then one may take $g$ and $H$ such that $g
		(y) = |y-b|$ for $y \in Y$,
		\begin{gather*}
			\text{$H (y)(v) = |y-b|^{-1} (y-b) \bullet v$ if
			$y \neq b$}, \quad \text{$H (y) = 0$ if $y=b$}
		\end{gather*}
		whenever $v, y \in Y$.
		\item \label{item:composite:1d} If $Y = \rel$ and $b \in
		\rel$ then one may take $g$ and $H$ such that
		\begin{gather*}
			g (y) = \sup \{ y, b \}, \quad \text{$H (y)(v) = v$ if
			$y > b$}, \quad \text{$H (y) = 0$ if $y \leq b$}
		\end{gather*}
		whenever $v,y \in \rel$.
	\end{enumerate}

	\eqref{item:composite:scalar_mult} and
	\eqref{item:composite:add_constant} are trivial.

	To prove \eqref{item:composite:mod}, assume $Y = \rel^l$ for some $l
	\in \nat$ by \eqref{item:composite:scalar_mult}, choose $\varrho \in
	\mathscr{D} ( Y, \rel )^+$ with $\int \varrho \ud \mathscr{L}^l = 1$
	and $\varrho (y) = \varrho (-y)$ for $y \in Y$ and take $\kappa=1$,
	$\Upsilon = Y$, and $g_i = \varrho_{1/i} \ast g$ in
	\ref{lemma:basic_v_weakly_diff} noting $(\varrho_{1/i} \ast g) (b-y) =
	(\varrho_{1/i} \ast g)(b+y)$ for $y \in Y$, hence $D ( \varrho_{1/i}
	\ast g ) (b)=0$.

	To prove \eqref{item:composite:1d}, choose $\varrho \in \mathscr{D} (
	\rel, \rel )^+$ with $\int \varrho \ud \mathscr{L}^1 = 1$, $\spt
	\varrho \subset \cball{0}{1}$ and $\varepsilon = \inf \spt \varrho >
	0$, and take $\kappa=1$, $\Upsilon = \classification{\rel}{y}{y \geq
	b}$, and $g_i = \varrho_{1/i} \ast g$ in
	\ref{lemma:basic_v_weakly_diff} noting $g_i (y) = b$ if $-\infty < y
	\leq b + \varepsilon/i$, hence $D g_i (y) = 0$ for $- \infty < y \leq
	b$.
\end{example}
\begin{lemma} \label{lemma:functional_analysis}
	Suppose $U$ is an open subset of $\rel^\adim$, $\mu$ is a Radon
	measure over $U$, $Y$ is a finite dimensional normed vectorspace, $g
	\in \Lp{1} ( \mu )$, $K$ denotes the set of all $f \in \Lp{1} ( \mu, Y
	)$ such that
	\begin{gather*}
		|f(x)| \leq g ( x) \quad \text{for $\mu$ almost all $x$},
	\end{gather*}
	$L_1 ( \mu, Y ) = \Lp{1} ( \mu , Y ) / \{ f \with \Lpnorm{\mu}{1}{f} =
	0 \}$ is the (usual) quotient Banach space, and $\pi : \Lp{1} ( \mu, Y
	) \to L_1 ( \mu, Y )$ denotes the canonical projection.

	Then $\pi \lIm K \rIm$ with the topology induced by the weak topology
	on $L_1 ( \mu, Y )$ is compact and metrisable.
\end{lemma}
\begin{proof}
	First, notice that $\pi \lIm K \rIm$ is convex and closed, hence
	weakly closed by \cite[\printRoman{5}.3.13]{MR90g:47001a}. Therefore
	one may assume that $Y = \rel$ as a basis of $Y$ induces a linear
	homeomorphism $L_1 ( \mu, \rel )^{\dim Y} \simeq L_1 ( \mu, Y )$ with
	respect to the weak topologies on $L_1 ( \mu, \rel )$ and $L_1 ( \mu,
	Y )$. Since $L_1 ( \mu, \rel )$ is separable, the conclusion now
	follows combining \cite[\printRoman{4}.8.9, \printRoman{5}.6.1,
	\printRoman{5}.6.3]{MR90g:47001a}.
\end{proof}
\begin{lemma} \label{lemma:comp_lip}
	Suppose $\vdim$, $\adim$, $U$, $V$, and $Y$ are as in
	\ref{def:v_weakly_diff}, $f \in \trunc (V, Y)$, $Z$ is a finite
	dimensional normed vectorspace, $\Upsilon$ is a closed subset of $Y$,
	$c$ is the characteristic function of $f^{-1} \lIm \Upsilon \rIm$, and
	$g : Y \to Z$ is a Lipschitzian function such that $g | \Upsilon$ is
	proper and $g | Y \without \Upsilon$ is locally constant.

	Then $g \circ f \in \trunc (V, Z)$ and
	\begin{gather*}
		\| \derivative{V}{( g \circ f )} (x) \| \leq \Lip ( g )
		c(x) \| \derivative{V}{f} (x) \| \quad \text{for $\| V \|$
		almost all $x$}.
	\end{gather*}
\end{lemma}
\begin{proof}
	Suppose $0 < \varepsilon \leq 1$ and abbreviate $\kappa = \Lip g$.

	Define $B = Y \cap \{ y \with \dist (y,\Upsilon) \leq \varepsilon \}$
	and let $b$ denote the characteristic function of $f^{-1} \lIm B
	\rIm$.  Since $g | B$ is proper, one may employ convolution to
	construct $g_i \in \mathscr{E} ( Y, Z )$ satisfying $\Lip g_i \leq
	\kappa$, $\spt Dg_i \subset B$ and
	\begin{gather*}
		\delta_i = \sup \{ | (g - g_i) (y) | \with y \in
		Y \} \to 0 \quad \text{as $i \to \infty$}.
	\end{gather*}
	Therefore, if $\delta_i < \infty$ then $g_i|B$ is proper and
	$g_i \circ f \in \trunc (V,Z)$ with
	\begin{gather*}
		\| \derivative{V}{( g_i \circ f )}(x) \| \leq \kappa \, b(x) \|
		\derivative{V}{f} (x) \| \quad \text{for $\| V \|$ almost all
		$x$}
	\end{gather*}
	by \ref{lemma:basic_v_weakly_diff} with $\Upsilon$, $g$, and $H$
	replaced by $B$, $g_i$, and $Dg_i$. Choose a sequence of compact sets
	$K_j$ such that $K_j \subset \Int K_{j+1}$ for $j \in \nat$ and $U =
	\bigcup_{j=1}^\infty K_j$ and define $E(j) = K_j \cap \{ x \with
	|f(x)| < j \}$ for $j \in \nat$. In view of
	\ref{lemma:functional_analysis}, possibly passing to a subsequence by
	means of a diagonal process, there exist functions $F_j \in \Lp{1} (
	\| V \| \restrict E(j), \Hom ( \rel^\adim, Z ) )$ such that
	\begin{gather*}
		\| F_j(x) \| \leq \kappa \, b(x) \| \derivative{V}{f} (x) \|
		\quad \text{for $\| V \|$ almost all $x \in E_j$}, \\
		\tint{E(j)}{} \left <
		\derivative{V}{(g_i \circ f )},G \right> \ud \| V \| \to
		\tint{E(j)}{} \left< F_j,G\right> \ud \| V \| \quad \text{as
		$i \to \infty$}
	\end{gather*}
	whenever $G \in \Lp{\infty} \big ( \| V \|, \Hom (\rel^\adim, Z)^\ast
	\big )$ and $j \in \nat$. Noting $E_j \subset E_{j+1}$ and $F_j (x) =
	F_{j+1} (x)$ for $\| V \|$ almost all $E(j)$ for $j \in \nat$, one may
	define a $\| V \|$ measurable function $F$ by $F(x) = \lim_{j \to
	\infty} F_j(x)$ whenever $x \in U$.
	
	In order to verify $g \circ f \in \trunc (V,Z)$ with $\derivative{V} (
	g \circ f )(x) = F(x)$ for $\| V \|$ almost all $x$, suppose $\theta
	\in \mathscr{D} ( U, \rel^\adim )$ and $\gamma \in \mathscr{E} ( Z,
	\rel)$ with $\spt D\gamma$ compact. Then there exists $j \in \nat$
	with $\spt \theta \subset K_j$ and $(g|B)^{-1} \lIm \spt D \gamma \rIm
	\subset \oball{0}{j}$.  Define $G_i, G \in \Lp{\infty} \big ( \| V \|,
	\Hom ( \rel^\adim, Z)^\ast \big )$ by the requirements
	\begin{gather*}
		\left <\sigma,G_i(x)\right> = \left < \theta (x), D\gamma (
		g_i(f(x))) \circ \sigma \right>, \quad \left < \sigma, G(x)
		\right > = \left < \theta (x), D\gamma ( g(f(x)) ) \circ
		\sigma \right>
	\end{gather*}
	whenever $x \in \dmn f$ and $\sigma \in \Hom ( \rel^\adim, Z )$,
	hence
	\begin{gather*}
		\Lpnorm{\| V \|}{\infty}{G_i-G} \to 0 \quad \text{as $i \to
		\infty$}.
	\end{gather*}
	Observing $( g_i|B )^{-1} \lIm \spt D \gamma \rIm \subset
	\oball{0}{j}$ for large $i$, one infers
	\begin{gather*}
		\begin{aligned}
			& \tint{}{} \left < \theta (x), D \gamma ( g (f(x)))
			\circ F (x) \right > \ud \| V \| x = \tint{E(j)}{}
			\left <F,G\right> \ud \| V \| \\
			& \qquad = \lim_{i \to \infty} \tint{E(j)}{} \left <
			\derivative{V}{(g_i \circ f)} (x),G_i(x) \right > \ud
			\| V \| x \\
			& \qquad = \lim_{i \to \infty} \tint{}{} \left <
			\theta (x), D\gamma ( g_i(f(x)) \circ
			\derivative{V}{(g_i \circ f )}(x) \right > \ud \| V
			\| x
		\end{aligned}
	\end{gather*}
	as $\left <F,G\right> (x) = 0 = \left < \derivative{V}{( g_i \circ
	f)}, G_i \right > (x) $ for $\| V \|$ almost all $x \in U \without
	E(j)$.
\end{proof}
\begin{remark} \label{remark:mod_tv}
	Taking $\Upsilon = Y$ and $g(y) = |y|$ for $y \in Y$ yields $|f| \in
	\trunc (V)$ and
	\begin{gather*}
		\| \derivative{V}{|f|}(x) \| \leq \| \derivative{V}{f} (x) \|
		\quad \text{for $\| V \|$ almost all $x$}.
	\end{gather*}
\end{remark}
\begin{miniremark} \label{miniremark:trunc}
	\emph{Whenever $Y$ is a finite dimensional normed vectorspace, there
	exists a family of functions $g_s \in \mathscr{D} ( Y, Y
	)$ with $0 < s < \infty$ satisfying
	\begin{gather*}
		g_s (y) = y \quad \text{whenever $y \in Y \cap
		\cball{0}{s}$, $0 < s < \infty$}, \\
		\sup \{ \Lip g_s : 0 < s < \infty \} < \infty;
	\end{gather*}}
	in fact, one may assume $Y = \rel^l$ for some $l \in \nat$, select
	$\omega \in \mathscr{D} ( \rel, \rel )$ such that
	\begin{gather*}
		0 \leq \omega (t) \leq t \quad \text{for $0 \leq t < \infty$},
		\qquad \omega (t) = t \quad \text{for $-1 \leq t \leq 1$},
	\end{gather*}
	define $\omega_s = s \omega \circ \boldsymbol{\mu}_{1/s}$ and $g_s
	\in \mathscr{D} ( Y, Y )$ by
	\begin{gather*}
		g_s (y) = 0 \quad \text{if $y=0$}, \qquad g_s (y) =
		\omega_s (|y|) |y|^{-1} y \quad \text{if $y \neq 0$},
	\end{gather*}
	whenever $y \in Y$ and $0 < s < \infty$, and conclude
	\begin{gather*}
		\omega_s (t) = t \quad \text{for $-s \leq t \leq s$}, \qquad
		\Lip \omega_s = \Lip \omega, \\
		D g_s (y)(v) = \omega_s' ( |y| ) ( |y|^{-1} y ) \bullet v (
		|y|^{-1} y ) + \omega_s ( |y|) |y|^{-1} \big ( v
		- ( |y|^{-1} y ) \bullet v ( |y|^{-1} y ) \big )
	\end{gather*}
	whenever $y \in Y \without \{ 0 \}$, $v \in Y$, and $0 < s < \infty$,
	hence $\Lip g_s \leq 2 \Lip \omega < \infty$.
\end{miniremark}
\begin{lemma} \label{lemma:integration_by_parts}
	Suppose $\vdim$, $\adim$, $U$, $V$, and $Y$ are as in
	\ref{def:v_weakly_diff}, $f \in \trunc (V,Y) \cap \Lploc{1} ( \| V \|
	+ \| \delta V \|, Y )$, $\derivative{V}{f} \in \Lploc{1} ( \| V \|,
	\Hom ( \rel^\adim, Y))$, and $\alpha \in \Hom ( Y, \rel )$.

	Then $\alpha \circ f \in \trunc (V)$ and
	\begin{gather*}
		\derivative{V}{(\alpha \circ f)} (x) = \alpha \circ
		\derivative{V}{f} (x) \quad \text{for $\| V \|$ almost all
		$x$}, \\
		( \delta V ) ( ( \alpha \circ f ) \theta ) = \alpha \big (
		\tint{}{} ( \project{P} \bullet D \theta (x) ) f(x) + \left <
		\theta(x), \derivative{V}{f} (x) \right > \ud V (x,P) \big )
	\end{gather*}
	whenever $\theta \in \mathscr{D} ( U, \rel^\adim )$.
\end{lemma}
\begin{proof}
	If $f$ is bounded the conclusion follows from
	\ref{remark:lipschitzian_theta}. The general case may be treated by
	approximation based on \ref{lemma:basic_v_weakly_diff} and
	\ref{miniremark:trunc}.
\end{proof}
\begin{remark} \label{remark:comparison_trunc_spaces}
	If $l=1$, $\vdim = \adim$, and $\| V \| ( A ) = \mathscr{L}^\vdim (A)$
	for $A \subset U$, then $f \in \trunc (V)$ if and only if $f$ belongs
	to the class $\mathscr{T}^{1,1}_{\mathrm{loc}} ( U )$ introduced by
	B{\'e}nilan, Boccardo, Gallou{\"e}t, Gariepy, Pierre and V{\'a}zquez
	in \cite[p.~244]{MR1354907} and in this case $\derivative{V}{f}$
	corresponds to ``the derivative $Df$ of $f \in
	\mathscr{T}^{1,1}_{\mathrm{loc}} ( U)$'' of \cite[p.~246]{MR1354907}
	as may be verified by use of
	\ref{lemma:basic_v_weakly_diff},
	\ref{example:composite}\,\eqref{item:composite:scalar_mult}\,\eqref{item:composite:1d}, \ref{lemma:integration_by_parts},
	and \cite[2.1, 2.3]{MR1354907}.
\end{remark}
\begin{theorem} \label{thm:addition}
	Suppose $\vdim$, $\adim$, $U$, $V$, and $Y$ are as in
	\ref{def:v_weakly_diff}, and $f \in \trunc (V, Y)$.

	Then the following four statements hold:
	\begin{enumerate}
		\item \label{item:addition:zero} If $A = \{ x \with f(x) = 0
		\}$, then
		\begin{gather*}
			\derivative{V}{f} (x) = 0 \quad \text{for $\| V \|$
			almost all $x \in A$}.
		\end{gather*}
		\item \label{item:addition:join} If $Z$ is a finite dimensional
		normed vectorspace, $g : U \to Z$ is locally Lipschitzian, and
		$h(x) = (f(x),g(x))$ for $x \in \dmn f$, then $h \in \trunc (V,
		Y \times Z )$ and
		\begin{gather*}
			\derivative{V}{h} (x)(u) = ( \derivative{V}{f}(x)(u),
			\derivative{V}{g} (x)(u)) \quad \text{whenever $u \in
			\rel^\adim$}
		\end{gather*}
		for $\| V \|$ almost all $x$.
		\item \label{item:addition:add} If $g : U \to Y$ is
		locally Lipschitzian, then $f+g \in \trunc (V,Y)$ and
		\begin{gather*}
			\derivative{V}{(f+g)} (x) = \derivative{V}{f} (x) +
			\derivative{V}{g} (x) \quad \text{for $\| V \|$ almost
			all $x$}.
		\end{gather*}
		\item \label{item:addition:mult} If $f \in \Lploc{1} ( \| V
		\|, Y )$, $\derivative{V}{f} \in \Lploc{1} ( \| V \|,
		\Hom ( \rel^\adim, Y ) )$, and $g : U \to \rel$ is
		locally Lipschitzian, then $gf \in \trunc ( V, Y )$ and
		\begin{gather*}
			\derivative{V}{(gf)} (x) = \derivative{V}{g} (x)\,f(x)
			+ g(x) \derivative{V}{f} (x) \quad \text{for $\| V \|$
			almost all $x$}.
		\end{gather*}
	\end{enumerate}
\end{theorem}
\begin{proof} [Proof of \eqref{item:addition:zero}]
	By \ref{miniremark:trunc} in conjunction with
	\ref{lemma:basic_v_weakly_diff} one may assume $f$ to be bounded and
	$\derivative{V}{f} \in \Lploc{1} ( \| V \|, \Hom ( \rel^\adim, Y))$,
	hence by \ref{lemma:integration_by_parts} also $Y = \rel$. In this
	case it follows from \ref{lemma:basic_v_weakly_diff} and
	\ref{example:composite}\,\eqref{item:composite:scalar_mult}\,\eqref{item:composite:1d}
	that $f^+$ and $f^-$ satisfy the same hypotheses as $f$, hence
	\begin{gather*}
		( \delta V ) ( g \theta ) = \tint{}{} ( \project{P} \bullet
		D\theta (x) ) g (x) + \left < \theta (x), \derivative{V}{g}
		(x) \right > \ud V (x,P)
	\end{gather*}
	for $\theta \in \mathscr{D} ( U, \rel )$ and $g \in \{ f, f^+, f^- \}$
	by \ref{lemma:integration_by_parts}. Since $f=f^+-f^-$, this implies
	\begin{gather*}
		\derivative{V}{f} (x) = \derivative{V}{f^+} (x) -
		\derivative{V}{f^-} (x) \quad \text{for $\| V \|$ almost all
		$x$}
	\end{gather*}
	and the formulae derived in
	\ref{example:composite}\,\eqref{item:composite:1d} yield the
	conclusion.
\end{proof}
\begin{proof} [Proof of \eqref{item:addition:join}]
	Assume $\dim Y>0$. Define a $\| V \|$ measurable function $H$ with
	values in $\Hom ( \rel^\adim, Y \times Z)$ by
	\begin{gather*}
		H(x)(u) = ( \derivative{V}{f} (x)(u), \derivative{V}{g} (x)(u)
		) \quad \text{for $u \in \rel^\adim$}
	\end{gather*}
	whenever $x \in \dmn \derivative{V}{f} \cap \dmn \derivative{V}{g}$.
	It will proven that
	\begin{gather*}
		( \delta V ) ( ( \gamma \circ h ) \theta ) = \tint{}{} \gamma
		(h(x)) \project P \bullet D \theta (x) + \left < \theta (x), D
		\gamma ( h (x) ) \circ H (x) \right > \ud V (x,P)
	\end{gather*}
	whenever $\gamma \in \mathscr{D} ( Y \times Z, \rel )$ and $\theta \in
	\mathscr{D} (U,\rel^\adim )$; in fact, in view of
	\ref{remark:ind-limit} and \ref{miniremark:distrib_on_products} the
	problem reduces to the case that, for some $\mu \in \mathscr{D}
	(Y,\rel)$ and some $\nu \in \mathscr{D} (Z,\rel)$,
	\begin{gather*}
		\gamma (y,z) = \mu (y) \nu (z) \quad \text{for $(y,z) \in Y
		\times Z$}
	\end{gather*}
	in which case one computes, using \ref{remark:lipschitzian_theta} and
	\ref{example:lipschitzian},
	\begin{gather*}
		\begin{aligned}
			& ( \delta V ) ( ( \gamma \circ h ) \theta ) = (
			\delta V ) ( ( \mu \circ f ) ( \nu \circ g )
			\theta ) \\
			& \quad = \tint{}{} \mu ( f(x)) \big ( \left <
			\theta (x), D \nu ( g(x) ) \circ
			\derivative{V}{g} (x) \right > + \nu ( g(x))
			\project{P} \bullet D \theta (x) \big ) \ud V (x,P) \\
			& \quad \phantom{=} \ + \tint{}{} \nu ( g (x))
			\left < \theta (x), D \mu ( f (x) ) \circ
			\derivative{V}{f} (x) \right > \ud \| V \| x \\
			& \quad = \tint{}{} \gamma ( h (x) ) \project{P}
			\bullet D \theta (x) + \left < \theta (x), D \gamma (
			h (x) ) \circ H (x) \right > \ud V(x,P).
		\end{aligned}
	\end{gather*}
	If $\dim Y \geq 2$ or $f$ is bounded the conclusion now follows from
	\ref{remark:eq_condition_weak_diff}. Finally, to approximate $f$ in
	case $\dim Y = 1$, one assumes $Y = \rel$ and employs the functions
	$f_i$ defined by $f_i(x) = \sup \{ \inf \{ f(x), i \}, -i \}$ for $x
	\in \dmn f$ and $i \in \nat$ and notices that $f_i \in \trunc (V)$ and
	\begin{gather*}
		| \derivative V{f_i}(x) | \leq | \derivative{V}f(x) |, \qquad
		\derivative V{f_i} (x) \to \derivative Vf(x) \quad \text{as $i
		\to \infty$}, \\
		\text{$|f_i(x)| < |f(x)|$ implies $\derivative V{f_i} (x) =
		0$}
	\end{gather*}
	for $\| V \|$ almost all $x$ by
	\ref{lemma:basic_v_weakly_diff},
	\ref{example:composite}\,\eqref{item:composite:scalar_mult}\,\eqref{item:composite:1d}.
\end{proof}
\begin{proof} [Proof of \eqref{item:addition:add}]
	Assume $\dim Y > 0$ and that $\im g \subset \cball 0t$ for some $0 < t
	< \infty$. Define $h$ as in \eqref{item:addition:join} and let $L : Y
	\times Y \to Y$ denote addition.

	The following assertion will be shown. \emph{If $\gamma \in
	\mathscr{D} (Y,\rel)$ and $\theta \in \mathscr{D} (U,\rel^\adim)$,
	then
	\begin{gather*}
		\begin{aligned}
			& ( \delta V ) ( ( \gamma \circ L \circ h ) \theta ) \\
			& \quad = \tint{}{} \gamma (L(h(x))) \project{P}
			\bullet D \theta (x) + \left <
			\theta(x), D(\gamma \circ L)(h(x)) \circ
			\derivative{V}h (x) \right > \ud V(x,P).
		\end{aligned}
	\end{gather*}}
	For this purpose define $D = Y \times ( Y \cap \cball 0{2t})$ and
	choose $\varrho \in \mathscr{D} (Y, \rel)$ with
	\begin{gather*}
		\cball 0t \subset \Int \{ z \with \varrho (z)= 1 \}, \quad
		\spt \varrho \subset \cball 0{2t}.
	\end{gather*}
	Let $\phi : Y \times Y \to Y$ be defined by
	\begin{gather*}
		\phi (y,z) = \varrho (z) ( \gamma \circ L ) (y,z) \quad
		\text{for $(y,z) \in Y \times Y$}.
	\end{gather*}
	Noting that
	\begin{gather*}
		\text{$L|D$ is proper}, \quad \spt \phi \subset D \cap L^{-1}
		\lIm \spt \gamma \rIm, \quad \text{$\spt D\phi$ is compact},
		\\
		\phi (y,z) = ( \gamma \circ L ) ( y,z) \quad \text{and} \quad
		D \phi (y,z) = D( \gamma \circ L ) ( y,z)
	\end{gather*}
	for $y \in Y$ and $z \in Y \cap \cball 0t$, one uses
	\eqref{item:addition:join} and
	\ref{def:v_weakly_diff}\,\eqref{item:v_weakly_diff:partial} with $f$
	and $\gamma$ replaced by $h$ and $\phi$ to infer the assertion.

	If $\dim Y \geq 2$ or $f$ is bounded the conclusion now follows from
	the assertion of the preceding paragraph in conjunction with
	\ref{remark:eq_condition_weak_diff}. Finally, the case $\dim Y = 1$
	may be treated by means of approximation as in
	\eqref{item:addition:join}.
\end{proof}
\begin{proof} [Proof of \eqref{item:addition:mult}]
	Assume $g$ to be bounded, define $h$ as in \eqref{item:addition:join}
	and let $\mu : Y \times \rel \to Y$ be defined by $\mu (y,t) = ty$ for
	$y \in Y$ and $t \in \rel$. If $f$ is bounded the conclusion follows
	from \eqref{item:addition:join} and \ref{remark:lipschitzian_theta}
	with $f$ and $g$ replaced by $h$ and $\mu$. The general case then
	follows by approximation using \ref{lemma:basic_v_weakly_diff} and
	\ref{miniremark:trunc}.
\end{proof}
\begin{remark}
	The approximation procedure in the proof of \eqref{item:addition:join}
	uses ideas from \cite[4.1.2, 4.1.3]{MR41:1976}.
\end{remark}
\begin{remark}
	The need for some strong restriction on $g$ in
	\eqref{item:addition:join} and \eqref{item:addition:add} will be
	illustrated in \ref{example:star}.
\end{remark}
\begin{miniremark} \label{miniremark:measurablity}
	If $\phi$ is a measure, $A$ is $\phi$ measurable and $B$ is a
	$\phi \restrict A$ measurable subset of $A$, then $B$ is $\phi$
	measurable.
\end{miniremark}
\begin{theorem} \label{thm:tv_on_decompositions}
	Suppose $\vdim$, $\adim$, $U$, $V$, and $Y$ are as in
	\ref{def:v_weakly_diff}, $\Xi$ is a decomposition of $V$, $\xi$ is
	associated to $\Xi$ as in \ref{remark:decomp_rep}, $f_W \in \trunc (W,
	Y )$ for $W \in \Xi$, and
	\begin{gather*}
		{\textstyle f = \bigcup \{ f_W | \xi (W) \with W \in \Xi \},
		\quad F = \bigcup \{ \derivative{W}{f_W} | \xi (W) \with W
		\in \Xi \}}.
	\end{gather*}
	
	Then the following three statements hold:
	\begin{enumerate}
		\item \label{item:tv_on_decompositions:f} $f$ is $\| V \| + \|
		\delta V \|$ measurable.
		\item \label{item:tv_on_decompositions:F} $F$ is $\| V \|$
		measurable.
		\item \label{item:tv_on_decompositions:tv} If $\tint{K \cap \{
		x \with |f(x)| \leq s \}}{} \| F \| \ud \| V \| < \infty$
		whenever $K$ is a compact subset of $U$ and $0 \leq s <
		\infty$, then $f \in \trunc (V,Y)$ and
		\begin{gather*}
			\derivative{V}{f} (x) = F (x) \quad \text{for $\| V
			\|$ almost all $x$}.
		\end{gather*}
	\end{enumerate}
\end{theorem}
\begin{proof}
	Clearly, $f | \xi (W) = f_W | \xi (W)$ and $F | \xi (W) =
	\derivative{W}{f_W} | \xi (W)$ for $W \in \Xi$.
	\eqref{item:tv_on_decompositions:f} and
	\eqref{item:tv_on_decompositions:F} readily follow by means of
	\ref{miniremark:measurablity}, since $\| W \| = \| V \| \restrict \xi
	( W )$ and $\| \delta W \| = \| \delta V \| \restrict \xi (W)$ for $W
	\in \Xi$. The hypothesis of \eqref{item:tv_on_decompositions:tv}
	implies
	\begin{gather*}
		\begin{aligned}
			& ( \delta V ) ( ( \gamma \circ f ) \theta ) = \tsum{W
			\in \Xi}{} ( \delta W ) ( ( \gamma \circ f_W ) \theta )
			\\
			& \ = \tsum{W \in \Xi}{} \tint{}{} \gamma ( f_W (x))
			\project{P} \bullet D \theta (x) + \left < \theta (x), D
			\gamma ( f_W ( x ) ) \circ \derivative{W}{f_W} (x)
			\right > \ud W (x,P) \\
			& \ = \tsum{W \in \Xi}{} \tint{\xi(W) \times
			\grass{\adim}{\vdim}}{} \gamma ( f(x) ) \project{P}
			\bullet D \theta (x) + \left < \theta (x), D\gamma
			(f(x)) \circ F(x) \right > \ud V (x,P) \\
			& \ = \tint{}{} \gamma (f(x)) \project{P} \bullet D
			\theta (x) + \left < \theta (x), D \gamma(f(x)) \circ
			F (x) \right > \ud V (x,P)
		\end{aligned}
	\end{gather*}
	whenever $\theta \in \mathscr{D} ( U, \rel^\adim )$, $\gamma \in
	\mathscr{E} ( Y, \rel )$ and $\spt D\gamma$ is compact.
\end{proof}
\begin{example} \label{example:star}
	Suppose $R_j$, $V_j$, and $V$ are as in
	\ref{remark:nonunique_decomposition}. Define $f : \mathbf{C} \to
	\rel$, $g : \mathbf{C} \to \rel$, and $h : \mathbf{C} \to \rel$ by
	\begin{gather*}
		f(x) = 1 \quad \text{if $x \in R_1 \cup R_3 \cup R_5$}, \quad
		f(x) = 0 \quad \text{else}, \\
		g(x) = 1 \quad \text{if $x \in R_1 \cup R_4$}, \qquad g(x) = 0
		\quad \text{else}, \\
		h (x) = (f(x),g(x))
	\end{gather*}
	whenever $x \in \mathbf{C}$.

	Then $f \in \trunc (V_1+V_3+V_5)$, $g \in \trunc (V_1+V_4)$, hence
	$f,g \in \trunc (V)$
	with
	\begin{gather*}
		\derivative{V}{f} (x) = 0 = \derivative{V}{g}(x) \quad
		\text{for $\| V \|$ almost all $x$}
	\end{gather*}
	by \ref{thm:tv_on_decompositions}. However, neither $h \notin \trunc
	(V, \rel \times \rel)$ nor $f+g \in \trunc (V)$ nor $gf \in \trunc
	(V)$; in fact the characteristic function of $R_1$ which equals $gf$
	does not belong to $\trunc (V)$, hence $f+g \notin \trunc (V)$ by
	\ref{lemma:basic_v_weakly_diff} and $h \in \trunc (V,\rel \times
	\rel)$ would imply $f+g \in \trunc (V)$ by
	\ref{remark:lipschitzian_theta} with $f$ and $g$ replaced by $h$ and
	the addition on $\rel$.
\end{example}
\begin{remark} \label{remark:too_big_sobolev}
	Here some properties of the class $\mathbf{W} ( V, Y )$
	consisting of all $f \in \Lploc{1} ( \| V \| + \| \delta V \|,
	Y)$ such that for some $F \in \Lploc{1} ( \| V \|, \Hom (
	\rel^\adim, Y ) )$
	\begin{gather*}
		( \delta V ) ( ( \alpha \circ f ) \theta ) = \alpha \big (
		\tint{}{} ( \project{P} \bullet D \theta (x)) f (x) + \left <
		\theta (x), F(x) \right > \ud V (x,P) \big )
	\end{gather*}
	whenever $\theta \in \mathscr{D} (U, \rel^\adim )$ and $\alpha \in
	\Hom ( Y, \rel )$, associated with $V$ and $Y$ whenever $\vdim$,
	$\adim$, $U$, $V$, and $Y$ are as in \ref{def:v_weakly_diff} will be
	discussed briefly.

	Clearly, $F$ is $\| V \|$ almost unique and $\mathbf{W} (V, Y )$ is a
	vectorspace. Note that
	\begin{gather*}
		\classification{\trunc (V, Y ) \cap \Lploc{1} ( \| V \| + \|
		\delta V \|, Y )}{f} { \derivative{V}{f} \in \Lploc{1} ( \| V
		\|, \Hom ( \rel^\adim, Y ) }
	\end{gather*}
	is contained in $\mathbf{W} (V,Y)$ by
	\ref{lemma:integration_by_parts}.  However, it may happen that $f \in
	\mathbf{W} (V, Y )$ but $\varrho \circ f \notin \mathbf{W} (V,Y)$ for
	some $\varrho \in \mathscr{D} ( Y, \rel )$; in fact the function $f+g$
	constructed in \ref{example:star} provides an example.
\end{remark}
\begin{example} \label{example:axioms}
	The considerations of \ref{example:star} may be axiomatised as
	follows.

	Suppose $\Phi$ denotes the family of stationary one dimensional integral
	varifolds in $\rel^2$ and, whenever $V \in \Phi$, $C(V)$ is a class of
	real valued functions on $\rel^2$ satisfying the following conditions.
	\begin{enumerate}
		\item \label{item:axioms:decomp} If $\Xi$ is a decomposition
		of $V$, $\xi$ is associated to $\Xi$ as in
		\ref{remark:decomp_rep}, $\upsilon : \Xi \to \rel$ and $f :
		\rel^2 \to \rel$ satisfies
		\begin{gather*}
			f(x) = \upsilon (W) \quad \text{if $x \in \xi(W)$,
			$W \in \Xi$}, \qquad f(x) = 0 \quad \text{if $x \in
			\rel^2 \without {\textstyle\bigcup} \im \xi$},
		\end{gather*}
		then $f \in C(V)$.
		\item \label{item:axioms:addition} If $f,g \in C(V)$, then
		$f+g \in C(V)$.
		\item \label{item:axioms:truncation} If $f \in C(V)$, $\omega
		\in \mathscr{E} ( \rel, \rel )$ and $\spt \omega'$ is compact,
		then $\omega \circ f \in C(V)$.  \intertextenum{
		
		Then, using the terminology of
		\ref{remark:nonunique_decomposition}, the characteristic
		function of any ray $R_j$ belongs to $C(V)$. Moreover, the
		same holds, if the conditions \eqref{item:axioms:addition} and
		\eqref{item:axioms:truncation} are replaced by the following
		condition.}
		\item \label{item:axioms:multiplication} If $f,g \in C(V)$,
		then $fg \in C(V)$.
	\end{enumerate}
\end{example}
\begin{lemma} \label{lemma:level_sets}
	Suppose $\vdim$, $\adim$, $U$, and $V$ are as in
	\ref{def:v_weakly_diff}, $f \in \trunc (V)$, $y \in \rel$, and $E = \{
	x \with f(x)>y \}$.

	Then there holds
	\begin{gather*}
		\boundary{V}{E}(\theta) = \lim_{\varepsilon \to 0+}
		\varepsilon^{-1} \tint{\{x \with y < f(x) \leq
		y+\varepsilon\}}{} \left < \theta, \derivative{V}{f} \right >
		\ud \| V \| \quad \text{for $\theta \in \mathscr{D} (U,
		\rel^\adim )$}.
	\end{gather*}
\end{lemma}
\begin{proof}
	Suppose $\theta \in \mathscr{D} (U,\rel^\adim)$. Define $g_\varepsilon
	: \rel \to \rel$ by $g_\varepsilon (\upsilon) = \varepsilon^{-1} \inf
	\{ \varepsilon , \sup \{ \upsilon-y, 0 \} \}$ whenever $\upsilon \in
	\rel$ and $0 < \varepsilon \leq 1$. Notice that
	\begin{gather*}
		g_\varepsilon (\upsilon) = 0 \quad \text{if $\upsilon \leq
		y$}, \qquad g_\varepsilon (\upsilon) = 1 \quad \text{if
		$\upsilon \geq y+\varepsilon$}, \\
		g_\varepsilon (\upsilon) \uparrow 1 \quad \text{as $\varepsilon
		\to 0+$ if $\upsilon > y$}
	\end{gather*}
	whenever $\upsilon \in \rel$. Consequently, one infers $g_\varepsilon
	\circ f \in \trunc (V)$ with
	\begin{gather*}
		\text{$\derivative{V}{( g_\varepsilon \circ f )} (x) =
		\varepsilon^{-1} \derivative{V}{f} (x)$ if $y < f(x) \leq
		y+\varepsilon$}, \quad \text{$\derivative{V}{(
		g_\varepsilon \circ f )} (x) = 0$ else}
	\end{gather*}
	for $\| V \|$ almost all $x$ from \ref{lemma:basic_v_weakly_diff},
	\ref{example:composite}\,\eqref{item:composite:scalar_mult}\,\eqref{item:composite:add_constant}\,\eqref{item:composite:1d}
	and \ref{thm:addition}\,\eqref{item:addition:zero}. It follows
	\begin{gather*}
		\begin{aligned}
			\boundary{V}{E} ( \theta ) & = \lim_{\varepsilon
			\to 0+} ( \delta V ) ( ( g_\varepsilon \circ f )
			\theta ) - \lim_{\varepsilon \to 0+} \tint{}{} (
			g_\varepsilon \circ f ) (x) \project{P} \bullet D
			\theta (x) \ud V (x,P) \\
			& = \lim_{\varepsilon \to 0+} \varepsilon^{-1}
			\tint{\{ x \with y < f (x) \leq y + \varepsilon \}}{}
			\left < \theta(x), \derivative{V}{f} (x) \right > \ud
			\| V \| x,
		\end{aligned}
	\end{gather*}
	where \ref{lemma:integration_by_parts} with $\alpha$ and $f$ replaced
	by $\id{\rel}$ and $g_\varepsilon \circ f$ was employed.
\end{proof}
\begin{theorem} \label{thm:tv_coarea}
	Suppose $\vdim$, $\adim$, $U$, and $V$ are as in
	\ref{def:v_weakly_diff}, $f \in \trunc(V)$, $T \in \mathscr{D}' ( U
	\times \rel, \rel^\adim )$ satisfies
	\begin{gather*}
		T ( \phi ) = \tint{}{} \left < \phi (x,f(x)),
		\derivative{V}{f} (x) \right > \ud \| V \|x \quad \text{for
		$\phi \in \mathscr{D} (U \times \rel, \rel^\adim )$},
	\end{gather*}
	and $E(y) = \{ x \with f(x) > y \}$ for $y \in \rel$.

	Then $T$ is representable by integration and
	\begin{gather*}
		T ( \phi ) = \tint{}{} \boundary{V}{E(y)} ( \phi (\cdot, y ) )
		\ud \mathscr{L}^1 y, \quad \tint{}{} g \ud \| T \| = \tint{}{}
		\tint{}{} g (x,y) \ud \| \boundary{V}{E(y)} \| x \ud
		\mathscr{L}^1 y
	\end{gather*}
	whenever $\phi \in \Lp{1} ( \| T \|, \rel^\adim )$ and $g$ is an
	$\overline{\rel}$ valued $\| T \|$ integrable function. In particular,
	for $\mathscr{L}^1$ almost all $y$, the distribution $\boundary
	V{E(y)}$ is representable by integration and $\| \boundary V{E(y)} \|
	( U \cap \{ x \with f(x) \neq y \} ) = 0$.
\end{theorem}
\begin{proof}
	Taking \ref{remark:associated_distribution} into account,
	\ref{lemma:meas_fct} yields
	\begin{gather*}
		\tint{}{} \omega \circ f \left < \theta, \derivative{V}{f}
		\right > \ud \| V \| = T_{(x,y)} ( \omega(y) \theta (x) ) =
		\tint{}{} \omega(y) \boundary{V}{E(y)} ( \theta ) \ud
		\mathscr{L}^1 y
	\end{gather*}
	for $\omega \in \mathscr{D} ( \rel,\rel )$ and $\theta \in \mathscr{D}
	(U,\rel^\adim)$. In view of \ref{lemma:level_sets}, the principal
	conclusion is implied by
	\ref{thm:distribution_on_product}\,\eqref{item:distribution_on_product:inequality}\,\eqref{item:distribution_on_product:absolute}
	with $J = \rel$ and $Z = \rel^\adim$. The postscript follows employing
	\ref{remark:associated_distribution} and noting that $( U \times \rel
	) \cap \{ (x,y) \with f(x) \neq y \}$ is $\| T \|$ measurable since
	$f$ is $\| V \|$ almost equal to a Borel function by
	\cite[2.3.6]{MR41:1976}.
\end{proof}
\begin{remark}
	The formulation of \ref{thm:tv_coarea} is modelled on
	\cite[4.5.9\,(13)]{MR41:1976}.
\end{remark}
\begin{remark} \label{remark:no_coarea_ineq_for_too_big_sobolev}
	The equalities in \ref{thm:tv_coarea} are not valid for functions in
	$\mathbf{W} (V)$ with $\derivative{V}{f}$ in the definition of $T$
	replaced by the function $F$ occurring in the definition of $\mathbf{W}
	(V)$, see \ref{remark:too_big_sobolev}; in fact, the function $f+g$
	constructed in \ref{example:star} provides an example.
\end{remark}
\begin{corollary} \label{corollary:boundary_controls_interior}
	Suppose $\vdim$, $\adim$, $U$, $V$, and $Y$ are as in
	\ref{def:v_weakly_diff} and $f \in \trunc (V,Y)$.

	Then there holds
	\begin{gather*}
		\eqLpnorm{\| \delta V \| \restrict X}{\infty}{f} \leq
		\eqLpnorm{\| V \| \restrict X}{\infty}{f}
	\end{gather*}
	whenever $X$ is an open subset of $U$.
\end{corollary}
\begin{proof}
	Assume $X = U$ and abbreviate $s = \Lpnorm{\| V \|}{\infty}{f}$.
	Recalling \ref{remark:mod_tv}, one applies \ref{thm:tv_coarea} with
	$f$ replaced by $|f|$ to infer
	\begin{gather*}
		\boundary{V}{E(t)} = 0 \quad \text{for $\mathscr{L}^1$ almost
		all $s < t < \infty$},
	\end{gather*}
	where $E(t) = \{ x \with |f(x)|>t \}$, hence $\| \delta V \| (E(t)) =
	0$ for those $t$.
\end{proof}
\begin{theorem} \label{thm:zero_derivative}
	Suppose $\vdim$, $\adim$, $U$, $V$, and $Y$ are as in
	\ref{def:v_weakly_diff}, $f \in \trunc ( V, Y )$, and
	\begin{gather*}
		\derivative{V}{f} (x) = 0 \quad \text{for $\| V \|$ almost all
		$x$}.
	\end{gather*}

	Then there exists a decomposition $\Xi$ of $V$ and $\upsilon : \Xi \to
	Y$ such that
	\begin{gather*}
		f(x)=\upsilon(W) \quad \text{for $\| W \| + \| \delta W \|$
		almost all $x$}
	\end{gather*}
	whenever $W \in \Xi$.
\end{theorem}
\begin{proof}
	Define $B(y,r) = \{ x \with | f(x) - y | \leq r \}$ for $y \in Y$
	and $0 \leq r < \infty$. First, one observes that \ref{thm:tv_coarea}
	in conjunction with \ref{lemma:basic_v_weakly_diff},
	\ref{example:composite}\,\eqref{item:composite:add_constant} and
	\ref{remark:mod_tv} implies
	\begin{gather*}
		\boundary{V}{B(y,r)} = 0 \quad \text{for $y \in Y$ and $0
		\leq r < \infty$}.
	\end{gather*}

	Next, a countable subset $\Upsilon$ of $Y$ such that
	\begin{gather*}
		f(x) \in \Upsilon \quad \text{for $\| V \|$ almost all $x$}
	\end{gather*}
	will be constructed. For this purpose define $\delta_i =
	\unitmeasure{\vdim} 2^{-\vdim-1} i^{-1-2\vdim}$, $\varepsilon_i =
	2^{-1} i^{-2}$ and Borel sets $A_i$ consisting of all $a \in
	\rel^\adim$ satisfying
	\begin{gather*}
		|a| \leq i, \quad \oball{a}{2\varepsilon_i} \subset U, \quad
		\density^\vdim ( \| V \|, a ) \geq 1/i, \\
		\measureball{\| \delta V \|}{\cball{a}{s}} \leq
		\unitmeasure{\vdim} is^\vdim \quad \text{for $0 < s <
		\varepsilon_i$}
	\end{gather*}
	whenever $i \in \nat$. Clearly, $A_i \subset A_{i+1}$ for $i \in \nat$
	and $\| V \| ( U \without \bigcup_{i=1}^\infty A_i ) = 0$ by Allard
	\cite[3.5\,(1a)]{MR0307015} and \cite[2.8.18, 2.9.5]{MR41:1976}.
	Abbreviating $\mu_i = f_\# ( \| V \| \restrict A_i )$, one obtains
	\begin{gather*}
		\| V \| ( B(y,r) \cap \{ x \with \dist ( x, A_i ) \leq
		\varepsilon_i \} ) \geq \delta_i
	\end{gather*}
	whenever $y \in \spt \mu_i$, $0 \leq r < \infty$ and $i \in \nat$; in
	fact, assuming $r > 0$ there exists $a \in A_i$ with $\density^\vdim (
	\| V \| \restrict B(y,r), a ) = \density^\vdim ( \| V \|, a ) \geq
	1/i$ by \cite[2.8.18, 2.9.11]{MR41:1976}, hence
	\begin{gather*}
		\tint{0}{\varepsilon_i} s^{-\vdim} \measureball{\| \delta ( V
		\restrict B(y,r) \times \grass{\adim}{\vdim} )
		\|}{\cball{a}{s}} \ud \mathscr{L}^1 s \leq \unitmeasure{\vdim}
		i \varepsilon_i
	\end{gather*}
	implying $\| V \| ( B (y,r) \cap \cball{a}{\varepsilon_i} ) \geq
	\delta_i$ by \ref{corollary:monotonicity} and
	\ref{remark:monotonicity}. As $\{ B(y,0) \with y \in \spt \mu_i \}$ is
	disjointed, one deduces
	\begin{gather*}
		\delta_i \card \spt \mu_i \leq \| V \| ( U \cap \{ x \with
		\dist (x,A_i) \leq \varepsilon_i \} ) < \infty
	\end{gather*}
	and one may take $\Upsilon = \bigcup_{i=1}^\infty \spt \mu_i$, since
	$\mu_i ( Y \without \spt \mu_i ) = 0$.

	Applying \ref{thm:decomposition} to $V \restrict B(y,0) \times
	\grass{\adim}{\vdim}$ for $y \in \Upsilon$ and recalling
	\ref{remark:iterated_boundaries} and \ref{remark:decomp_rep}, one
	constructs a decomposition $\Xi$ of $V$ and a function $\xi$ mapping
	$\Xi$ into the class of all $\| V \| + \| \delta V \|$ measurable sets
	such that distinct members of $\Xi$ are mapped onto disjoint sets,
	\begin{gather*}
		W = V \restrict \xi(W) \times \grass{\adim}{\vdim},
		\quad \boundary{V}{\xi(W)} = 0
	\end{gather*}
	whenever $W \in \Xi$, and each $\xi (W)$ is contained in some
	$B(y,0)$. Defining $\upsilon : \Xi \to Y$ by the requirement $\{
	\upsilon (W) \} = f \lIm \xi ( W ) \rIm$ for $W \in \Xi$ and noting $(
	\| W \| + \| \delta W \| ) ( U \without \xi (W)) = 0$ for $W \in \Xi$,
	the conclusion is now evident.
\end{proof}
\begin{remark}
	Clearly, the second paragraph of the proof has conceptual overlap with
	the second and third paragraph of the proof of
	\ref{thm:decomposition}.
\end{remark}
\section{Zero boundary values} \label{sec:zero}
In this section a notion of zero boundary values for nonnegative weakly
differentiable functions based on the behaviour of superlevel sets is
introduced. Stability of this class under composition, see
\ref{lemma:trunc_tg}, convergence, see \ref{lemma:closeness_tg} and
\ref{remark:closeness_tg}, and multiplication by a nonnegative Lipschitzian
function, see \ref{thm:mult_tg}, are investigated. The deeper parts of this
study rest on a characterisation of such functions in terms of an associated
distribution built from certain distributional boundaries of superlevel sets,
see \ref{thm:char_tg}.
\begin{definition} \label{def:trunc_g}
	Suppose $\vdim, \adim \in \nat$, $\vdim \leq \adim$, $U$ is an open
	subset of $\rel^\adim$, $V \in \RVar_\vdim ( U )$, $\| \delta V \|$ is
	a Radon measure, and $G$ is a relatively open subset of $\Bdry U$.

	Then $\trunc_G (V)$ is defined to be the set of all nonnegative
	functions $f \in \trunc ( V)$ such that with
	\begin{gather*}
		B = ( \Bdry U ) \without G \qquad \text{and} \qquad E(y) = \{
		x \with f(x) > y \} \quad \text{for $0 < y < \infty$}
	\end{gather*}
	the following conditions hold for $\mathscr{L}^1$ almost all $0 < y <
	\infty$, see
	\ref{def:variation_measure}--\ref{def:restriction_distribution},
	\begin{gather*}
		( \| V \| + \| \delta V \| ) ( E(y) \cap K ) + \|
		\boundary{V}{E(y)} \| ( U \cap K ) < \infty, \\
		\tint{E(y) \times \grass{\adim}{\vdim}}{} \project{P} \bullet
		D \theta (x) \ud V(x,P) = ( (\delta V) \restrict E(y) ) (
		\theta | U ) - \boundary{V}{E(y)} ( \theta|U )
	\end{gather*}
	whenever $K$ is a compact subset of $\rel^\adim \without B$ and
	$\theta \in \mathscr{D} ( \rel^\adim \without B, \rel^\adim )$.
\end{definition}
\begin{remark} \label{remark:trunc}
	Notice that $|f| \in \trunc_\varnothing (V)$ whenever $Y$ is a finite
	dimensional normed vectorspace and $f \in \trunc (V,Y)$ by
	\ref{remark:mod_tv} and \ref{thm:tv_coarea}.
\end{remark}
\begin{remark}
	The condition on $E(y)$ has been studied in Section \ref{sec:rel_iso}
	under the supplementary hypothesis on $V$ that $\density^\vdim ( \| V
	\|, x ) \geq 1$ for $\| V \|$ almost all $x$, see
	\ref{miniremark:zero_boundary}. In the present section this hypothesis
	will not occur leaving those properties of $\trunc_G (V)$ employing
	the additionally hypothesis on $V$ to Section \ref{sec:embeddings}.
\end{remark}
\begin{lemma} \label{lemma:boundary}
	Suppose $\vdim$, $\adim$, $U$, $V$, $G$, $B$, and $E(y)$ are as in
	\ref{def:trunc_g}, $f \in \trunc_G (V)$, and
	\begin{gather*}
		\Clos \spt \big ( ( \| V \| + \| \delta V \| ) \restrict E(y)
		\big ) \subset \rel^\adim \without B \quad \text{for
		$\mathscr{L}^1$ almost all $0 < y < \infty$}.
	\end{gather*}

	Then $f \in \trunc_{\Bdry U} ( V )$.
\end{lemma}
\begin{proof}
	Define $A(y) = \Clos \spt \big ( ( \| V \| + \| \delta V \| )
	\restrict E(y) \big )$ for $0 < y < \infty$. Suppose $y$ satisfies the
	conditions of \ref{def:trunc_g} and $A(y) \subset \rel^\adim \without
	B$. One concludes
	\begin{gather*}
		( \| V\| + \| \delta V \| ) ( E(y) \cap K ) + \| \boundary
		V{E(y)}\| ( U \cap K ) < \infty
	\end{gather*}
	whenever $K$ is a compact subset of $\rel^\adim$. Hence one may
	define $S \in \mathscr{D}' (\rel^\adim,\rel^\adim)$ by
	\begin{gather*}
		S(\theta) = \tint{E(y) \times \grass \adim \vdim}{} \project
		P \bullet D \theta (x) \ud V(x,P) - ( ( \delta V) \restrict
		E(y) ) ( \theta | U ) + \boundary V{E(y)} ( \theta | U )
	\end{gather*}
	whenever $\theta \in \mathscr{D} ( \rel^\adim, \rel^\adim )$. Notice
	that $\spt S \subset A(y) \subset \rel^\adim \without B$. On the other
	hand the conditions of \ref{def:trunc_g} imply $\spt S \subset B$. It
	follows $\spt S = \varnothing$ and $S = 0$.
\end{proof}
\begin{lemma} \label{lemma:restriction_tg}
	Suppose $\vdim$, $\adim$, $U$, $V$, $G$, and $B$ are as in
	\ref{def:trunc_g}, $f \in \trunc_G ( V )$, $X$ is an open subset of
	$\rel^\adim \without B$, $H = X \cap \Bdry U$, and $W = V |
	\mathbf{2}^{( U \cap X ) \times \grass{\adim}{\vdim}}$.

	Then $f | X \in \trunc_H ( W )$.
\end{lemma}
\begin{proof}
	Notice that $H = X \cap \Bdry ( U \cap X )$. In particular, $H$ is a
	relatively open subset of $\Bdry ( U \cap X )$. Let $C = ( \Bdry ( U
	\cap X ) ) \without H$ and observe the inclusions
	\begin{gather*}
		( \rel^\adim \without C ) \cap \Clos ( U \cap X ) \subset X
		\subset \rel^\adim \without ( B \cup C ).
	\end{gather*}

	Suppose $0 < y < \infty$ satisfies the conditions of
	\ref{def:trunc_g}. Define $E = \{ x \with f(x) > y \}$ and notice
	that
	\begin{gather*}
		( \| W \| + \| \delta W \| ) ( E \cap X \cap K ) +
		\| \boundary W {(E \cap X)} \| ( U \cap X \cap K ) < \infty
	\end{gather*}
	whenever $K$ is a compact subset of $\rel^\adim \without C$ by the
	first inclusion of the first paragraph. Therefore one may define $S
	\in \mathscr{D}' ( \rel^\adim \without C, \rel^\adim )$ by
	\begin{align*}
		S(\theta) & = \tint{(E \cap X) \times \grass \adim \vdim}{}
		\project P \bullet D \theta (x) \ud W(x,P) \\
		& \phantom = \ - ( ( \delta W ) \restrict E \cap X) ( \theta |
		U \cap X ) + ( \boundary W{(E \cap X)} ) ( \theta | U \cap X )
	\end{align*}
	whenever $\theta \in \mathscr{D} ( \rel^\adim \without C, \rel^\adim
	)$. Notice that $\spt S \subset ( \rel^\adim \without C ) \cap \Clos ( U
	\cap X ) \subset X$. On the other hand the conditions of
	\ref{def:trunc_g} in conjunction with the second inclusion of the
	first paragraph imply
	\begin{align*}
		& S( \theta | \rel^\adim \without C ) \\
		& \qquad = \tint{E \times \grass
		\adim \vdim}{} \project P \bullet D \theta (x) \ud V(x,P) - (
		( \delta V ) \restrict E ) ( \theta | U ) + \boundary VE(
		\theta |U ) = 0
	\end{align*}
	whenever $\theta \in \mathscr{D} ( \rel^\adim, \rel^\adim )$ and $\spt
	\theta \subset X$, hence $X \cap \spt S = \varnothing$. It follows
	$\spt S = \varnothing$ and $S = 0$.
\end{proof}
\begin{remark} \label{remark:restriction_tg}
	Recalling the first paragraph of the proof of
	\ref{lemma:restriction_tg}, one notices that
	\begin{gather*}
		\classification{U}{x}{ \oball{x}{r} \subset X } \subset
		\classification{U \cap X}{x}{\oball{x}{r} \subset \rel^\adim
		\without C } \quad \text{for $0 < r < \infty$};
	\end{gather*}
	this fact will be useful for localisation in
	\ref{thm:sob_poin_summary}\,\eqref{item:sob_poin_summary:interior:p=m<q}.
\end{remark}
\begin{example}
	Suppose $\vdim = \adim = 1$, $U = \rel \without \{ 0 \}$, $V \in
	\RVar_\vdim ( U )$ with $\| V \| ( A ) = \mathscr{L}^1 ( A )$ for $A
	\subset U$, $f : U \to \rel$ with $f(x) = 1$ for $x \in U$, $X =
	\classification{U}{x}{x<0}$, and $W = V | \mathbf{2}^{X \times
	\grass{\adim}{\vdim}}$.

	Then $\delta V = 0$ and $f \in \trunc_{\{ 0 \}} ( V )$ but $f | X
	\notin \trunc_{\{ 0 \}} (W)$.
\end{example}
\begin{remark}
	The preceding example shows that \ref{lemma:restriction_tg} may not be
	sharpened by allowing $H$ to be an arbitrary relatively open subset of
	$\Bdry ( U \cap X )$ contained in $G$.
\end{remark}
\begin{lemma} \label{lemma:trunc_tg}
	Suppose $\vdim$, $\adim$, $U$, $V$, and $G$ are as in
	\ref{def:trunc_g}, $f \in \trunc_G ( V )$, $\Upsilon$ is a closed
	subset of $\{ y \with 0 \leq y < \infty \}$, $g : \{ y \with 0
	\leq y < \infty \} \to \{ z \with 0 \leq z < \infty \}$ is a
	Lipschitzian function such that $g(0)=0$, $g|\Upsilon$ is proper and
	$g | \rel \without \Upsilon$ is locally constant.

	Then $g \circ f \in \trunc_G ( V )$.
\end{lemma}
\begin{proof}
	Let $h = g \cup \{ (y,0) \with - \infty < y < 0 \}$. Since $\Lip h =
	\Lip g$ and $h | \rel \without \Upsilon$ is locally constant,
	\ref{lemma:comp_lip} yields $g \circ f = h \circ f \in \trunc (V)$.

	Define $F$ to be the class of all Borel subsets $Y$ of $\{ y \with 0 <
	y < \infty \}$ such that $E = f^{-1} \lIm Y \rIm$ satisfies
	\begin{gather*}
		( \| V \| + \| \delta V \| ) ( E \cap K ) + \| \boundary{V}{E}
		\| ( U \cap K ) < \infty, \\
		\tint{E \times \grass{\adim}{\vdim}}{} \project{P} \bullet D
		\theta (x) \ud V(x,P) = ( ( \delta V ) \restrict E ) (
		\theta|U ) - ( \boundary{V}{E} ) ( \theta|U )
	\end{gather*}
	whenever $K$ is a compact subset of $\rel^\adim \without B$ and
	$\theta \in \mathscr{D} ( \rel^\adim \without B, \rel^\adim )$. If $Y
	\in F$, $G$ is a finite disjointed subfamily of $F$ and $\bigcup G
	\subset Y$, then $Y \without \bigcup G \in F$; as one readily verifies
	using \ref{remark:v_boundary}. Let $O$ denote the set of $0 < b <
	\infty$ such that either $\{ y \with b < y < \infty \}$ or $\{ y \with
	b \leq y < \infty \}$ does not belong to $F$. Since $( \| V \| + \|
	\delta V \| ) ( f^{-1} \lIm \{ b \} \rIm ) = 0$ for all but countably
	many $b$, one obtains $\mathscr{L}^1 ( O ) = 0$.

	Next, it will be shown that
	\begin{gather*}
		D = g^{-1} \lIm \{ z \with c < z < \infty \} \rIm \in F
	\end{gather*}
	whenever $c$ satisfies $0 < c \in \rel \without g \lIm O \rIm$ and
	$N(g,c) < \infty$. For this purpose assume $D \neq \varnothing$. Since
	$\inf D > 0$ and $\Bdry D \subset \{ y \with g(y) = c \}$, one infers
	that $\Bdry D$ is a finite subset of $\{ b \with 0 < b < \infty \}
	\without O$, in particular $Y = \{ y \with \inf D < y < \infty \} \in
	F$. Let $G$ denote the family of connected components of $Y \without
	D$. Observe that $G$ is a finite disjointed family of possibly
	degenerated closed intervals. Since
	\begin{gather*}
		\Bdry I \subset \Bdry D, \quad I = \{ y \with \inf I \leq y <
		\infty \} \without \{ y \with \sup I < y < \infty \}
	\end{gather*}
	whenever $I \in G$, it follows $G \subset F$ and $D = Y \without
	\bigcup G \in F$.

	Finally, notice that $\mathscr{L}^1$ almost all $0 < c < \infty$
	satisfy $N (g|\{ y \with y \leq i \}, c ) < \infty$ for $i \in \nat$
	by \cite[3.2.3\,(1)]{MR41:1976}, hence $\infty > N(g|\Upsilon,c) =
	N(g,c)$ for such $c$.
\end{proof}
\begin{example}
	Suppose $\vdim = \adim = 1$, $U = \rel \without \{ 0 \}$, and $V \in
	\RVar_\vdim ( U )$ with $\| V \| ( A ) = \mathscr{L}^1 (A)$ for $A
	\subset U$.
	
	Then $\delta V = 0$ and the following two statements hold.
	\begin{enumerate}
		\item If $f = \sign | U$ then $f \in \trunc (V)$ and $|f| \in
		\trunc_{\{0\}} ( V )$ but $f^+ \notin \trunc_{\{0\}} (V)$.
		\item If $y,b \in \rel^2$, $|y|=|b|$, $\nu$ is a norm on
		$\rel^2$, $\nu(y) \neq \nu (b)$, $f(x)=y$ for $0>x \in \rel$
		and $f(x)=b$ for $0<x \in \rel$, then $f \in \trunc(V,\rel^2)$
		and $|f| \in \trunc_{\{0\}} ( V )$ but $\nu \circ f \notin
		\trunc_{\{0\}} ( V )$.
	\end{enumerate}
\end{example}
\begin{remark} \label{remark:signed_functions_with_zero_boundary}
	The preceding example shows that the class $\trunc (V,Y) \cap \{ f
	\with |f| \in \trunc_G (V) \}$, where $Y$ is a finite dimensional
	normed vectorspace, does not satisfy stability properties similar to
	those proven for $\trunc_G (V)$ in \ref{lemma:trunc_tg}.
\end{remark}
\begin{theorem} \label{thm:char_tg}
	Suppose $\vdim$, $\adim$, $U$, $V$, $G$, and $B$ are as in
	\ref{def:trunc_g}, $f \in \trunc (V)$, $f$ is nonnegative, $J = \{ y
	\with 0 < y < \infty \}$, $A = f^{-1} \lIm J \rIm$, $E(y) =
	\classification{U}{x}{f(x) > y }$ for $y \in J$,
	\begin{gather*}
		( \| V \| + \| \delta V \| ) ( K \cap E(y) ) < \infty
	\end{gather*}
	whenever $K$ is a compact subset of $\rel^\adim \without B$ and $y \in
	J$, and the distributions $S(y) \in \mathscr{D}' ( \rel^\adim \without
	B, \rel^\adim )$ and $T \in \mathscr{D}' ( ( \rel^\adim \without B )
	\times J, \rel^\adim )$ satisfy
	\begin{gather*}
		S(y)(\theta) = ( ( \delta V ) \restrict E(y) ) ( \theta | U )
		- \tint{E(y) \times \grass{\adim}{\vdim}}{} \project{P}
		\bullet D \theta (x) \ud V (x,P) \quad \text{for $y \in J$},
		\\
		T ( \phi ) = \tint J{} S(y) ( \phi (\cdot,y) ) \ud
		\mathscr{L}^1 y
	\end{gather*}
	whenever $\theta \in \mathscr{D} ( \rel^\adim \without B, \rel^\adim
	)$ and $\phi \in \mathscr{D} ( ( \rel^\adim \without B ) \times J,
	\rel^\adim )$.

	Then the following three conditions are equivalent:
	\begin{enumerate}
		\item \label{item:char_tg:def} $\tint{K \cap \{ x \with f(x)
		\in I \}}{} | \derivative Vf(x) | \ud \| V \| x < \infty$
		whenever $K$ is a compact subset of $\rel^\adim \without B$
		and $I$ is a compact subset of $J$, and $f \in \trunc_G (V)$.
		\item \label{item:char_tg:slices} $\tint{I}{} \| S(y) \| ( K )
		\ud \mathscr{L}^1 y < \infty$ whenever $K$ is a compact subset
		of $\rel^\adim \without B$ and $I$ is a compact subset of $J$,
		and $\| S(y) \| ( ( \rel^\adim \without B ) \without U ) = 0$
		for $\mathscr{L}^1$ almost all $y \in J$.
		\item \label{item:char_tg:graph} $T$ is representable by
		integration and $\| T \| ( ( ( \rel^\adim \without B ) \without
		U ) \times J ) = 0$.
	\intertextenum{If these conditions are satisfied then the following
	three statements hold:}
		\item \label{item:char_tg:relate_sy_boundary} For
		$\mathscr{L}^1$ almost all $y \in J$, there holds
		\begin{gather*}
			S(y)(\theta) = \boundary V{E(y)} ( \theta|U ) \quad
			\text{whenever $\theta \in \mathscr{D} ( \rel^\adim
			\without B, \rel^\adim )$}.
		\end{gather*}
		\item \label{item:char_tg:T} If $\phi \in \Lp{1} ( \| T \|,
		\rel^\adim )$, then
		\begin{gather*}
			T(\phi) = \tint J{} \boundary V{E(y)} ( \phi
			(\cdot,y) | U ) \ud \mathscr{L}^1 y = \tint A{} \left
			< \phi(x,f(x)), \derivative Vf(x) \right > \ud \| V \|
			x.
		\end{gather*}
		\item \label{item:char_tg:bar_T} If $g$ is an $\overline \rel$
		valued $\| T \|$ integrable function, then
		\begin{align*}
			\tint{}{} g \ud \| T \| & = \tint J{} \tint{}{} g(x,y)
			\ud \| \boundary V{E(y)} \| x \ud \mathscr{L}^1 y \\
			& = \tint A{} g (x,f(x)) | \derivative Vf(x) | \ud \|
			V \| x.
		\end{align*}
	\end{enumerate}
\end{theorem}
\begin{proof}
	Notice that $S(y)(\theta) = \boundary V{E(y)} ( \theta | U )$ whenever
	$\theta \in \mathscr{D} ( \rel^\adim \without B, \rel^\adim )$ and
	$\spt \theta \subset U$. From \ref{remark:associated_distribution} and
	\ref{thm:tv_coarea} one infers
	\begin{gather*}
		\tint{K \cap \{ x \with f(x) \in I \}}{} | \derivative Vf(x) |
		\ud \| V \| x = \tint I{} \| \boundary V{E(y)} \| ( U \cap K )
		\ud \mathscr{L}^1 y
	\end{gather*}
	whenever $K$ is a compact subset of $\rel^\adim \without B$ and $I$ is
	a compact subset of $J$. Consequently, \eqref{item:char_tg:def} and
	\eqref{item:char_tg:slices} are equivalent.

	Moreover, one remarks that $S$ is $\mathscr{L}^1 \restrict J$
	measurable by \ref{example:distrib_lusin} and, employing a countable
	sequentially dense subset of $\mathscr{D} ( \rel^\adim \without B,
	\rel^\adim )$, one obtains that, for $\mathscr{L}^1$ almost all $y \in
	J$,
	\begin{gather*}
		S(y)(\theta) = \lim_{\varepsilon \to 0+} \varepsilon^{-1}
		\tint{y}{y+\varepsilon} S ( \upsilon ) ( \theta ) \ud
		\mathscr{L}^1 \upsilon \quad \text{whenever $\theta \in
		\mathscr{D} ( \rel^\adim \without B, \rel^\adim )$}
	\end{gather*}
	by \ref{remark:lusin}, \ref{remark:ind-limit}, and \cite[2.8.17,
	2.9.8]{MR41:1976}. Defining $R_\theta \in \mathscr{D}' ( J,\rel )$ by
	\begin{gather*}
		R_\theta ( \omega ) = T_{(x,y)} ( \omega (y) \theta (x) )
		\quad \text{whenever $\omega \in \mathscr{D} (J,\rel )$ and
		$\theta \in \mathscr{D} ( \rel^\adim \without B, \rel^\adim
		)$},
	\end{gather*}
	one notes that $\| R_\theta \|$ is absolutely continuous with respect
	to $\mathscr{L}^1 | \mathbf{2}^J$ for $\theta \in \mathscr{D} (
	\rel^\adim \without B, \rel^\adim )$. If \eqref{item:char_tg:slices}
	holds, then $T$ is representable by integration and
	\ref{thm:distribution_on_product}\,\eqref{item:distribution_on_product:absolute}
	with $U$ and $Z$ replaced by $\rel^\adim \without B$ and $\rel^\adim$
	yields \eqref{item:char_tg:graph}. Conversely, if
	\eqref{item:char_tg:graph} holds, then \eqref{item:char_tg:slices}
	follows similarly from
	\ref{thm:distribution_on_product}\,\eqref{item:distribution_on_product:inequality}.

	Suppose now \eqref{item:char_tg:def}--\eqref{item:char_tg:graph} hold.
	Then \eqref{item:char_tg:relate_sy_boundary} is evident from
	\ref{def:trunc_g} and implies
	\begin{gather*}
		T ( \phi ) = \tint A{} \left < \phi(x,f(x)), \derivative Vf(x)
		\right > \ud \| V \| x \quad \text{for $\phi \in \mathscr{D} (
		( \rel^\adim \without B ) \times J, \rel^\adim )$}
	\end{gather*}
	by \ref{remark:associated_distribution} and \ref{thm:tv_coarea}.
	Finally, \eqref{item:char_tg:T} and \eqref{item:char_tg:bar_T} follow
	from \ref{lemma:push_on_product} and
	\ref{thm:distribution_on_product}\,\eqref{item:distribution_on_product:absolute}
	with $U$ replaced by $\rel^\adim \without B$.
\end{proof}
\begin{lemma} \label{lemma:closeness_tg}
	Suppose $\vdim$, $\adim$, $U$, $V$, $G$, and $B$ are as in
	\ref{def:trunc_g}, $f \in \trunc (V)$, $f_i$ is a sequence in
	$\trunc_G ( V )$, $J = \{ y \with 0 < y < \infty \}$, and
	\begin{gather*}
		( \| V \| + \| \delta V \| ) ( \classification{K}{x}{ f(x)>b
		} ) < \infty, \\
		f_i \to f \quad \text{as $i \to \infty$ in $( \| V \| + \|
		\delta V \| ) \restrict U \cap K$ measure}, \\
		\varrho ( K, I, b, \delta ) < \infty \quad \text{for $0 <
		\delta < \infty$}, \qquad \varrho ( K, I, b, \delta ) \to 0
		\quad \text{as $\delta \to 0+$}
	\end{gather*}
	whenever $K$ is a compact subset of $\rel^\adim \without B$, $I$ is a
	compact subset of $J$, and $\inf I > b \in J$, where $\varrho (K, I, b,
	\delta )$ denotes the supremum of all numbers
	\begin{gather*}
		\limsup_{i \to \infty} \tint{\classification{K \cap A}{x}{f_i
		(x) \in I}}{} | \derivative{V}{f_i} | \ud \| V \|
	\end{gather*}
	corresponding to $\| V \|$ measurable sets $A$ with $\| V \| (
	\classification{A \cap K}{x}{f(x) > b} ) \leq \delta$.

	Then $f \in \trunc_G ( V )$.
\end{lemma}
\begin{proof}
	Define $C = (U \times J) \cap \{ (x,y) \with f(x)>y \}$, $C_i = ( U
	\times J ) \cap \{ (x,y) \with f_i(x) > y \}$ for $i \in \nat$, $E(y)
	= \{ x \with (x,y) \in C \}$ and $E_i(y) = \{ x \with (x,y) \in C_i
	\}$ for $y \in J$ and $i \in \nat$, and $g : ( U \times \grass \adim
	\vdim ) \times J \to U \times J$ by $g((x,P),y) = (x,y)$ for $x \in
	U$, $P \in \grass \adim \vdim$, and $y \in J$. As in
	\ref{lemma:meas_fct} one derives that $C$ and $C_i$ are $\| \delta V
	\| \times \mathscr{L}^1$ measurable and that $g^{-1} \lIm C \rIm$ and
	$g^{-1} \lIm C_i \rIm$ are $V \times \mathscr{L}^1$ measurable.
	Defining $T \in \mathscr{D}' ( \rel^\adim \without B, \rel^\adim )$ as
	in \ref{thm:char_tg}, Fubini's theorem yields that (see
	\ref{miniremark:situation_general_varifold})
	\begin{align*}
		T ( \phi ) & = \tint C{} \eta (V,x) \bullet \phi (x,y) \ud (
		\| \delta V \| \times \mathscr{L}^1 ) (x,y) \\
		& \phantom = \ - \tint{g^{-1} \lIm C \rIm}{} \project P
		\bullet D \phi (\cdot,y) (x) \ud ( V \times \mathscr{L}^1 )
		((x,P),y)
	\end{align*}
	for $\phi \in \mathscr{D} (( \rel^\adim \without B ) \times J,
	\rel^\adim )$. Observe that
	\begin{gather*}
		\big (( \| V \| + \| \delta V \| ) \times \mathscr{L}^1 \big )
		\big ((( U \cap K ) \times I ) \cap C \big ) < \infty, \\
		\lim_{i \to \infty} \big (( \| V \| + \| \delta V \| ) \times
		\mathscr{L}^1 \big ) \big ((( U \cap K ) \times I ) \cap (( C
		\without C_i ) \cup ( C_i \without C ) ) \big ) = 0
	\end{gather*}
	whenever $K$ is a compact subset of $U$ and $I$ a compact subset of
	$J$; in fact, if $b,\varepsilon \in J$, $b+\varepsilon < \inf I$, and
	$A = \{ x \with |f(x)-f_i(x)| > \varepsilon \}$, then
	\begin{gather*}
		( U \times I ) \cap ( C \without C_i ) \subset (A \times I)
		\cup \{ (x,y) \with b < f(x)-\varepsilon \leq y < f(x)
		\}, \\
		( U \times I ) \cap ( C_i \without C ) \subset ( A \times I )
		\cup \{ (x,y) \with b < f(x) \leq y < f(x)+\varepsilon
		\}.
	\end{gather*}
	Consequently, one employs Fubini's theorem and \ref{def:trunc_g} to
	compute
	\begin{align*}
		T ( \phi ) & = \lim_{i \to \infty} \Big ( \tint{C_i}{} \eta
		(V,x) \bullet \phi(x,y) \ud ( \| \delta V \| \times
		\mathscr{L}^1 ) ( x,y) \\
		& \phantom{ = \lim_{i \to \infty} \big (} \ - \tint{g^{-1} \lIm
		C_i \rIm}{} \project P \bullet D \phi (\cdot,y) (x) \ud ( V
		\times \mathscr{L}^1 ) ((x,P),y) \Big ) \\
		& = \lim_{i \to \infty} \tint{J}{} \big ( (( \delta V )
		\restrict E_i(y)) ( \phi ( \cdot,y)|U) \\
		& \phantom{ = \lim_{i \to \infty} \tint{J}{} \big (} \ -
		\tint{E_i (y) \times \grass \adim \vdim}{} \project P \bullet
		D \phi (\cdot,y) (x) \ud V (x,P) \big ) \ud \mathscr{L}^1 y \\
		& = \lim_{i \to \infty} \tint{J}{} \boundary V{E_i(y)} ( \phi
		( \cdot, y ) |U) \ud \mathscr{L}^1 y.
	\end{align*}
	In view of \ref{thm:tv_coarea}, one infers
	\begin{gather*}
		\| T \| \big ( ( ( \Int K ) \without A ) \times \Int I \big )
		\leq \varrho ( K, I, b, \delta )
	\end{gather*}
	whenever $K$ is a compact subset of $\rel^\adim \without B$, $I$ is a
	compact subset of $J$, $\inf I > b \in J$, $0 < \delta < \infty$, $A$
	is a compact subset of $U$, and $\| V \| ( ( K \without A ) \cap E(b)
	) \leq \delta$. In particular, taking $A = \varnothing$ and $\delta$
	sufficiently large, one concludes that $T$ is representable by
	integration and taking $A$ such that $\| V \| ( ( K \without A ) \cap
	E (b) )$ is small yields
	\begin{gather*}
		\| T \| ( ( ( \rel^\adim \without B ) \without U ) \times J )
		= 0.
	\end{gather*}
	The conclusion now follows from
	\ref{thm:char_tg}\,\eqref{item:char_tg:def}\,\eqref{item:char_tg:graph}.
\end{proof}
\begin{remark} \label{remark:closeness_tg}
	The conditions on $\varrho$ are satisfied for instance if for any
	compact subset $K$ of $\rel^\adim \without B$ there holds
	\begin{gather*}
		\begin{aligned}
			\text{either} & \quad \tint{U \cap K}{} |
			\derivative{V}{f} | \ud \| V \| < \infty, \quad
			\lim_{i \to \infty} \eqLpnorm{\| V \| \restrict U \cap
			K}{1}{ \derivative{V}{f} - \derivative{V}{f_i} } = 0,
			\\
			\text{or} & \quad \limsup_{i \to \infty} \eqLpnorm{\|
			V \| \restrict U \cap K}{q}{\derivative{V}{f_i}} <
			\infty \quad \text{for some $1 < q \leq \infty$};
		\end{aligned}
	\end{gather*}
	in fact if $I$ is a compact subset of $J$ and $\inf I > b \in J$ then,
	in the first case,
	\begin{gather*}
		\limsup_{i \to \infty} \tint{\classification{K \cap A}{x}{f_i
		(x) \in I }}{} | \derivative{V}{f_i} | \ud \| V \| \leq
		\tint{\classification{K \cap A}{x}{f(x) > b}}{} |
		\derivative{V}{f} | \ud \| V \|
	\end{gather*}
	whenever $A$ is $\| V \|$ measurable and, in the second case,
	\begin{gather*}
		\varrho ( K, I, b, \delta ) \leq \delta^{1-1/q} \limsup_{i \to
		\infty} \eqLpnorm{\| V \| \restrict U \cap
		K}{q}{\derivative{V}{f_i}} \quad \text{for $0 < \delta <
		\infty$}.
	\end{gather*}
\end{remark}
\begin{miniremark} \label{miniremark:stepfunction}
	If $f$ is a nonnegative $\mathscr{L}^1$ measurable $\overline{\rel}$
	valued function, $O \subset \rel$, $\mathscr{L}^1 ( O ) = 0$,
	$\varepsilon > 0$, and $j \in \nat$, then there exist $b_1, \ldots,
	b_j$ such that 
	\begin{gather*}
		\varepsilon (i-1) < b_i < \varepsilon i \quad \text{and} \quad
		b_i \notin O \qquad \text{for $i=1, \ldots, j$}, \\
		\tsum{i=1}{j} ( b_i-b_{i-1} ) f(b_i) \leq 2 \tint{}{} f \ud
		\mathscr{L}^1,
	\end{gather*}
	where $b_0=0$; in fact, it is sufficient to choose $b_i$ such that
	\begin{gather*}
		\varepsilon (i-1) < b_i < \varepsilon i, \quad b_i \notin O,
		\quad \varepsilon f(b_i) \leq \tint{\varepsilon
		(i-1)}{\varepsilon i} f \ud \mathscr{L}^1
	\end{gather*}
	for $i=1, \ldots, j$, and note $b_i - b_{i-1} \leq 2 \varepsilon$.
\end{miniremark}
\begin{theorem} \label{thm:mult_tg}
	Suppose $\vdim$, $\adim$, $U$, $V$, $G$, and $B$ are as in
	\ref{def:trunc_g}, $f \in \trunc_G ( V )$, $g : U \to \{ y \with 0
	\leq y < \infty \}$, and
	\begin{gather*}
		\tint{U \cap K}{} | f | + | \derivative{V}{f} | \ud \| V \| <
		\infty, \quad \Lip ( g | K ) < \infty
	\end{gather*}
	whenever $K$ is a compact subset of $\rel^\adim \without B$.

	Then $gf \in \trunc_G ( V )$.
\end{theorem}
\begin{proof}
	Define $h= gf$ and note $h \in \trunc ( V )$ by
	\ref{thm:addition}\,\eqref{item:addition:mult}. Define a function $c$
	by $c = ( ( \rel^\adim \without B ) \times \rel ) \cap  \Clos g$ and
	note
	\begin{gather*}
		\dmn c = U \cup G, \quad c | U = g, \quad \Lip ( c | K ) <
		\infty
	\end{gather*}
	whenever $K$ is a compact subset of $\rel^\adim \without B$. Moreover,
	let
	\begin{gather*}
		J = \{ y \with 0 < y < \infty \}, \quad A =
		\eqclassification{U \times J}{(x,y)}{h(x) > y }.
	\end{gather*}
	and define $p : ( \rel^\adim \without B ) \times J \to \rel^\adim
	\without B$ by
	\begin{gather*}
		p (x,y) = x \quad \text{for $x \in \rel^\adim \without B$ and
		$y \in J$}.
	\end{gather*}
	Noting $( \| V \| + \| \delta V \| ) ( \classification{K}{x}{ h(x) > y
	} ) < \infty$ whenever $K$ is a compact subset of $\rel^\adim \without
	B$ and $y \in J$, the proof may be carried out by showing that the
	distribution $T \in \mathscr{D}' ( ( \rel^\adim \without B) \times J,
	\rel^\adim )$ defined by
	\begin{gather*}
		\begin{aligned}
			T ( \phi ) & = \tint{J}{} \big ( ( ( \delta V )
			\restrict \{ x \with  h(x)  > y \} ) ( \phi ( \cdot, y
			) | U ) \\ & \phantom{= \tint{J}{} \big (} \ -
			\tint{\{ x \with  h(x)  > y \} \times \grass \adim
			\vdim}{} \project{P} \bullet D \phi ( \cdot, y ) (x)
			\ud V (x,P) \big ) \ud \mathscr{L}^1 y
		\end{aligned}
	\end{gather*}
	for $\phi \in \mathscr{D} ( ( \rel^\adim \without B ) \times J,
	\rel^\adim )$ satisfies the conditions of
	\ref{thm:char_tg}\,\eqref{item:char_tg:graph} with $f$ replaced by
	$h$.

	For this purpose, define subsets $D(y)$, $E(\upsilon)$, and
	$F(y,\upsilon)$ of $U \cup G$, varifolds $W_y \in \RVar_\vdim (
	\rel^\adim \without B )$, and distributions $S(y) \in \mathscr{D}' (
	\rel^\adim \without B, \rel^\adim )$ and $\Sigma(y,\upsilon) \in
	\mathscr{D}' ( \rel^\adim \without B, \rel^\adim )$ by
	\begin{gather*}
		D (y) = \{ x \with f(x) > y \}, \quad E (\upsilon) = \{ x
		\with c(x)  > \upsilon \}, \\
		F (y,\upsilon) = D(y) \cap E(\upsilon), \quad W_y ( k ) =
		\tint{D(y) \times \grass{\adim}{\vdim}}{} k \ud V, \\
		S(y) ( \theta ) = ( ( \delta V ) \restrict D(y) ) ( \theta |
		U) - \tint{D(y) \times \grass \adim \vdim}{} \project{P}
		\bullet D \theta (x) \ud V (x,P), \\
		\Sigma(y,\upsilon) ( \theta ) = ( ( \delta V ) \restrict
		F(y,\upsilon) ) ( \theta | U ) - \tint{F(y,\upsilon) \times
		\grass \adim \vdim}{} \project{P} \bullet D \theta (x) \ud V
		(x,P)
	\end{gather*}
	whenever $y,\upsilon \in J$, $k \in \mathscr{K} ( ( \rel^\adim
	\without B ) \times \grass{\adim}{\vdim} )$, and $\theta \in
	\mathscr{D} ( \rel^\adim \without B, \rel^\adim )$. Let $O$ consist of
	all $b \in J$ violating the following condition:
	\begin{gather*}
		\text{$S(b)$ is representable by integration and $\| S (b) \|
		( ( \rel^\adim \without B ) \without U ) = 0$}.
	\end{gather*}
	Note $\mathscr{L}^1 ( J \without O ) = 0$ by \ref{def:trunc_g} and
	\begin{gather*}
		\text{$\| \delta W_b \|$ is a Radon measure}, \quad
		\Sigma(b,y) = S(b) \restrict E(y) + \boundary{W_b}{E(y)}
	\end{gather*}
	whenever $b \in J \without O$ and $y \in J$. One readily verifies by
	means of \ref{example:lipschitzian} in conjunction with
	\cite[2.10.19\,(4), 2.10.43]{MR41:1976} that $c \in \trunc ( W )$ and
	\begin{gather*}
		\derivative{W}{c} (x) = \derivative{V}{g} (x) \quad \text{for
		$\| V \|$ almost all $x \in D$}
	\end{gather*}
	whenever $D$ is $\| V \|$ measurable, $W \in \RVar_\vdim ( \rel^\adim
	\without B )$, $W ( k ) = \tint{D \times \grass{\adim}{\vdim}}{} k \ud
	V$ for $k \in \mathscr{K} ( ( \rel^\adim \without B ) \times
	\grass{\adim}{\vdim} )$ and $\| \delta W \|$ is a Radon measure.

	Whenever $N$ is a nonempty finite subset of $J \without O$ define
	functions $f_N : \dmn f \to \{ y \with 0 \leq y < \infty \}$ and $h_N
	: \dmn f \to \{ y \with 0 \leq y < \infty \}$ by
	\begin{gather*}
		f_N (x) = \sup ( \{ 0 \} \cup ( \classification{N}{y}{ x \in
		D(y) } )), \quad h_N (x) = f_N (x) g(x)
	\end{gather*}
	whenever $x \in \dmn f$ and distributions $\Theta_N \in \mathscr{D}' (
	( \rel^\adim \without B ) \times J, \rel^\adim )$ and $T_N \in
	\mathscr{D}' ( ( \rel^\adim \without B ) \times J, \rel^\adim )$ by
	\begin{gather*}
		\begin{aligned}
			\Theta_N ( \phi ) & = \tint{J}{} \big ( ( ( \delta V )
			\restrict \{ x \with  f_N (x)  > y \} ) ( \phi (
			\cdot, y ) | U ) \\
			& \phantom{= \tint{J}{} \big (} \ - \tint{\{ x \with
			f_N (x) > y \} \times \grass \adim \vdim}{}
			\project{P} \bullet D \phi ( \cdot , y ) ( x ) \ud V
			(x,P) \big ) \ud \mathscr{L}^1 y, \\
			T_N ( \phi ) & = \tint{J}{} \big ( ( ( \delta V )
			\restrict \{ x \with  h_N (x)  > y \} ) ( \phi (
			\cdot, y ) | U) \\
			& \phantom{= \tint{J}{} \big (} \ - \tint{\{ x \with
			h_N (x) > y \} \times \grass \adim \vdim}{}
			\project{P} \bullet D \phi ( \cdot , y ) ( x ) \ud
			V(x,P) \big ) \ud \mathscr{L}^1 y
		\end{aligned}
	\end{gather*}
	whenever $\phi \in \mathscr{D} ( ( \rel^\adim \without B ) \times J,
	\rel^\adim )$.

	Next, it will be shown \emph{if $N$ is a nonempty finite subset of
	$J \without O$ then
	\begin{gather*}
		\| T_N \| ( X \times J ) \leq ( ( p_\# \| \Theta_N \| )
		\restrict X) ( c ) + \tint{U \cap X}{} f | \derivative{V}{g} |
		\ud \| V \|
	\end{gather*}
	whenever $X$ is an open subset of $\rel^\adim \without B$}. For this
	purpose suppose $j \in \nat$ and $0 = b_0 < b_1 < \ldots < b_j <
	\infty$ satisfy $N = \{ b_i \with i = 1, \ldots, j \}$ and notice that
	\begin{gather*}
		f_N (x) = b_i \quad \text{if $x \in D(b_i) \without
		D(b_{i+1})$ for some $i=1, \ldots, j-1$}, \\
		f_N (x) = 0 \quad \text{if $x \in (\dmn f) \without D(b_1)$},
		\qquad f_N (x) = b_j \quad \text{if $x \in D(b_j)$}
	\end{gather*}
	and define distributions $\Psi_1, \Psi_2 \in \mathscr{D}' ( (
	\rel^\adim \without B ) \times J, \rel^\adim )$ by
	\begin{align*}
		\Psi_1 ( \phi ) & = \tint{J}{} \big ( \big ( S (b_1) \restrict
		E(y/b_1) \\
		& \phantom{ = \tint{J}{} \big ( \big ( } \ + \tsum{i=2}{j} (
		S(b_i) \restrict E (y/b_i) \without E (y/b_{i-1}) ) \big ) (
		\phi ( \cdot, y ) ) \big ) \ud \mathscr{L}^1 y, \\
		\Psi_2 ( \phi ) & = \tint{J}{} \big ( \tsum{i=1}{j-1}
		\boundary{( W_{b_i} - W_{b_{i+1}} )}{E(y/b_i)} +
		\boundary{W_{b_j}}{E(y/b_j) } \big ) ( \phi (\cdot, y ) ) \ud
		\mathscr{L}^1 y
	\end{align*}
	for $\phi \in \mathscr{D} ( ( \rel^\adim \without B ) \times J,
	\rel^\adim )$. One computes
	\begin{gather*}
		\Theta_N ( \phi ) = \tsum{i=1}{j} \tint{b_{i-1}}{b_i} S(b_i) (
		\phi ( \cdot, y ) ) \ud \mathscr{L}^1 y \quad \text{for $\phi
		\in \mathscr{D} ( ( \rel^\adim \without B ) \times J,
		\rel^\adim)$}
	\end{gather*}
	and deduces from
	\ref{thm:distribution_on_product}\,\eqref{item:distribution_on_product:absolute}
	with $U$ replaced by $\rel^\adim \without B$ that
	\begin{gather*}
		\| \Theta_N \| ( d ) = \tsum{i=1}{j} \tint{b_{i-1}}{b_i}
		\| S (b_i) \| ( d ( \cdot, y ) ) \ud \mathscr{L}^1 y \quad
	\end{gather*}
	whenever $d$ is an $\overline \rel$ valued $\| \Theta_N \|$ integrable
	function. Noting
	\begin{gather*}
		\{ x \with  h_N (x)  > y \} \cap ( D (b_i) \without D
		(b_{i+1}) ) = F (b_i,y/b_i) \without F(b_{i+1},y/b_i), \\
		\{ x \with  h_N (x)  > y \} \cap ( U \without D(b_1) ) =
		\varnothing, \quad \{ x \with  h_N (x)  > y \} \cap D (b_j) =
		F (b_j,y/b_j)
	\end{gather*}
	for $i=1, \ldots, j-1$ and $y \in J$, one obtains
	\begin{gather*}
		T_N ( \phi ) = \tint{J}{} \big ( \tsum{i=1}{j-1} ( \Sigma
		(b_i,y/b_i) - \Sigma (b_{i+1}, y/b_i ) ) + \Sigma ( b_j,y/b_j
		) \big)( \phi ( \cdot, y )) \ud \mathscr{L}^1 y
	\end{gather*}
	whenever $\phi \in \mathscr{D} ( ( \rel^\adim \without B ) \times J,
	\rel^\adim )$. Computing with the help of \ref{remark:boundary_of_sum}
	that
	\begin{gather*}
		\begin{aligned}
			& \tsum{i=1}{j-1} ( \Sigma(b_i,y/b_i) - \Sigma
			(b_{i+1}, y/b_i)) + \Sigma (b_j, y/b_j) \\
			& \qquad = \tsum{i=1}{j} S (b_i) \restrict E (y/b_i) +
			\tsum{i=1}{j} \boundary{W_{b_i}}{E (y/b_i)} \\
			& \qquad \phantom{=} \ - \tsum{i=1}{j-1} S (b_{i+1})
			\restrict E ( y/b_i ) - \tsum{i=1}{j-1}
			\boundary{W_{b_{i+1}}}{E(y/b_i)} \\
			& \qquad = S(b_1) \restrict E (y/b_1) + \tsum{i=2}{j}
			S (b_i) \restrict ( E ( y/b_i) \without E (y/b_{i-1}))
			\\
			& \qquad \phantom{=} \ + \tsum{i=1}{j-1}
			\boundary{(W_{b_i}-W_{b_{i+1}})}{E ( y/b_i )} +
			\boundary{W_{b_j}}{E(y/b_j)}
		\end{aligned}
	\end{gather*}
	whenever $y \in J$ yields
	\begin{gather*}
		T_N = \Psi_1 + \Psi_2.
	\end{gather*}
	Moreover, the quantity $\| \Psi_1 \| ( X \times J )$ does not exceed
	\begin{gather*}
		\begin{aligned}
			& \tint{J}{} \big ( \| S(b_1) \| \restrict E (y/b_1) +
			\tsum{i=2}{j} \| S (b_i ) \| \restrict ( E (y/b_i)
			\without E(y/b_{i-1}) ) \big ) (X) \ud \mathscr{L}^1 y
			\\
			& \qquad = \tsum{i=1}{j} \tint{J}{} ( \| S (b_i) \|
			\restrict X ) ( E (y/b_i) ) \ud \mathscr{L}^1 y \\
			& \qquad \phantom{=} \ - \tsum{i=2}{j} \tint{J}{} ( \|
			S (b_i) \| \restrict X ) ( E (y/b_{i-1}) ) \ud
			\mathscr{L}^1 y \\
			& \qquad = \tsum{i=1}{j} (b_i-b_{i-1}) ( \| S (b_i) \|
			\restrict X ) ( c ) = ( ( p_\# \| \Theta_N \| )
			\restrict X ) ( c )
		\end{aligned}
	\end{gather*}
	and, using \ref{remark:associated_distribution} and
	\ref{thm:tv_coarea}, the quantity $\| \Psi_2 \| ( X \times J )$ may be
	bounded by
	\begin{gather*}
		\begin{aligned}
			& \tint{J}{} \big ( \tsum{i=1}{j-1} \|
			\boundary{(W_{b_i}-W_{b_{i+1}})}{E (y/b_i)} \| + \|
			\boundary{W_{b_j}}{E(y/b_j)} \| \big ) (X) \ud
			\mathscr{L}^1 y \\
			& \qquad = \tsum{i=1}{j-1} b_i \tint{X}{} |
			\derivative{(W_{b_i}-W_{b_{i+1}})}{c} | \ud \|
			W_{b_i}-W_{b_{i+1}} \| + b_j \tint{X}{} |
			\derivative{W_{b_j}}{c} | \ud \| W_{b_j} \| \\
			& \qquad = \tsum{i=1}{j-1} b_i \tint{X \cap ( D(b_i)
			\without D(b_{i+1}))}{} | \derivative{V}{g} | \ud \| V
			\| + b_j \tint{X \cap D(b_j)}{} | \derivative{V}{g} |
			\ud \| V \| \\
			& \qquad = \tint{U \cap X}{} f | \derivative{V}{g} |
			\ud \| V \|.
		\end{aligned}
	\end{gather*}

	Next, \emph{it will be proven
	\begin{gather*}
		\| T \| ( X \times J ) \leq \tint{U \cap X}{} 2 g |
		\derivative{V}{f} | + f | \derivative{V}{g} | \ud \| V \|
	\end{gather*}
	whenever $X$ is an open subset of $\rel^\adim \without B$}. Recalling
	the formula for $\| \Theta_N \|$, one may use
	\ref{miniremark:stepfunction} with $f(y)$ replaced by $( \| S (y) \|
	\restrict X ) ( c )$ for $y \in J$ to construct a sequence $N(i)$ of
	nonempty finite subsets of $J \without O$ such that
	\begin{gather*}
		( ( p_\# \| \Theta_{N(i)} \| ) \restrict X ) ( c ) \leq 2
		\tint{J}{} ( \| S (y) \| \restrict X ) (  c  ) \ud
		\mathscr{L}^1 y, \\
		\dist ( y, N (i) ) \to 0 \quad \text{as $i \to \infty$ for $y
		\in J$}.
	\end{gather*}
	Define $A(i) = \eqclassification{U \times J}{(x,y)}{ h_{N(i)} (x)  > y
	}$ for $i \in \nat$. Noting
	\begin{gather*}
		h_{N(i)} (x)  \to  h(x)  \quad \text{as $i \to \infty$ for $x
		\in \dmn f$}
	\end{gather*}
	and recalling $( \| V \| + \| \delta V \| ) ( \classification{K}{x}{
	h(x) > y } ) < \infty$ for $y \in J$ whenever $K$ is a compact subset
	of $\rel^\adim \without B$, one infers
	\begin{gather*}
		\big ( ( \| V \| + \| \delta V \| ) \times \mathscr{L}^1 \big)
		( C \cap A \without A (i) ) \to 0
	\end{gather*}
	as $i \to \infty$ whenever $C$ is a compact subset of $( \rel^\adim
	\without B ) \times J$. Since $A(i) \subset A$ for $i \in \nat$, it
	follows by means of Fubini's theorem that
	\begin{gather*}
		T_{N(i)} \to T \quad \text{as $i \to \infty$}
	\end{gather*}
	and, in conjunction with the assertion of the preceding paragraph,
	\begin{gather*}
		\| T \| ( X \times J ) \leq 2 \tint{J}{} ( \| S (y) \|
		\restrict X ) (  c  ) \ud \mathscr{L}^1 y + \tint{U \cap X}{}
		f | \derivative{V}{g} | \ud \| V \|.
	\end{gather*}
	Therefore the assertion of the present paragraph is implied by
	\ref{thm:char_tg}\,\eqref{item:char_tg:def}\,\eqref{item:char_tg:relate_sy_boundary}\,\eqref{item:char_tg:bar_T}.
	
	Finally, the assertion of the preceding paragraph extends to all
	Borel subsets $X$ of $\rel^\adim \without B$ by approximation and the
	conclusion follows.
\end{proof}
\section{Embeddings into Lebesgue spaces} \label{sec:embeddings}
In this section a variety of Sobolev Poincar{\'e} type inequalities for weakly
differentiable functions are established by means of the relative
isoperimetric inequalities \ref{corollary:rel_iso_ineq} and
\ref{corollary:true_rel_iso_ineq}. The key are local estimates under a
smallness condition on set of points where the nonnegative function is
positive, see
\ref{thm:sob_poin_summary}\,\eqref{item:sob_poin_summary:interior}. These
estimates are formulated in such a way as to improve in case the function
satisfies a zero boundary value condition on an open part of the boundary.
Consequently, Sobolev inequalities are essentially a special case, see
\ref{thm:sob_poin_summary}\,\eqref{item:sob_poin_summary:global}. Local
summability results also follow, see \ref{corollary:integrability}. Finally,
versions without the previously hypothesised smallness condition are derived
in \ref{thm:sob_poincare_q_medians} and \ref{thm:sob_poin_several_med}.

The differentiability results which will be derived in \ref{thm:approx_diff}
and \ref{thm:diff_lebesgue_spaces} are based on
\ref{thm:sob_poin_summary}\,\eqref{item:sob_poin_summary:interior} whereas the
oscillation estimates which will be proven in \ref{thm:mod_continuity} and
\ref{thm:hoelder_continuity} employ \ref{thm:sob_poin_several_med}.

\begin{theorem} \label{thm:sob_poin_summary}
	Suppose $1 \leq M < \infty$.

	Then there exists a positive, finite number $\Gamma$ with the
	following property.

	If $\vdim$, $\adim$, $p$, $U$, $V$, and $\psi$ are as in
	\ref{miniremark:situation_general}, $\adim \leq M$, $G$ is a
	relatively open subset of $\Bdry U$, $B = ( \Bdry U ) \without G$, and
	$f \in \trunc_G ( V )$, then the following two statements
	hold:
	\begin{enumerate}
		\item \label{item:sob_poin_summary:interior} Suppose $1 \leq Q
		\leq M$, $0 < r < \infty$, $E = \{ x \with f(x) > 0 \}$,
		\begin{gather*}
			\| V \| ( E ) \leq ( Q-M^{-1} ) \unitmeasure{\vdim}
			r^\vdim, \\
			\| V \| ( E \cap \{ x \with \density^\vdim ( \| V \|,
			x ) < Q \} ) \leq \Gamma^{-1} r^\vdim,
		\end{gather*}
		and $A = U \cap \{ x \with \oball{x}{r} \subset \rel^\adim
		\without B \}$. Then the following four implications hold:
		\begin{enumerate}
			\item \label{item:sob_poin_summary:interior:p=1} If $p
			= 1$, $\beta = \infty$ if $\vdim = 1$ and $\beta =
			\vdim/(\vdim-1)$ if $\vdim > 1$, then
			\begin{gather*}
				\eqLpnorm{\| V \| \restrict A}{\beta}{f} \leq
				\Gamma \big ( \Lpnorm{\| V \|}{1}{
				\derivative{V}{f} } + \| \delta V \|(f) \big
				).
			\end{gather*}
			\item \label{item:sob_poin_summary:interior:p=m=1} If
			$p = \vdim = 1$ and $\psi ( E ) \leq \Gamma^{-1}$,
			then
			\begin{gather*}
				\eqLpnorm{\| V \| \restrict A}{\infty}{f} \leq
				\Gamma \, \Lpnorm{\| V \|}{1}{
				\derivative{V}{f} }.
			\end{gather*}
			\item \label{item:sob_poin_summary:interior:q<m=p} If
			$1 \leq q < \vdim = p$ and $\psi ( E ) \leq
			\Gamma^{-1}$, then
			\begin{gather*}
				\eqLpnorm{\| V \| \restrict A}{\vdim
				q/(\vdim-q)}{f} \leq \Gamma (\vdim-q)^{-1}
				\Lpnorm{\| V \|}{q}{ \derivative{V}{f} }.
			\end{gather*}
			\item \label{item:sob_poin_summary:interior:p=m<q} If
			$1 < p = \vdim < q \leq \infty$ and $\psi ( E ) \leq
			\Gamma^{-1}$, then
			\begin{gather*}
				\eqLpnorm{\| V \| \restrict A}{\infty}{f} \leq
				\Gamma^{1/(1/\vdim-1/q)} \| V \| ( E
				)^{1/\vdim-1/q} \Lpnorm{\| V \|}{q}{
				\derivative{V}{f} }.
			\end{gather*}
		\end{enumerate}
		\item \label{item:sob_poin_summary:global} Suppose $G = \Bdry
		U$, $E = \{ x \with f(x) > 0 \}$, and $\| V \| ( E ) <
		\infty$. Then the following four implications hold:
		\begin{enumerate}
			\item \label{item:sob_poin_summary:global:p=1} If $p
			= 1$, $\beta = \infty$ if $\vdim = 1$ and $\beta =
			\vdim/(\vdim-1)$ if $\vdim > 1$, then
			\begin{gather*}
				\Lpnorm{\| V \|}{\beta}{f} \leq \Gamma \big
				( \Lpnorm{\| V \|}{1}{ \derivative{V}{f} } +
				\| \delta V \|(f) \big ).
			\end{gather*}
			\item \label{item:sob_poin_summary:global:p=m=1} If
			$p = \vdim = 1$ and $\psi ( E ) \leq \Gamma^{-1}$,
			then
			\begin{gather*}
				\Lpnorm{\| V \|}{\infty}{f} \leq \Gamma \,
				\Lpnorm{\| V \|}{1}{ \derivative{V}{f} }.
			\end{gather*}
			\item \label{item:sob_poin_summary:global:q<m=p} If
			$1 \leq q < \vdim = p$ and $\psi ( E ) \leq
			\Gamma^{-1}$, then
			\begin{gather*}
				\Lpnorm{\| V \|}{\vdim q/(\vdim-q)}{f} \leq
				\Gamma (\vdim-q)^{-1} \Lpnorm{\| V \|}{q}{
				\derivative{V}{f} }.
			\end{gather*}
			\item \label{item:sob_poin_summary:global:p=m<q} If
			$1 < p = \vdim < q \leq \infty$ and $\psi ( E ) \leq
			\Gamma^{-1}$, then
			\begin{gather*}
				\Lpnorm{\| V \|}{\infty}{f} \leq
				\Gamma^{1/(1/\vdim-1/q)} \| V \| (
				E)^{1/\vdim-1/q} \Lpnorm{\| V \|}{q}{
				\derivative{V}{f} }.
			\end{gather*}
		\end{enumerate}
	\end{enumerate}
\end{theorem}
\begin{proofinsteps}
	Denote by \eqref{item:sob_poin_summary:interior:q<m=p}$'$
	[respectively \eqref{item:sob_poin_summary:global:q<m=p}$'$] the
	implication resulting from
	\eqref{item:sob_poin_summary:interior:q<m=p} [respectively
	\eqref{item:sob_poin_summary:global:q<m=p}] through omission of the
	factor $(\vdim-q)^{-1}$ and addition of the requirement $q=1$. It is
	sufficient to construct functions
	$\Gamma_{\eqref{item:sob_poin_summary:interior:p=1}}$,
	$\Gamma_{\eqref{item:sob_poin_summary:interior:p=m=1}}$,
	$\Gamma_{\eqref{item:sob_poin_summary:interior:q<m=p}}$,
	$\Gamma_{\eqref{item:sob_poin_summary:interior:q<m=p}'}$,
	$\Gamma_{\eqref{item:sob_poin_summary:interior:p=m<q}}$,
	$\Gamma_{\eqref{item:sob_poin_summary:global:p=1}}$,
	$\Gamma_{\eqref{item:sob_poin_summary:global:p=m=1}}$,
	$\Gamma_{\eqref{item:sob_poin_summary:global:q<m=p}}$,
	$\Gamma_{\eqref{item:sob_poin_summary:global:q<m=p}'}$, and
	$\Gamma_{\eqref{item:sob_poin_summary:global:p=m<q}}$ corresponding to
	the implications \eqref{item:sob_poin_summary:interior:p=1},
	\eqref{item:sob_poin_summary:interior:p=m=1},
	\eqref{item:sob_poin_summary:interior:q<m=p},
	\eqref{item:sob_poin_summary:interior:q<m=p}$'$,
	\eqref{item:sob_poin_summary:interior:p=m<q},
	\eqref{item:sob_poin_summary:global:p=1},
	\eqref{item:sob_poin_summary:global:p=m=1},
	\eqref{item:sob_poin_summary:global:q<m=p},
	\eqref{item:sob_poin_summary:global:q<m=p}$'$, and
	\eqref{item:sob_poin_summary:global:p=m<q} whose value at $M$ for $1
	\leq M < \infty$ is a positive, finite number such that the respective
	implication is true for $M$ with $\Gamma$ replaced by this value.
	Define
	\begin{gather*}
		\Gamma_{\eqref{item:sob_poin_summary:interior:p=1}} (M) =
		\Gamma_{\eqref{item:sob_poin_summary:interior:p=m=1}} (M) =
		\Gamma_{\eqref{item:sob_poin_summary:interior:q<m=p}'} (M) =
		\Gamma_{\ref{thm:rel_iso_ineq}} ( M ), \quad
		\Gamma_{\eqref{item:sob_poin_summary:global:p=1}} ( M )=
		\Gamma_{\eqref{item:sob_poin_summary:interior:p=1}} ( \sup \{
		2, M \} ), \\
		\Gamma_{\eqref{item:sob_poin_summary:global:p=m=1}} ( M ) =
		\Gamma_{\eqref{item:sob_poin_summary:interior:p=m=1}} ( \sup
		\{ 2, M \} ), \quad
		\Gamma_{\eqref{item:sob_poin_summary:global:q<m=p}'} ( M ) =
		\Gamma_{\eqref{item:sob_poin_summary:interior:q<m=p}'} ( \sup
		\{ 2, M \} ), \\
		\Gamma_{\eqref{item:sob_poin_summary:global:q<m=p}} (M) = M^2
		\Gamma_{\eqref{item:sob_poin_summary:global:q<m=p}'} ( M ), \\
		\Delta_1 (M) = ( \sup \{ 1/2, 1- 1/( 2M^2-1) \})^{1/M},
		\quad \Delta_2 (M) = 2 M / (1-\Delta_1(M)), \\
		\Delta_3 (M) = M \sup \{ \unitmeasure{\vdim}
		\with M \geq \vdim \in \nat \}, \\
		\Delta_4 (M) = \inf \{ \unitmeasure{\vdim} \with M \geq \vdim
		\in \nat \} / 2, \\
		\Delta_5 (M) = \sup \{ 1,
		\Gamma_{\eqref{item:sob_poin_summary:global:q<m=p}} (M)(
		\Delta_2(M) \Delta_3(M) + 1) \}, \\
		\Delta_6 (M) = 4^{M+1} \Delta_3(M)^M \Delta_5(M)^{M+1}, \\
		\Gamma_{\eqref{item:sob_poin_summary:interior:q<m=p}} (M) =
		\sup \{ 2
		\Gamma_{\eqref{item:sob_poin_summary:interior:q<m=p}'} ( 2M),
		\Delta_6 (M) ( 1 + \Delta_3(M)
		\Gamma_{\eqref{item:sob_poin_summary:interior:q<m=p}'} (2M) )
		\}, \\
		\Delta_7 (M) = \sup \{ 1,
		\Gamma_{\eqref{item:sob_poin_summary:interior:q<m=p}'} (2M)
		\}, \\
		\Delta_8 (M) = \inf \big \{ \inf \{ 1, \Delta_3(M)^{-1}
		\Delta_4(M) \Delta_1(M)^M \} ( 1 - \Delta_1(M) )^M, 1/2 \big
		\}, \\
		\Delta_9 ( M ) = 2 \Delta_7(M) \Delta_8(M)^{-1}, \\
		\Gamma_{\eqref{item:sob_poin_summary:interior:p=m<q}} (M) =
		\sup \big \{ \Delta_9 (M), \Delta_1 (M)^{-M}
		\Gamma_{\eqref{item:sob_poin_summary:interior:q<m=p}'} ( 2M )
		\big \}, \\
		\Gamma_{\eqref{item:sob_poin_summary:global:p=m<q}} (M) =
		\Gamma_{\eqref{item:sob_poin_summary:interior:p=m<q}} ( \sup
		\{ 2, M \} )
	\end{gather*}
	whenever $1 \leq M < \infty$.

	In order to verify the asserted properties suppose $1 \leq M < \infty$
	and abbreviate $\delta_i = \Delta_i(M)$ whenever $i = 1, \ldots, 9$.
	In the verification of assertion ($X$) [respectively ($X$)$'$], where
	$X$ is one of \ref{item:sob_poin_summary:interior:p=1},
	\ref{item:sob_poin_summary:interior:p=m=1},
	\ref{item:sob_poin_summary:interior:q<m=p},
	\ref{item:sob_poin_summary:interior:p=m<q},
	\ref{item:sob_poin_summary:global:p=1},
	\ref{item:sob_poin_summary:global:p=m=1},
	\ref{item:sob_poin_summary:global:q<m=p}, and
	\ref{item:sob_poin_summary:global:p=m<q}, [respectively
	\ref{item:sob_poin_summary:interior:q<m=p} and
	\ref{item:sob_poin_summary:global:q<m=p}] it will be assumed that the
	quantities occurring in its hypotheses are defined and satisfy these
	hypotheses with $\Gamma$ replaced by $\Gamma_{(X)} (M)$ [respectively
	$\Gamma_{(X)'} (M)$].

	\begin{step} \label{step:sob_poin_summary:1}
		Verification of the property of
		$\Gamma_{\eqref{item:sob_poin_summary:interior:p=1}}$.
	\end{step}

	Define $E(b) = \{ x \with f(x) > b \}$ for $0 \leq b < \infty$. If
	$\vdim = 1$ then $\Gamma_{\ref{thm:rel_iso_ineq}}(M)^{-1} \leq \|
	\boundary{V}{E(b)} \| (U) + \| \delta V \| ( E(b))$ for
	$\mathscr{L}^1$ almost all $b$ with $0 < b < \eqLpnorm{\| V \|
	\restrict A}{\beta}{f}$ by \ref{corollary:rel_iso_ineq} and the
	conclusion follows from \ref{thm:tv_coarea}. If $\vdim > 1$, define
	\begin{gather*}
		f_b = \inf \{ f, b \}, \quad g (b) = \eqLpnorm{\| V \|
		\restrict A}{\beta}{f_b} \leq b \, \| V \| ( E )^{1/\beta} <
		\infty
	\end{gather*}
	for $0 \leq b < \infty$ and use Minkowski's inequality to conclude
	\begin{gather*}
		0 \leq g (b+y) - g (b) \leq \eqLpnorm{\| V \| \restrict
		A}{\beta}{f_{b+y}-f_b} \leq y \, \| V \| ( A \cap E(b)
		)^{1/\beta}
	\end{gather*}
	for $0 \leq b < \infty$ and $0 < y < \infty$. Therefore $g$ is
	Lipschitzian and one infers from \cite[2.9.19]{MR41:1976} and
	\ref{corollary:rel_iso_ineq} that
	\begin{gather*}
		0 \leq g'(b) \leq \| V \| ( A \cap E(b) )^{1-1/\vdim} \leq
		\Gamma_{\ref{thm:rel_iso_ineq}}(M) \big ( \|
		\boundary{V}{E(b)} \| ( U ) + \| \delta V \| (E(b)) \big)
	\end{gather*}
	for $\mathscr{L}^1$ almost all $0 < b < \infty$, hence $\eqLpnorm{\| V
	\| \restrict A}{\beta}{f} = \lim_{b \to \infty} g (b) =
	\tint{0}{\infty} g' \ud \mathscr{L}^1$ by \cite[2.9.20]{MR41:1976}
	and \ref{thm:tv_coarea} implies the conclusion.

	\begin{step} \label{step:sob_poin_summary:2}
		Verification of the properties of
		$\Gamma_{\eqref{item:sob_poin_summary:interior:p=m=1}}$ and
		$\Gamma_{\eqref{item:sob_poin_summary:interior:q<m=p}'}$.
	\end{step}

	Omitting the terms involving $\delta V$ from the proof of Step
	\ref{step:sob_poin_summary:1} and using
	\ref{corollary:true_rel_iso_ineq} instead of
	\ref{corollary:rel_iso_ineq}, a proof of Step
	\ref{step:sob_poin_summary:2} results.
	
	\begin{step}
		Verification of the properties of
		$\Gamma_{\eqref{item:sob_poin_summary:global:p=1}}$,
		$\Gamma_{\eqref{item:sob_poin_summary:global:p=m=1}}$, and
		$\Gamma_{\eqref{item:sob_poin_summary:global:q<m=p}'}$.
	\end{step}

	Taking $Q = 1$, one may apply
	\eqref{item:sob_poin_summary:interior:p=1},
	\eqref{item:sob_poin_summary:interior:p=m=1}, and
	\eqref{item:sob_poin_summary:interior:q<m=p}$'$ with $M$ replaced by
	$\sup \{ 2, M \}$ and a sufficiently large number $r$.

	\begin{step}
		Verification of the property of
		$\Gamma_{\eqref{item:sob_poin_summary:global:q<m=p}}$.
	\end{step}

	One may assume $f$ to be bounded by \ref{lemma:basic_v_weakly_diff},
	\ref{example:composite}\,\eqref{item:composite:1d}. Noting
	$f^{q(\vdim-1)/(\vdim-q)} \in \trunc_{\Bdry U} ( V )$ by
	\ref{lemma:trunc_tg}, one now applies
	\eqref{item:sob_poin_summary:global:q<m=p}$'$ with $f$ replaced by
	$f^{q(\vdim-1)/(\vdim-q)}$ to deduce the assertion by the method of
	\cite[4.5.15]{MR41:1976}.

	\begin{step}
		Verification of the property of
		$\Gamma_{\eqref{item:sob_poin_summary:interior:q<m=p}}$.
	\end{step}

	One may assume $\Lpnorm{\| V \|}{q}{ \derivative{V}{f} } < \infty$
	and, possibly using homotheties, also $r=1$. Moreover, one may assume
	$f$ to be bounded by \ref{lemma:basic_v_weakly_diff},
	\ref{example:composite}\,\eqref{item:composite:1d}, hence
	\begin{gather*}
		\tint{}{} |f| + | \derivative{V}{f} | \ud \| V \| < \infty
	\end{gather*}
	by \ref{thm:addition}\,\eqref{item:addition:zero} and H\"older's
	inequality. Define
	\begin{gather*}
		r_i = \delta_1 + (i-1)(1-\delta_1)/M, \quad
		A_i = \classification{U}{x}{ \oball{x}{r_i} \subset \rel^\adim
		\without B } \quad
	\end{gather*}
	whenever $i=1, \ldots, \vdim$. Observe that
	\begin{gather*}
		Q-M^{-1} \leq r_1^\vdim \big (Q-(2M)^{-1}\big ), \quad
		r_1^\vdim \geq 1/2, \quad r_\vdim \leq 1, \\
		r_{i+1} - r_i = 2 / \delta_2 \quad \text{for $i=1, \ldots,
		\vdim-1$}.
	\end{gather*}
	Choose $g_i \in \mathscr{E} ( U, \rel)$ with
	\begin{gather*}
		\spt g_i \subset A_i, \qquad g_i(x) = 1 \quad \text{for $x \in
		A_{i+1}$}, \\
		0 \leq g_i (x) \leq 1 \quad \text{for $x \in U$}, \qquad \Lip
		g_i \leq \delta_2
	\end{gather*}
	whenever $i=1, \ldots, \vdim-1$. Notice that $g_i f \in \trunc (V)$ by
	\ref{thm:addition}\,\eqref{item:addition:mult} with
	\begin{gather*}
		\derivative{V}{(g_if)} (x) = \derivative{V}{g_i} (x) f(x) +
		g_i (x) \derivative{V}{f} (x) \quad \text{for $\| V \|$ almost
		all $x$}.
	\end{gather*}
	Moreover, one infers $g_i f \in \trunc_{\Bdry U} ( V )$ by
	\ref{thm:mult_tg} and \ref{lemma:boundary}.

	If $i \in \{ 1, \ldots, \vdim-1 \}$ and $1-(i-1)/\vdim \geq 1/q \geq
	1-i/\vdim$ then by \eqref{item:sob_poin_summary:global:q<m=p} and
	H\"older's inequality
	\begin{gather*}
		\begin{aligned}
			& \eqLpnorm{\| V \| \restrict A_{i+1}}{\vdim
			q/(\vdim-q)}{f} \leq \Lpnorm{\| V \|}{\vdim
			q/(\vdim-q)} { g_i f } \\
			& \qquad \leq
			\Gamma_{\eqref{item:sob_poin_summary:global:q<m=p}}
			(M) (\vdim-q)^{-1} \big ( \delta_2 \eqLpnorm{\| V \|
			\restrict A_i}{q}{f} + \Lpnorm{\| V \|}{q}{
			\derivative{V}{f} } \big ) \\
			& \qquad \leq \delta_5 ( \vdim-q )^{-1} \big (
			\eqLpnorm{\| V \| \restrict A_i }{\vdim/(\vdim-i)} {f}
			+ \Lpnorm{\| V \|}{q}{ \derivative{V}{f} } \big ).
		\end{aligned}
	\end{gather*}
	In particular, replacing $q$ by $\vdim/(\vdim-i)$ and using H\"older's
	inequality in conjunction with
	\ref{thm:addition}\,\eqref{item:addition:zero}, one infers
	\begin{gather*}
		\eqLpnorm{ \| V \| \restrict A_{i+1} }{\vdim/(\vdim-i-1)} {f}
		\leq 2 \delta_3 \delta_5 \big ( \eqLpnorm{\| V \| \restrict
		A_i }{\vdim/(\vdim-i)} {f} + \Lpnorm{\| V \|}{q}
		{\derivative{V}{f}} \big )
	\end{gather*}
	whenever $i \in \{ 1, \ldots, \vdim-2 \}$ and $1-i/\vdim \geq 1/q$,
	choosing $j \in \{ 1, \ldots, \vdim-1 \}$ with $1-(j-1)/\vdim \geq 1/q
	> 1-j/\vdim$ and iterating this $j-1$ times, also
	\begin{gather*}
		\begin{aligned}
			& \eqLpnorm{ \| V \| \restrict A_j }{\vdim/(\vdim-j)}
			{f} \\
			& \qquad \leq ( 4 \delta_3 \delta_5)^{j-1} \big (
			\eqLpnorm{\| V \| \restrict A_1 }{\vdim/(\vdim-1)} {f}
			+ \Lpnorm{\| V \|}{q} {\derivative{V}{f}} \big ).
		\end{aligned}
	\end{gather*}
	Together this yields
	\begin{gather*}
		\begin{aligned}
			& \eqLpnorm{\| V \| \restrict A_{j+1} }{\vdim
			q/(\vdim-q)} {f} \\
			& \qquad \leq \delta_6 (\vdim-q)^{-1} \big (
			\eqLpnorm{\| V \| \restrict A_1}{\vdim/(\vdim-1)} {f}
			+ \Lpnorm{\| V \|}{q}{ \derivative{V}{f} } \big ).
		\end{aligned}
	\end{gather*}
	and the conclusion then follows from
	\eqref{item:sob_poin_summary:interior:q<m=p}$'$ with $M$ and $r$
	replaced by $2M$ and $r_1$.

	\begin{step}
		Verification of the property of
		$\Gamma_{\eqref{item:sob_poin_summary:interior:p=m<q}}$.
	\end{step}

	Assume
	\begin{gather*}
		r=1, \quad \| V \| ( E ) > 0, \quad \Lpnorm{\| V
		\|}{q}{\derivative{V}{f}} < \infty,
	\end{gather*}
	abbreviate $\beta = \vdim / ( \vdim-1 )$ and $\alpha = 1/(1/\vdim-1/q)$,
	and suppose
	\begin{gather*}
		\lambda = \| V \| ( E ), \quad \Lpnorm{\| V \|}{q}{
		\derivative{V}{f} } < \kappa < \infty.
	\end{gather*}
	Notice that $0 < \delta_1 < 1$, $0 < \delta_8 < 1$, $\alpha \geq 1$ and
	$\lambda > 0$, and define
	\begin{gather*}
		r_i = 1 - (1-\delta_1)^i, \quad s_i = \delta_9^\alpha
		\lambda^{1/\vdim-1/q} \kappa ( 1 - 2^{-i} ), \\
		X_i = \classification{\rel^\adim}{x}{\dist (x,B) > r_i },
		\quad U_i = U \cap X_i, \\
		H_i = X_i \cap \Bdry U, \quad C_i = ( \Bdry U_i ) \without
		H_i, \quad W_i = V | \mathbf{2}^{U_i \times
		\grass{\adim}{\vdim} }, \\
		t_i = \| V \| ( \classification{U_i}{x}{f(x)>s_i} ), \\
		f_i (x) = \sup \{ f(x) - s_i, 0 \} \quad \text{whenever $x
		\in \dmn f$}
	\end{gather*}
	whenever $i$ is a nonnegative integer. Notice that
	\begin{gather*}
		r_{i+1} - r_i = \delta_1 (1-\delta_1)^i, \quad U_{i+1} \subset
		U_i, \quad s_{i+1} - s_i = \delta_9^\alpha \lambda^{1/\vdim-1/q}
		\kappa 2^{-i-1} > 0
	\end{gather*}
	whenever $i$ is a nonnegative integer. The conclusion is readily
	deduced from the assertion,
	\begin{gather*}
		t_i \leq \lambda \delta_8^{\alpha i} \quad \text{\emph{whenever $i$
		is a nonnegative integer}},
	\end{gather*}
	which will be proven by induction.

	The case $i=0$ is trivial. To prove the case $i=1$, one applies
	\eqref{item:sob_poin_summary:interior:q<m=p}$'$ with $M$ and $r$
	replaced by $2M$ and $\delta_1$ and H\"older's inequality in
	conjunction with \ref{thm:addition}\,\eqref{item:addition:zero} to
	obtain
	\begin{gather*}
		\eqLpnorm{\| V \| \restrict U_1}{\beta}{f} \leq
		\Gamma_{\eqref{item:sob_poin_summary:interior:q<m=p}'} (2M)
		\Lpnorm{\| V \|}{1}{ \derivative{V}{f} } \leq \delta_7
		\lambda^{1-1/q} \kappa,
	\end{gather*}
	hence
	\begin{gather*}
		t_1^{1-1/\vdim} \leq (s_1-s_0)^{-1} \eqLpnorm{\| V \|
		\restrict U_1}{\beta}{f} \leq \delta_9^{-\alpha} \lambda^{1-1/\vdim} 2
		\delta_7 \leq \lambda^{1-1/\vdim} \delta_8^{\alpha ( 1-1/\vdim )}.
	\end{gather*}
	Assuming the assertion to be true for some $i \in \nat$, notice that
	\begin{gather*}
		\lambda \leq \delta_3, \quad \alpha i \geq 1, \quad t_i \leq \lambda
		\delta_8^{\alpha i} \leq \delta_4 \delta_1^\vdim
		(1-\delta_1)^{i \vdim} \leq (1/2) \unitmeasure{\vdim}
		(r_{i+1}-r_i)^\vdim, \\
		f_i \in \trunc_G ( V ), \quad f_i | X_i \in
		\trunc_{H_i} ( W_i), \\
		H_i = X_i \cap \Bdry U_i, \quad U_{i+1} \subset
		\classification{U_i}{x}{ \oball{x}{r_{i+1}-r_i} \subset
		\rel^\adim \without C_i },
	\end{gather*}
	by \ref{lemma:basic_v_weakly_diff},
	\ref{example:composite}\,\eqref{item:composite:1d} and
	\ref{lemma:restriction_tg}, \ref{remark:restriction_tg} with $f$, $X$,
	$H$, $W$, $C$, and $r$ replaced by $f_i$, $X_i$, $H_i$, $W_i$, $C_i$,
	and $r_{i+1}-r_i$. Therefore
	\eqref{item:sob_poin_summary:interior:q<m=p}$'$ applied with $U$, $V$,
	$Q$, $M$, $G$, $B$, $f$, and $r$ replaced by $U_i$, $X_i$, $1$, $2M$,
	$H_i$, $C_i$, $f_i | X_i$, and $r_{i+1}-r_i$ yields
	\begin{gather*}
		\eqLpnorm{\| V \| \restrict U_{i+1} }{\beta}{f_i} \leq
		\Gamma_{\eqref{item:sob_poin_summary:interior:q<m=p}'} (2M)
		\eqLpnorm{\| V \| \restrict U_i }{1}{ \derivative{V}{f_i} },
	\end{gather*}
	hence, using H\"older's inequality in conjunction with
	\ref{lemma:basic_v_weakly_diff},
	\ref{example:composite}\,\eqref{item:composite:1d}
	and
	\ref{thm:addition}\,\eqref{item:addition:zero},
	\begin{gather*}
		\eqLpnorm{\| V \| \restrict U_{i+1}}{\beta}{f_i} \leq \delta_7
		t_i^{1-1/q} \kappa \leq \delta_7 \lambda^{1-1/q}
		\delta_8^{\alpha i (1-1/q)} \kappa.
	\end{gather*}
	Noting $\delta_9^{-\alpha} 2 \delta_7 \leq \delta_8^{\alpha
	(1-1/\vdim)}$ and $2 \delta_8^{\alpha (1-1/q)} \leq \delta_8^{\alpha
	(1-1/\vdim)}$, it follows
	\begin{gather*}
		\begin{aligned}
			t_{i+1}^{1-1/\vdim} & \leq ( s_{i+1} - s_i )^{-1}
			\eqLpnorm{\| V \| \restrict U_{i+1}}{\beta}{f_i} \\
			& \leq \lambda^{1-1/\vdim} \delta_9^{-\alpha} 2 \delta_7
			\big ( 2 \delta_8^{\alpha(1-1/q)} \big )^i \leq
			\lambda^{1-1/\vdim} \delta_8^{\alpha (1-1/\vdim) (i+1)}
		\end{aligned}
	\end{gather*}
	and the assertion is proven.

	\begin{step}
		Verification of the property of
		$\Gamma_{\eqref{item:sob_poin_summary:global:p=m<q}}$.
	\end{step}

	Taking $Q = 1$, one may apply
	\eqref{item:sob_poin_summary:interior:p=m<q} with $M$ replaced by
	$\sup \{ 2, M \}$ and a sufficiently large number $r$.
\end{proofinsteps}
\begin{remark}
	The role of \ref{corollary:rel_iso_ineq} and
	\ref{corollary:true_rel_iso_ineq} respectively in the Steps
	\ref{step:sob_poin_summary:1} and \ref{step:sob_poin_summary:2} of the
	preceding proof is identical to the one of Allard
	\cite[7.1]{MR0307015} in Allard \cite[7.3]{MR0307015}.

	The method of deduction of \eqref{item:sob_poin_summary:global:q<m=p}
	from \eqref{item:sob_poin_summary:global:q<m=p}$'$ is classical, see
	\cite[4.5.15]{MR41:1976}.

	In the context of Lipschitz functions the method of deduction of
	\eqref{item:sob_poin_summary:interior:q<m=p} from
	\eqref{item:sob_poin_summary:interior:q<m=p}$'$ and
	\eqref{item:sob_poin_summary:global:q<m=p} is outlined by Hutchinson
	in \cite[pp.~59--60]{MR1066398}.

	The iteration procedure employed in the proof
	\eqref{item:sob_poin_summary:interior:p=m<q} bears formal resemblance
	with Stampacchia \cite[Lemma 4.1\,(i)]{MR0251373}. However, here the
	estimate of $t_{i+1}$ in terms of $t_i$ requires a smallness
	hypothesis on $t_i$.
\end{remark}
\begin{corollary} \label{corollary:integrability}
	Suppose $\vdim$, $\adim$, $U$, $V$, and $p$ are as in
	\ref{miniremark:situation_general}, $1 \leq q \leq \infty$, $Y$ is a
	finite dimensional normed vectorspace, $f \in \trunc ( V, Y )$,
	and $\derivative{V}{f} \in \Lploc{q} ( \| V \|, \Hom ( \rel^\adim,
	Y ) )$.

	Then the following four statements hold:
	\begin{enumerate}
		\item \label{item:integrability:m>1} If $\vdim > 1$ and $f \in
		\Lploc{1} ( \| \delta V \|, Y )$, then $f \in
		\Lploc{\vdim/(\vdim-1)} ( \| V \|, Y )$.
		\item \label{item:integrability:m=1} If $\vdim = 1$, then $f
		\in \Lploc{\infty} ( \| V \| + \| \delta V \|, Y )$.
		\item \label{item:integrability:q<m=p} If $1 \leq q < \vdim =
		p$, then $f \in \Lploc{\vdim q/(\vdim-q)} ( \| V \|, Y
		)$.
		\item \label{item:integrability:m=p<q} If $1 < \vdim = p < q
		\leq \infty$, then $f \in \Lploc{\infty} ( \| V \|, Y )$.
	\end{enumerate}
\end{corollary}
\begin{proof}
	In view of \ref{remark:mod_tv} and \ref{remark:trunc} it is sufficient
	to consider the case $Y = \rel$ and $f \in \trunc_\varnothing (V)$.
	Assume $p = 1$ in case of \eqref{item:integrability:m>1} or
	\eqref{item:integrability:m=1} and define $\psi$ as in
	\ref{miniremark:situation_general}. Moreover, assume $( \| V \| + \psi
	) ( U ) + \Lpnorm{\| V \|}{q}{ \derivative{V}{f} } < \infty$ and, in
	case of \eqref{item:integrability:m>1}, also $\Lpnorm{\| \delta V
	\|}{1}{f} < \infty$. Suppose $K$ is a compact subset of $U$.  Choose
	$0 < r < \infty$ with $\oball{x}{r} \subset U$ for $x \in K$ and $0 <
	s < \infty$ with
	\begin{gather*}
		\| V \| ( E ) \leq (1/2) \unitmeasure{\vdim} r^\vdim, \quad
		\psi ( E ) \leq \Gamma_{\ref{thm:sob_poin_summary}} ( \sup \{
		2, \adim \})^{-1},
	\end{gather*}
	where $E = \{ x \with f(x) > s \}$. Select $g \in \mathscr{E} ( \rel,
	\rel )$ with $g (t) = 0$ if $t \leq s$ and $g (t) = t$ if $t \geq s+1$
	and notice that $g \circ f \in \trunc_\varnothing ( V )$ and
	\begin{gather*}
		\| \derivative{V}{(g \circ f)}(x) \| \leq ( \Lip g ) \|
		\derivative{V}{f} \| ( x ) \quad \text{for $\| V \|$ almost
		all $x$}
	\end{gather*}
	by \ref{lemma:basic_v_weakly_diff} with $\Upsilon = \{ t \with t \leq
	s \}$ and \ref{remark:trunc}. Applying
	\ref{thm:sob_poin_summary}\,\eqref{item:sob_poin_summary:interior}
	with $M$, $G$, $f$, and $Q$ replaced by $\sup \{ 2, \adim \}$,
	$\varnothing$, $g \circ f$, and $1$ and, in case of
	\eqref{item:integrability:m=1}, noting
	\ref{corollary:boundary_controls_interior}, the conclusion follows.
\end{proof}
\begin{miniremark} \label{miniremark:lpnorms}
	\emph{If $1 \leq q \leq p \leq \infty$, $q < \infty$, $\mu$ measures
	$X$, $G$ is a countable collection of $\mu$ measurable $\{ t \with 0
	\leq t \leq \infty \}$ valued functions, $1 \leq \kappa < \infty$,
	\begin{gather*}
		\card ( G \cap \{ g \with g(x) > 0 \} ) \leq \kappa \quad
		\text{for $\mu$ almost all $x$},
	\end{gather*}
	and $f (x) = \sum_{g \in G} g(x)$ for $\mu$ almost all $x$, then 
	\begin{gather*}
		\Lpnorm{\mu}{p}{f} \leq \kappa \big ( \tsum{g \in G}{}
		\Lpnorm{\mu}{p}{g}^q \big )^{1/q}, \quad
		\Lpnorm{\mu}{\infty}{f} \leq \kappa \sup ( \{ 0 \} \cup \{
		\Lpnorm{\mu }{\infty}{g} \with g \in G \} );
	\end{gather*}}
	in fact, abbreviating $h(x) = \sum_{g \in G} g(x)^q$ for $\mu$ almost
	all $x$, one estimates
	\begin{gather*}
		\Lpnorm{\mu}{p}{f}^q \leq \kappa^q \big ( \tsum{g \in G}{}
		\tint{}{}
		g^p \ud \mu \big )^{q/p} \leq \kappa^q \tsum{g \in G}{}
		\Lpnorm{\mu}{p}{g}^q \quad \text{if $p < \infty$}, \\
		\Lpnorm{\mu}{\infty}{f}^q \leq \kappa^q
		\Lpnorm{\mu}{\infty}{h} \leq \kappa^q \tsum{g\in G}{}
		\Lpnorm{\mu}{\infty}{g}^q.
	\end{gather*}
\end{miniremark}
\begin{lemma} \label{lemma:partition_of_unity_normed}
	Suppose $Y$ is a finite dimensional normed vectorspace, $\Phi$ is a
	family of open subsets $Y$, and $h : \bigcup \Phi \to \{ t \with 0 < t
	< \infty \}$ satisfies
	\begin{gather*}
		h(y) = {\textstyle\frac 1{20}} \sup \{ \inf \{ 1 , \dist ( y,
		Y \without \Upsilon ) \} \with \Upsilon \in \Phi \} \quad
		\text{whenever ${\textstyle y \in \bigcup \Phi}$}.
	\end{gather*}

	Then there exists a countable subset $B$ of $\bigcup \Phi$ and for
	each $b \in B$ an associated function $g_b : Y \to \rel$ with the
	following properties.
	\begin{enumerate}
		\item If $y \in B$, then $B_y = B \cap \{ b \with \cball
		b{10h(b)} \cap \cball y{10h(y)} \neq \varnothing \}$ satisfies
		\begin{gather*}
			\card B_y \leq (129)^{\dim Y}.
		\end{gather*}
		\item If $b \in B$, then $0 \leq g_b(y) \leq 1$ for $y \in Y$,
		$\spt g_b \subset \cball b{10h(b)}$, and
		\begin{gather*}
			\Lip ( g_b | \cball y{10h(y)}) \leq (130)^{\dim Y}
			h(y)^{-1} \quad \text{whenever ${\textstyle y \in
			\bigcup \Phi}$}.
		\end{gather*}
		\item If $y \in Y$, then $\sum_{b \in B} g_b(y) = 1$.
	\end{enumerate}
\end{lemma}
\begin{proof}
	Assume $\dim Y > 0$ and observe that \cite[3.1.13]{MR41:1976} remains
	valid with $\rel^m$ and $m$ replaced by $Y$ and $\dim Y$; in fact, one
	modifies the proof by choosing a nonzero translation invariant Radon
	measure $\nu$ over $Y$ and replacing $\unitmeasure \vdim$ by
	$\measureball \nu {\cball 01}$. Then one modifies the proof of
	\cite[3.1.14]{MR41:1976} by taking $B$ to be the set named ``$S$''
	there, $\gamma (t) = \sup \{ 0, \inf \{ 1, 2-t \} \}$ for $t \in
	\rel$, hence $\Lip \mu =1$, and adding the estimates
	\begin{gather*}
		\Lip u_b \leq h(y)^{-1}, \quad \Lip ( \sigma | \cball
		y{10h(y)} ) \leq (129)^{\dim Y} h(y)^{-1}, \\
		\Lip (v_b | \cball y{10h(y)} ) \leq \Lip u_b + \Lip ( \sigma |
		\cball y{10h(y)} ) \leq (130)^{\dim Y} h(y)^{-1}
	\end{gather*}
	whenever $y \in \bigcup \Phi$ and $b \in B_y$; hence one may take $g_b
	= v_b$ for $b \in B$.
\end{proof}
\begin{lemma} \label{lemma:partition_of_unity}
	Suppose $Y$ is a finite dimensional normed vectorspace, $D$ is a
	closed subset of $Y$, and $0 < s < \infty$.
	
	Then there exists a countable family $G$ with the following
	properties.
	\begin{enumerate}
		\item \label{item:partition_of_unity:contained} If $g \in G$,
		then $g : Y \to \rel$ and $\spt g \subset \oball{d}{s}$ for
		some $d \in D$.
		\item \label{item:partition_of_unity:card} If $y \in Y$
		then $\card ( G \cap \{ g \with \cball{y}{s/4} \cap \spt g
		\neq \varnothing \} ) \leq (129)^{\dim Y}$.
		\item \label{item:partition_of_unity:sum} There holds $D
		\subset \Int \big \{ \sum_{g \in G} g(y) = 1 \big \}$ and
		$\sum_{g \in G} g(y) \leq 1$ for $y \in Y$.
		\item \label{item:partition_of_unity:estimate} If $g \in G$,
		then $0 \leq g(y) \leq 1$ for $y \in Y$ and $\Lip g \leq 40
		\cdot (130)^{\dim Y} s^{-1}$.
	\end{enumerate}
\end{lemma}
\begin{proof}
	Assume $s=1$. Defining $\Phi = \{ Y \without D \} \cup \{ \oball{d}{1}
	\with d \in D \}$, observe that the function $h$ resulting from
	\ref{lemma:partition_of_unity_normed} satisfies $h(y) \geq
	\frac{1}{40}$ for $y \in Y$, hence one verifies that one can take $G =
	\{ g_b \with D \cap \spt g_b \neq \varnothing \}$.
\end{proof}
\begin{theorem} \label{thm:sob_poincare_q_medians}
	Suppose $1 \leq M < \infty$.

	Then there exists a positive, finite number $\Gamma$ with the
	following property.

	If $\vdim$, $\adim$, $p$, $U$, $V$, and $\psi$ are as in
	\ref{miniremark:situation_general}, $Y$ is a finite dimensional normed
	vectorspace, $\sup \{ \dim Y, \adim \} \leq M$, $f \in \trunc
	(V,Y)$, $1 \leq Q \leq M$, $N \in \nat$, $0 < r < \infty$,
	\begin{gather*}
		\| V \| ( U ) \leq ( Q-M^{-1} ) (N+1) \unitmeasure{\vdim}
		r^\vdim, \\
		\| V \| ( \{ x \with \density^\vdim ( \| V \|, x ) < Q \} )
		\leq \Gamma^{-1} r^\vdim,
	\end{gather*}	
	$A = \{ x \with \oball{x}{r} \subset U \}$, and $f_\Upsilon : \dmn f
	\to \rel$ is defined by
	\begin{gather*}
		f_\Upsilon (x) = \dist (f(x),\Upsilon) \quad \text{whenever $x
		\in \dmn f$, $\varnothing \neq \Upsilon \subset Y$},
	\end{gather*}
	then the following four statements hold.
	\begin{enumerate}
		\item \label{item:sob_poincare_q_medians:p=1} If $p = 1$,
		$\beta = \infty$ if $\vdim = 1$ and $\beta = \vdim/(\vdim-1)$
		if $\vdim > 1$, then there exists a subset $\Upsilon$ of $Y$
		such that $1 \leq \card \Upsilon \leq N$ and
		\begin{gather*}
			\eqLpnorm{\| V \| \restrict A}{\beta}{f_\Upsilon} \leq
			\Gamma N^{1/\beta} \big ( \Lpnorm{\| V
			\|}{1}{\derivative{V}{f}} + \| \delta V \|(f_\Upsilon)
			\big).
		\end{gather*}
		\item \label{item:sob_poincare_q_medians:p=m=1} If $p = \vdim
		= 1$ and $\psi (U) \leq \Gamma^{-1}$, then there exists a
		subset $\Upsilon$ of $Y$ such that $1 \leq \card \Upsilon \leq
		N$ and
		\begin{gather*}
			\eqLpnorm{\| V \| \restrict A}{\infty}{f_\Upsilon} \leq
			\Gamma \, \Lpnorm{\| V \|}{1}{\derivative{V}{f}}.
		\end{gather*}
		\item \label{item:sob_poincare_q_medians:q<m=p} If $1 \leq q <
		\vdim = p$ and $\psi (U) \leq \Gamma^{-1}$, then there exists
		a subset $\Upsilon$ of $Y$ such that $1 \leq \card \Upsilon
		\leq N$ and
		\begin{gather*}
			\eqLpnorm{\| V \| \restrict A}{\vdim
			q/(\vdim-q)}{f_\Upsilon} \leq \Gamma N^{1/q-1/\vdim}
			(\vdim-q)^{-1} \Lpnorm{\| V \|}{q}{\derivative{V}{f}}.
		\end{gather*}
		\item \label{item:sob_poincare_q_medians:p=m<q} If $1 < p =
		\vdim < q \leq \infty$ and $\psi ( U ) \leq \Gamma^{-1}$, then
		there exists a subset $\Upsilon$ of $Y$ such that $1 \leq \card
		\Upsilon \leq N$ and
		\begin{gather*}
			\eqLpnorm{\| V \| \restrict A}{\infty}{f_\Upsilon} \leq
			\Gamma^{1/(1/\vdim-1/q)} r^{1-\vdim/q} \Lpnorm{\| V
			\|}{q}{\derivative{V}{f}}.
		\end{gather*}
	\end{enumerate}
\end{theorem}
\begin{proof}
	Define
	\begin{gather*}
		\Delta_1 = \sup \{ 1, \Gamma_{\ref{thm:rel_iso_ineq}} ( M )
		\}^M, \quad \Delta_2 = ( \sup \{ 1/2, 1-1/(2M^2-1) \} )^{1/M},
		\\
		\Delta_3 = 40 \cdot ( 130 )^M, \quad \Delta_4 = \Delta_2^{-M}
		\Gamma_{\ref{thm:sob_poin_summary}} ( 2M), \\
		\Delta_5 = (1/2) \inf \{ \unitmeasure{\vdim} \with M \geq
		\vdim \in \nat \}, \\
		\Delta_6 = \sup \{ 1, \Gamma_{\ref{thm:sob_poin_summary}} (
		2M) \} ( \Delta_3^3 + 1 ) \Delta_5^{-1} (1-\Delta_2)^{-M}, \\
		\Delta_7 = \sup \{ \unitmeasure{\vdim} \with M \geq \vdim \in
		\nat \}, \quad \Delta_8 = 4 M \Delta_6 \Delta_7 +
		\Gamma_{\ref{thm:sob_poin_summary}} ( 2M ), \\
		\Delta_9 = M \Delta_7 \Gamma_{\ref{thm:sob_poin_summary}} ( 2M
		), \quad \Gamma = \sup \{ \Delta_1, \Delta_4, \Delta_8, 2
		\Delta_6 ( 1 + \Delta_9 ) \}
	\end{gather*}
	and notice that $\Delta_5 \leq 1 \leq \inf \{ \Delta_6, \Delta_7 \}$.

	In order to verify that $\Gamma$ has the asserted property, suppose
	$\vdim$, $\adim$, $p$, $U$, $V$, $\psi$, $Y$, $f$, $Q$, $N$, $r$, $A$,
	and $f_\Upsilon$ are related to $\Gamma$ as in the body of the
	theorem.  Abbreviate $\iota = \vdim q/(\vdim-q)$ in case of
	\eqref{item:sob_poincare_q_medians:q<m=p} and $\alpha = 1/\vdim-1/q$
	in case of \eqref{item:sob_poincare_q_medians:p=m<q}.  Moreover,
	define
	\begin{align*}
		\kappa & = \Lpnorm{\| V \|}{1}{\derivative{V}{f}} && \quad
		\text{in case of \eqref{item:sob_poincare_q_medians:p=1} or
		\eqref{item:sob_poincare_q_medians:p=m=1}}, \\
		\kappa & = (\vdim-q)^{-1} \Lpnorm{\| V
		\|}{q}{\derivative{V}{f}} && \quad \text{in case of
		\eqref{item:sob_poincare_q_medians:q<m=p}}, \\
		\kappa & = \Delta_9^{1/\alpha} \Lpnorm{\| V
		\|}{q}{\derivative{V}{f}} && \quad \text{in case of
		\eqref{item:sob_poincare_q_medians:p=m<q}}
	\end{align*}
	and assume $r = 1$ and $\dmn f = U$.
	
	First, the \emph{case $\kappa = 0$} will be considered.
	
	For this purpose choose $\Xi$ and $\upsilon$ as in
	\ref{thm:zero_derivative} and let
	\begin{gather*}
		\Pi = \Xi \cap \{ W \with \| W \| ( U ) \leq ( Q-M^{-1} )
		\unitmeasure{\vdim} \}.
	\end{gather*}
	Applying \ref{corollary:rel_iso_ineq} and
	\ref{corollary:true_rel_iso_ineq} with $V$, $E$, and $B$ replaced by
	$W$, $U$, and $\Bdry U$, one infers
	\begin{align*}
		& \| W \| ( A )^{1/\beta} \leq \Delta_1 \, \| \delta W \| ( U
		) && \quad \text{in case of
		\eqref{item:sob_poincare_q_medians:p=1}, where $0^0=0$,}
		\\
		& \| W \| ( A ) = 0 && \quad \text{in case of
		\eqref{item:sob_poincare_q_medians:p=m=1} or
		\eqref{item:sob_poincare_q_medians:q<m=p} or
		\eqref{item:sob_poincare_q_medians:p=m<q}}
	\end{align*}
	whenever $W \in \Pi$. Since $\card ( \Xi \without \Pi ) \leq N$, there
	exists $\Upsilon$ such that
	\begin{gather*}
		\upsilon \lIm \Xi \without \Pi \rIm \subset \Upsilon \subset Y
		\quad \text{and} \quad 1 \leq \card \Upsilon \leq N.
	\end{gather*}
	In case of \eqref{item:sob_poincare_q_medians:p=1}, it follows that,
	using \ref{remark:decomp_rep},
	\begin{align*}
		& \eqLpnorm{\| V \| \restrict A}{\beta}{f_\Upsilon} \leq
		\tsum{W \in \Pi}{} \dist (\upsilon(W),\Upsilon) \| W \|
		(A)^{1/\beta} \\
		& \qquad \leq \Delta_1 \tsum{W \in \Pi}{} \dist
		(\upsilon(W),\Upsilon) \| \delta W \| ( U ) = \Delta_1 \|
		\delta V \|(f_\Upsilon).
	\end{align*}
	In case of \eqref{item:sob_poincare_q_medians:p=m=1} or
	\eqref{item:sob_poincare_q_medians:q<m=p} or
	\eqref{item:sob_poincare_q_medians:p=m<q}, the corresponding estimate
	is trivial.

	Second, the \emph{case $\kappa > 0$} will be considered.

	For this purpose assume $\kappa < \infty$ and define
	\begin{gather*}
		X = \{ x \with \cball{x}{\Delta_2} \subset U \}, \quad s =
		\Delta_6 \kappa, \\
		B = Y \cap \big \{ y \with \| V \| ( f^{-1} \lIm \oball
		ys \rIm ) > (Q-M^{-1}) \unitmeasure{\vdim} \big \}.
	\end{gather*}
	Choose $\Upsilon \subset Y$ satisfying
	\begin{gather*}
		1 \leq \card \Upsilon \leq N \quad \text{and} \quad B \subset
		\{ y \with \dist (y,\Upsilon) < 2s \};
	\end{gather*}
	in fact, if $B \neq \varnothing$ then one may take $\Upsilon$ to be a
	maximal subset of $B$ with respect to inclusion such that $\{ \oball
	ys \with y \in \Upsilon \}$ is disjointed. Abbreviate $\gamma = \dist
	( \cdot, \Upsilon )$ and, in case of
	\eqref{item:sob_poincare_q_medians:p=1}, define
	\begin{gather*}
		\lambda = \kappa + \| \delta V \| ( \gamma \circ f )
	\end{gather*}
	and assume $\lambda < \infty$. In case of
	\eqref{item:sob_poincare_q_medians:p=m=1} or
	\eqref{item:sob_poincare_q_medians:q<m=p} or
	\eqref{item:sob_poincare_q_medians:p=m<q}, define $\lambda = \kappa$.
	Let
	\begin{gather*}
		D = Y \cap \{ y \with \gamma(y) \geq 2 \Delta_6 \lambda \},
		\quad E = f^{-1} \lIm D \rIm.
	\end{gather*}

	Next, \emph{it will be shown that
	\begin{align*}
		& \| V \| ( X \cap E ) \leq (1/2) \unitmeasure{\vdim} (
		1-\Delta_2)^\vdim && \quad \text{in case of
		\eqref{item:sob_poincare_q_medians:p=1} or
		\eqref{item:sob_poincare_q_medians:q<m=p}}, \\
		& \| V \| ( X \cap E ) = 0 && \quad \text{in case of
		\eqref{item:sob_poincare_q_medians:p=m=1} or
		\eqref{item:sob_poincare_q_medians:p=m<q}}.
	\end{align*}}
	To prove this assertion, choose $G$ as in
	\ref{lemma:partition_of_unity} and let $E_g = \{ x \with g(f(x))
	> 0 \}$ for $g \in G$. Since $D \cap B = \varnothing$ as $\Delta_6
	\lambda \geq s$, it follows that
	\begin{gather*}
		\| V \|( E_g ) \leq ( Q-M^{-1} ) \unitmeasure{\vdim} \leq \big
		(Q-(2M)^{-1} \big ) \unitmeasure{\vdim} \Delta_2^\vdim, \\
		\| V \| ( E_g \cap \{ x \with \density^\vdim ( \| V \| , x ) <
		Q \} ) \leq \Gamma^{-1} \leq \Delta_4^{-1} \leq
		\Gamma_{\ref{thm:sob_poin_summary}} ( 2M )^{-1}
		\Delta_2^\vdim, \\
		\psi ( E_g ) \leq \Gamma^{-1} \leq
		\Gamma_{\ref{thm:sob_poin_summary}} ( 2M )^{-1} \quad \text{in
		case of
		\eqref{item:sob_poincare_q_medians:p=m=1} or
		\eqref{item:sob_poincare_q_medians:q<m=p} or
		\eqref{item:sob_poincare_q_medians:p=m<q}}
	\end{gather*}
	whenever $g \in G$. Denoting by $c_g$ the characteristic function of
	$E_g$, applying \ref{lemma:comp_lip} with $\Upsilon$ replaced by $\spt
	g$ and noting \ref{thm:addition}\,\eqref{item:addition:zero} and
	\ref{remark:trunc} yields $g \circ f \in \trunc_\varnothing(V)$ with
	\begin{gather*}
		| \derivative{V}{(g \circ f)} (x) | \leq \Delta_3 s^{-1}
		c_g(x) \| \derivative{V}{f} (x) \| \quad \text{for $\| V \|$
		almost all $x$}
	\end{gather*}
	whenever $g \in G$. Moreover, denoting the function whose domain is
	$U$ and whose value at $x$ equals $\sum_{g \in G} (g \circ f) (x)$ by
	$\sum_{g\in G} g \circ f$, notice that
	\begin{gather*}
		\card ( G \cap \{ g \with x \in E_g \} ) \leq \Delta_3 \quad
		\text{for $x \in U$}, \\
		\big (\tsum{g \in G}{} g \circ f \big )| E = 1, \qquad \spt g
		\subset \{ y \with \gamma(y) \geq \Delta_6 \lambda \} \quad
		\text{for $g \in G$}.
	\end{gather*}
	In case of \eqref{item:sob_poincare_q_medians:p=1}, one estimates,
	using
	\ref{thm:sob_poin_summary}\,\eqref{item:sob_poin_summary:interior:p=1}
	with $M$, $G$, $f$, and $r$ replaced by $2M$, $\varnothing$, $g \circ
	f$, and $\Delta_2$ and \cite[2.4.18\,(1)]{MR41:1976},
	\begin{align*}
		& \| V \| ( X \cap E )^{1/\beta} \leq \eqLpnorm{\| V \|
		\restrict X}{\beta}{ \tsum{g \in G}{} g \circ f } \leq \tsum{g
		\in G}{} \eqLpnorm{\| V \| \restrict X}{\beta}{ g \circ f} \\
		& \qquad \leq \Gamma_{\ref{thm:sob_poin_summary}} ( 2M )
		\tsum{g \in G}{} \big ( \Delta_3 s^{-1} \eqLpnorm{\| V \|
		\restrict E_g}{1}{ \derivative{V}{f} } + \| \delta V \| ( g
		\circ f ) \big ), \\
		& \qquad \leq \Gamma_{\ref{thm:sob_poin_summary}} ( 2M ) \big
		( \Delta_6^{-1} \Delta_3^2 + ( f_\# \| \delta V \| ) ( \{ y
		\with \gamma (y) \geq \Delta_6 \lambda \} ) \big ) \\
		& \qquad \leq \Gamma_{\ref{thm:sob_poin_summary}} ( 2M ) (
		\Delta_3^2 + 1 ) \Delta_6^{-1} \leq \Delta_5 ( 1- \Delta_2)^M
		\leq \big ((1/2) \unitmeasure{\vdim} (1-\Delta_2)^\vdim \big
		)^{1/\beta},
	\end{align*}
	where $0^0=0$. In case of \eqref{item:sob_poincare_q_medians:p=m=1},
	using
	\ref{thm:sob_poin_summary}\,\eqref{item:sob_poin_summary:interior:p=m=1}
	with $M$, $G$, $f$, and $r$ replaced by $2M$, $\varnothing$, $g \circ
	f$, and $\Delta_2$, one estimates
	\begin{align*}
		& \eqLpnorm{\| V \| \restrict X}{\infty}{ \tsum{g \in G}{} g
		\circ f } \leq \tsum{g \in G}{} \eqLpnorm{\| V \| \restrict
		X}{\infty}{g \circ f} \\
		& \qquad \leq \Gamma_{\ref{thm:sob_poin_summary}} ( 2M )
		\tsum{g \in G}{} \Delta_3 s^{-1} \eqLpnorm{ \| V \| \restrict
		E_g }{1}{ \derivative{V}{f} } \\
		& \qquad \leq \Gamma_{\ref{thm:sob_poin_summary}} ( 2M )
		\Delta_3^2 \Delta_6^{-1} \leq \Delta_5 ( 1-\Delta_2 )^M < 1,
	\end{align*}
	hence $\| V \| ( X \cap E ) = 0$. In case of
	\eqref{item:sob_poincare_q_medians:q<m=p}, using
	\ref{miniremark:lpnorms} and
	\ref{thm:sob_poin_summary}\,\eqref{item:sob_poin_summary:interior:q<m=p}
	with $M$, $G$, $f$, and $r$ replaced by $2M$, $\varnothing$, $g \circ
	f$, and $\Delta_2$, one estimates
	\begin{align*}
		& \| V \| ( X \cap E )^{1/\iota} \leq \eqLpnorm{\| V \|
		\restrict X}{\iota}{ \tsum{g \in G}{} g \circ f } \\
		& \qquad \leq \Delta_3 \big ( \tsum{g \in G}{} \eqLpnorm{\| V
		\| \restrict X}{\iota}{ g \circ f }^q \big )^{1/q} \\
		& \qquad \leq \Gamma_{\ref{thm:sob_poin_summary}} ( 2M )
		(\vdim-q)^{-1} \Delta_3^2 s^{-1} \big ( \tsum{g \in G}{}
		\eqLpnorm{\| V \| \restrict E_g}{q}{ \derivative{V}{f} }^q
		\big )^{1/q} \\
		& \qquad \leq \Gamma_{\ref{thm:sob_poin_summary}} ( 2M )
		\Delta_3^3 \Delta_6^{-1} \leq \Delta_5 (1-\Delta_2)^M \leq
		\big ( (1/2) \unitmeasure{\vdim} ( 1-\Delta_2)^\vdim \big
		)^{1/\iota}.
	\end{align*}
	In case of \eqref{item:sob_poincare_q_medians:p=m<q}, using
	\ref{thm:sob_poin_summary}\,\eqref{item:sob_poin_summary:interior:p=m<q}
	with $M$, $G$, $f$, and $r$ replaced by $2M$, $\varnothing$, $g \circ
	f$, and $\Delta_2$ and noting $\Gamma_{\ref{thm:sob_poin_summary}} (
	2M)^{1/\alpha} (M\Delta_7)^\alpha \leq \Delta_9^{1/\alpha}$, one obtains
	\begin{gather*}
		\eqLpnorm{\| V \| \restrict X}{\infty}{g \circ f} \leq
		\Delta_9^{1/\alpha} \Delta_3 s^{-1} \eqLpnorm{\| V \| \restrict
		E_g}{q}{\derivative{V}{f}} \quad \text{for $g \in G$}
	\end{gather*}
	and therefore $\| V \| (X \cap E) = 0$, since if $q < \infty$ then,
	using \ref{miniremark:lpnorms},
	\begin{align*}
		& \eqLpnorm{\| V \| \restrict X}{\infty}{ \tsum{g \in G}{} g
		\circ f } \leq \Delta_3 \big ( \tsum{g \in G}{} \eqLpnorm{\| V
		\| \restrict X}{\infty}{g \circ f}^q \big )^{1/q} \\
		& \qquad \leq \Delta_9^{1/\alpha} \Delta_3^2 s^{-1} \big (
		\tsum{g \in G}{} \eqLpnorm{\| V \| \restrict
		E_g}{q}{\derivative{V}{f}}^q \big)^{1/q} \\
		& \qquad \leq \Delta_3^3 \Delta_6^{-1} \leq \Delta_5
		(1-\Delta_2)^M < 1,
	\end{align*}
	and if $q = \infty$ then $\eqLpnorm{ \| V \| \restrict X}{\infty}{
	\sum_{g \in G} g \circ f} < 1$ follows similarly.

	In case of \eqref{item:sob_poincare_q_medians:p=m=1} or
	\eqref{item:sob_poincare_q_medians:p=m<q}, noting $2 \Delta_6
	\Delta_9^{1/\alpha} \leq \Gamma^{1/\alpha}$ in case of
	\eqref{item:sob_poincare_q_medians:p=m<q}, the conclusion is evident
	from the assertion of the preceding paragraph. Now consider the case
	of \eqref{item:sob_poincare_q_medians:p=1} or
	\eqref{item:sob_poincare_q_medians:q<m=p} and define $h : Y \to \rel$
	by
	\begin{gather*}
		h(y) = \sup \{ 0, \gamma(y)- 2 \Delta_6 \lambda \} \quad
		\text{for $y \in Y$}.
	\end{gather*}
	Notice that $h \circ f \in \trunc_\varnothing \big (V|\mathbf{2}^{X
	\times \grass{\adim}{\vdim}} \big )$ with
	\begin{gather*}
		| \derivative{V}{(h \circ f)} (x) | \leq \| \derivative{V}{f}
		(x) \| \quad \text{for $\| V \|$ almost all $x$}
	\end{gather*}
	by \ref{lemma:comp_lip} with $\Upsilon$ replaced by $Y$ and
	\ref{remark:trunc}. In case of
	\eqref{item:sob_poincare_q_medians:p=1}, applying
	\ref{thm:sob_poin_summary}\,\eqref{item:sob_poin_summary:interior:p=1}
	with $M$, $U$, $V$, $G$, $f$, $Q$, and $r$ replaced by $2M$, $X$, $V |
	\mathbf{2}^{X \times \grass{\adim}{\vdim}}$, $\varnothing$, $h \circ
	f|X$, $1$, and $1-\Delta_2$ implies
	\begin{gather*}
		\eqLpnorm{\| V \| \restrict A}{\beta}{ h \circ f} \leq
		\Gamma_{\ref{thm:sob_poin_summary}} ( 2M ) \big ( \Lpnorm{\| V
		\|}{1}{ \derivative{V}{f} } + \| \delta V \| ( h \circ f )
		\big ), \\
		\eqLpnorm{\| V \| \restrict A}{\beta}{\gamma \circ f} \leq 4 M
		\Delta_6 \Delta_7 N^{1/\beta} \lambda +
		\Gamma_{\ref{thm:sob_poin_summary}} ( 2M) \lambda \leq \Gamma
		N^{1/\beta} \lambda
	\end{gather*}
	since $\| V \| ( A )^{1/\beta} \leq 2 M \Delta_7 N^{1/\beta}$. In case
	of \eqref{item:sob_poincare_q_medians:q<m=p}, applying
	\ref{thm:sob_poin_summary}\,\eqref{item:sob_poin_summary:interior:q<m=p}
	with $M$, $U$, $V$, $G$, $f$, $Q$, and $r$ replaced by $2M$, $X$, $V |
	\mathbf{2}^{X \times \grass{\adim}{\vdim}}$, $\varnothing$, $h \circ
	f|X$, $1$, and $1-\Delta_2$ implies
	\begin{gather*}
		\eqLpnorm{\| V \| \restrict A}{\iota}{ h \circ f} \leq
		\Gamma_{\ref{thm:sob_poin_summary}} ( 2M ) \kappa, \\
		\eqLpnorm{\| V \| \restrict A}{\iota}{ \gamma \circ f} \leq 4 M
		\Delta_6 \Delta_7 N^{1/\iota} \kappa +
		\Gamma_{\ref{thm:sob_poin_summary}} ( 2M ) \kappa \leq \Gamma
		N^{1/\iota} \kappa
	\end{gather*}
	since $\| V \| (A)^{1/\iota} \leq 2M \Delta_7 N^{1/\iota}$.
\end{proof}
\begin{remark}
	Comparing the preceding theorem to previous Sobolev Poincar{\'e}
	inequalities of the author, see \cite[4.4--4.6]{snulmenn.poincare},
	which treated the particular case of $f$ being the orthogonal
	projection onto an $\codim$ dimensional plane in the context of
	integral varifolds, the present approach is significantly more
	general.  Yet, for the case of orthonormal projections the previous
	results give somewhat more precise information as they include, for
	instance, the intermediate cases $1 < p < \vdim$ as well as Lorentz
	spaces.  Evidently, the question arises whether the present theorem
	can be correspondingly extended.
\end{remark}
\begin{theorem} \label{thm:sob_poin_several_med}
	Suppose $\vdim$, $\adim$, $p$, $U$, $V$, and $\psi$ are as in
	\ref{miniremark:situation_general}, $\adim \leq M < \infty$, $N \in
	\nat$, $1 \leq Q \leq M$, $f \in \trunc (V)$, $0 < r < \infty$, $X$ is
	a finite subset of $U$,
	\begin{gather*}
		\| V \| ( U ) \leq ( Q-M^{-1} ) (N+1) \unitmeasure{\vdim}
		r^\vdim, \\
		\| V \| ( \{ x \with \density^\vdim ( \| V \|, x ) < Q \} )
		\leq \Gamma^{-1} r^\vdim,
	\end{gather*}
	and $A = \{ x \with \oball{x}{r} \subset U \}$, then there exists a
	subset $\Upsilon$ of $\rel$ with $1 \leq \card \Upsilon \leq N + \card
	X$ such that the following four statements hold with $g = \dist (
	\cdot, \Upsilon ) \circ f$:
	\begin{enumerate}
		\item \label{item:sob_poin_several_med:p=1} If $p = 1$, $\beta
		= \infty$ if $\vdim = 1$ and $\beta = \vdim/(\vdim-1)$ if
		$\vdim > 1$, then
		\begin{gather*}
			\eqLpnorm{\| V \| \restrict A}{\beta}{g} \leq
			\Gamma_{\ref{thm:sob_poin_summary}} (M) \big (
			\Lpnorm{\| V \|}{1}{ \derivative{V}{f} } + \| \delta V
			\|(g) \big ).
		\end{gather*}
		\item \label{item:sob_poin_several_med:p=m=1} If $p = \vdim =
		1$ and $\psi ( U \without X ) \leq
		\Gamma_{\ref{thm:sob_poin_summary}} (M)^{-1}$, then
		\begin{gather*}
			\eqLpnorm{\| V \| \restrict A}{\infty}{g} \leq
			\Gamma_{\ref{thm:sob_poin_summary}} (M) \, \Lpnorm{\|
			V \|}{1}{ \derivative{V}{f} }.
		\end{gather*}
		\item \label{item:sob_poin_several_med:q<m=p} If $1 \leq q <
		\vdim = p$ and $\psi ( U ) \leq
		\Gamma_{\ref{thm:sob_poin_summary}} (M)^{-1}$, then
		\begin{gather*}
			\eqLpnorm{\| V \| \restrict A}{\vdim q/(\vdim-q)}{g}
			\leq \Gamma_{\ref{thm:sob_poin_summary}} (M)
			(\vdim-q)^{-1} \Lpnorm{\| V \|}{q}{ \derivative{V}{f}
			}.
		\end{gather*}
		\item \label{item:sob_poin_several_med:p=m<q} If $1 < p =
		\vdim < q \leq \infty$ and $\psi ( U ) \leq
		\Gamma_{\ref{thm:sob_poin_summary}} (M)^{-1}$, then
		\begin{gather*}
			\eqLpnorm{\| V \| \restrict A}{\infty}{g} \leq
			\Gamma^{1/(1/\vdim-1/q)} Q^{1/\vdim-1/q} r^{1-\vdim/q}
			\Lpnorm{\| V \|}{q}{ \derivative{V}{f} },
		\end{gather*}
		where $\Gamma = \Gamma_{\ref{thm:sob_poin_summary}} (M) \sup
		\{ \unitmeasure{i} \with M \geq i \in \nat \}$.
	\end{enumerate}
\end{theorem}
\begin{proof}
	Choose a nonempty subset $\Upsilon$ of $\rel$ satisfying
	\begin{gather*}
		\card \Upsilon \leq N + \card X, \quad f \lIm X \rIm \subset
		\Upsilon, \\
		\| V \| ( U \cap \{ x \with f(x) \in I \} ) \leq (Q-M^{-1})
		\unitmeasure{\vdim} r^\vdim \quad \text{for $I \in \Phi$},
	\end{gather*}
	where $\Phi$ is the family of connected components of $\rel \without
	\Upsilon$.  Notice that
	\begin{gather*}
		\dist (b,\Upsilon) = \dist (b, \rel\without I) \quad
		\text{and} \quad \dist (b,\rel \without J) = 0
	\end{gather*}
	whenever $b \in I \in \Phi$ and $I \neq J \in \Phi$, and
	\begin{gather*}
		\dist (b,\Upsilon) = \dist (b,\rel \without I) = 0 \quad
		\text{whenever $b \in \Upsilon$ and $I \in \Phi$}.
	\end{gather*}
	Defining $f_I = \dist ( \rel \without I ) \circ f$, one infers $f_I
	\in \trunc_\varnothing (V)$ and
	\begin{gather*}
		| \derivative{V}{f} (x) | = | \derivative{V}{f_I}(x) | \quad
		\text{for $\| V \|$ almost all $x \in f^{-1} \lIm I \rIm$}, \\
		\derivative{V}{f} (x) = 0 \quad \text{for $\| V \|$ almost all
		$x \in f^{-1} \lIm \Upsilon \rIm$}, \\
		\derivative{V}{f_I} (x) = 0 \quad \text{for $\| V \|$ almost
		all $x \in f^{-1} \lIm \rel \without I \rIm$}
	\end{gather*}
	whenever $I \in \Phi$ by \ref{lemma:basic_v_weakly_diff},
	\ref{example:composite},
	\ref{thm:addition}\,\eqref{item:addition:zero}, and
	\ref{remark:trunc}.

	The conclusion now will be obtained employing \ref{miniremark:lpnorms}
	and applying
	\ref{thm:sob_poin_summary}\,\eqref{item:sob_poin_summary:interior}
	with $G$ and $f$ replaced by $\varnothing$ and $f_I$ for $I \in \Phi$.
	For instance, in case of \eqref{item:sob_poin_several_med:q<m=p} one
	estimates
	\begin{gather*}
		\begin{aligned}
			\eqLpnorm{\| V \| \restrict A}{\vdim q/(\vdim-q)}{g} &
			\leq \big ( \tsum{I \in \Phi}{} \eqLpnorm{\| V \|
			\restrict A}{\vdim q/(\vdim-q)}{ f_I }^q \big )^{1/q}
			\\
			& \leq \lambda \big ( \tsum{I \in \Phi}{} \Lpnorm{\| V
			\|}{q}{ \derivative{V}{f_I} }^q \big )^{1/q} = \lambda
			\, \Lpnorm{\| V \|}{q}{ \derivative{V}{f} },
		\end{aligned}
	\end{gather*}
	where $\lambda = \Gamma_{\ref{thm:sob_poin_summary}} (M)
	(\vdim-q)^{-1}$. The cases \eqref{item:sob_poin_several_med:p=1} and
	\eqref{item:sob_poin_several_med:p=m=1} follow similarly. Defining
	$\Delta = \sup \{ \unitmeasure{i} \with M \geq i \in \nat \}$ and
	noting $\Delta \geq 1$, hence $\unitmeasure{\vdim}^{(1/\vdim-1/q)^2}
	\leq \Delta^{(1/\vdim-1/q)^2} \leq \Delta$, the same holds for
	\eqref{item:sob_poin_several_med:p=m<q}.
\end{proof}
\begin{remark}
	The method of deduction of \ref{thm:sob_poin_several_med} from
	\ref{thm:sob_poin_summary} is that of Hutchinson \cite[Theorem
	3]{MR1066398} which is derived from Hutchinson \cite[Theorem
	1]{MR1066398}.
\end{remark}
\begin{remark}
	A nonempty choice of $X$ will occur in \ref{thm:mod_continuity}.
\end{remark}
\section{Differentiability properties}
In this section approximate differentiability, see \ref{thm:approx_diff}, and
differentiability in Lebesgue spaces, see \ref{thm:diff_lebesgue_spaces}, are
established for weakly differentiable functions. The primary ingredient is the
Sobolev Poincar{\'e} inequality
\ref{thm:sob_poin_summary}\,\eqref{item:sob_poin_summary:interior}.
\begin{lemma} \label{lemma:approx_diff}
	Suppose $\vdim, \adim \in \nat$, $\vdim \leq \adim$, $U$ is an open
	subset of $U$, $V \in \RVar_\vdim (U)$, $Y$ is a finite dimensional
	normed vectorspace, $f$ is a $\| V \|$ measurable $Y$ valued function,
	and $A$ is the set of points at which $f$ is $( \| V \|, \vdim )$
	approximately differentiable.

	Then the following four statements hold.
	\begin{enumerate}
		\item \label{item:approx_diff:measurable} The set $A$ is $\| V
		\|$ measurable and $( \| V \|, \vdim ) \ap Df(x) \circ
		\project{\Tan^\vdim ( \| V \|, x )}$ depends $\| V \|
		\restrict A$ measurably on $x$.
		\item \label{item:approx_diff:cover} There exists a sequence
		of functions $f_i : U \to Y$ of class $1$ such that
		\begin{gather*}
			\| V \| ( A \without \{ x \with \text{$f(x) = f_i(x)$
			for some $i$} \} ) = 0.
		\end{gather*}
		\item \label{item:approx_diff:unrectifiable} If $g : U \to
		Y$ is locally Lipschitzian, then
		\begin{gather*}
			\| V \| ( U \cap \{ x \with f(x) = g(x) \} \without A
			) = 0.
		\end{gather*}
		\item \label{item:approx_diff:comp} If $g$ is a $\| V \|$
		measurable $Y$ valued function and $B = U \cap \{ x \with
		f(x) = g(x) \}$, then $B \cap A$ is $\| V \|$ almost equal to
		$B \cap \dmn ( \| V \|, \vdim ) \ap Dg$ and
		\begin{gather*}
			( \| V \|, \vdim ) \ap Df(x) = ( \| V \|, \vdim ) \ap
			Dg(x) \quad \text{for $\| V \|$ almost all $x \in
			B \cap A$}.
		\end{gather*}
	\end{enumerate}
\end{lemma}
\begin{proof}
	Assume $Y = \rel$. Then \eqref{item:approx_diff:measurable} is
	\cite[4.5\,(1)]{snulmenn.decay}.

	In order to prove \eqref{item:approx_diff:cover}, one may reduce the
	problem. Firstly, to the case that $\| V \| ( U \without M ) = 0$ for
	some $\vdim$ dimensional submanifold $M$ of $\rel^\adim$ of class $1$
	by \cite[2.10.19\,(4), 3.2.29]{MR41:1976}. Secondly, to the case that
	for some $1 < \lambda < \infty$ the varifold satisfies additionally
	that $\lambda^{-1} \leq \density^\vdim ( \| V \|, x ) \leq \lambda$
	for $\| V \|$ almost all $x$ by \cite[2.10.19\,(4)]{MR41:1976} and
	Allard \cite[3.5\,(1a)]{MR0307015}, hence thirdly to the case that
	$\density^\vdim ( \| V \|, x ) = 1$ for $\| V \|$ almost all $x$ by
	\cite[2.10.19\,(1)\,(3)]{MR41:1976}. Finally, to the case $\| V \| =
	\mathscr{L}^\vdim$ by \cite[3.1.19\,(4), 3.2.3, 2.8.18,
	2.9.11]{MR41:1976} which may be treated by means of
	\cite[3.1.16]{MR41:1976}.

	\eqref{item:approx_diff:comp} follows from
	\cite[2.10.19\,(4)]{MR41:1976} and implies
	\eqref{item:approx_diff:unrectifiable} by
	\cite[4.5\,(2)]{snulmenn.decay}.
\end{proof}
\begin{theorem} \label{thm:approx_diff}
	Suppose $\vdim$, $\adim$, $U$, and $V$ are as in
	\ref{miniremark:situation_general}, $Y$ is a finite dimensional normed
	vectorspace, and $f \in \trunc (V,Y)$.

	Then $f$ is $( \| V \|, \vdim )$ approximately differentiable with
	\begin{gather*}
		\derivative{V}{f} (a) = ( \| V \|, \vdim ) \ap Df(a) \circ
		\project{\Tan^\vdim ( \| V \|, a )}
	\end{gather*}
	at $\| V \|$ almost all $a$.
\end{theorem}
\begin{proof}
	In view of \ref{lemma:approx_diff}\,\eqref{item:approx_diff:comp}, one
	employs \ref{miniremark:trunc} and \ref{lemma:basic_v_weakly_diff} to
	reduce the problem to the case that $f$ is bounded and
	$\derivative{V}{f} \in \Lploc{1} ( \| V \|, \Hom ( \rel^\adim,
	Y))$. In particular, one may assume $Y = \rel$ by
	\ref{lemma:integration_by_parts}.  Define $\beta = \infty$ if $\vdim
	= 1$ and $\beta = \vdim/(\vdim-1)$ if $\vdim > 1$. Let $g_a : \dmn f
	\to \rel$ be defined by $g_a(x) = f(x)-f(a)$ for $a,x \in \dmn f$.

	First, \emph{it will be shown that
	\begin{gather*}
		\limsup_{s \to 0+} s^{-1} \| V \| ( \cball as )^{-1}
		\tint{\cball as}{} g_a \ud \| V \| < \infty
		\quad \text{for $\| V \|$ almost all $a$}.
	\end{gather*}}
	For this purpose define $C = \{ (a, \cball{a}{r} ) \with \cball{a}{r}
	\subset U \}$ and consider a point $a$ satisfying for some $M$
	the conditions
	\begin{gather*}
		\sup \{ 4, \adim \} \leq M < \infty, \quad 1 \leq
		\density^\vdim ( \| V \|, a ) \leq M, \quad \eqLpnorm{\| V \|
		+ \| \delta V \|}{\infty}{f} \leq M, \\
		\limsup_{s \to 0+} s^{-\vdim} \big ( \tint{\cball{a}{s}}{} |
		\derivative{V}{f} | \ud \| V \| + \measureball{\| \delta V
		\|}{ \cball{a}{s} } \big ) < M, \\
		\text{$\density^\vdim ( \| V \|, \cdot )$ and $f$ are $(\| V
		\|, C )$ approximately continuous at $a$}
	\end{gather*}
	which are met by $\| V \|$ almost all $a$ by \cite[2.10.19\,(1)\,(3),
	2.8.18, 2.9.13]{MR41:1976}. Define $\Delta =
	\Gamma_{\ref{thm:sob_poin_summary}} (M)$. Choose $1 \leq Q \leq M$ and
	$1 < \lambda \leq 2$ subject to the requirements $\density^\vdim ( \| V
	\|, a ) < 2 \lambda^{-\vdim} ( Q- 1/4)$ and
	\begin{gather*}
		\text{either $Q = \density^\vdim ( \| V \|, a ) = 1$} \quad
		\text{or $Q < \density^\vdim ( \| V \|, a )$}.
	\end{gather*}
	Then pick $0 < r < \infty$ such that $\cball{a}{\lambda r} \subset U$
	and
	\begin{gather*}
		\measureball{\| V \|}{\cball{a}{s}} \geq (1/2)
		\unitmeasure{\vdim} s^\vdim, \quad \measureball{\| V
		\|}{\oball{a}{\lambda s}} \leq 2 (Q-M^{-1}) \unitmeasure{\vdim}
		s^\vdim, \\
		\| V \| ( \classification{\oball{a}{\lambda s}}{x}{
		\density^\vdim ( \| V \|, x ) < Q } ) \leq \Delta^{-1}
		s^\vdim, \\
		\tint{\cball{a}{\lambda s}}{} | \derivative{V}{f} | \ud \| V \| +
		\measureball{\| \delta V \|}{ \cball{a}{\lambda s} } \leq M
		\lambda^\vdim s^\vdim
	\end{gather*}
	for $0 < s \leq r$. Choose $y (s) \in \rel$ such that
	\begin{gather*}
		\| V \| ( \classification{\oball{a}{\lambda s}}{x}{ f ( x ) < y
		(s) } ) \leq (1/2) \measureball{\| V \|}{ \oball{a}{\lambda s} },
		\\
		\| V \| ( \classification{\oball{a}{\lambda s}}{x}{ f ( x ) > y
		(s) } ) \leq (1/2) \measureball{\| V \|}{ \oball{a}{\lambda s} }
	\end{gather*}
	for $0 < s \leq r$, in particular
	\begin{gather*}
		| y (s) | \leq M \quad \text{and} \quad f(a) = \lim_{s \to 0+}
		y(s).
	\end{gather*}
	Define $f_s ( x ) = f(x)-y(s)$ whenever $0 < s \leq r$ and $x \in
	\dmn f$. Recalling \ref{lemma:basic_v_weakly_diff},
	\ref{example:composite}\,\eqref{item:composite:1d}, and
	\ref{remark:trunc}, one applies
	\ref{thm:sob_poin_summary}\,\eqref{item:sob_poin_summary:interior:p=1}
	with $U$, $G$, $f$, and $r$ replaced by $\oball{a}{\lambda s}$,
	$\varnothing$, $f_s^+$ respectively $f_s^-$, and $s$ to infer
	\begin{align*}
		\eqLpnorm{\| V \| \restrict \cball{a}{(\lambda-1)
		s}}{\beta}{f_s} & \leq \Delta \tint{\cball{a}{\lambda s}}{} |
		\derivative V{f_s} | \ud \| V \| + \tint{\cball{a}{\lambda
		s}}{} |f_s| \ud \| \delta V \| \\
		& \leq \kappa s^\vdim
	\end{align*}
	for $0 < s \leq r$, where $\kappa = \Delta M \lambda^\vdim$. Noting
	that
	\begin{gather*}
		\begin{aligned}
			& | y(s)-y(s/2) | \cdot \| V \| ( \cball{a}{(\lambda-1)
			s/2})^{1/\beta} \\
			& \leq \eqLpnorm{\| V \| \restrict
			\cball{a}{(\lambda-1) s/2}}{\beta}{f_{s/2}} +
			\eqLpnorm{\| V \| \restrict \cball{a}{(\lambda-1)
			s}}{\beta}{f_s} \leq 2 \kappa s^\vdim,
		\end{aligned} \\
		| y(s)-y(s/2) | \leq 2^{\vdim+1} \kappa
		\unitmeasure{\vdim}^{-1/\beta} ( \lambda-1 )^{1-\vdim} s,
	\end{gather*}
	one obtains for $0 < s \leq r$ that
	\begin{gather*}
		|y(s)-f(a)| \leq 2^{\vdim+2} \kappa
		\unitmeasure{\vdim}^{-1/\beta} (\lambda-1)^{1-\vdim} s, \\
		\eqLpnorm{\| V \| \restrict \cball{a}{(\lambda-1)s}
		}{\beta}{g_a} \leq \kappa \big ( 1 + 2^{\vdim+3} M (
		\lambda-1 )^{1-\vdim} \big ) s^\vdim.
	\end{gather*}

	Combining the assertion of the preceding paragraph with
	\cite[3.7\,(i)]{snulmenn.isoperimetric} applied with $\alpha$, $q$,
	and $r$ replaced by $1$, $1$, and $\infty$, one obtains a sequence of
	locally Lipschitzian functions $f_i : U \to \rel$ such that
	\begin{gather*}
		{\textstyle \| V \| \big ( U \without \bigcup \{ B_i \with i
		\in \nat \} \big ) = 0, \quad \text{where $B_i = U \cap \{ x
		\with f(x)=f_i(x) \}$}}.
	\end{gather*}
	Since $f-f_i \in \trunc (V)$ with
	\begin{gather*}
		\derivative Vf (a) - \derivative{V}{f_i} (a) =
		\derivative{V}{(f-f_i)} (a) = 0 \quad \text{for $\| V \|$
		almost all $a \in B_i$}
	\end{gather*}
	by
	\ref{thm:addition}\,\eqref{item:addition:zero}\,\eqref{item:addition:add},
	the conclusion follows from of \ref{example:lipschitzian} and
	\ref{lemma:approx_diff}\,\eqref{item:approx_diff:comp}.
\end{proof}
\begin{miniremark} \label{miniremark:tangent_spaces}
	\emph{If $\vdim$, $\adim$, $U$, $V$, $\psi$, and $p$ are as in
	\ref{miniremark:situation_general}, $p = \vdim$, and $a \in U$, then
	\begin{gather*}
		\Tan^\vdim ( \| V \|, a ) = \Tan ( \spt \| V \|, a );
	\end{gather*}}
	in fact, if $0 < r < \infty$, $\oball a{2r} \subset U$, $\psi ( \oball
	a{2r} \without \{ a \} )^{1/\vdim} \leq ( 2 \isoperimetric
	\vdim)^{-1}$, and $x \in \oball ar \cap \spt \| V \|$, then
	$\measureball{\| V \|}{\cball xs} \geq ( 2 \vdim \isoperimetric
	\vdim)^{-\vdim} s^\vdim$ for $0 < s < |x-a|$ by
	\cite[2.5]{snulmenn.isoperimetric}.
\end{miniremark}
\begin{theorem} \label{thm:diff_lebesgue_spaces}
	Suppose $\vdim$, $\adim$, $p$, $U$, $V$ are as in
	\ref{miniremark:situation_general}, $1 \leq q \leq \infty$, $Y$ is a
	finite dimensional normed vectorspace, $f \in \trunc (V,Y)$,
	$\derivative Vf \in \Lploc{q} ( \| V \|, \Hom (\rel^\adim, Y) )$,
	\begin{gather*}
		C = \{ ( a, \cball{a}{r} ) \with \cball{a}{r} \subset U \},
	\end{gather*}
	and $X$ is the set of points in $\spt \| V \|$ at which $f$ is $( \| V
	\|,C )$ approximately continuous.

	Then $\| V \| ( U \without X ) = 0$ and the following four statements
	hold.
	\begin{enumerate}
		\item \label{item:diff_lebesgue_spaces:m>1=p} If $\vdim > 1$,
		$\beta = \vdim/(\vdim-1)$, and $f \in \Lploc{1} ( \| \delta V
		\|, Y )$, then
		\begin{gather*}
			\lim_{r \to 0+} r^{-\vdim} \tint{\cball ar}{} ( |
			f(x)-f(a)- \left <x-a, \derivative Vf (a) \right > | /
			| x-a| )^\beta \ud \| V \| x = 0
		\end{gather*}
		for $\| V \|$ almost all $a$.
		\item \label{item:diff_lebesgue_spaces:m=p=1} If $\vdim = 1$,
		then $f|X$ is differentiable relative to $X$ at $a$ with
		\begin{gather*}
			D (f|X) (a) = \derivative Vf (a) | \Tan^\vdim ( \| V
			\|, a) \quad \text{for $\| V \|$ almost all $a$}.
		\end{gather*}
		\item \label{item:diff_lebesgue_spaces:q<m=p} If $q < \vdim =
		p$ and $\iota = \vdim q/(\vdim-q)$, then
		\begin{gather*}
			\lim_{r \to 0+} r^{-\vdim} \tint{\cball ar}{} ( |
			f(x)-f(a)- \left <x-a, \derivative Vf (a) \right > | /
			| x-a| )^\iota \ud \| V \| x = 0
		\end{gather*}
		for $\| V \|$ almost all $a$.
		\item \label{item:diff_lebesgue_spaces:p=m<q} If $p = \vdim <
		q$, then $f|X$ is differentiable relative to $X$ at $a$ with
		\begin{gather*}
			D (f|X) (a) = \derivative Vf (a) | \Tan^\vdim ( \| V
			\|, a) \quad \text{for $\| V \|$ almost all $a$}.
		\end{gather*}
	\end{enumerate}
\end{theorem}
\begin{proof}
	Clearly, $\| V \| (U \without X) = 0$ by \cite[2.8.18,
	2.9.13]{MR41:1976}.

	Assume $p = q = 1$ in case of \eqref{item:diff_lebesgue_spaces:m>1=p}
	or \eqref{item:diff_lebesgue_spaces:m=p=1}. In case of
	\eqref{item:diff_lebesgue_spaces:p=m<q} also assume $p > 1$ and $q <
	\infty$. Define $\alpha = \beta$ in case of
	\eqref{item:diff_lebesgue_spaces:m>1=p}, $\alpha = \iota$ in case of
	\eqref{item:diff_lebesgue_spaces:q<m=p}, and $\alpha = \infty$ in case
	of \eqref{item:diff_lebesgue_spaces:m=p=1} or
	\eqref{item:diff_lebesgue_spaces:p=m<q}. Moreover, let $h_a : \dmn f
	\to Y$ be defined by
	\begin{gather*}
		h_a(x) = f(x)-f(a) - \left <x-a, \derivative Vf (a) \right > 
	\end{gather*}
	whenever $a \in \dmn f \cap \dmn \derivative Vf$ and $x \in \dmn f$.
	
	The following assertion will be shown. \emph{There holds
	\begin{gather*}
		\lim_{s \to 0+} s^{-1} \eqLpnorm{s^{-\vdim} \| V \| \restrict
		\cball{a}{s}}{\alpha}{h_a} = 0 \quad \text{for $\| V \|$ almost
		all $a$}.
	\end{gather*}}
	 In the special case that $f$ is of class
	$1$, in view of \ref{example:lipschitzian}, it is sufficient to
	prove that
	\begin{gather*}
		\lim_{s \to 0+} s^{-1} \eqLpnorm{s^{-\vdim} \| V \| \restrict
		\cball{a}{s}}{\alpha}{\project{\Nor^\vdim ( \| V \|, a )} (
		\cdot - a )} = 0
	\end{gather*}
	for $\| V \|$ almost all $a$; and if $\varepsilon > 0$ and $\Tan^\vdim
	( \| V \|, a ) \in \grass{\adim}{\vdim}$, then
	\begin{gather*}
		\density^\vdim ( \| V \| \restrict U \cap \{ x \with |
		\project{\Nor^\vdim ( \| V \|, a )} (x-a) | > \varepsilon
		|x-a| \}, a ) = 0
	\end{gather*}
	in case of \eqref{item:diff_lebesgue_spaces:m>1=p} by
	\cite[3.2.16]{MR41:1976} and
	\begin{gather*}
		\cball{a}{s} \cap \spt \| V \| \subset \{ x \with |
		\project{\Nor^\vdim ( \| V \|, a )} (x-a) | \leq \varepsilon
		|x-a| \} \quad \text{for some $s>0$}
	\end{gather*}
	in case of \eqref{item:diff_lebesgue_spaces:m=p=1} or
	\eqref{item:diff_lebesgue_spaces:q<m=p} or
	\eqref{item:diff_lebesgue_spaces:p=m<q} by
	\ref{miniremark:tangent_spaces} and \cite[3.1.21]{MR41:1976}. To treat
	the general case, one obtains a sequence of functions $f_i : U \to
	Y$ of class $1$ such that
	\begin{gather*}
		{\textstyle \| V \| \big ( U \without \bigcup \{ B_i \with i
		\in \nat \} \big ) = 0, \quad \text{where $B_i = U \cap \{ x
		\with f(x)=f_i(x) \}$}}
	\end{gather*}
	from \ref{thm:approx_diff} and
	\ref{lemma:approx_diff}\,\eqref{item:approx_diff:cover}. Define $g_i =
	f - f_i$ and notice that $g_i \in \trunc ( V, Y )$ and
	\begin{gather*}
		\derivative{V}{g_i} (a) = \derivative{V}{f} (a) -
		\derivative{V}{f_i} (a) \quad \text{for $\| V \|$ almost all
		$a$}
	\end{gather*}
	for $i \in \nat$ by \ref{thm:addition}\,\eqref{item:addition:add}.
	Define Radon measures $\mu_i$ over $U$ by
	\begin{gather*}
		\mu_i ( A ) = \tint{A}{\ast} | \derivative{V}{g_i} | \ud \| V
		\| + \tint{A}{\ast} |g_i| \ud \| \delta V \| \quad \text{in
		case of \eqref{item:diff_lebesgue_spaces:m>1=p}}, \\
		\mu_i ( A ) = \tint{A}{\ast} | \derivative{V}{g_i} |^q \ud \|
		V \| \quad \text{in case of
		\eqref{item:diff_lebesgue_spaces:m=p=1} or
		\eqref{item:diff_lebesgue_spaces:q<m=p} or
		\eqref{item:diff_lebesgue_spaces:p=m<q}}
	\end{gather*}
	whenever $A \subset U$ and $i \in \nat$ and notice that $\mu_i ( B_i )
	= 0$ by \ref{thm:addition}\,\eqref{item:addition:zero} (or
	alternately by \ref{thm:approx_diff} and
	\ref{lemma:approx_diff}\,\eqref{item:approx_diff:comp}). The assertion
	will be shown to hold at a point $a$ satisfying for some $i \in \nat$
	that
	\begin{gather*}
		f(a) = f_i (a), \quad \derivative Vf (a) = \derivative V{f_i}
		(a), \\
		\lim_{s \to 0+} s^{-1} \eqLpnorm{s^{-\vdim} \| V \| \restrict
		\cball as }{\alpha}{f_i(\cdot)-f_i(a) - \left < \cdot - a ,
		\derivative V{f_i} (a) \right >} = 0, \\
		\density^\vdim ( \| V \| \restrict
		U \without B_i, a ) = 0, \quad \psi ( \{ a \} ) = 0, \quad
		\density^\vdim ( \mu_i, a ) = 0.
	\end{gather*}
	These conditions are met by $\| V \|$ almost all $a$ in view of the
	special case, \cite[2.10.19\,(4)]{MR41:1976}. Choosing $0 < r <
	\infty$ with $\cball{a}{2r} \subset U$ and
	\begin{gather*}
		\| V \| ( \oball{a}{2s} \without B_i ) \leq (1/2)
		\unitmeasure{\vdim} s^\vdim, \\
		\measureball{\psi}{\oball{a}{2r}} \leq
		\Gamma_{\ref{thm:sob_poin_summary}} ( 2\adim )^{-1} \quad
		\text{in case of \eqref{item:diff_lebesgue_spaces:m=p=1} or
		\eqref{item:diff_lebesgue_spaces:q<m=p} or
		\eqref{item:diff_lebesgue_spaces:p=m<q}}
	\end{gather*}
	for $0 < s \leq r$, one infers from \ref{remark:trunc} and
	\ref{thm:sob_poin_summary}\,\eqref{item:sob_poin_summary:interior}
	with $U$, $M$, $G$, $f$, $Q$, and $r$ replaced by $\oball a{2s}$,
	$2\adim$, $\varnothing$, $g_i$, $1$, and $s$ that
	\begin{gather*}
		\eqLpnorm{\| V \| \restrict \cball{a}{s} }{\alpha}{g_i} \leq
		\Delta \mu_i ( \cball{a}{2s} )^{1/q} \quad \text{in case of
		\eqref{item:diff_lebesgue_spaces:m>1=p} or
		\eqref{item:diff_lebesgue_spaces:m=p=1} or
		\eqref{item:diff_lebesgue_spaces:q<m=p}}, \\
		\eqLpnorm{\| V \| \restrict \cball{a}{s} }{\alpha}{g_i} \leq
		\Delta s^{1-\vdim/q} \mu_i ( \cball{a}{2s} )^{1/q} \quad
		\text{in case of \eqref{item:diff_lebesgue_spaces:p=m<q}}
	\end{gather*}
	for $0 < s \leq r$, where $\Delta =
	\Gamma_{\ref{thm:sob_poin_summary}} ( 2 \adim )$ in case of
	\eqref{item:diff_lebesgue_spaces:m>1=p} or
	\eqref{item:diff_lebesgue_spaces:m=p=1}, $\Delta = (\vdim-q)^{-1}
	\Gamma_{\ref{thm:sob_poin_summary}} ( 2 \adim )$ in case of
	\eqref{item:diff_lebesgue_spaces:q<m=p}, and $\Delta =
	\Gamma_{\ref{thm:sob_poin_summary}} (2 \adim )^{1/(1/\vdim-1/q)}
	\unitmeasure{\vdim}^{1/\vdim-1/q}$ in case of
	\eqref{item:diff_lebesgue_spaces:p=m<q}. Consequently,
	\begin{gather*}
		\lim_{s \to 0+} s^{-1} \eqLpnorm { s^{-\vdim} \| V \|
		\restrict \cball{a}{s} }{\alpha} {g_i} = 0
	\end{gather*}
	and the assertion follows.

	From the assertion of the preceding paragraph one obtains
	\begin{gather*}
		\lim_{s \to 0+} \eqLpnorm{s^{-\vdim} \| V \| \restrict
		\cball{a}{s} }{\alpha}{ | \cdot - a |^{-1} h_a ( \cdot ) } = 0
		\quad \text{for $\| V \|$ almost all $a$};
	\end{gather*}
	in fact, if $\cball{a}{r} \subset U$ and $\kappa = \sup \{ s^{-1}
	\eqLpnorm{s^{-\vdim} \| V \| \restrict \cball{a}{s} }{\alpha}{h_a}
	\with 0 < s \leq r \}$, then one estimates
	\begin{gather*}
		\begin{aligned}
			& \eqLpnorm{s^{-\vdim} \| V \| \restrict \cball{a}{s}
			}{\alpha}{ |\cdot-a|^{-1} h_a} \\
			& \qquad \leq s^{-\vdim/\alpha} \tsum{i=1}{\infty}
			2^{i} s^{-1} \eqLpnorm{\| V \| \restrict (
			\cball{a}{2^{1-i} s} \without \cball{a}{2^{-i}s}
			)}{\alpha}{h_a} \\
			& \qquad \leq 2 \kappa \tsum{i=1}{\infty}
			2^{(1-i)\vdim/\alpha} = 2 \kappa / (1-2^{-\vdim/\alpha})
		\end{aligned}
	\end{gather*}
	in case of \eqref{item:diff_lebesgue_spaces:m>1=p} or
	\eqref{item:diff_lebesgue_spaces:q<m=p} and
	$\eqLpnorm{s^{-\vdim} \| V \| \restrict \cball{a}{s} }{\alpha}{
	|\cdot-a|^{-1} h_a} \leq \kappa$ in case of
	\eqref{item:diff_lebesgue_spaces:m=p=1} or
	\eqref{item:diff_lebesgue_spaces:p=m<q} for $0 < s \leq r$. This
	yields the conclusion in case of
	\eqref{item:diff_lebesgue_spaces:m>1=p} or
	\eqref{item:diff_lebesgue_spaces:q<m=p} and that $f|X$ is
	differentiable relative to $X$ at $a$ with
	\begin{gather*}
		D(f|X) (a) = \derivative Vf (a) | \Tan (X,a) \quad \text{for
		$\| V \|$ almost all $a$}
	\end{gather*}
	in case of \eqref{item:diff_lebesgue_spaces:m=p=1} or
	\eqref{item:diff_lebesgue_spaces:p=m<q} by \cite[3.1.22]{MR41:1976}.
	To complete the proof, note that $X$ is dense in $\spt \| V \|$ and
	hence if $p = \vdim$, then
	\begin{gather*}
		\Tan (X,a) = \Tan ( \spt \| V \|, a ) = \Tan^\vdim ( \| V \|,
		a ) \quad \text{for $a \in U$}
	\end{gather*}
	by \cite[3.1.21]{MR41:1976} and \ref{miniremark:tangent_spaces}.
\end{proof}
\begin{remark}
	The usage of $h_a$ in the last paragraph of the proof is adapted from
	the proof of \cite[4.5.9\,(26)\,(\printRoman{2})]{MR41:1976}.
\end{remark}
\section{Coarea formula} \label{sec:coarea}
In this section rectifiability properties of the distributional boundary of
almost all superlevel sets of real valued weakly differentiable functions are
established, see \ref{corollary:coarea}. The result rests on the approximate
differentiability of such functions, see \ref{thm:approx_diff}. To underline
this fact, it is derived as corollary to a general result for approximately
differentiable functions, see \ref{thm:ap_coarea}.
\begin{theorem} \label{thm:ap_coarea}
	Suppose $\vdim, \adim \in \nat$, $\vdim \leq \adim$, $U$ is an open
	subset of $\rel^\adim$, $V \in \RVar_\vdim ( U )$, $f$ is a $\| V \|$
	measurable real valued function which is $( \| V \|, \vdim )$
	approximately differentiable at $\| V \|$ almost all points, $F$ is a
	$\| V \|$ measurable $\Hom ( \rel^\adim, \rel )$ valued function with
	\begin{gather*}
		F(x) = ( \| V \|, \vdim ) \ap Df(x) \circ \project{\Tan^\vdim
		( \| V \|, x )} \quad \text{for $\| V \|$ almost all $x$},
	\end{gather*}
	and $\tint{K \cap \{ x \with |f(x)| \leq s \}}{} |F| \ud \| V \| <
	\infty$ whenever $K$ is a compact subset of $U$ and $0 \leq s <
	\infty$, and $T \in \mathscr{D}' ( U \times \rel, \rel^\adim )$
	and $S(y) : \mathscr{D} ( U, \rel^\adim ) \to \rel$ satisfy
	\begin{gather*}
		T ( \phi ) = \tint{}{} \left < \phi (x,f(x)), F(x) \right >
		\ud \| V \| x \quad \text{for $\phi \in \mathscr{D} ( U \times
		\rel, \rel^\adim )$}, \\
		S(y)( \theta ) = \lim_{\varepsilon \to 0+} \varepsilon^{-1}
		\tint{\{ x \with y < f(x) \leq y+\varepsilon \}}{} \left <
		\theta, F \right > \ud \| V \| \in \rel \quad \text{for
		$\theta \in \mathscr{D} (U, \rel^\adim )$}
	\end{gather*}
	whenever $y \in \rel$, that is $y \in \dmn S$ if and only if the limit
	exists and belongs to $\rel$ for $\theta \in \mathscr{D} ( U,
	\rel^\adim )$.

	Then the following two statements hold.
	\begin{enumerate}
		\item \label{item:ap_coarea:equation} If $\phi \in \Lp{1} ( \|
		T \|, \rel^\adim )$ and $g$ is an $\overline{\rel}$ valued $\|
		T \|$ integrable function, then
		\begin{gather*}
			T ( \phi )= \tint{}{} S(y) ( \phi (\cdot, y ) ) \ud
			\mathscr{L}^1 y, \quad
			\tint{}{} g \ud \| T \| = \tint{}{} \tint{}{} g(x,y)
			\ud \| S(y) \| x \ud \mathscr{L}^1 y.
		\end{gather*}
		\item \label{item:ap_coarea:structure} There exists an
		$\mathscr{L}^1$ measurable function $W$ with values in
		$\RVar_{\vdim-1} (U)$ endowed with the weak topology such that
		for $\mathscr{L}^1$ almost all $y$ there holds
		\begin{gather*}
			\Tan^{\vdim-1} ( \| W(y) \|, x ) = \Tan^\vdim ( \| V
			\|, x ) \cap \ker F(x) \in \grass{\adim}{\vdim-1}, \\
			\density^{\vdim-1} ( \| W(y) \|, x ) = \density^\vdim
			( \| V \|, x )
		\end{gather*}
		for $\| W(y) \|$ almost all $x$ and
		\begin{gather*}
			S(y)(\theta) = \tint{}{} \left < \theta, |F|^{-1} F
			\right > \ud \| W(y) \| \quad \text{for $\theta \in
			\mathscr{D} (U,\rel^\adim )$}.
		\end{gather*}
	\end{enumerate}
\end{theorem}
\begin{proof}
	First, notice that \ref{lemma:push_on_product} with $J = \rel$ implies
	that $T$ is representable by integration and
	\begin{gather*}
		T ( \phi ) = \tint{}{} \left < \phi(x,f(x)), F(x)
		\right > \ud \| V \| x, \quad \| T \| (g) = \tint{}{}
		g(x,f(x)) | F(x) | \ud \| V \| x
	\end{gather*}
	whenever $\phi \in \Lp{1} ( \| T \|, \rel^\adim )$ and $g$ is an
	$\overline{\rel}$ valued $\| T \|$ integrable function.

	Next, the following assertion will be shown. \emph{There exists an
	$\mathscr{L}^1$ measurable function $W$ with values in
	$\RVar_{\vdim-1} (U)$ such that for $\mathscr{L}^1$ almost all $y$
	there holds
	\begin{gather*}
		\Tan^{\vdim-1} ( \| W(y) \|, x ) = \Tan^\vdim ( \| V \|, x )
		\cap \ker F(x) \in \grass{\adim}{\vdim-1}, \\
		\density^{\vdim-1} ( \| W(y) \|, x ) = \density^\vdim ( \| V
		\|, x )
	\end{gather*}
	for $\| W(y) \|$ almost all $x$ and
	\begin{gather*}
		T( \phi ) = \tint{}{} \tint{}{} \left < \phi(x,f(x)),
		|F(x)|^{-1} F(x) \right > \ud \| W(y) \| x \ud \mathscr{L}^1
		y, \\
		\tint{}{} g \ud \| T \| = \tint{}{} \tint{}{} g(x,y) \ud \| W
		(y) \| x \ud \mathscr{L}^1 y
	\end{gather*}
	whenever $\phi \in \Lp{1} ( \| T \|, \rel^\adim )$ and $g$ is an
	$\overline{\rel}$ valued $\| T \|$ integrable function.} For this
	purpose choose a disjoint sequence of Borel subsets $B_i$ of $U$,
	sequences $M_i$ of $\vdim$ dimensional submanifolds of $\rel^\adim$ of
	class $1$ and functions $f_i : M_i \to \rel$ of class $1$ satisfying
	$\| V \| \big ( U \without \bigcup_{i=1}^\infty B_i \big ) = 0$ and
	\begin{gather*}
		B_i \subset M_i, \quad f_i (x) = f(x), \quad Df_i (x) = ( \| V
		\|, \vdim ) \ap Df(x) = F(x) | \Tan^\vdim ( \| V \|, x )
	\end{gather*}
	whenever $i \in \nat$ and $x \in B_i$, see \ref{lemma:approx_diff},
	\cite[2.8.18, 2.9.11, 3.2.17, 3.2.29]{MR41:1976} and Allard
	\cite[3.5\,(2)]{MR0307015}. Let $B = \bigcup_{i=1}^\infty B_i$. If $y
	\in \rel$ satisfies
	\begin{gather*}
		\mathscr{H}^{\vdim-1} ( B  \cap \{ x \with \text{$f(x) = y$
		and $( \| V \|, \vdim ) \ap Df (x)=0$} \}) = 0, \\
		\tint{B \cap K \cap \{ x \with f(x) = y \}}{} \density^\vdim (
		\| V \|, x ) \ud \mathscr{H}^{\vdim-1} x< \infty
	\end{gather*}
	whenever $K$ is a compact subset of $U$, then define $W(y) \in
	\RVar_{\vdim-1} ( U)$ by
	\begin{gather*}
		W(y)(k) = \tint{B \cap \{ x \with f(x) = y \}}{} k (x, \ker \,
		( \| V \|, \vdim ) \ap Df(x) ) \density^{\vdim-1} ( \| V \|, x
		) \ud \mathscr{H}^{\vdim-1} x
	\end{gather*}
	for $k \in \mathscr{K} ( U \times \grass{\adim}{\vdim-1} )$. Applying 
	\cite[3.2.22]{MR41:1976} with $f$ replaced by $f_i$ and summing over
	$i$, one infers that
	\begin{gather*}
		\tint{A \cap \{ x \with s < f(x) < t \}}{} |F| \ud \| V \| =
		\tint{s}{t} \tint{A \cap B \cap \{ x \with f(x) = y \}}{}
		\density^\vdim ( \| V \|, x ) \ud \mathscr{H}^{\vdim-1} x \ud
		\mathscr{L}^1 y
	\end{gather*}
	whenever $A$ is $\| V \|$ measurable and $- \infty < s < t < \infty$,
	hence
	\begin{gather*}
		\tint{}{} g(x,f(x)) | F(x) | \ud \| V \| x = \tint{}{}
		\tint{}{} g (x,y) \ud \| W(y) \| x \ud \mathscr{L}^1 y
	\end{gather*}
	whenever $g$ is an $\overline{\rel}$ valued $\| T \|$ integrable
	function. The remaining parts of the assertion now follow by
	considering appropriate choices of $g$ and recalling
	\ref{example:kx_lusin}.
	
	Consequently, \eqref{item:ap_coarea:equation} is implied by
	\ref{thm:distribution_on_product}\,\eqref{item:distribution_on_product:absolute}
	with $J = \rel$ and $Z = \rel^\adim$. Noting
	\begin{gather*}
		S(y)(\theta) = \lim_{\varepsilon \to 0+} \varepsilon^{-1}
		\tint{y}{y+\varepsilon} \tint{}{} \left < \theta, |F|^{-1} F
		\right > \ud \| W(\upsilon) \| \ud \mathscr{L}^1 \upsilon
		\quad \text{for $\theta \in \mathscr{D} (U, \rel^\adim )$}
	\end{gather*}
	whenever $y \in \dmn S$, \eqref{item:ap_coarea:structure} follows
	using \ref{remark:lusin}, \ref{example:distrib_lusin}, and
	\cite[2.8.17, 2.9.8]{MR41:1976}.
\end{proof}
\begin{corollary} \label{corollary:coarea}
	Suppose $\vdim$, $\adim$, $U$, and $V$ are as in
	\ref{miniremark:situation_general}, $f \in \trunc (V)$, and $E(y) = \{
	x \with f(x) > y \}$ for $y \in \rel$.

	Then there exists an $\mathscr{L}^1$ measurable function $W$ with
	values in $\RVar_{\vdim-1} (U)$ endowed with the weak topology such
	that for $\mathscr{L}^1$ almost all $y$ there holds
	\begin{gather*}
		\Tan^{\vdim-1} ( \| W(y) \|, x ) = \Tan^\vdim ( \| V \|, x )
		\cap \ker \derivative{V}{f} (x) \in \grass{\adim}{\vdim-1}, \\
		\density^{\vdim-1} ( \| W(y) \|, x ) = \density^\vdim ( \| V
		\|, x )
	\end{gather*}
	for $\| W(y) \|$ almost all $x$ and
	\begin{gather*}
		\boundary{V}{E(y)} (\theta) = \tint{}{} \left < \theta,
		|\derivative{V}{f}|^{-1} \derivative{V}{f} \right > \ud \|
		W(y) \| \quad \text{for $\theta \in \mathscr{D} (U,\rel^\adim
		)$}.
	\end{gather*}
\end{corollary}
\begin{proof}
	In view of \ref{lemma:level_sets} and \ref{thm:approx_diff}, this is a
	consequence of \ref{thm:ap_coarea}.
\end{proof}
\begin{remark} \label{remark:no_rectifiable_structure}
	The formulation of \ref{corollary:coarea} is modelled on similar
	results for sets of locally finite perimeter, see \cite[4.5.6\,(1),
	4.5.9\,(12)]{MR41:1976}. Observe however that the structural
	description of $\boundary{V}{E(y)}$ given here for $\mathscr{L}^1$
	almost all $y$ does not extend to arbitrary $\| V \| + \| \delta V \|$
	measurable sets $E$ such that $\boundary{V}{E}$ is representable by
	integration; in fact, $\| \boundary{V}{E} \|$ does not even need to be
	the weight measure of some member of $\RVar_{\vdim-1} ( U )$. (Using
	\cite[2.10.28]{MR41:1976} to construct $V \in \IVar_1 ( \rel^2 )$ such
	that $\| \delta V \|$ is a nonzero Radon measure with $\| \delta V \|
	( \{ x \} ) = 0$ for $x \in \rel^2$ and $\| V \| ( \spt \delta V) =
	0$, one may take $E = \spt \delta V$.)
\end{remark}
\section{Oscillation estimates} \label{sec:oscillation}
In this section two situations are studied where the oscillation of a
generalised weakly differentiable function may be controlled by its weak
derivative assuming a suitable summability of the mean curvature of the
underlying varifold. In general, such control is necessarily rather weak, see
\ref{thm:mod_continuity} and \ref{remark:mod_continuity}, but under special
circumstances one may obtain H\"older continuity, see
\ref{thm:hoelder_continuity}. The main ingredients are the analysis of the
connectedness structure of a varifold, see \ref{corollary:conn_structure}, and
the Sobolev Poincar{\'e} inequalities with several medians, see
\ref{thm:sob_poin_several_med}.
\begin{theorem} \label{thm:mod_continuity}
	Suppose $\vdim$, $\adim$, $p$, $U$, and $V$ are as in
	\ref{miniremark:situation_general}, $p = \vdim$, $K$ is a compact
	subset of $U$, $0 < \varepsilon \leq \dist ( K, \rel^\adim \without U
	)$, $\varepsilon < \infty$, and either
	\begin{enumerate}
		\item \label{item:mod_continuity:m=1} $1 = \vdim = q$ and
		$\lambda = 2^{\vdim+5} \unitmeasure{\vdim}^{-1}
		\Gamma_{\ref{thm:sob_poin_summary}} ( 2 \adim ) \sup \{
		\density^1 ( \| V \|, a ) \with a \in K \}$, or
		\item \label{item:mod_continuity:m>1} $1 < \vdim < q$ and
		$\lambda = \varepsilon$.
	\end{enumerate}

	Then there exists a positive, finite number $\Gamma$ with the
	following property.

	If $Y$ is a finite dimensional normed vectorspace, $f : \spt \| V \|
	\to Y$ is a continuous function, $f \in \trunc (V,Y)$, and
	$\kappa = \sup \{ \eqLpnorm{\| V \| \restrict
	\oball{a}{\varepsilon}}{q}{ \derivative{V}{f} } \with a \in K \}$,
	then
	\begin{gather*}
		|f(x)-f(\chi)| \leq \lambda \kappa \quad \text{whenever
		$x,\chi \in K \cap \spt \| V \|$ and $|x-\chi| \leq
		\Gamma^{-1}$}.
	\end{gather*}
\end{theorem}
\begin{proof}
	Let $\psi$ be as in \ref{miniremark:situation_general}.  Abbreviate $A
	= \spt \| V \|$ and $\delta = 1-\vdim/q$. Abbreviate
	\begin{gather*}
		\Delta_1 = 2^{\vdim+3} \unitmeasure{\vdim}^{-1}, \quad
		\Delta_2 = \Gamma_{\ref{thm:sob_poin_summary}} ( 2 \adim ).
	\end{gather*}
	Notice that $\density_\ast^\vdim ( \| V \|, a ) \geq 1/2$ for $a \in
	A$ by \ref{corollary:density_1d}\,\eqref{item:density_1d:lower_bound}
	and \ref{remark:kuwert_schaetzle}. If $\vdim >1$, then
	\begin{gather*}
		\dmn \density^1 ( \| V \|, \cdot ) = U, \quad \Delta_3 = \sup
		\{ \density^1 ( \| V \|, a ) \with a \in K \} < \infty
	\end{gather*}
	by
	\ref{corollary:density_1d}\,\eqref{item:density_1d:upper_bound}\,\eqref{item:density_1d:real}.
	If $\vdim > 1$, then $\density^{\vdim-\delta} (
	\| V \|, a ) = 0$ for $a \in A$ by
	\ref{corollary:density_ratio_estimate}. Moreover, if $a \in A$, $0 < r
	< \infty$, and $\oball{a}{r} \subset U$ then, by
	\ref{corollary:conn_structure}\,\eqref{item:conn_structure:open},
	there exists $0 < s \leq r$ such that $A \cap \oball{a}{s}$ is a
	subset of the connected component of $A \cap \oball{a}{r/2}$ which
	contains $a$. Since $K \cap A$ is compact, one may therefore construct
	$j \in \nat$ and $a_i$, $s_i$ and $r_i$ for $i = 1, \ldots, j$ such
	that $K \cap A \subset \bigcup_{i=1}^j \oball{a_i}{s_i}$ and
	\begin{gather*}
		a_i \in K \cap A, \quad 0 < s_i \leq r_i \leq \varepsilon,
		\quad A \cap \oball{a_i}{s_i} \subset C_i, \quad
		\oball{a_i}{r_i} \subset U, \\
		\psi ( \oball{a_i}{r_i} \without \{ a_i \} ) \leq
		\Delta_2^{-1}, \quad \measureball{\| V \|}{ \oball{a_i}{r_i} }
		\geq (1/2) \unitmeasure{\vdim} (r_i/2)^\vdim, \\
		r_i^{-1} \measureball{\| V \|}{\cball{a_i}{r_i}} \leq 4
		\Delta_3 \quad \text{if $\vdim = 1$}, \\
		\Delta_1
		\Gamma_{\ref{thm:sob_poin_several_med}\,\eqref{item:sob_poin_several_med:p=m<q}}
		( 2\adim )^{\vdim/\delta} r_i^{\delta-\vdim}
		\measureball{\| V \|}{ \oball{a_i}{r_i} } \leq \varepsilon
		\quad \text{if $\vdim > 1$}
	\end{gather*}
	for $i = 1, \ldots, j$, where $C_i$ is the connected component of $A
	\cap \oball{a_i}{r_i/2}$ which contains $a_i$. Since $K \cap A$ is
	compact, there exists a positive, finite number $\Gamma$ with the
	following property. If $x,\chi \in K \cap A$ and $|x-\chi| \leq \Gamma^{-1}$
	then $\{ x,\chi \} \subset \oball{a_i}{s_i}$ for some $i$.

	In order to verify that $\Gamma$ has the asserted property, suppose
	that $Y$, $f$, $\kappa$, $x$, and $\chi$ are related to $\Gamma$ as in
	the body of the theorem.

	Assume $\kappa < \infty$. Since $|y| = \sup \{ \alpha ( y ) \with
	\alpha \in \Hom ( Y, \rel ), \| \alpha \| \leq 1 \}$ for $y \in Y$ by
	\cite[2.4.12]{MR41:1976}, one may also assume $Y = \rel$ by
	\ref{lemma:integration_by_parts}. Choose $i$ such that $\{ x,\chi \}
	\subset C_i$. Choose $N \in \nat$ satisfying
	\begin{gather*}
		(1/2) N \unitmeasure{\vdim} (r_i/2)^\vdim \leq \measureball{\|
		V \|}{ \oball{a_i}{r_i} } \leq (1/2) (N+1) \unitmeasure{\vdim}
		(r_i/2)^\vdim.
	\end{gather*}
	Applying
	\ref{thm:sob_poin_several_med}\,\eqref{item:sob_poin_several_med:p=m=1}\,\eqref{item:sob_poin_several_med:p=m<q}
	with $U$, $M$, $Q$, $r$, and $X$ replaced by $\oball{a_i}{r_i}$,
	$2 \adim$, $1$, $r_i/2$, and $\{ a_i \}$ yields a subset $\Upsilon$ of $\rel$
	such that $\card \Upsilon \leq N +1$ and
	\begin{gather*}
		{\textstyle f \lIm \oball{a_i}{r_i/2} \rIm \subset \bigcup \{
		\cball{y}{\kappa_i} \with y \in \Upsilon \}},
	\end{gather*}
	where $\kappa_i = \Delta_2 \kappa$ if $\vdim = 1$
	and $\kappa_i =
	\Gamma_{\ref{thm:sob_poin_several_med}\,\eqref{item:sob_poin_several_med:p=m<q}}
	( 2 \adim )^{\vdim/\delta} r_i^\delta \kappa$ if $\vdim > 1$. Defining
	$I$ to be the connected component of $\bigcup \{ \cball{y}{\kappa_i}
	\with y \in \Upsilon \}$ which contains $f(a)$, one infers
	\begin{gather*}
		|f(x)-f(\chi)| \leq \diam I = \mathscr{L}^1 (I) \leq 2 (N+1)
		\kappa_i \leq \Delta_1 \kappa_i r_i^{-\vdim} \measureball{\| V
		\|}{ \oball{a_i}{r_i} } \leq \lambda \kappa
	\end{gather*}
	since $\{ f(x), f(\chi) \} \subset f \lIm C_i \rIm \subset I$ and $I$
	is an interval.
\end{proof}
\begin{remark} \label{remark:mod_continuity}
	Considering varifolds corresponding to two parallel planes, it is
	clear that $\Gamma$ may not be chosen independently of $V$.
	Also the continuity hypothesis on $f$ is essential as may be seen
	considering $V$ associated to two transversely intersecting lines.
\end{remark}
\begin{theorem} \label{thm:hoelder_continuity}
	Suppose $1 \leq M < \infty$.

	Then there exists a positive, finite number $\Gamma$ with the
	following property.

	If $\vdim$, $\adim$, $p$, $U$, $V$, and $\psi$ are as in
	\ref{miniremark:situation_general}, $p = \vdim$, $\adim \leq M$,
	$0 < r < \infty$, $A = \{
	x \with \oball xr \subset U \}$,
	\begin{gather*}
		\measureball{\psi}{ \oball{a}{r} } \leq \Gamma^{-1}, \qquad
		\measureball{\| V \|}{ \oball as } \leq (2 - M^{-1} )
		\unitmeasure{\vdim} s^\vdim \quad \text{for $0 < s \leq r$}
	\end{gather*}
	whenever $a \in A \cap \spt \| V \|$, $Y$ is a finite dimensional
	normed vectorspace, $f \in \trunc (V,Y)$, $C = \{ (x,\cball xs ) \with
	\cball xs \subset U \}$, $X$ is the set of $a \in A \cap \spt \| V \|$
	such that $f$ is $(\| V \|,C)$ approximately continuous at $a$, $\vdim
	< q \leq \infty$, $\delta = 1-\vdim/q$, and $\kappa = \sup \{
	\eqLpnorm{ \| V \| \restrict \oball ar}{q}{ \derivative Vf } \with a
	\in A \cap \spt \| V \| \}$, then
	\begin{gather*}
		| f(x)-f(\chi) | \leq \lambda |x-\chi|^\delta \kappa \quad
		\text{whenever $x,\chi \in X$ and $|x-\chi| \leq r/\Gamma$},
	\end{gather*}
	where $\lambda = \Gamma$ if $\vdim = 1$ and $\lambda =
	\Gamma^{1/\delta}$ if $\vdim > 1$.
\end{theorem}
\begin{proof}
	Define
	\begin{gather*}
		\Delta_1 = (1-1/(4M-1))^{1/M}, \quad \Delta_2 = \sup \{ 1,
		\Gamma_{\ref{thm:sob_poin_several_med}\,\eqref{item:sob_poin_several_med:p=m<q}}
		(2M) \}^M, \\
		\Gamma = 4 (1-\Delta_1)^{-1} \sup \{ 4
		\Gamma_{\ref{thm:sob_poin_summary}} (2M), \Delta_2 \}
	\end{gather*}
	and notice that $\Gamma \geq 4 (1-\Delta_1)^{-1}$.

	In order to verify that $\Gamma$ has the asserted property, suppose
	that $\vdim$, $\adim$, $p$, $U$, $V$, $\psi$, $M$, $r$, $A$, $Y$, $f$,
	$C$, $X$, $q$, $\delta$, $\kappa$, $x$, $\chi$, and $\lambda$ are
	related to $\Gamma$ as in the body of the theorem.

	In view of \ref{lemma:integration_by_parts} and
	\ref{corollary:integrability}\,\eqref{item:integrability:m=1}\,\eqref{item:integrability:m=p<q}
	one may assume $Y = \rel$. Define $s = 2 |x-\chi|$ and $t =
	(1-\Delta_1)^{-1} s$ and notice that $0 < s < t \leq r$ and
	\begin{gather*}
		\begin{aligned}
			& \measureball{\| V \|}{ \oball{x}{t}} \leq (2-M^{-1})
			\unitmeasure{\vdim} (1-\Delta_1)^{-\vdim} s^\vdim \\
			& \qquad \leq \big (2- (2M)^{-1} \big )
			\unitmeasure{\vdim} \Delta_1^\vdim
			(1-\Delta_1)^{-\vdim} s^\vdim = \big (2-(2M)^{-1} \big
			) \unitmeasure{\vdim} (t-s)^\vdim,
		\end{aligned} \\
		\measureball{\psi}{ \oball xt } \leq \Gamma^{-1} \leq
		\Gamma_{\ref{thm:sob_poin_summary}} ( 2M )^{-1}.
	\end{gather*}
	Therefore applying
	\ref{thm:sob_poin_several_med}\,\eqref{item:sob_poin_several_med:p=m=1}\,\eqref{item:sob_poin_several_med:p=m<q}
	with $U$, $M$, $N$, $Q$, $r$, and $X$ replaced by $\oball{x}{t}$,
	$2M$, $1$, $1$, $t-s$, and $\varnothing$ yields a subset $\Upsilon$ of $\rel$
	with $\card \Upsilon = 1$ such that $\eqLpnorm{\| V \| \restrict
	\oball{x}{s}}{\infty}{f_\Upsilon}$ is bounded by, using H{\"o}lder's
	inequality,
	\begin{gather*}
		\Gamma_{\ref{thm:sob_poin_summary}} ( 2M ) \eqLpnorm{\| V \|
		\restrict \oball xt}{1}{ \derivative{V}{f} } \leq 4
		\Gamma_{\ref{thm:sob_poin_summary}} ( 2M ) t^\delta \kappa
		\quad \text{if $\vdim = 1$, and by} \\
		\Gamma_{\ref{thm:sob_poin_several_med}\,\eqref{item:sob_poin_several_med:p=m<q}}
		( 2M )^{\vdim/\delta} t^\delta \eqLpnorm{\| V \| \restrict
		\oball xt}{q}{ \derivative Vf } \leq \Delta_2^{1/\delta}
		t^\delta \kappa \quad \text{if $\vdim > 1$}.
	\end{gather*}
	Noting $|f(x)-f(\chi)| \leq 2 \eqLpnorm{\| V \| \restrict \oball xs
	}{\infty}{f_\Upsilon}$ and $t = 2 (1-\Delta_1)^{-1} |x-\chi|$, the conclusion
	follows.
\end{proof}
\section{Geodesic distance} \label{sec:geodesic_distance}
Reconsidering the support of the weight measure of a varifold whose mean
curvature satisfies a suitable summability of the mean curvature, the
oscillation estimate \ref{thm:mod_continuity} is used to establish in this
section that its connected components agree with the components induced by its
geodesic distance, see \ref{thm:conn_path_finite_length}.
\begin{example} \label{example:space_fill}
	Whenever $1 \leq p < \vdim < \adim$, $U$ is an open subset of
	$\rel^\adim$ and $X$ is an open subset of $U$, there exists $V$
	related to $\vdim$, $\adim$, $U$, and $p$ as in
	\ref{miniremark:situation_general} such that $\spt \| V \|$ equals the
	closure of $X$ relative to $U$ as is readily seen taking the into
	account the behaviour of $\psi$ and $V$ under homotheties; compare
	\cite[1.2]{snulmenn.isoperimetric}.
\end{example}
\begin{theorem} \label{thm:conn_path_finite_length}
	Suppose $\vdim$, $\adim$, $U$, $V$, and $p$ are as in
	\ref{miniremark:situation_general}, $p = \vdim$, and $C$ is a
	connected component of $\spt \| V \|$, and $a,x \in C$.

	Then there exist $- \infty < b \leq y < \infty$ and a Lipschitzian
	function $g : \{ \upsilon \with b \leq \upsilon \leq y \} \to \spt \|
	V \|$ such that $g(b) = a$ and $g(y) = x$.
\end{theorem}
\begin{proof}
	In view of
	\ref{corollary:conn_structure}\,\eqref{item:conn_structure:piece}
	one may assume $C = \spt \| V \|$. Whenever $c \in C$ denote by $X(c)$
	the set of $\chi \in C$ such that there exist $- \infty < b \leq y <
	\infty$ and a Lipschitzian function $g : \{ \upsilon \with b \leq
	\upsilon \leq y \} \to C$ such that $g(b) = c$ and $g(y) = \chi$.
	Observe that it is sufficient to prove that $c$ belongs to the
	interior of $X(c)$ relative to $C$ whenever $c \in C$.
	
	For this purpose suppose $c \in C$, choose $0 < r < \infty$ with
	$\oball{c}{2r} \subset U$ and let $K = \cball{c}{r}$. If $\vdim = 1$,
	define $\lambda$ as in
	\ref{thm:mod_continuity}\,\eqref{item:mod_continuity:m=1} and notice
	that $\lambda < \infty$ by
	\ref{corollary:density_1d}\,\eqref{item:density_1d:upper_bound}.
	Define $q = 1$ if $\vdim = 1$ and $q = 2 \vdim$ if $\vdim > 1$. Choose
	$0 < \varepsilon \leq r$ such that
	\begin{gather*}
		\lambda \sup \{ \| V \| ( \oball{\chi}{\varepsilon} )^{1/q}
		\with \chi \in K \} \leq r,
	\end{gather*}
	where $\lambda = \varepsilon$ in accordance with
	\ref{thm:mod_continuity}\,\eqref{item:mod_continuity:m>1} if $\vdim >
	1$, and define
	\begin{gather*}
		s = \inf \{ \Gamma_{\ref{thm:mod_continuity}} ( \vdim, \adim,
		\vdim, U, V, K, \varepsilon, q )^{-1}, r \}.
	\end{gather*}

	Next, define functions $f_\delta : C \to \overline{\rel}$ by letting
	$f_\delta ( \chi )$ for $\chi \in C$ and $\delta > 0$ equal the
	infimum of sums
	\begin{gather*}
		\sum_{i=1}^j |x_i-x_{i-1}|
	\end{gather*}
	corresponding to all finite sequences $x_0, x_1, \ldots, x_j$ in $C$
	with $x_0 = c$, $x_j = \chi$ and $| x_i-x_{i-1} | \leq \delta$ for $i
	= 1, \ldots, j$ and $j \in \nat$. Notice that
	\begin{gather*}
		f_\delta (\chi) \leq f_\delta (\zeta) + |\zeta-\chi| \quad
		\text{whenever $\zeta,\chi \in C$ and $|\zeta-\chi| \leq
		\delta$}.
	\end{gather*}
	Since $C$ is connected and $f_\delta (c) = 0$, it follows that
	$f_\delta$ is a locally Lipschitzian real valued function
	satisfying
	\begin{gather*}
		\Lip ( f_\delta | A ) \leq 1 \quad \text{whenever $A
		\subset U$ and $\diam A \leq \delta$},
	\end{gather*}
	in particular $f_\delta \in \trunc (V)$ and $\Lpnorm{\| V
	\|}{\infty}{\derivative{V}{f_\delta}} \leq 1$ by
	\ref{example:lipschitzian}. One infers
	\begin{gather*}
		f_\delta (\chi) \leq r \quad \text{whenever $\chi \in C
		\cap \cball{c}{s}$ and $\delta > 0$}
	\end{gather*}
	from \ref{thm:mod_continuity}. It follows that $C \cap \cball{c}{s}
	\subset X(c)$; in fact, if $\chi \in C \cap \cball{c}{s}$  one readily
	constructs $g_\delta : \{ u \with 0 \leq \upsilon \leq r \} \to
	\rel^\adim$ satisfying
	\begin{gather*}
		g_\delta (0)=c, \quad g_\delta (r) = \chi, \quad \Lip
		g_\delta \leq 1+\delta, \\
		\dist ( g_\delta (\upsilon), C ) \leq \delta \quad
		\text{whenever $0 \leq \upsilon \leq r$},
	\end{gather*}
	hence, noting that $\im g_\delta \subset \cball{c}{(1+\delta)r}$, the
	existence of $g : \{ \upsilon \with 0 \leq \upsilon \leq r \} \to C$
	satisfying $g (0)=c$, $g (r) = \chi$, and $\Lip g \leq 1$ now is a
	consequence of \cite[2.10.21]{MR41:1976}.
\end{proof}
\begin{remark} \label{remark:metric_spaces}
	The deduction of \ref{thm:conn_path_finite_length} from
	\ref{thm:mod_continuity} is adapted from Cheeger \cite[\S
	17]{MR1708448} who attributes the argument to Semmes, see also David
	and Semmes \cite{MR1132876}.
\end{remark}
\begin{remark} \label{remark:topping}
	In view of \ref{example:space_fill} it is not hard to construct
	examples showing that the hypothesis ``$p = \vdim$'' in
	\ref{thm:mod_continuity} may not be replaced by ``$p \geq q$'' for any
	$1 \leq q < \vdim$ if $\vdim < \adim$. Yet, for indecomposable $V$,
	the study of possible extensions of \ref{thm:conn_path_finite_length}
	as well as related questions seems to be most natural under the
	hypothesis ``$p = \vdim-1$'', see Topping \cite{MR2410779}.
\end{remark}
\section{Curvature varifolds} \label{sec:curvature_varifolds}
In this section Hutchinson's concept of curvature varifold is rephrased in
terms of the concept of weakly differentiable function proposed in the present
paper, see \ref{thm:curvature_varifolds}. To indicate possible benefits of
this perspective, a result on the differentiability of the tangent plane map
is included, see \ref{corollary:curv_var_diff}.
\begin{miniremark} \label{miniremark:trace}
	Suppose $\adim \in \nat$ and $Y = \Hom ( \rel^\adim, \rel^\adim ) \cap
	\{ \sigma \with \sigma = \sigma^\ast \}$. Then $T : \Hom ( \rel^\adim,
	Y ) \to \rel^\adim$ denotes the linear map which is given by the
	composition (see \cite[1.1.1, 1.1.2, 1.1.4, 1.7.9]{MR41:1976})
	\begin{align*}
		\Hom ( \rel^\adim, Y ) & \xrightarrow{\subset}
		\Hom ( \rel^\adim, \Hom (\rel^\adim, \rel^\adim)) \\
		& \simeq \Hom ( \rel^\adim, \rel ) \otimes \Hom ( \rel^\adim,
		\rel ) \otimes \rel^\adim \xrightarrow{L \otimes
		\id{\rel^\adim}} \rel \otimes \rel^\adim \simeq \rel^\adim,
	\end{align*}
	where the linear map $L$ is induced by the inner product on $\Hom (
	\rel^\adim, \rel )$, hence $T ( g ) = \sum_{i=1}^\adim \left < u_i,
	g(u_i) \right >$ whenever $g \in \Hom (\rel^\adim, Y )$ and $u_1,
	\ldots, u_\adim$ form an orthonormal basis of $\rel^\adim$.
\end{miniremark}
\begin{miniremark} \label{miniremark:mean_curv}
	Suppose $\adim$, $Y$ and $T$ are as in \ref{miniremark:trace}, $\adim
	\geq \vdim \in \nat$, $M$ is an $\vdim$ dimensional submanifold of
	$\rel^\adim$ of class $2$, and $\tau : M \to Y$ is defined by $\tau(x) =
	\project{\Tan ( M,x)}$ for $x \in M$. Then one computes
	\begin{gather*}
		\mathbf{h} ( M,x ) = T ( D\tau(x) \circ \tau(x) ) \quad
		\text{whenever $x \in M$};
	\end{gather*}
	in fact, differentiating the equation $\tau(x) \circ \tau(x) =
	\tau(x)$ for $x \in M$, one obtains
	\begin{gather*}
		\tau(x) \circ \left < u, D\tau(x) \right > \circ \tau(x) = 0
		\quad \text{for $u \in \Tan(M,x)$}, \\
		T(D\tau(x) \circ \tau(x)) \in \Nor (M,x)
	\end{gather*}
	for $x \in M$ and, denoting by $u_1, \ldots, u_\adim$ an orthonormal
	base of $\rel^\adim$, one computes
	\begin{align*}
		& \tau(x) \bullet ( Dg(x) \circ \tau(x) ) = \tsum{i=1}{\adim}
		\left < \tau(x)(u_i), Dg(x) \right > \bullet \tau(x)(u_i) \\
		& \qquad = - g(x) \bullet \tsum{i=1}{\adim} \left <
		\tau(x)(u_i), D\tau(x) \right > (u_i) = - g(x) \bullet T (
		D\tau(x) \circ \tau(x))
	\end{align*}
	whenever $g : M \to \rel^\adim$ is of class $1$ and $g(x) \in \Nor
	(M,x)$ for $x \in M$.
\end{miniremark}
\begin{lemma} \label{lemma:gen_sym_endo}
	Suppose $\adim$ and $Y$ are as in \ref{miniremark:trace} and $\adim >
	\vdim \in \nat$.

	Then the vectorspace $Y$ is generated by $\{\project{P} \with P \in
	\grass{\adim}{\vdim} \}$.
\end{lemma}
\begin{proof}
	If $u_1, \ldots, u_{\vdim+1}$ are orthonormal vectors in $\rel^\adim$,
	$u_{\vdim+1} \in L \in \grass{\adim}{1}$, and
	\begin{gather*}
		P_i = \Span \big \{ u_j \with \text{$j \in \{ 1,\ldots,\vdim
		+1 \}$ and $j \neq i$} \big \} \in \grass{\adim}{\vdim}
	\end{gather*}
	for $i=1, \ldots, \vdim+1$, then $\project{L} +
	\eqproject{P_{\vdim+1}} = \frac{1}{\vdim} \sum_{i=1}^{\vdim+1}
	\eqproject{P_i}$. Hence one may assume $\vdim=1$ in which case
	\cite[1.7.3]{MR41:1976} yields the assertion.
\end{proof}
\begin{definition} \label{def:curvature_varifold}
	Suppose $\adim$, $Y$ and $T$ are as in \ref{miniremark:trace}, $\adim
	\geq \vdim \in \nat$, $U$ is an open subset of $\rel^\adim$, $V \in
	\IVar_\vdim (U)$, $X = U \cap \{ x \with \Tan^\vdim ( \| V \|, x ) \in
	\grass{\adim}{\vdim} \}$, and $\tau : X \to Y$ satisfies
	\begin{gather*}
		\tau(x) = \project{\Tan^\vdim ( \| V \|, x )} \quad
		\text{whenever $x \in X$}.
	\end{gather*}

	Then the varifold $V$ is called a \emph{curvature varifold} if and
	only if there exists $F \in \Lploc{1} \big ( \| V \|, \Hom (
	\rel^\adim, Y) \big )$ satisfying
	\begin{gather*}
		\tint{}{} \left < (\tau(x)(u), F(x)(u)), D \phi (x,\tau(x))
		\right > + \phi (x,\tau(x)) T(F(x)) \bullet u \ud \| V \| x =
		0
	\end{gather*}
	whenever $u \in \rel^\adim$ and $\phi \in \mathscr{D} ( U \times Y,
	\rel)$.
\end{definition}
\begin{remark} \label{remark:hutchinson_reformulations}
	Using approximation and the fact that $\tau$ is a bounded function,
	one may verify that the same definition results if the equation is
	required for every $\phi : U \times Y \to \rel$ of class $1$ such that
	\begin{gather*}
		\Clos ( U \cap \{ (x,\sigma) \with \text{$\phi(x,\sigma) \neq
		0$ for some $\sigma \in Y$} \} )
	\end{gather*}
	is compact. Extending $T$ accordingly, one may also replace $Y$ by
	$\Hom ( \rel^\adim, \rel^\adim )$ in the definition; in fact, if $F
	\in \Lploc{1} \big ( \| V \|, \Hom ( \rel^\adim, \Hom ( \rel^\adim,
	\rel^\adim) ) \big )$ satisfies the condition of the modified
	definition then $\im F(x) \subset Y$ for $\| V \|$ almost all $x$ as
	may be seen by enlarging the class of $\phi$ as before and considering
	$\phi ( x, \sigma ) = \zeta ( x ) \alpha ( \sigma )$ for $x \in U$,
	$\sigma \in \Hom ( \rel^\adim, \rel^\adim )$, $\zeta \in \mathscr{D} (
	U, \rel)$ and $\alpha : \Hom ( \rel^\adim, \rel^\adim ) \to \rel$ is
	linear with $Y \subset \ker \alpha$, compare Hutchinson
	\cite[5.2.4\,(i)]{MR825628}.

	Consequently, the definition \ref{def:curvature_varifold} is
	equivalent to Hutchinson's definition in \cite[5.2.1]{MR825628}. The
	present formulation is motivated by \ref{lemma:gen_sym_endo}.
\end{remark}
\begin{theorem} \label{thm:curvature_varifolds}
	Suppose $\adim$, $Y$ and $T$ are as in \ref{miniremark:trace},
	$\adim \geq \vdim \in \nat$, $U$ is an open subset of $\rel^\adim$, $V
	\in \IVar_\vdim (U)$, $X = U \cap \{ x \with \Tan^\vdim ( \| V \|, x )
	\in \grass{\adim}{\vdim} \}$, and $\tau : X \to Y$ satisfies
	\begin{gather*}
		\tau(x) = \project{\Tan^\vdim ( \| V \|, x )} \quad
		\text{whenever $x \in X$}.
	\end{gather*}

	Then $V$ is a curvature varifold if and only if $\| \delta V \|$ is a
	Radon measure absolutely continuous with respect to $\| V \|$ and $\tau
	\in \trunc(V, Y )$. In this case, there holds
	\begin{gather*}
		\mathbf{h} ( V,x ) = T ( \derivative{V}{\tau} (x) ) \quad
		\text{for $\| V \|$ almost all $x$}, \\
		( \delta V )_x ( \phi (x,\tau(x)) u ) = \tint{}{} \left <
		( \tau(x)(u), \derivative{V}{\tau}(x)(u) ), D \phi (x,\tau(x))
		\right > \ud \| V \| x
	\end{gather*}
	whenever $\phi : U \times Y \to \rel$ is of class $1$ such that
	\begin{gather*}
		\Clos ( U \cap \{ (x,\sigma) \with \text{$\phi(x,\sigma) \neq 0$
		for some $\sigma \in Y$} \} )
	\end{gather*}
	is compact and $u \in \rel^\adim$.
\end{theorem}
\begin{proof}
	Suppose $V$ is a curvature varifold and $F$ is as in
	\ref{def:curvature_varifold}. If $\zeta \in \mathscr{D} ( U, \rel)$,
	$\gamma \in \mathscr{E} ( Y, \rel)$ and $u \in \rel^\adim$, one may take
	$\phi (x,\sigma) = \zeta(x) \gamma (\sigma)$ in
	\ref{def:curvature_varifold}, \ref{remark:hutchinson_reformulations}
	to obtain
	\begin{multline*}
		\tint{}{} \gamma ( \tau(x)) \left < (\tau(x)(u),D\zeta(x)
		\right > + \zeta (x) \left < F(x)(u), D \gamma (\tau(x))
		\right > \ud \| V \| x \\
		= - \tint{}{} \zeta (x) \gamma (\tau(x)) T (F(x)) \bullet u
		\ud \| V \| x.
	\end{multline*}
	In particular, taking $\gamma = 1$, one infers that $\| \delta V \|$ is
	a Radon measure absolutely continuous with respect to $\| V \|$ with
	\begin{gather*}
		T(F(x)) = \mathbf{h}(V,x) \quad \text{$\| V \|$ almost all
		$x$}.
	\end{gather*}
	It follows that
	\begin{gather*}
		( \delta V ) ( ( \gamma \circ \tau ) \zeta \cdot u ) = -
		\tint{}{} \zeta (x) \gamma (\tau(x)) T ( F(x) ) \bullet u \ud
		\| V \|x
	\end{gather*}
	for $\zeta \in \mathscr{D} ( U, \rel)$, $\gamma \in \mathscr{E} ( Y,
	\rel )$ and $u \in \rel^\adim$. Together with the first equation this
	implies that $\tau \in \trunc(V,Y)$ with $\derivative{V}{\tau} (x) =
	F(x)$ for $\| V \|$ almost all $x$ by
	\ref{remark:eq_condition_weak_diff}.
	
	To prove the converse, suppose $\| \delta V \|$ is a Radon measure
	absolutely continuous with respect to $\| V \|$ and $\tau \in
	\trunc(V, Y)$. Since $\tau$ is a bounded function,
	$\derivative{V}{\tau} \in \Lploc{1} \big ( \| V \|, \Hom ( \rel^\adim,
	Y ) \big )$ by
	\ref{def:v_weakly_diff}\,\eqref{item:v_weakly_diff:int}. In order to
	prove the equation for the generalised mean curvature vector of $V$,
	in view of \ref{thm:approx_diff} and \cite[4.8]{snulmenn.c2}, it is
	sufficient to prove that
	\begin{gather*}
		\mathbf{h} (M,x) = T \big ( ( \| V \|, \vdim ) \ap D\tau(x)
		\circ \tau(x) \big ) \quad \text{for $\| V \|$ almost all $x
		\in U \cap M$}
	\end{gather*}
	whenever $M$ is an $\vdim$ dimensional submanifold of $\rel^\adim$ of
	class $2$. The latter equation however is evident from
	\ref{lemma:approx_diff}\,\eqref{item:approx_diff:comp} and
	\ref{miniremark:mean_curv} since
	\begin{gather*}
		\Tan (M,x) = \Tan^\vdim ( \| V \|, x ) \quad \text{for $\| V
		\|$ almost all $x \in U \cap M$}
	\end{gather*}
	by \cite[2.8.18, 2.9.11, 3.2.17]{MR41:1976} and Allard
	\cite[3.5\,(2)]{MR0307015}.

	Next, define $f : X \to \rel^\adim \times Y$ by
	\begin{gather*}
		f(x) = (x,\tau(x)) \quad \text{for $x \in \dmn \tau$}
	\end{gather*}
	and notice that $f \in \trunc ( V, \rel^\adim \times Y )$ with
	\begin{gather*}
		\derivative{V}{f} (x) (u) = ( \tau(x)(u),
		\derivative{V}{\tau}(x)(u)) \quad \text{for $u \in
		\rel^\adim$}
	\end{gather*}
	for $\| V \|$ almost all $x$ by
	\ref{thm:addition}\,\eqref{item:addition:join} and
	\ref{example:lipschitzian}. The proof may be concluded by establishing
	the last equation of the postscript. Using approximation
	and the fact that $\tau$ is a bounded function yields that is sufficient
	to consider $\phi \in \mathscr{D} ( U \times Y, \rel )$. Choose $\zeta
	\in \mathscr{D} (U, \rel)$ with
	\begin{gather*}
		\Clos ( U \cap \{ x \with \text{$\phi (x,\sigma) \neq 0$ for
		some $\sigma \in Y$} \} ) \subset \Int \{ x \with \zeta (x) =
		1 \}
	\end{gather*}
	and denote by $\gamma$ the extension of $\phi$ to $\rel^\adim \times
	Y$ by $0$. Now, applying \ref{remark:eq_condition_weak_diff} with $Y$
	replaced by $\rel^\adim \times Y$ yields the equation in question.
\end{proof}
\begin{remark}
	The first paragraph of the proof is essentially contained in
	Hutchinson in \cite[5.2.2, 5.2.3]{MR825628} and is included here for
	completeness.
\end{remark}
\begin{remark}
	If $\vdim = \adim$ then $\density^\vdim ( \| V \|, \cdot )$ is a
	locally constant, integer valued function whose domain is $U$; in
	fact, $\tau$ is constant, hence $V$ is stationary by
	\ref{thm:curvature_varifolds} and the asserted structure follows from
	Allard \cite[4.6\,(3)]{MR0307015}.
\end{remark}
\begin{corollary} \label{corollary:curv_var_diff}
	Suppose $\vdim, \adim \in \nat$, $1 < \vdim < \adim$, $\beta =
	\vdim/(\vdim-1)$, $U$ is an open subset of $\rel^\adim$, $V \in
	\IVar_\vdim (U)$ is a curvature varifold, $X = U \cap \{ x \with
	\Tan^\vdim ( \| V \|, x ) \in \grass{\adim}{\vdim} \}$, and $\tau : X
	\to Y$ satisfies
	\begin{gather*}
		\tau(x) = \project{\Tan^\vdim ( \| V \|, x )} \quad
		\text{whenever $x \in X$}.
	\end{gather*}

	Then $\| V \|$ almost all $a$ satisfy
	\begin{gather*}
		\lim_{r \to 0+} r^{-\vdim} \tint{\cball ar}{} ( |
		\tau(x)-\tau(a)- \left <x-a, \derivative V\tau (a) \right > | /
		| x-a| )^\beta \ud \| V \| x = 0.
	\end{gather*}
\end{corollary}
\begin{proof}
	This is an immediate consequence of \ref{thm:curvature_varifolds} and
	\ref{thm:diff_lebesgue_spaces}\,\eqref{item:diff_lebesgue_spaces:m>1=p}.
\end{proof}
\begin{remark} \label{remark:curv_flatness_m=1}
	Using
	\ref{thm:diff_lebesgue_spaces}\,\eqref{item:diff_lebesgue_spaces:m=p=1},
	one may formulate a corresponding result for the case $\vdim = 1$.
\end{remark}
\begin{remark} \label{remark:curv_flatness}
	Notice that one may deduce decay results for height quantities
	from this result by use of \cite[4.11\,(1)]{snulmenn.poincare}.
\end{remark}
\begin{remark} \label{remark:curv_to_do}
	If $1 < p < \vdim$ and $\derivative{V}{\tau} \in \Lploc{p} \big ( \| V
	\| , \Hom ( \rel^\adim, Y ) \big )$, see \ref{miniremark:trace}, one
	may investigate whether the conclusion still holds with $\beta$
	replaced by $\vdim p/(\vdim-p)$.
\end{remark}

\addcontentsline{toc}{section}{\numberline{}References}
\bibliography{UlrichMenne3v3}\bibliographystyle{alpha}

\newcommand{\etalchar}[1]{$^{#1}$}
\newcommand{\noopsort}[1]{} \newcommand{\singleletter}[1]{#1} \def\cprime{$'$}
  \def\cprime{$'$}
\begin{thebibliography}{BBG{\etalchar{+}}95}

\bibitem[ADS96]{MR1441622}
G.~Anzellotti, S.~Delladio, and G.~Scianna.
\newblock B{V} functions over rectifiable currents.
\newblock {\em Ann. Mat. Pura Appl. (4)}, 170:257--296, 1996.

\bibitem[AF03]{MR2424078}
Robert~A. Adams and John J.~F. Fournier.
\newblock {\em Sobolev spaces}, volume 140 of {\em Pure and Applied Mathematics
  (Amsterdam)}.
\newblock Elsevier/Academic Press, Amsterdam, second edition, 2003.

\bibitem[AFP00]{MR2003a:49002}
Luigi Ambrosio, Nicola Fusco, and Diego Pallara.
\newblock {\em Functions of bounded variation and free discontinuity problems}.
\newblock Oxford Mathematical Monographs. The Clarendon Press Oxford University
  Press, New York, 2000.

\bibitem[All72]{MR0307015}
William~K. Allard.
\newblock On the first variation of a varifold.
\newblock {\em Ann. of Math. (2)}, 95:417--491, 1972.

\bibitem[Alm86]{MR855173}
F.~Almgren.
\newblock Optimal isoperimetric inequalities.
\newblock {\em Indiana Univ. Math. J.}, 35(3):451--547, 1986.

\bibitem[Alm00]{MR1777737}
Frederick~J. Almgren, Jr.
\newblock {\em Almgren's big regularity paper}, volume~1 of {\em World
  Scientific Monograph Series in Mathematics}.
\newblock World Scientific Publishing Co. Inc., River Edge, NJ, 2000.
\newblock $Q$-valued functions minimizing Dirichlet's integral and the
  regularity of area-minimizing rectifiable currents up to codimension 2, With
  a preface by Jean E.\ Taylor and Vladimir Scheffer.

\bibitem[AM03]{MR1959769}
Luigi Ambrosio and Simon Masnou.
\newblock A direct variational approach to a problem arising in image
  reconstruction.
\newblock {\em Interfaces Free Bound.}, 5(1):63--81, 2003.

\bibitem[BBG{\etalchar{+}}95]{MR1354907}
Philippe B{\'e}nilan, Lucio Boccardo, Thierry Gallou{\"e}t, Ron Gariepy, Michel
  Pierre, and Juan~Luis V{\'a}zquez.
\newblock An {$L^1$}-theory of existence and uniqueness of solutions of
  nonlinear elliptic equations.
\newblock {\em Ann. Scuola Norm. Sup. Pisa Cl. Sci. (4)}, 22(2):241--273, 1995.

\bibitem[BG72]{MR0308945}
E.~Bombieri and E.~Giusti.
\newblock Harnack's inequality for elliptic differential equations on minimal
  surfaces.
\newblock {\em Invent. Math.}, 15:24--46, 1972.

\bibitem[Bou87]{MR910295}
N.~Bourbaki.
\newblock {\em Topological vector spaces. {C}hapters 1--5}.
\newblock Elements of Mathematics (Berlin). Springer-Verlag, Berlin, 1987.
\newblock Translated from the French by H. G. Eggleston and S. Madan.

\bibitem[Bou98]{MR1726779}
Nicolas Bourbaki.
\newblock {\em General topology. {C}hapters 1--4}.
\newblock Elements of Mathematics (Berlin). Springer-Verlag, Berlin, 1998.
\newblock Translated from the French, Reprint of the 1989 English translation.

\bibitem[Bra78]{MR485012}
Kenneth~A. Brakke.
\newblock {\em The motion of a surface by its mean curvature}, volume~20 of
  {\em Mathematical Notes}.
\newblock Princeton University Press, Princeton, N.J., 1978.

\bibitem[Che99]{MR1708448}
J.~Cheeger.
\newblock Differentiability of {L}ipschitz functions on metric measure spaces.
\newblock {\em Geom. Funct. Anal.}, 9(3):428--517, 1999.

\bibitem[DGA88]{MR1152641}
Ennio De~Giorgi and Luigi Ambrosio.
\newblock New functionals in the calculus of variations.
\newblock {\em Atti Accad. Naz. Lincei Rend. Cl. Sci. Fis. Mat. Natur. (8)},
  82(2):199--210 (1989), 1988.

\bibitem[DS88]{MR90g:47001a}
Nelson Dunford and Jacob~T. Schwartz.
\newblock {\em Linear operators. {P}art {I}}.
\newblock Wiley Classics Library. John Wiley \& Sons Inc., New York, 1988.
\newblock General theory, With the assistance of William G. Bade and Robert G.
  Bartle, Reprint of the 1958 original, A Wiley-Interscience Publication.

\bibitem[DS93]{MR1132876}
Guy David and Stephen Semmes.
\newblock Quantitative rectifiability and {L}ipschitz mappings.
\newblock {\em Trans. Amer. Math. Soc.}, 337(2):855--889, 1993.

\bibitem[Fed69]{MR41:1976}
Herbert Federer.
\newblock {\em Geometric measure theory}.
\newblock Die Grundlehren der ma\-the\-ma\-ti\-schen Wissenschaften, Band 153.
  Springer-Verlag New York Inc., New York, 1969.

\bibitem[Fed70]{MR0260981}
Herbert Federer.
\newblock The singular sets of area minimizing rectifiable currents with
  codimension one and of area minimizing flat chains modulo two with arbitrary
  codimension.
\newblock {\em Bull. Amer. Math. Soc.}, 76:767--771, 1970.

\bibitem[Fed75]{MR0388226}
Herbert Federer.
\newblock A minimizing property of extremal submanifolds.
\newblock {\em Arch. Rational Mech. Anal.}, 59(3):207--217, 1975.

\bibitem[HT00]{MR1803974}
John~E. Hutchinson and Yoshihiro Tonegawa.
\newblock Convergence of phase interfaces in the van der
  {W}aals-{C}ahn-{H}illiard theory.
\newblock {\em Calc. Var. Partial Differential Equations}, 10(1):49--84, 2000.

\bibitem[Hut86]{MR825628}
John~E. Hutchinson.
\newblock Second fundamental form for varifolds and the existence of surfaces
  minimising curvature.
\newblock {\em Indiana Univ. Math. J.}, 35(1):45--71, 1986.

\bibitem[Hut90]{MR1066398}
John Hutchinson.
\newblock Poincar\'e-{S}obolev and related inequalities for submanifolds of
  {${\bf R}^N$}.
\newblock {\em Pacific J. Math.}, 145(1):59--69, 1990.

\bibitem[Kel75]{MR0370454}
John~L. Kelley.
\newblock {\em General topology}.
\newblock Springer-Verlag, New York, 1975.
\newblock Reprint of the 1955 edition [Van Nostrand, Toronto, Ont.], Graduate
  Texts in Mathematics, No. 27.

\bibitem[KS04]{MR2119722}
Ernst Kuwert and Reiner Sch{\"a}tzle.
\newblock Removability of point singularities of {W}illmore surfaces.
\newblock {\em Ann. of Math. (2)}, 160(1):315--357, 2004.

\bibitem[Men09]{snulmenn.isoperimetric}
Ulrich Menne.
\newblock Some applications of the isoperimetric inequality for integral
  varifolds.
\newblock {\em Adv. Calc. Var.}, 2(3):247--269, 2009.

\bibitem[Men10]{snulmenn.poincare}
Ulrich Menne.
\newblock A {S}obolev {P}oincar\'e type inequality for integral varifolds.
\newblock {\em Calc. Var. Partial Differential Equations}, 38(3-4):369--408,
  2010.

\bibitem[Men12a]{snulmenn.decay}
Ulrich Menne.
\newblock Decay estimates for the quadratic tilt-excess of integral varifolds.
\newblock {\em Arch. Ration. Mech. Anal.}, 204(1):1--83, 2012.

\bibitem[Men12b]{snulmenn.mfo1230}
Ulrich Menne.
\newblock A sharp lower bound on the mean curvature integral with critical
  power for integral varifolds, 2012.
\newblock In abstracts from the workshop held July 22--28, 2012, Organized by
  Camillo De Lellis, Gerhard Huisken and Robert Jerrard, Oberwolfach Reports.
  Vol.~9, no.~3.

\bibitem[Men13]{snulmenn.c2}
Ulrich Menne.
\newblock Second order rectifiability of integral varifolds of locally bounded
  first variation.
\newblock {\em J. Geom. Anal.}, 23:709--763, 2013.

\bibitem[ML98]{MR1712872}
Saunders Mac~Lane.
\newblock {\em Categories for the working mathematician}, volume~5 of {\em
  Graduate Texts in Mathematics}.
\newblock Springer-Verlag, New York, second edition, 1998.

\bibitem[Mon14]{MR3148123}
Andrea Mondino.
\newblock Existence of integral {$m$}-varifolds minimizing {$\int\vert A\vert
  ^p$} and {$\int\vert H\vert ^p,\,p>m,$} in {R}iemannian manifolds.
\newblock {\em Calc. Var. Partial Differential Equations}, 49(1-2):431--470,
  2014.

\bibitem[Mos01]{62659}
Roger Moser.
\newblock {A generalization of Rellich's theorem and regularity of varifolds
  minimizing curvature}.
\newblock {Preprint}, {Max Planck Institute for Mathematics in the Sciences},
  2001.

\bibitem[MS73]{MR0344978}
J.~H. Michael and L.~M. Simon.
\newblock Sobolev and mean-value inequalities on generalized submanifolds of
  {$R\sp{n}$}.
\newblock {\em Comm. Pure Appl. Math.}, 26:361--379, 1973.

\bibitem[Oss86]{MR852409}
Robert Osserman.
\newblock {\em A survey of minimal surfaces}.
\newblock Dover Publications, Inc., New York, second edition, 1986.

\bibitem[RT08]{MR2377408}
Matthias R{\"o}ger and Yoshihiro Tonegawa.
\newblock Convergence of phase-field approximations to the {G}ibbs-{T}homson
  law.
\newblock {\em Calc. Var. Partial Differential Equations}, 32(1):111--136,
  2008.

\bibitem[Sch66]{MR0209834}
Laurent Schwartz.
\newblock {\em Th\'eorie des distributions}.
\newblock Publications de l'Institut de Math\'ematique de l'Universit\'e de
  Strasbourg, No. IX-X. Nouvelle \'edition, enti\'erement corrig\'ee, refondue
  et augment\'ee. Hermann, Paris, 1966.

\bibitem[Sch73]{MR0426084}
Laurent Schwartz.
\newblock {\em Radon measures on arbitrary topological spaces and cylindrical
  measures}.
\newblock Published for the Tata Institute of Fundamental Research, Bombay by
  Oxford University Press, London, 1973.
\newblock Tata Institute of Fundamental Research Studies in Mathematics, No. 6.

\bibitem[Shi59]{MR0105613}
Tameharu Shirai.
\newblock Sur les topologies des espaces de {L}. {S}chwartz.
\newblock {\em Proc. Japan Acad.}, 35:31--36, 1959.

\bibitem[Sim83]{MR756417}
Leon Simon.
\newblock {\em Lectures on geometric measure theory}, volume~3 of {\em
  Proceedings of the Centre for Mathematical Analysis, Australian National
  University}.
\newblock Australian National University Centre for Mathematical Analysis,
  Canberra, 1983.

\bibitem[Sta66]{MR0251373}
Guido Stampacchia.
\newblock {\em \`{E}quations elliptiques du second ordre \`a coefficients
  discontinus}.
\newblock S\'eminaire de Math\'ematiques Sup\'erieures, No. 16 (\'Et\'e, 1965).
  Les Presses de l'Universit\'e de Montr\'eal, Montreal, Que., 1966.

\bibitem[Top08]{MR2410779}
Peter Topping.
\newblock Relating diameter and mean curvature for submanifolds of {E}uclidean
  space.
\newblock {\em Comment. Math. Helv.}, 83(3):539--546, 2008.

\bibitem[Wic14]{10.1007/s00526-013-0695-4}
Neshan Wickramasekera.
\newblock A sharp strong maximum principle and a sharp unique continuation
  theorem for singular minimal hypersurfaces.
\newblock {\em Calc. Var. Partial Differential Equations}, 51(3-4):799--812,
  2014.

\end{thebibliography}

\medskip

{\small \noindent Max Planck Institute for Gravitational Physics (Albert
Einstein Institute) \newline Am M{\"u}hlen\-berg 1, 14476 Potsdam, Germany
\newline \texttt{Ulrich.Menne@aei.mpg.de} \medskip \newline University of
Potsdam, Institute for Mathematics, \newline Am Neuen Palais 10, 14469
Potsdam, Germany \newline \texttt{Ulrich.Menne@math.uni-potsdam.de} }
\end{document}